\documentclass{amsart}

\usepackage{graphicx}
\usepackage{amsmath}
\usepackage{amsthm}
\usepackage{enumerate}
\usepackage{diagrams}
\usepackage{amssymb,bbm}%
\usepackage[numbers, square]{natbib}
\usepackage[font=small,labelfont=sc,tableposition=top]{caption}
\usepackage{ifpdf}

\ifpdf
\usepackage{hyperref}
\hypersetup{
colorlinks=false,
breaklinks=true,
pdfstartview={FitH},
unicode=true
}
\else
\usepackage{hyperref}
\fi

\makeatletter
\def\@tocline#1#2#3#4#5#6#7{\relax
  \ifnum #1>\c@tocdepth 
  \else
    \par \addpenalty\@secpenalty\addvspace{#2}%
    \begingroup \hyphenpenalty\@M
    \@ifempty{#4}{%
      \@tempdima\csname r@tocindent\number#1\endcsname\relax
    }{%
      \@tempdima#4\relax
    }%
    \parindent\z@ \leftskip#3\relax \advance\leftskip\@tempdima\relax
    \rightskip\@pnumwidth plus4em \parfillskip-\@pnumwidth
    #5\leavevmode\hskip-\@tempdima
      \ifcase #1
       \or\or \hskip 1em \or \hskip 2em \else \hskip 3em \fi%
      #6\nobreak\relax
    \dotfill\hbox to\@pnumwidth{\@tocpagenum{#7}}\par
    \nobreak
    \endgroup
  \fi}
\makeatother

\newcommand{\E}{\mathbb E}
\newcommand{\R}{\mathbb{R}}
\newcommand{\N}{\mathbb{N}}
\newcommand{\T}{\mathbb{T}}
\renewcommand{\C}{\mathbb{C}}
\newcommand{\Z}{\mathbb{Z}}

\newcommand{\bS}{\mathbf{S}}
\newcommand{\bZ}{\mathbf{Z}}
\newcommand{\bz}{\mathbf{z}}

\renewcommand{\P}{\mathbb{P}}
\newcommand{\XXX}{\mathbb{X}}
\newcommand{\SSS}{\mathbb{S}}

\newcommand{\ZZZ}{\mathcal{Z}}
\newcommand{\Var}{\mathop{\mathrm{Var}}\nolimits}
\newcommand{\supp}{\mathop{\mathrm{supp}}\nolimits}
\newcommand{\Cov}{\mathop{\mathrm{Cov}}\nolimits}

\newcommand{\sgn}{\mathop{\mathrm{sgn}}\nolimits}
\newcommand{\spec}{\mathop{\mathrm{spec}}\nolimits}

\newcommand{\calD}{\mathcal{D}}
\newcommand{\calF}{\mathcal{F}}
\newcommand{\calA}{\mathcal{A}}
\newcommand{\calO}{\mathcal{O}}

\newcommand{\MMM}{\mathcal{M}}
\newcommand{\NNN}{\mathcal{N}}
\newcommand{\HHH}{\mathcal{H}}

\newcommand{\nn}{\mathfrak{n}}

\newcommand{\eps}{\varepsilon}
\newcommand{\eqdistr}{\stackrel{d}{=}}
\newcommand{\todistr}{\overset{d}{\underset{n\to\infty}\longrightarrow}}
\newcommand{\toT}{\overset{}{\underset{T\to\infty}\longrightarrow}}
\newcommand{\todistrT}{\overset{d}{\underset{T\to\infty}\longrightarrow}}
\newcommand{\toweak}{\overset{w}{\underset{n\to\infty}\longrightarrow}}
\newcommand{\toprobab}{\overset{P}{\underset{n\to\infty}\longrightarrow}}

\newcommand{\toj}{\overset{}{\underset{j\to\infty}\longrightarrow}}

\newcommand{\bsl}{\backslash}
\newcommand{\ind}{\mathbbm{1}}
\renewcommand{\Re}{\operatorname{Re}}
\renewcommand{\Im}{\operatorname{Im}}
\newcommand{\dd}{{\rm d}}
\newcommand{\eee}{{\rm e}}
\newcommand{\Zeros}{\text{\textup{\textbf{Zeros}}}}
\DeclareMathOperator*{\plim}{\text{$P$-lim}}

\newcommand{\xxx}{}

\theoremstyle{plain}
\newtheorem{theorem}{Theorem}[section]
\newtheorem{lemma}[theorem]{Lemma}
\newtheorem{corollary}[theorem]{Corollary}
\newtheorem{proposition}[theorem]{Proposition}

\theoremstyle{definition}

\newtheorem{remark}[theorem]{Remark}

\theoremstyle{remark}

\numberwithin{equation}{section}

\begin{document}

\author{Zakhar Kabluchko}
\address{Zakhar Kabluchko\\ University of Ulm\\ Institute of Stochastics\\ Helmholtzstr.\ 18\\ 89069 Ulm\\ Germany}
\author{Anton Klimovsky}
\address{Anton Klimovsky\\AG Wahrscheinlichkeitstheorie\\Faculty of Mathematics\\University of Duisburg--Essen\\ D-45117 Essen \\Germany}
\title[Generalized Random
Energy Model]{Generalized Random Energy Model at Complex Temperatures}



\begin{abstract}
Motivated by the Lee--Yang approach to phase transitions, we study the partition
function of the Generalized Random Energy Model (GREM) at \textit{complex}
inverse temperature $\beta$. We compute the limiting log-partition function and describe
the fluctuations of the partition function. For the GREM with $d$ levels, in
total, there are $\frac 12 (d+1)(d+2)$ phases, each of which can symbolically be
encoded as $G^{d_1}F^{d_2}E^{d_3}$ with $d_1,d_2,d_3\in\mathbb{N}_0$ such that
$d_1+d_2+d_3=d$. In phase $G^{d_1}F^{d_2}E^{d_3}$, the first $d_1$ levels (counting from
the root of the GREM tree) are in the \textit{glassy} phase~(G), the next
$d_2$ levels are dominated by \textit{fluctuations}~(F), and the last $d_3$
levels are dominated by the \textit{expectation}~(E). Only the phases of the form
$G^{d_1}E^{d_3}$ intersect the real $\beta$ axis. We describe the limiting
distribution of the zeros of the partition function in the complex $\beta$ plane
(= Fisher zeros). It turns out that the complex zeros densely touch the positive real axis
at $d$ points at which the GREM is known to undergo phase transitions. Our
results confirm rigorously and considerably extend the replica-method
predictions from the physics literature.
\end{abstract}

\maketitle

\begin{center}
\includegraphics[width=0.85\textwidth]{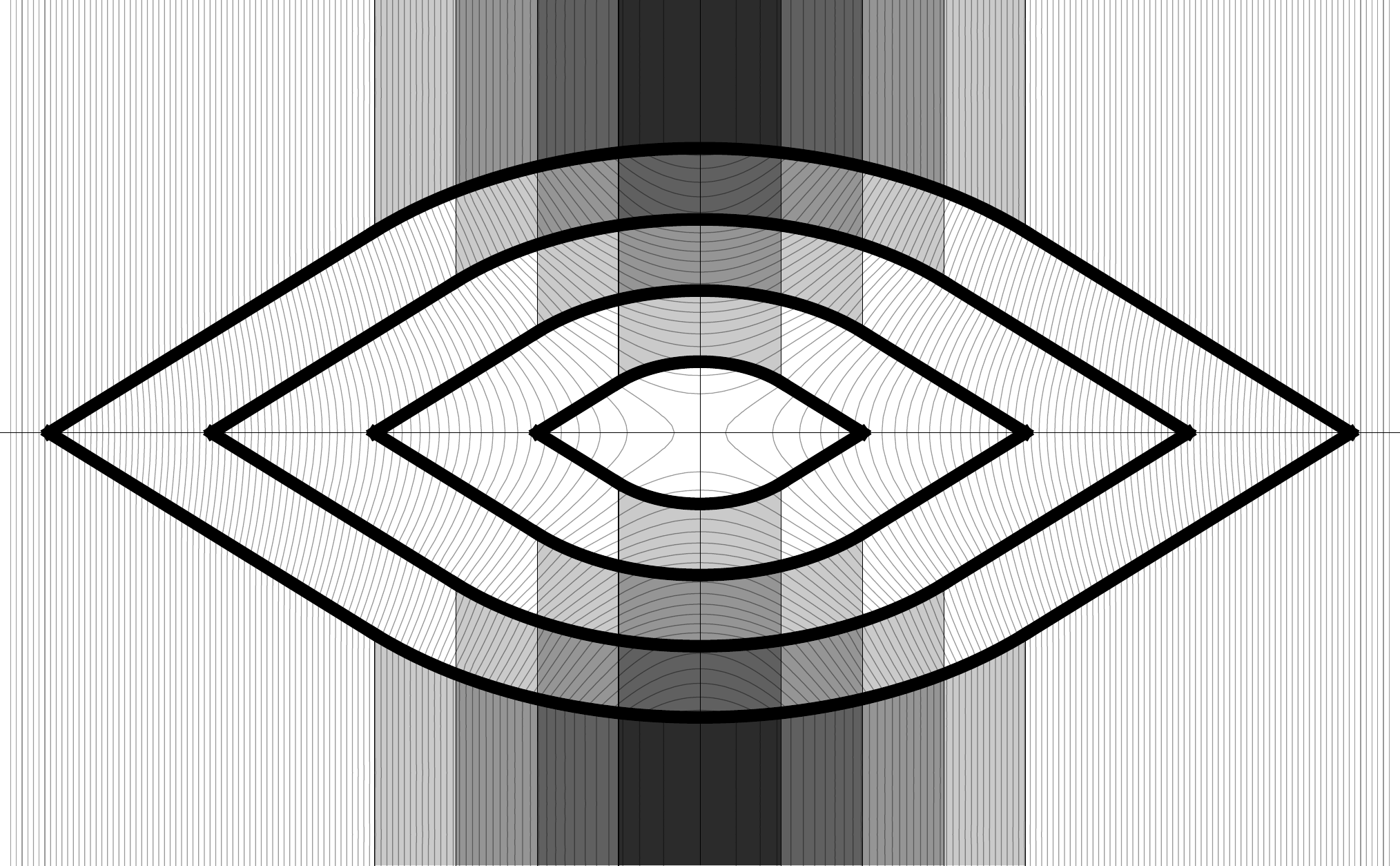}
\captionof{figure}
{\small Phase diagram of the GREM in the complex $\beta$ plane together with the level lines of the limiting log-partition function.
See Figure~\ref{fig:phases_grem_1} for details.}
\label{fig:phases_grem}
\end{center}

\newpage

\tableofcontents


\newpage

\section{Introduction and definition of the model}\label{sec:introduction}

\subsection{Introduction} The study of phase transitions is one of the central
topics in statistical physics. Phase transitions are usually defined as the
values of physical parameters (for example, the inverse temperature $\beta$) at
which the \textit{limiting} \textit{log-partition function} (or equivalently the
\textit{free energy}) is not real analytic (= non-analytic in any neighborhood
of the phase transition point). However, at any \textit{finite} system size, the
log-partition function is real analytic. In order to explain why the infinite
system log-partition function looses analyticity (while the finite system
log-partition function does not), Lee and Yang~\cite{lee_yang1,lee_yang2}
suggested to look at \textit{complex} values of the inverse temperature $\beta$.
At complex temperatures, the partition function may have \textit{zeros} and
hence, the log-partition function has singularities, even for finite system
sizes.  If in the infinite system limit these singularities accumulate around
the real axis at some $\beta_c \in \R$, then the limiting log-partition function
may loose analyticity at $\beta_c$, even though $\beta_c$ itself is never a
point of singularity of the log-partition function. Thus, the approach of Lee
and Yang relates phase transitions to the distribution of \textit{complex zeros}
of the partition function. The study of the complex zeros of the partition
function is usually referred to as the Lee--Yang program; see for
example~\cite{biskup_etal1,biskup_etal2}, where a large class of lattice spin
models is considered from this point of view.

The aim of the present work is to study a special model of spin glass, the
\textit{Generalized Random Energy Model} (GREM) within the
Lee--Yang program. The simplest model of a spin glass is the \textit{Random
Energy Model} (REM) introduced by~\citet{Derrida_REM3,Derrida_REM1}. In this
model, the energies of the system are assumed to be \textit{independent} Gaussian random variables. The behavior of the REM at real inverse
temperature is well understood; see~\citet{bovier_kurkova_loewe}
and~\citet[Chapter~9]{bovier_book}. For the REM at \textit{complex} inverse
temperature, \citet{Derrida_zeros} derived the limiting free energy, obtained
the phase diagram and computed the limiting distribution of complex zeros of the
partition function. The present authors refined Derrida's results and provided
rigorous proofs in \cite{kabluchko_klimovsky}.

Although the REM contains some of the physics of the spin glasses, e.g., it
displays the freezing phenomenon, the REM does not exhibit such phenomena as
multiple freezing transitions and chaos which are observed, e.g., in the
celebrated Sherrington--Kirkpatrick (SK) model of a spin glass. In order to
obtain a solvable model with multiple freezing transitions, Derrida introduced
the \textit{Generalized Random Energy Model}~(GREM);
see~\cite{Derrida_GREM,Derrida_Gardner1,Derrida_Gardner2}. Rigorous results on
the GREM  at real inverse temperatures were obtained by~\citet{capocaccia_etal}
and in a series of works by Bovier and
Kurkova~\cite{bovier_kurkova1,bovier_kurkova2,bovier_kurkova3}. For a review of
these results, we refer to~\citet{bovier_kurkova_review}
and~\citet[Chapter~10]{bovier_book}. We note in passing that the recent progress
in rigorous understanding of the SK model draws heavily on the analysis of the
GREM, see~\cite{PanchenkoBook2013} for a review.

In the theoretical physics literature, there is a strong interest in studying
spin glass models at complex temperatures. Besides the Lee--Yang program, the
motivation comes here from quantum physics and concretely from the studies of
interference in inhomogeneous media. See, e.g., the recent works of Takahashi
and Obuchi~\cite{obuchi_takahashi,takahashi,obuchi_takahashi_short},
\citet{saakian1,saakian2}, \citet{DobrinevskiLeDoussalWiese2011}. In particular,
\citet{takahashi}, developed a complex version of the (non-rigorous) replica
method and used it to identify the phase diagram of the GREM.


As for the rigorous works, beyond the uncorrelated case of the REM, to our
knowledge, only two models of disordered systems with correlated complex random
energies have been studied to some extent: the Branching Random Walk and the
Gaussian Multiplicative Chaos. See~\citet{derrida_evans_speer} and the recent
works of~\citet{lacoin_rhodes_vargas}
and~\citet{madaule_rhodes_vargas,madaule_rhodes_vargas1}. Both models have
correlations of logarithmic type and their complex-plane phase diagrams are
quite similar to that of the REM (see Section~\ref{sec:crem} for more details).

The GREM seems to be a natural candidate to be tackled next from the Lee--Yang
viewpoint. On the one hand, as we show below, the complex GREM is a rather
tractable model even at the level of fluctuations, and, on the other hand, it
exhibits multiple freezing phase transitions and has a much richer phase diagram
than that of the REM.

The main results of this paper can be summarized as follows:  
\begin{enumerate}

\item  we compute the limiting log-partition function $p(\beta): = \lim_{n\to\infty} \frac 1n \log |\ZZZ_n(\beta)|$;

\item  we describe the global limiting distribution of complex zeros of $\ZZZ_n(\beta)$;

\item  we identify the limiting fluctuations of $\ZZZ_n(\beta)$;

\item  we prove  functional limit theorems for $\ZZZ_n(\beta)$ in a suitably rescaled neighborhood of a fixed $\beta_*\in\C$;

\item  we describe the local limiting distribution of complex zeros of  $\ZZZ_n(\beta)$ in a suitably rescaled neighborhood of a fixed $\beta_*\in\C$.

\end{enumerate}
These results give the complete phase diagram of the GREM; see
Figures~\ref{fig:phases_grem} and \ref{fig:phases_grem_1}. Our results confirm
the replica-method predictions of~\citet{takahashi} and extend these
considerably. We also indicate how to pass to the limit of continuous
hierarchies (\textit{Continuous Random Energy Model}, CREM), see
Section~\ref{sec:crem}, which allows us to compare our results with the ones on
on the Branching Random Walk~\cite{madaule_rhodes_vargas} and the Gaussian
Multiplicative Chaos~\cite{lacoin_rhodes_vargas,madaule_rhodes_vargas1}. We
hope that our results shed more light on the complex plane phase diagrams and on fluctuations in strongly
correlated random energy models.

\subsection{Notation: Definition of the GREM}\label{subsec:intro}
We start by introducing the notation which will be used throughout the paper. Fix the following parameters:
\begin{enumerate}
\item the \textit{number of levels} $d\in \N$;
\item the \textit{variances} of the levels $a_1,\ldots,a_d>0$ (energetic parameters);
\item the \textit{branching exponents} $\alpha_1,\ldots,\alpha_d>1$ (entropic parameters).
\end{enumerate}
We also fix $d$ sequences $\{N_{n,1}\}_{n \in \N}, \ldots, \{N_{n,d}\}_{n \in \N}$ of natural numbers (called the \textit{branching
numbers}) such that for every $1\leq j\leq d$,
\begin{equation}\label{eq:asympt_N_nk}
\lim_{n\to\infty} \frac {N_{n,j}}{\alpha_j^n} = 1.
\end{equation}
The reader may simply take $N_{n,1}=[\alpha_1^n], \ldots, N_{n,d}=[\alpha_d^n]$.
Consider a \textit{rooted tree}, denoted by $\T_n$, which is constructed in the following
way. The root of the tree is located at level $1$ and is connected by edges to
$N_{n,1}$ vertices (descendants) at level $2$. Any vertex at level $2$ is
connected to $N_{n,2}$ vertices at level $3$, and so on. Finally, any vertex at
level $d$ is connected to $N_{n,d}$ terminal vertices (leaf nodes) which have no
descendants. We label the edges of the tree by $d$ levels $1,\ldots,d$ so that
the edges issuing from the root are at level $1$, whereas the leaf edges of the
tree are at level $d$. The set of paths in $\T_n$ connecting the root to the
terminal vertices is denoted by
\begin{equation}\label{eq:parameter_set_def}
\SSS_n=\{\eps=(\eps_1,\ldots,\eps_d)\in \N^d \colon 1\leq \eps_1 \leq N_{n,1}, \ldots, 1\leq \eps_d \leq N_{n,d}\}.
\end{equation}
The total number of elements in $\SSS_n$ and its growth exponent are given by
\begin{equation}\label{eq:def_Nn}
N_n := \# \SSS_n = N_{n,1} \cdot N_{n,2} \cdot \ldots \cdot N_{n,d}\sim \alpha^n, \;\;\; \alpha:=\alpha_1 \cdot \alpha_2 \cdot \ldots \cdot \alpha_d.
\end{equation}
Consider independent real standard normal random variables attached to the edges of the tree and denoted by
\begin{equation}
\{\xi_{\eps_1\ldots\eps_j}\colon 1\leq j\leq d, 1\leq \eps_1 \leq N_{n,1}, \ldots, 1\leq \eps_j \leq N_{n,j}\}.
\end{equation}
Define a zero-mean Gaussian random field $\xxx X = \{X_{\eps}\colon \eps\in \SSS_n\}$ by
\begin{equation}\label{eq:gaussian_field_def}
X_{\eps}= \sqrt {a_1}\, \xi_{\eps_1} + \sqrt {a_2}\, \xi_{\eps_1 \eps_2} + \ldots + \sqrt {a_d}\, \xi_{\eps_1\ldots\eps_d}.
\end{equation}
Note that the variance of this random field is constant:
\begin{equation}\label{eq:def_a}
a:=a_1+\ldots+a_d=\Var X_{\eps}, \;\;\; \eps\in \SSS_n.
\end{equation}
In the literature on the GREM, one usually assumes that the total number of energies in $\SSS_n$ is $N_n=2^n$ (so that $\alpha=\log 2$)  and that the variance is $a=1$. Since we will often use induction over the number of levels of the GREM, it is more convenient to us to consider the general case omitting these assumptions.

Let us write the complex inverse temperature $\beta$ in the form
$$
\beta=\sigma+i \tau\in \C, \;\;\; \sigma=\Re \beta \in \R, \;\;\; \tau=\Im \beta \in\R.
$$
The \textit{partition function} of the Generalized Random Energy Model at inverse temperature $\beta\in\C$ is defined by
\begin{equation}\label{eq:ZZZ_n_beta_def}
\ZZZ_n(\beta)=\sum_{\eps\in \SSS_n} \eee^{\beta \sqrt n X_{\eps}}.
\end{equation}
Define the \textit{critical inverse temperatures}
\begin{equation}\label{eq:beta_k_def}
\sigma_{j}=\sqrt{\frac{2\log \alpha_j}{a_j}}\in\R, \quad  1\leq j\leq d.
\end{equation}
To make the notation consistent, we make the convention $\sigma_0=0$ and
$\sigma_{d+1}=+\infty$.
Throughout the whole paper, we assume that
\begin{align}
\label{eq:convexity}
\sigma_1<\ldots<\sigma_d.
\end{align}
Geometrically, this condition means that
the broken line joining the points
$$
(a_1+\ldots+a_j, \log \alpha_1 + \ldots + \log \alpha_j)
,
\;\;\;
0\leq j\leq d,
$$
is strictly concave.  If \eqref{eq:convexity} is not satisfied, one has to coarse grain the GREM levels by replacing the above broken line by its concave hull; see
\cite{bovier_kurkova1} for details in the real $\beta$ case. If \eqref{eq:convexity} does not hold, there are less phase transition temperatures than $d$. In order to avoid complicated notation, we assume~\eqref{eq:convexity}.

Often, we can restrict ourselves to the quarter-plane $\sigma\geq 0$ and $\tau\geq 0$ because of the straightforward distributional equalities
\begin{align}
\{\ZZZ_n(-\beta)\colon \beta\in \C\} &\eqdistr \{\ZZZ_n(\beta)\colon \beta\in \C\}, \label{eq:symmetry1}\\
\{\ZZZ_n(\bar \beta)\colon \beta\in \C\} &\eqdistr \{\overline{\ZZZ_n(\beta)}\colon \beta\in \C\}. \label{eq:symmetry2}
\end{align}

\subsection{Notation: Spaces and modes of convergence}
In this section, we briefly recall several
notions of convergence which will be frequently used below.  For more
information,  we refer to the classical books~\cite{billingsley_book}
and~\cite{kallenberg_book}. The reader may skip this section and return to it when necessary.

Let $(D, \rho_D)$ be a locally compact metric space
with metric $\rho_D$. If not stated otherwise, all measures on $D$ are defined on the Borel $\sigma$-algebra generated by the metric $\rho_D$.

 \subsubsection*{Space of Radon measures.} A Radon measure on $D$ is a measure
$\mu$ on $D$  having the property that $\mu(K)<\infty$ for every compact set
$K\subset D$. Let $\MMM(D)$ be the set of all Radon measures on $D$. A sequence
of Radon measures $\mu_1,\mu_2,\ldots\in \MMM(D)$ converges \textit{vaguely} to
a Radon measure $\mu\in \MMM(D)$ if for every continuous compactly supported
function $f\colon D\to\R$ we have $\lim_{k\to\infty} \int_{D} f d\mu_k = \int_{D} f
d\mu$. Endowed with the topology of vague convergence, $\MMM(D)$ becomes a
Polish space. A \textit{random measure} on $D$ is a random variable defined on
some probability space $(\Omega, \calF, \P)$ and taking values in $\MMM(D)$.

 \subsubsection*{Space of integer-valued Radon measures.} Let $\NNN(D)$ be the
subset of $\MMM(D)$ consisting of all measures $\mu$ such that $\mu(K)\in \N_0$
for every compact set $K\subset D$. Measures with this property are called \textit{integer-valued.} Every measure $\mu\in\NNN(D)$ can be
represented as $\mu=\sum_{i\in I} \delta(x_i)$, where $\{x_i\}_{i\in I}$ is at
most countable collection of points in $D$ having no accumulation points in $D$.
Here, $\delta(x)$ is the Dirac delta-measure at $x\in D$.  It is well-known that
$\NNN(D)$ is a closed subset of $\MMM(D)$. We endow $\NNN(D)$ with the induced
vague topology. A \textit{point process} on $D$ is a random variable defined on some
probability space $(\Omega, \calF, \P)$ and taking values in $\NNN(D)$.

 \subsubsection*{Space of continuous functions.} Recall that $(D,\rho_D)$ is a locally
compact metric space with metric $\rho_D$.  Let $C(D)$ be the space of all (not
necessarily bounded) continuous complex-valued functions on $D$. A sequence of continuous functions on $D$ converges \textit{locally uniformly} if it converges uniformly on every compact set $K\subset D$.  Endowed with
the topology of locally uniform convergence, the space $C(D)$ becomes a
Polish space. A \textit{random continuous function} on $D$ is a random variable defined on some probability space $(\Omega, \calF, \P)$ and taking values in $C(D)$.

If $D$ is an open subset of $\C^d$,  let $\HHH(D)$ be
the set of all complex-valued functions which are analytic on $D$. Note that $\HHH(D)$ is a closed
linear subspace of $C(D)$. We endow $\HHH(D)$ with the topology of locally uniform convergence induced from $C(D)$. A \textit{random analytic function} on $D$ is a random variable defined on some probability space $(\Omega, \calF, \P)$ and taking values in $\HHH(D)$.

 \subsubsection*{Weak convergence.} Let $(M, \rho_M)$ be a metric space. A
sequence of random elements $Z_1,Z_2,\ldots$ taking values in $M$ converges
\textit{weakly} to a random element $Z$ with values in $M$ if for every
continuous, bounded function $f\colon M\to\R$, we have $\lim_{k\to\infty} \E f(Z_k)=\E
f(Z)$. In the case when $M$ is $\MMM(D)$, $\NNN(D)$, $C(D)$, or  $\HHH(D)$, we
speak of weak convergence of random measures, point processes, random continuous
functions, or random analytic functions, respectively.

 \subsubsection*{Zeros of analytic functions.} For an analytic function $f$
which is defined on some domain (=connected open set) $D\subset \C$ and does not vanish identically, we denote by
$\Zeros\{f(\beta) \colon \beta\in D\}\in \NNN(D)$ an integer-valued Radon measure on $D$ which counts the zeros of $f$ in $D$ according to their
multiplicities.

\subsubsection*{Real and complex Gaussian distribution.}
The \textit{real Gaussian distribution} $N_{\R}(0,\theta^2)$ with mean zero and variance $\theta^2>0$ has density
$$
\varphi_{\R}(t) = \frac {1}{\sqrt{2\pi} \theta} \eee^{-\frac{t^2}{2\theta^2}}, \;\;\; t\in\R,
$$
w.r.t.\ the Lebesgue measure on $\R$. The \textit{complex Gaussian distribution} $N_{\C}(0,\theta^2)$ with mean zero and variance $\theta^2>0$ has  density
$$
\varphi_{\C}(t) = \frac 1{\pi \theta^2} \eee^{-\frac{|t|^2}{\theta^2}}, \;\;\; t\in\C,
$$
w.r.t.\ the Lebesgue measure on $\C$. Note that $Z\sim N_{\C}(0,\theta^2)$ iff $Z=X+iY$, where $X,Y\sim
N_{\R}(0,\frac 12 \theta^2)$ are independent. A zero mean real or complex Gaussian distribution is called \textit{standard} if $\theta=1$.

\vspace*{2mm} Throughout the paper, $C, C_1, \ldots$ denote positive constants
whose values may change from line to line. Let $\R_+=(0,\infty)$. We write $a_n\sim b_n$ if $\lim_{n\to\infty} a_n/b_n = 1$.

\section{Statement of results}\label{sec:results}

\subsection{Limiting log-partition function}\label{subsec:free_energy}
\begin{figure}
\includegraphics[width=0.65\textwidth]{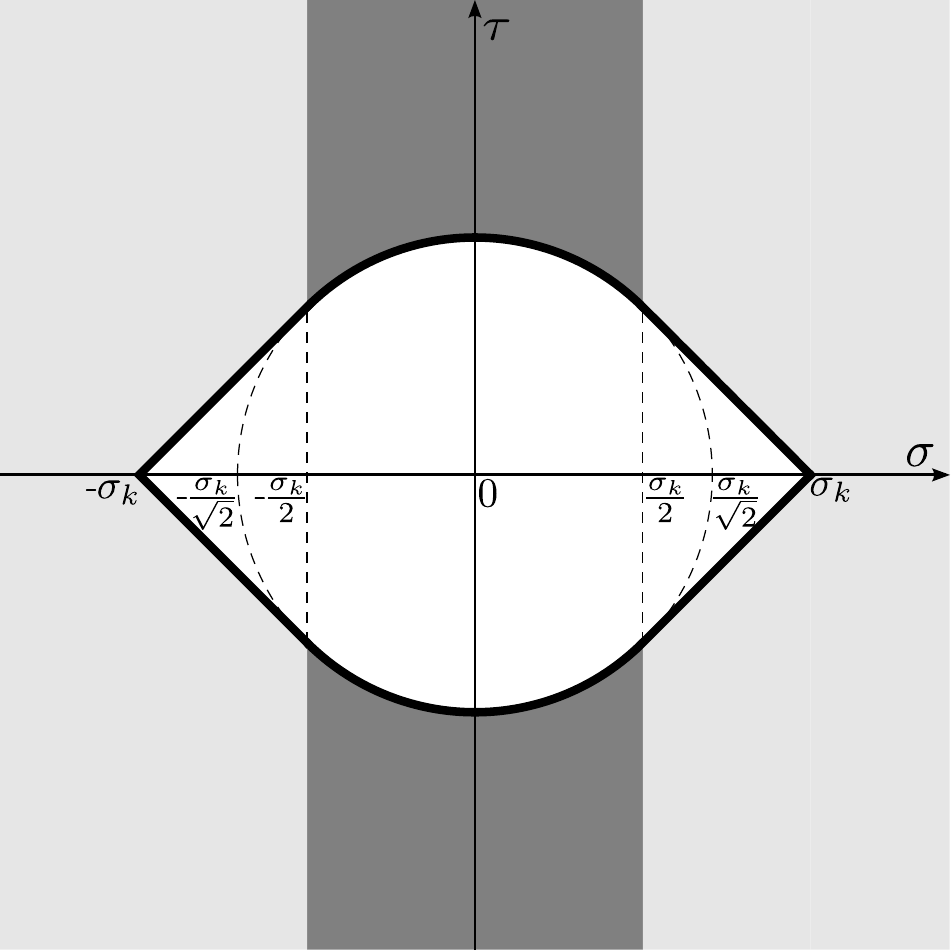}
 \caption
{\small  Complex $\beta$ phase diagram of the REM with the partition function $\ZZZ_n^{(k)}(\beta)$, see \citet{Derrida_zeros} and also~\cite{kabluchko_klimovsky}.}
\label{fig:phases_rem}
\end{figure}
In this section, we state a formula for the limiting log-partition function of the GREM. To
understand this formula heuristically, imagine a GREM with $d$ levels as a
``superposition'' of $d$ independent copies of the REM.
(Note that the random
field $X_{\eps}$ which generates the partition function of the GREM, cf.\
\eqref{eq:ZZZ_n_beta_def},  has strong correlations.)
Namely, with every level
$k=1,\ldots,d$ of the GREM  we can associate a REM whose partition function is
given by
\begin{equation}\label{eq:ZZZ_n_k_beta_def}
\xxx \ZZZ_n^{(k)} (\beta) = \sum_{j=1}^{N_{n,k}} \eee^{\beta \sqrt {n a_k} \eta^{(k)}_j}, \quad 1\leq k\leq d,
\end{equation}
where $\xxx \eta^{(k)}_1, \eta^{(k)}_2, \ldots, \eta^{(k)}_{N_{n,k}}$ are independent real standard normal random variables. The complex plane phase diagram of the REM has been described by~\citet{Derrida_zeros}; see also~\cite{kabluchko_klimovsky}. There are three phases, see Figure~\ref{fig:phases_rem}, which we will denote by
\begin{enumerate}
\item [(a)] $E_k$ (\textit{expectation} dominated phase),
\item [(b)] $F_k$ (\textit{fluctuations} dominated phase),
\item [(c)] $G_k$ (\textit{``glassy phase"} = extreme values dominated phase).
\end{enumerate}
Concretely, the phases are given by
\begin{align}
G_k
&=
\{\beta\in\C \colon 2|\sigma| >  \sigma_k, |\sigma|+|\tau| > \sigma_k\},
\label{eq:phases_G_k}\\
F_k
&=
\{\beta \in \C \colon 2|\sigma| < \sigma_k, 2 (\sigma^2+\tau^2) > \sigma_k^2\},
\label{eq:phases_F_k}\\
E_k
&= \C\bsl \overline{G_k\cup F_k},
\label{eq:phases_E_k}
\end{align}
where $\bar A$ is the closure of the set $A$. The phases $G_k$ and $E_k$ intersect the real axis, while the phase $F_k$ is special for the complex $\beta$ case. By definition, the sets $G_k$, $F_k$, $E_k$ are open.

\citet{Derrida_zeros}, see also~\cite{kabluchko_klimovsky} for a rigorous proof, computed the limiting log-partition function of the REM at complex $\beta$. Namely, for the log-partition function of the REM corresponding to the $k$-th level of the GREM,
\begin{equation}\label{eq:p_k_lim_ZZZ_n_k}
p_k(\beta) = \lim_{n\to\infty} \frac 1n \log |\ZZZ_n^{(k)}(\beta)|,
\end{equation}
Derrida's formula takes the form
\begin{equation} \label{eq:def_mk}
p_k(\beta) =
\begin{cases}
|\sigma| \sqrt{2 a_k\log \alpha_k}, & \text{if } \beta\in \bar G_k,\\
\frac 12 \log \alpha_k + a_k \sigma^2,  & \text{if } \beta\in \bar F_k,\\
\log \alpha_k + \frac 12 a_k (\sigma^2-\tau^2),  & \text{if } \beta\in \bar E_k.
\end{cases}
\end{equation}
It is easy to check that the function $p_k$ is continuous and strictly positive.

The next result shows that the limiting log-partition function of the GREM can be computed as the sum of the log-partition functions of the REM's corresponding to the $d$ levels of the GREM.
\begin{theorem}\label{theo:free_energy}
For every $\beta\in\C$, the following limit exists in probability and in $L^q$, for all $q\geq 1$:
\begin{equation}\label{eq:free_energy}
p(\beta):=\lim_{n\to\infty} \frac 1 n \log |\ZZZ_n(\beta)| = \sum_{k=1}^d p_k(\beta),
\end{equation}
where $p_k(\beta)$, the contribution of the $k$-th level, is given by~\eqref{eq:def_mk}.
\end{theorem}
\begin{remark} Restricting~\eqref{eq:free_energy} and~\eqref{eq:def_mk} to the real temperature case $\beta\geq 0$, we obtain, for $\beta \in [\sigma_m, \sigma_{m+1})$ with $0\leq m\leq d$,
$$
p(\beta) = \sigma \sum_{k=1}^{m} \sqrt{2 a_k\log \alpha_k} + \sum_{k=m+1}^{d}\left(\log \alpha_k + \frac 12 a_k (\sigma^2-\tau^2)\right).
$$
Thus, we recovered the previously known formula obtained in~\cite{capocaccia_etal}; see
also~\cite{Derrida_Gardner1,bovier_kurkova1,bovier_kurkova_review}.
\end{remark}

\begin{figure}
\begin{tabular}{cc}
\includegraphics[width=0.49\textwidth]{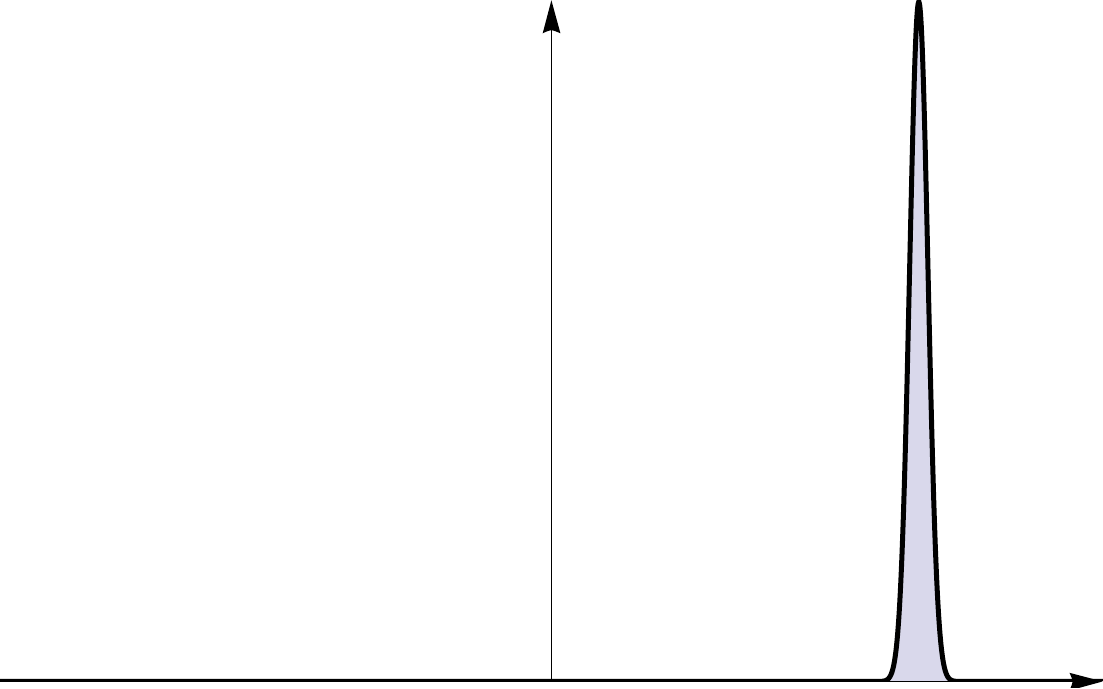}
&
\includegraphics[width=0.49\textwidth]{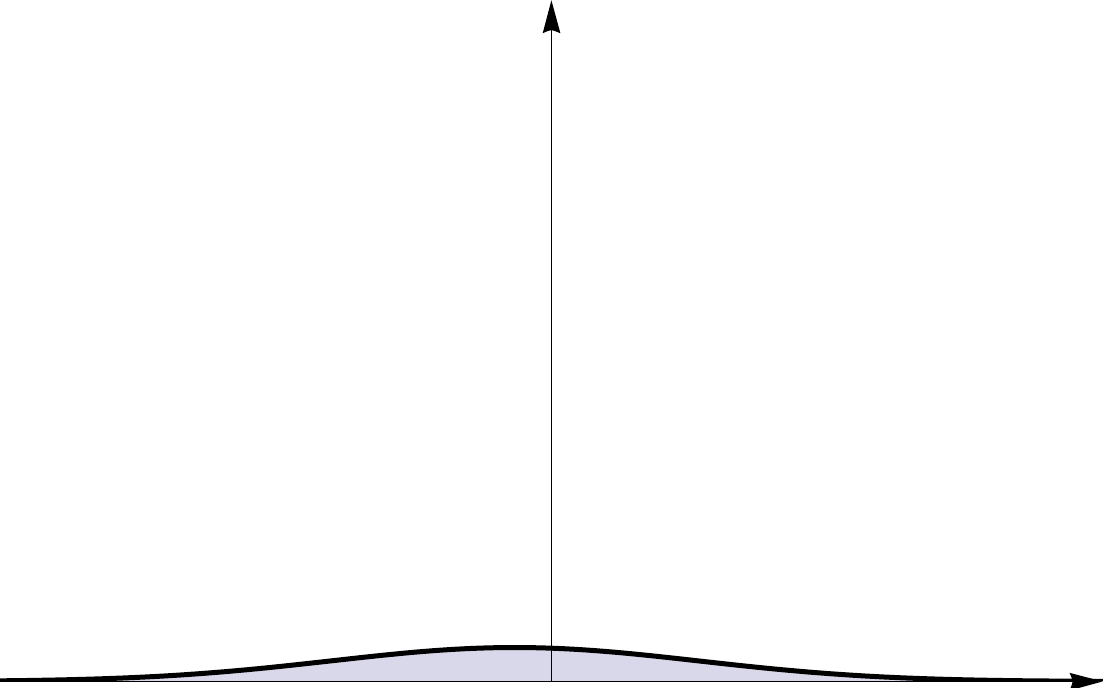}
\\
{\small
Expectation dominates.} &
{\small Fluctuations dominate.}
\end{tabular}
\caption{\small Caricatures of the probability density of $\ZZZ^{(k)}_n(\beta)$ in the regimes with light tails.}
\label{fig:location_vs_fluctuation}
\end{figure}

\subsection{Heuristics}
The reader may find the following heuristics useful.
There are three natural guesses on the asymptotic behavior of $\ZZZ^{(k)}_n(\beta)$:
\begin{enumerate}
\item [(a)] \textit{expectation dominates:} $\ZZZ^{(k)}_n(\beta)$ behaves approximately as its \textit{expectation}; see Figure~\ref{fig:location_vs_fluctuation}, left.  This guess turns out to be  correct in phase $E_k$.
\end{enumerate}
However, it can happen that the fluctuations of $\ZZZ_n^{(k)}(\beta)$ around its expectation are of larger order than the expectation. In this case, we end up in the following regime:
\begin{enumerate}
\item[(b)] \textit{fluctuations dominate:} $\ZZZ^{(k)}_n(\beta)$ behaves approximately as its \textit{standard deviation}; see Figure~\ref{fig:location_vs_fluctuation}, right. This guess turns out to be correct in phase $F_k$.
\end{enumerate}
Still, it can happen that due to the presence of heavy tails neither the expectation nor the standard deviation are adequate to estimate the true magnitude of the  partition function. In this case, one can make the following guess:
\begin{enumerate}
\item[(c)] \textit{extremes dominate:} $\ZZZ^{(k)}_n(\beta)$ behaves approximately as the \textit{maximal summand} in~\eqref{eq:ZZZ_n_k_beta_def}. This guess turns out to be correct in phase $G_k$.
\end{enumerate}
Summarizing, we arrive at the following three guesses for the limiting log-partition function $p_k(\beta) = \lim_{n \to \infty} \frac 1n \log |\ZZZ_n^{(k)}(\beta)|$:
\begin{align}
\text{Expectation} \quad p_k(\beta)&=\lim_{n\to\infty} \frac 1n \log \left|\E \ZZZ_n^{(k)}(\beta)\right| = \log \alpha_k + \frac 12 a_k (\sigma^2-\tau^2),\\
\text{Fluctuations} \quad p_k(\beta)&=\lim_{n\to\infty} \frac 1n \log \sqrt {\Var \ZZZ_n^{(k)}(\beta)} = \frac 12 \log \alpha_k + a_k \sigma^2,\\
\text{Extremes}  \quad p_k(\beta)&=\lim_{n\to\infty} \frac 1n \log \max_{j=1,\ldots,N_{n,k}} \left|\eee^{\beta \sqrt{na_k} \xxx \eta^{(k)}_j}\right| =|\sigma| \sqrt{2 a_k\log \alpha_k}.
\end{align}
It turns out that these formulae indeed give the correct value of $p_k(\beta)$ in phases $E_k$, $F_k$, $G_k$, respectively.

\subsection{Global limiting distribution of complex zeros}\label{subsec:zeros_global}
Using Theorem~\ref{theo:free_energy}, it is possible to obtain the limiting distribution of complex zeros of the GREM partition function $\ZZZ_n(\beta)$.

For $\ZZZ_n^{(k)}(\beta)$, the partition function of the REM corresponding to the $k$-th level of the GREM, the limiting distribution of zeros has been computed by~\citet{Derrida_zeros}; see also~\cite{kabluchko_klimovsky} for a rigorous proof.
The main idea is to use the Poincar\'e--Lelong formula (see, e.g.,~\cite[\S2.4.1]{peres_etal_book}). It states that the measure counting the complex zeros of any analytic function $f$ (which is not everywhere $0$) can be represented as
\begin{equation}\label{eq:poincare_lelong}
\Zeros\{f(\beta)\colon \beta\in \C\} = \frac 1{2\pi} \Delta \log |f(\beta)|.
\end{equation}
Here, $\Delta=\frac{\partial^2}{\partial \sigma^2} + \frac{\partial^2}{\partial
\tau^2}$ is the Laplace operator in the complex $\beta$-plane. The Laplace
operator should be understood in the sense of generalized functions
(distributions). Applying this formula to $f(\beta)=\ZZZ_n^{(k)}(\beta)$,
dividing by $n$, interchanging the large $n$ limit and the Laplacian (which
should be justified), and using~\eqref{eq:p_k_lim_ZZZ_n_k}, one can show that
weakly on $\MMM(\C)$,
$$
\frac{1}{n} \Zeros \{\ZZZ_n^{(k)}(\beta)\colon \beta\in\C\}
\toweak
\frac{1}{2\pi} \Delta p_k.
$$
The distributional Laplacian of $p_k$ (see, e.g.,\ Section~\ref{sec:proof_theo:zeros_global} for the details of the computation), is a measure $\Xi_k$ on $\C$ given by
\begin{equation}\label{eq:Xi_k}
\Xi_k: = \Delta p_k = \Xi_k^{F}+\Xi_{k}^{EF}+\Xi_{k}^{EG},
\end{equation}
where  $\Xi_k^{F}$, $\Xi_{k}^{EF}$, $\Xi_{k}^{EG}$ are measures on the complex plane defined as follows:
\begin{enumerate}
\item [(a)] $\Xi_{k}^F$ is $2a_k$ times the two-dimensional Lebesgue
measure restricted to $F_k$.
\item [(b)] $\Xi_{k}^{EF}$ is $\sqrt{a_k \log \alpha_k}$ times the one-dimensional
length measure on the boundary between $E_k$ and $F_k$ (which consists of two
circular arcs).
\item [(c)] $\Xi^{EG}_k$ is a measure having the density $\sqrt
2 a_k |\tau|$ with respect to the one-dimensional length measure restricted to the
boundary between  $E_k$ and $G_k$ (which consists of four line segments).
\end{enumerate}
Thus, the zeros of $\ZZZ_n^{(k)}(\beta)$ fill the \textit{two-dimensional}
region $F_k$ asymptotically uniformly with density $2a_kn$, but some zeros
concentrate around the boundary of $E_k$ with \textit{one-dimensional} density
asymptotically proportional to $n$.  The term $\Xi_k^F$ is just the pointwise
Laplacian of $p_k$, whereas the terms $\Xi_{k}^{EF}$ and $\Xi_{k}^{EG}$ appear
because the normal derivative of the function $p_k$ has a jump discontinuity  on
the boundary of the phase $E_k$. On the boundary between $F_k$ and $G_k$, the
normal derivative of $p_k$ is continuous, hence this boundary makes no
one-dimensional contribution to $\Xi$.

We now proceed to the complex zeros of $\ZZZ_n(\beta)$, the partition function of the GREM. In view of Theorem~\ref{theo:free_energy}, it is not surprising that the limiting distribution of zeros of $\ZZZ_n(\beta)$ can be obtained as a \textit{superposition} of the limiting zeros distributions of the corresponding REM's.
\begin{theorem}\label{theo:zeros_global}
The following convergence of random measures holds weakly on the space $\MMM(\C)$:
\begin{equation}\label{eq:theo:zeros_global}
\frac{1}{n} \Zeros \{\ZZZ_n(\beta)\colon \beta\in\C\}
\toweak
\frac{1}{2\pi} \Xi, 
\end{equation}
where $\Xi=\Delta p = \sum_{k=1}^d \Xi_k$.
\end{theorem}



\subsection{Phase diagram}
\begin{figure}
\includegraphics[width=\textwidth]{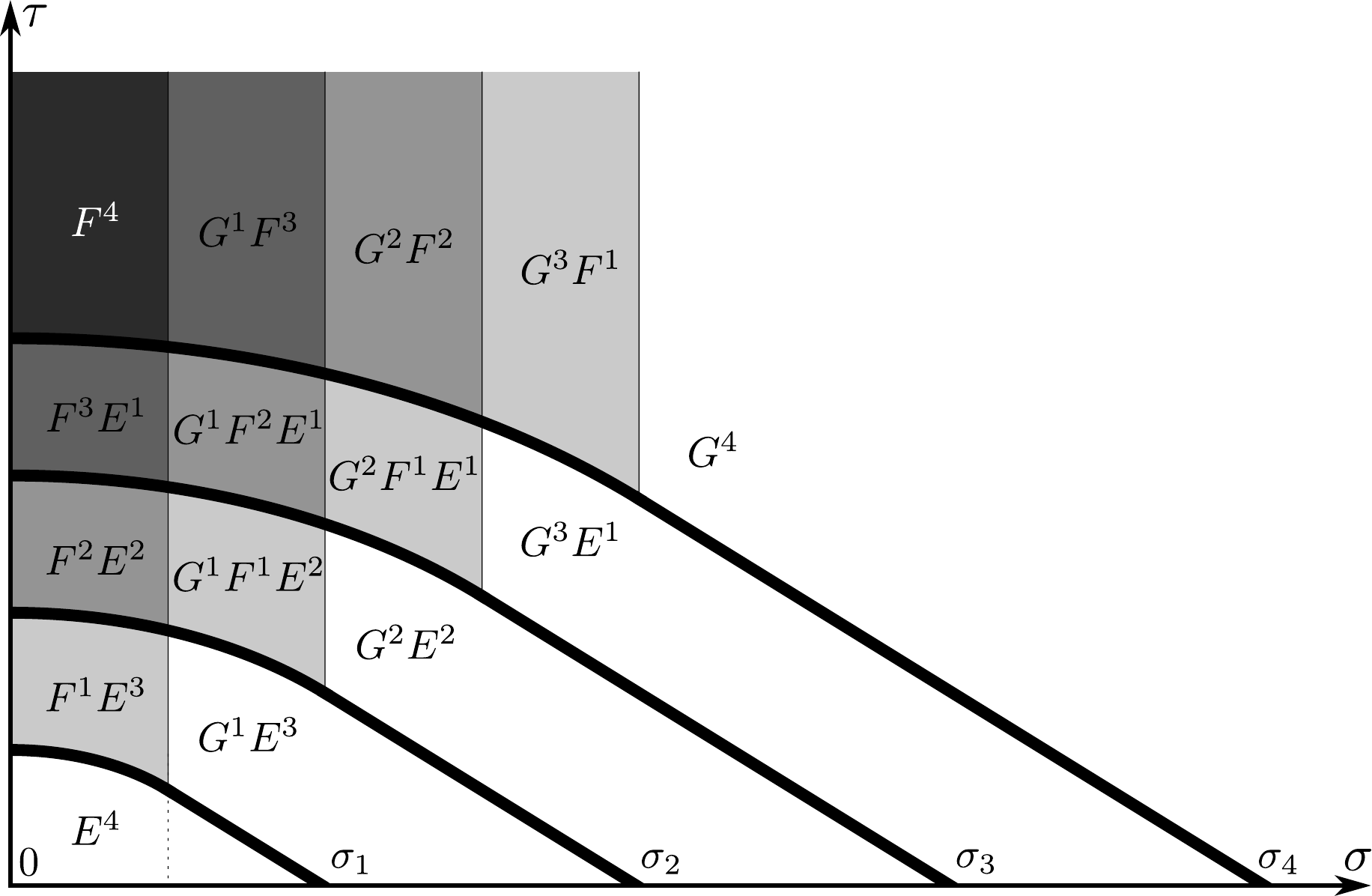}
 \caption
{\small Phase diagram of a GREM with $d=4$ levels in the complex $\beta$ plane. Only the quarter-plane $\sigma\geq 0$, $\tau\geq 0$ is shown. Darker regions have larger density of partition function zeros.}
\label{fig:phases_grem_1}
\end{figure}
We can now describe the phase diagram of the GREM in the complex $\beta$ plane; see Figure~\ref{fig:phases_grem_1}. It is obtained  as a superposition of the phase diagrams of the corresponding REM's. Take some $\beta\in\C$. For every $k=1,\ldots,d$, we can determine the phase ($G_k$, $F_k$, or $E_k$) to which $\beta$ belongs and write the result in form of a sequence of length $d$ over the alphabet $\{G,F,E\}$. However, it is easy to see that only phases of the following form are possible:
$$
G^{d_1}F^{d_2}E^{d_3} = \underbrace{G\ldots G}_{d_1} \underbrace{F\ldots F}_{d_2} \underbrace{E \ldots E}_{d_3},
$$
where $d_1,d_2,d_3\in \{0,\ldots,d\}$ are such that $d_1+d_2+d_3=d$.
In other words,  we have an ordering of the level phases which can be
symbolically expressed as
$$
G \succ F \succ E.
$$
For example, it is not possible that a level in $E$-phase is followed by a level in $F$- or in $G$-phase. This follows from the fact that if $\beta\in E_k$ for some $k$, then $\beta\notin F_l$ and $\beta\notin G_l$ for $l\geq k$.  This ordering of phases agrees with the observation
of~\citet{saakian1}.
The phases of the GREM are therefore given by
\begin{equation*}
G^{d_1}F^{d_2}E^{d_3} =  (G_1 \cap \ldots \cap G_{d_1})\cap (F_{d_1+1}\cap \ldots \cap F_{d_1+d_2}) \cap  (E_{d_1+d_2+1}\cap \ldots \cap E_d),
\end{equation*}
where $d_1,d_2,d_3\in\{0,\ldots,d\}$ are such that $d_1+d_2+d_3=d$. If $\beta\in G^{d_1}F^{d_2}E^{d_3}$, then we say that the levels $1,\ldots,d_1$ are in the $G$-phase, the levels $d_1+1,\ldots,d_1+d_2$ are in the $F$-phase, and the levels $d_1+d_2+1,\ldots,d$ are in the $E$-phase. Note that each $G^{d_1}F^{d_2}E^{d_3}$ is an open subset of the complex plane. The union of the closures of these sets is the entire complex plane.
The total number of phases is $\frac 12
(d+1)(d+2)$.  Only $d+1$ of these phases, namely those of the form
$G^{d_1} E^{d_3}$, intersect the real axis.



\subsection{Central limit theorem in the strip $|\sigma| < \frac {\sigma_1}2$}
In this and subsequent sections, we identify the limiting fluctuations of the partition function $\ZZZ_n(\beta)$. We can view $\ZZZ_n(\beta)$ as a sum of random variables in a triangular summation scheme. Although these random variables are dependent (unless $d=1$),  the limiting distribution of their sum $\ZZZ_n(\beta)$ is infinitely divisible, as we shall see.  It is well known that an infinitely divisible distribution can be decomposed into a superposition of a Gaussian and a Poissonian component. In this section, we consider the case in which only the Gaussian component is present.
The next result states that in the strip $|\sigma| < \frac {\sigma_1}2$ the partition function $\ZZZ_n(\beta)$ satisfies a central limit theorem.
\begin{theorem}\label{theo:clt}
Let $\beta=\sigma+i\tau\in\C\bsl\{0\}$ be such that $|\sigma|<\frac {\sigma_1}2$. Then,
\begin{equation}\label{eq:CLT}
\frac{\ZZZ_n(\beta)-\E \ZZZ_n(\beta)}{\sqrt{\Var \ZZZ_n(\beta)}} \todistr
\begin{cases}
N_{\C}(0,1), & \text{if } \tau\neq 0,\\
N_{\R}(0,1), & \text{if } \tau = 0.
\end{cases}
\end{equation}
\end{theorem}
To draw corollaries from Theorem~\ref{theo:clt}, we need to obtain expressions for  $\E \ZZZ_n(\beta)$ and $\Var \ZZZ_n(\beta)$. Recall that $a=a_1+\ldots+a_d$ denotes the variance of $X_{\eps}$, $\eps\in \SSS_n$, and $\alpha=\alpha_1\cdot \ldots \cdot \alpha_d$. Recall the convention that  $\sigma_{d+1} = +\infty$.
\begin{proposition}\label{prop:asympt_expect}
For every $\beta\in \C$, $\E \ZZZ_n(\beta) = N_n \eee^{\frac 12 \beta^2 a n}$.
\end{proposition}
\begin{proof}
If $X\sim N_{\R}(0,\theta^2)$ is real normal random variable with mean zero and variance $\theta^2$, then
$\E e^{t X} = e^{\frac 12 \theta^2 t^2}$, $t\in\C$.
Since every Gaussian random variable $X_{\eps}$ in~\eqref{eq:ZZZ_n_beta_def} has variance $a$, we  immediately obtain the required formula.
\end{proof}
Next, we establish an asymptotic formula for $\Var \ZZZ_n(\beta)$, as
$n\to\infty$.  The asymptotic behavior of the variance displays several regimes
(see Figure~\ref{fig:clt_simple}, left) which are separated by the circles
$$
|\beta|=\frac{\sigma_k}{\sqrt 2}, \quad 1\leq k\leq d.
$$
\begin{figure}
\begin{tabular*}{\textwidth}{p{0.5\textwidth}p{0.5\textwidth}}
\includegraphics[width=0.5\textwidth]{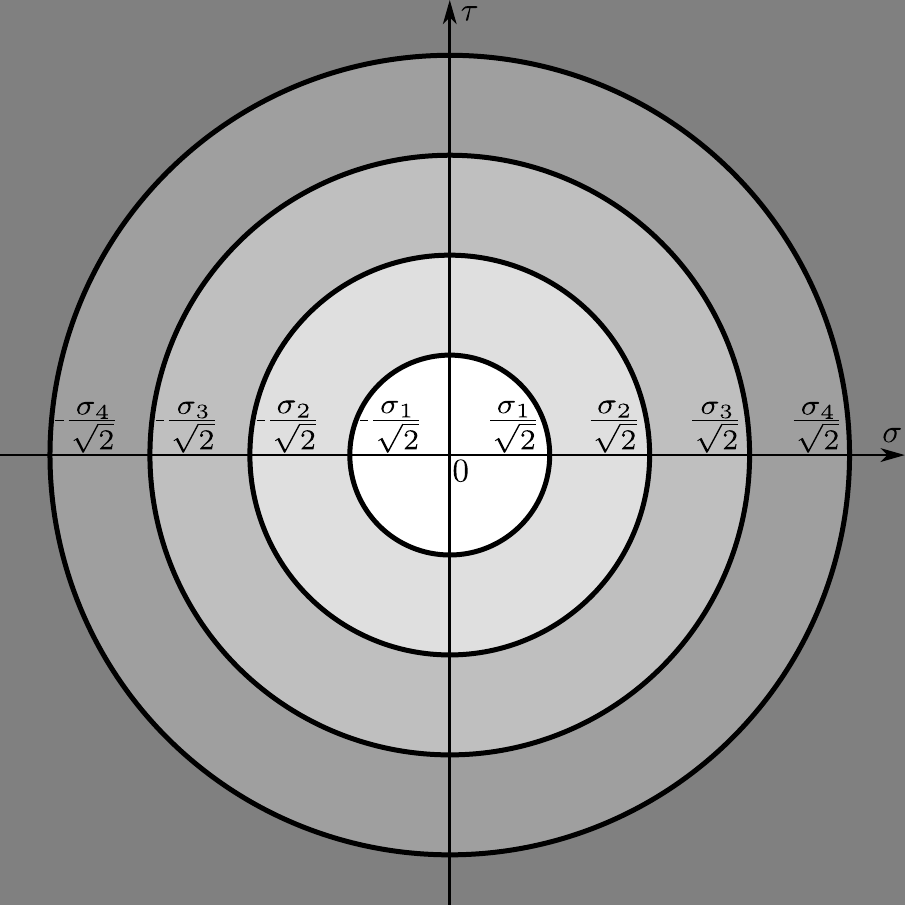}
&
\includegraphics[width=0.5\textwidth]{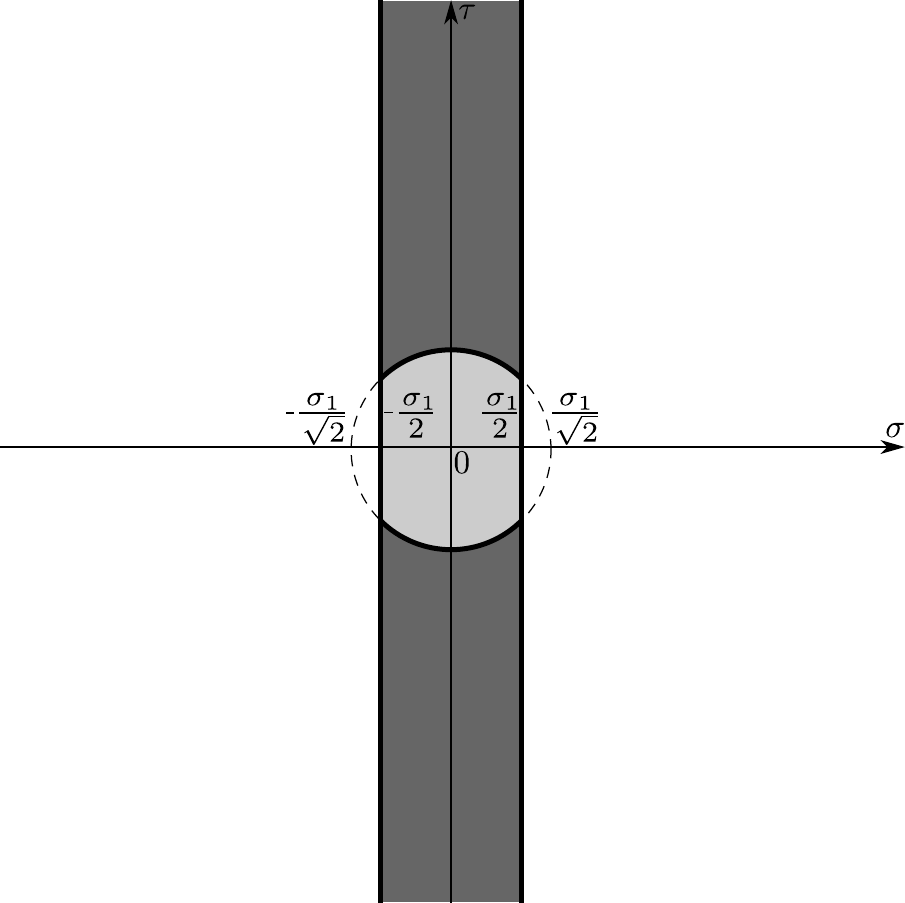}
\\
{\small
Regimes of the asymptotic behavior of $\Var \ZZZ_n(\beta)$; see Proposition~\ref{prop:asympt_exp_variance_log_scale}. Darker regions have stronger local correlations of $\ZZZ_n(\beta)$; see Section~\ref{sec:cov}}. &
{\small Two cases in the central limit theorem for $\ZZZ_n(\beta)$. See Propositions~\ref{prop:clt_simple_1} and~\ref{prop:clt_simple_2}.}
\end{tabular*}
 \caption
{Variance and CLT.}
\label{fig:clt_simple}
\end{figure}
\begin{proposition}\label{prop:asympt_exp_variance_log_scale}
Let $\beta\in \C$ be arbitrary. For $0\leq k\leq d$, write
\begin{equation*}
b_k=\log \alpha +  2\sigma^2 a +  \sum \limits_{m=k+1}^{d} (\log \alpha_m - |\beta|^2 a_m).
\end{equation*}
Then,
\begin{equation*}
\Var \ZZZ_n(\beta)
\sim
\begin{cases}
\eee^{b_k n},
&\text{if } \frac{\sigma_k} {\sqrt 2} < |\beta| < \frac{\sigma_{k+1}} {\sqrt 2}, \;\; 1\leq k\leq d,\\
\eee^{b_1 n},
&\text{if } 0 <  |\beta| < \frac{\sigma_1} {\sqrt 2},\\
2 \eee^{b_k n},
&\text{if }  |\beta| = \frac{\sigma_{k}} {\sqrt 2},\;\; 2\leq k\leq d,\\
\eee^{b_1 n},
&\text{if }  |\beta| = \frac{\sigma_1} {\sqrt 2}.
\end{cases}
\end{equation*}
In the first two cases, the formula holds locally uniformly as long as $\beta$ stays in the specified region.
\end{proposition}



As an immediate corollary of Proposition~\ref{prop:asympt_exp_variance_log_scale}, we obtain the following result comparing the expectation and the standard deviation of $\ZZZ_n(\beta)$.
\begin{proposition}\label{prop:exp_vs_stand_dev}
For any $\beta\in\C\bsl \{0\}$,
$$
\lim_{n\to\infty}\frac{|\E \ZZZ_n(\beta)|}{\sqrt{\Var \ZZZ_n(\beta)}}
=
\begin{cases}
\infty, &\text{ for } |\beta| < \frac{\sigma_1}{\sqrt 2},\\
1, &\text{ for } |\beta| = \frac{\sigma_1}{\sqrt 2},\\
0, &\text{ for } |\beta| > \frac{\sigma_1}{\sqrt 2}.\\
\end{cases}
$$
\end{proposition}
Depending on which quantity, the expectation or the standard deviation, has larger order of magnitude, we can  derive from Theorem~\ref{theo:clt} the following two corollaries. The corresponding domains are shown in  Figure~\ref{fig:clt_simple}, right. 
\begin{proposition}\label{prop:clt_simple_1}
If  $|\beta| > \frac{\sigma_1}{{\sqrt 2}}$ and $|\sigma|<\frac
{\sigma_1}2$ (which means that $\beta\in F^{d_2}E^{d_3}$ with $d_2> 0$), then we can
drop the expectation in~\eqref{eq:CLT}:
$$
\frac{\ZZZ_n(\beta)}{\sqrt{\Var \ZZZ_n(\beta)}} \todistr
N_{\C}(0,1).
$$
\end{proposition}
\begin{proposition}\label{prop:clt_simple_2}
If  $|\beta| < \frac{\sigma_1}{{\sqrt 2}}$ and $|\sigma|<\frac
{\sigma_1}2$ (which implies but is not equivalent to  $\beta\in E^d$), then
$$
\frac{\ZZZ_n(\beta)}{\E \ZZZ_n(\beta)} \todistr 1.
$$
\end{proposition}
If $\beta\in (-\frac {\sigma_1}2, +\frac {\sigma_1}2)$ is real, then the result of Proposition~\ref{prop:clt_simple_2} is contained in~\cite[Theorem~1.7]{bovier_kurkova1}. Theorem~\ref{theo:clt} (which is stronger than Proposition~\ref{prop:clt_simple_2}) seems to be new even in the case $\beta \in \R$.

\subsection{Central limit theorem for $|\sigma| = \frac {\sigma_1}2$}
We will show that on the boundary of the strip, i.e.\ for  $|\sigma| = \frac {\sigma_1}2$, the central limit theorem still holds, but with a non-standard limiting  variance.    In order to have the right  ``resolution'' on the boundary, let us assume that  $\sigma=\sigma(n)$ depends on $n$ in such a way that for some constant $u\in\R$,
\begin{equation}\label{eq:sigma_boundary_CLT}
\sigma(n) = \frac{\sigma_1}{2} - \frac{u}{2\sqrt{na_1}} + o\left(\frac 1 {\sqrt n} \right).
\end{equation}
\begin{theorem}\label{theo:clt_boundary}
Let $\beta=\beta(n)=\sigma(n)+i\tau$ be such that $\tau\in\R$ is constant and $\sigma=\sigma(n)$ satisfies~\eqref{eq:sigma_boundary_CLT}.
Then,
\begin{equation}\label{eq:CLT_boundary}
\frac{\ZZZ_n(\beta)-\E \ZZZ_n(\beta)}{\sqrt{\Var \ZZZ_n(\beta)}} \todistr
\begin{cases}
N_{\C}(0,\Phi(u)), & \text{if } \tau \neq  0,\\
N_{\R}(0,\Phi(u)), & \text{if } \tau = 0.
\end{cases}
\end{equation}
Here, $\Phi(u)$ is the standard normal distribution function.
\end{theorem}
In particular, if $\sigma=\frac {\sigma_1}{2}$ does not depend on $n$, then
$u=0$ and the variance of the limiting distribution is $\frac 12$. For the case
of the REM and real $\beta$, this fact was discovered
in~\cite{bovier_kurkova_loewe}. See also~\cite{kabluchko_FCLT_geom_BM} for a
version with a fine ``resolution'' as in~\eqref{eq:sigma_boundary_CLT}.  For the
case of the REM and complex $\beta$, see~\cite{kabluchko_klimovsky}.  In the
case of the GREM, Theorem~\ref{eq:CLT_boundary} is new even in the real $\beta$
case. The appearance of the ``truncated variance'' in~\eqref{eq:CLT_boundary}
can be explained as follows. For $\sigma<\frac {\sigma_1}{2}$, the limiting
distribution is Gaussian, whereas it turns out that for $\sigma>\frac
{\sigma_1}{2}$ the first level of the GREM  contributes only to the Poissonian
component of the limiting distribution.  In the boundary case, some energies at
the first level of the GREM have left the Gaussian part, but have not arrived
yet at the Poissonian part. This is why the variance of the limiting Gaussian
distribution is smaller than $1$ in the boundary case.

\subsection{Poisson cascade zeta function} \label{subsec:zeta_P}
The fluctuations of $\ZZZ_n(\beta)$ in phases of the form $G^{d_1} F^{d_2} E^{d_3}$
with $d_1>0$ will be described using a random zeta function associated to the Poisson
cascades. In this section, we define this function and state results on its
meromorphic continuation.

Let $P_1,P_2,\ldots$ be the points of a unit intensity Poisson point process on $(0,\infty)$. The points are always arranged in an increasing order.  The \textit{Poisson process zeta function} is defined by
$$
\zeta_P(z) = \sum_{k=1}^{\infty} P_k^{-z}, \;\;\; \Re z > 1.
$$
With probability $1$, the above series converges absolutely and uniformly on compact subsets of the half-plane $\{\Re z > 1\}$ since $\lim_{k\to\infty} P_k/k=1$ a.s.\ by the law of large numbers. However, with probability $1$, the function $\zeta_P$ admits a meromorphic continuation to the half-plane $\{\Re z> 1/2 \}$.
Namely, by~\cite[Theorem~2.6]{kabluchko_klimovsky}, with probability $1$, we have
\begin{equation}\label{eq:zeta_P_anal_cont_d_1}
\sum_{P_k\leq T} P_k^{-z} - \int_1^T t^{-z}\dd t  \toT  \zeta_P(z) - \frac{1}{z-1}
\text{ on }  \HHH(\{\Re z> 1/2\}).
\end{equation}

We will need a  multivariate generalization of the
Poisson process zeta function which will be called the \textit{Poisson cascade zeta function}.
First, we need to define the Poisson cascade point processes; see
Figure~\ref{fig:poi2D}. These and related point processes appeared for example in~\cite{bovier_kurkova1},
\cite{ruelle_cascades}. Fix dimension $d\in\N$. Start with a unit intensity
Poisson point process $\sum_{i=1}^{\infty} \delta(P_{i})$ on $(0,\infty)$. Then,
for every $m=1,\ldots,d-1$ and every $\eps_1,\ldots,\eps_m\in\N$ let
$\sum_{i=1}^{\infty} \delta(P_{\eps_1\ldots\eps_m i})$ be a unit intensity
Poisson point process on $(0,\infty)$. Assume that all point processes introduced
above are independent. Consider the following point process  $\Pi$ on
$(0,\infty)^d$,
\begin{equation}\label{eq:def_Pi}
\Pi = \sum_{\eps=(\eps_1,\ldots,\eps_d)\in \N^d} \delta(P_{\eps_1}, P_{\eps_1\eps_2}, \ldots, P_{\eps_1 \ldots \eps_d}).
\end{equation}
Of course, $\Pi$ is not a Poisson process (unless $d=1$) since $\Pi$ contains
infinitely many collinear point with probability $1$. The next lemma  states
that $\Pi$ has the same first order intensity as the homogeneous Poisson process on
$(0,\infty)^d$. It can easily be proven by induction over $d$.

\begin{figure}
\begin{tabular*}{\textwidth}{p{0.5\textwidth}p{0.5\textwidth}}
\includegraphics[width=0.5\textwidth]{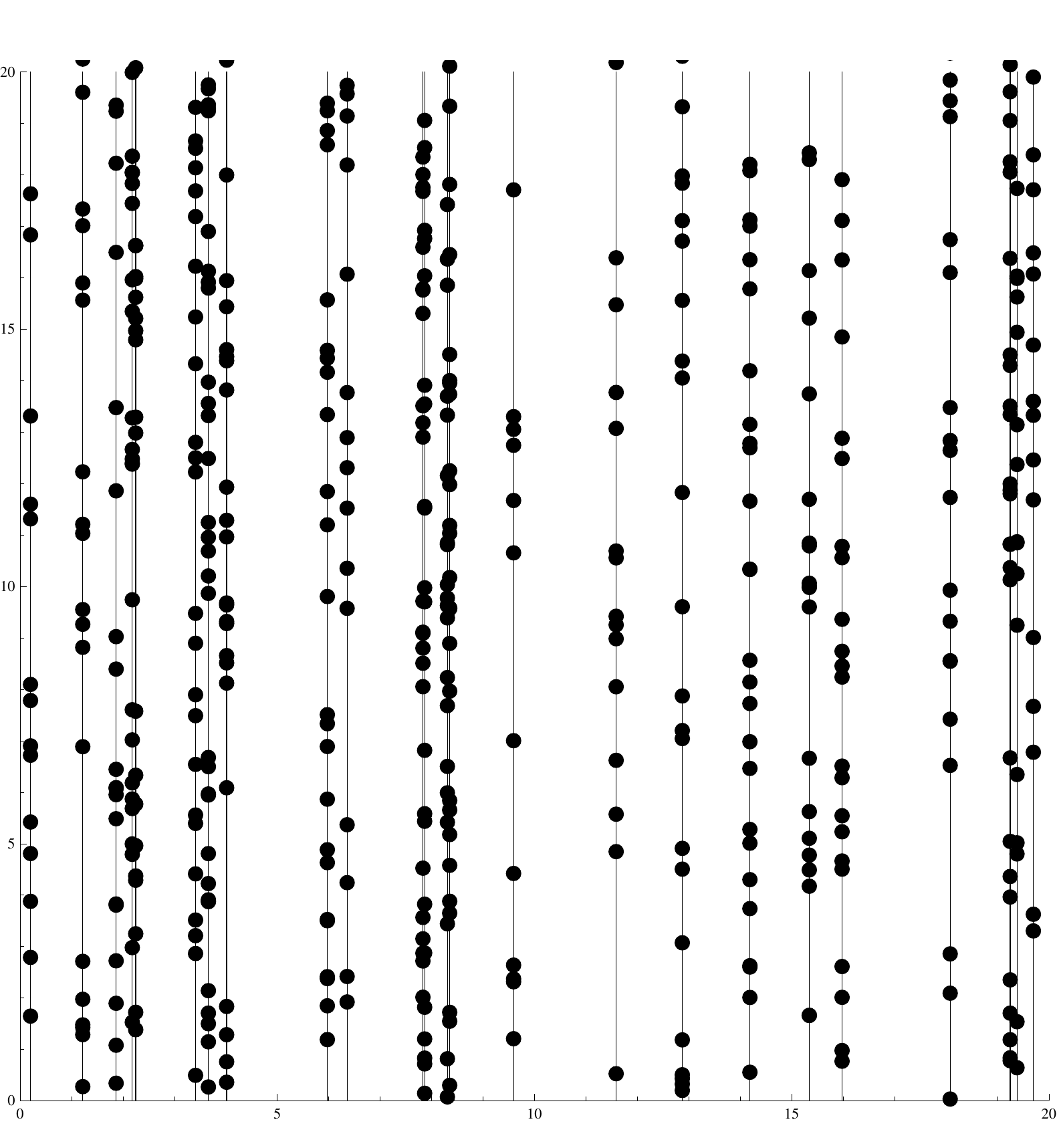}
&
\includegraphics[width=0.49\textwidth]{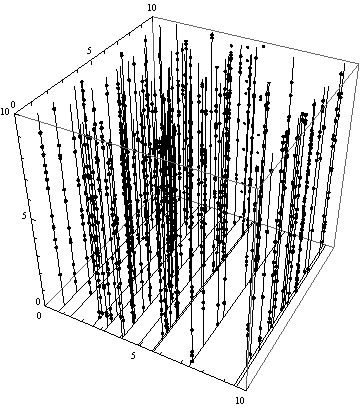}
\\
\begin{center}
{\small
$d=2$ levels.
}
\end{center}
&
\begin{center}
{\small $d=3$ levels.}
\end{center}
\end{tabular*}
\caption{\small Poisson cascade point process.}
\label{fig:poi2D}
\end{figure}


\begin{lemma}\label{lem:poisson_cascade_moments}
Let $\varphi$ be an integrable or non-negative function on $(0,\infty)^d$. Then,
$$
\E \left[\sum_{x\in \Pi} \varphi(x)\right] = \int_{(0,\infty)^d} \varphi(x)\dd x.
$$
\end{lemma}

The random zeta function $\zeta_P$ associated to the Poisson cascade point process $\Pi$ is a stochastic process defined by the series
\begin{equation}\label{eq:zeta_P_def}
\zeta_P(z_1,\ldots,z_d)=\sum_{\eps\in \N^d} P_{\eps_1}^{-z_1}P_{\eps_1\eps_2}^{-z_2}\ldots P_{\eps_1\ldots\eps_d}^{-z_d}.
\end{equation}
\begin{theorem}\label{theo:zeta_abs_conv}
With probability $1$, the series~\eqref{eq:zeta_P_def} converges absolutely and uniformly on any compact subset of the domain
\begin{equation}\label{eq:def_calD}
\calD = \{(z_1,\ldots,z_d)\in\C^d \colon \Re z_1>\ldots > \Re z_d>1\}.
\end{equation}
In particular, the function $\zeta_P$ is analytic on $\calD$ with probability $1$.
\end{theorem}

Theorem~\ref{theo:zeta_abs_conv}  would be sufficient to treat the GREM at real inverse
temperature $\beta$, as in~\cite{bovier_kurkova1}. However, for complex $\beta$,
we need a meromorphic continuation of $\zeta_P$ to a larger domain.
\begin{theorem}\label{theo:zeta_mero_cont}
With probability $1$, the function
$
\zeta_P(z_1,\ldots,z_d)
$
defined originally on $\calD$ admits a meromorphic continuation to the domain
$$
\frac 12 \calD = \{(z_1,\ldots,z_d)\in\C^d\colon \Re z_1>\ldots > \Re z_d>1/2\}.
$$
Moreover, the function $(z_d-1)\zeta_P(z_1,\ldots,z_d)$ is analytic on $\frac 12
\calD$ with probability $1$.
\end{theorem}
We conjecture that with probability $1$ there is no meromorphic continuation beyond
$\frac 12 \calD$. In the sequel, we use the notation $z=(z_1,\ldots,z_d)\in \C^d$.

\begin{remark}
The value of $(z_d-1)\zeta_P(z)$ in the case $z_d=1$ is understood by
continuity. In the case $d=1$, this value is equal to $1$, whereas, for $d\geq
2$, it is a non-degenerate random variable. (The non-degeneracy follows from the
fact that a degenerate random variable cannot
satisfy~\eqref{eq:zeta_P_operator_stable}, see below, with $\Re z_1>z_d=1$).
\end{remark}

\begin{proposition}\label{prop:zeta_P_operator_stable}
Consider $m\in \N$ independent copies of the random analytic function $\{(z_d-1)\zeta_P(z)\colon z\in\frac 12 \calD\}$ denoted by
$
\{(z_d-1)\zeta_P^{(j)}(z)\colon z\in \frac 12 \calD\}$, $1\leq j\leq m$.
Then, the following distributional equality on $\HHH(\frac 12 \calD)$ holds:
\begin{equation}\label{eq:zeta_P_operator_stable}
\left\{ \sum_{j=1}^m (z_d-1) \zeta_P^{(j)}(z) \colon z\in \frac 12 \calD\right\} \eqdistr \left\{m^{z_1} (z_d-1)\zeta_P(z) \colon z\in \frac 12 \calD\right\}.
\end{equation}
\end{proposition}
From Proposition~\ref{prop:zeta_P_operator_stable}, we can draw several
conclusions about the finite-dimensional distributions of $\zeta_P$. If $z\in
\frac 12 \calD\cap \R^d$, then the distribution of the real-valued random
variable $(z_d-1)\zeta_P(z)$ is \textit{stable} with exponent $1/z_1$;
see~\cite[Chapter 1]{samorodnitsky_taqqu_book}. In fact, it is even
\textit{strictly stable} meaning that no additive constant is needed
in~\eqref{eq:zeta_P_operator_stable}. If $z\in\frac 12 \calD$ is such that
$z_1\in\R$ (but $z_2,\ldots,z_d$ are not necessarily real), then the term
$m^{z_1}$ is real and hence, $(z_d-1)\zeta_P(z)$ (which is considered as a
random vector with values in $\C\equiv \R^2$) has a two-dimensional stable
distribution (which need not be isotropic); see~\cite[Chapter
2]{samorodnitsky_taqqu_book}. In general, for $z\in\frac 12 \calD$ without any
additional assumptions on the components, the distribution of the random
variable $(z_d-1)\zeta_P(z)$ (again considered as a random vector with values in
$\C\equiv \R^2$)  is \textit{strictly complex stable} in the sense
of~\citet{hudson_veeh}.  A random variable with values in $\C$ is called
strictly complex stable, see~\cite{hudson_veeh}, if for every $m\in\N$ the sum
of $m$ independent copies of this random variable, after dividing it by an
appropriate complex number, has the same law as the original random variable.
More generally, all finite-dimensional distributions of the stochastic process
$\{(z_d-1) \zeta_P(z)\colon z\in\frac 12 \calD\}$  are \textit{strictly operator
stable} (and hence, infinitely divisible). Recall that a random vector with
values in $\R^k$ is called strictly operator stable, if for every $m\in\N$ the
sum of $m$ copies of this random vector, after applying to it an appropriate
linear transformation of $\R^k$, has the same law as the original random vector;
see~\cite[Definition 3.3.24]{meerschaert_book}. The same conclusions apply to
the random variable $\zeta_P(z)$ and the stochastic process $\{\zeta_P(z)\colon
z\in\frac 12 \calD\}$ if we additionally assume that $z_d\neq 1$.


The following property of the moments of $\zeta_P(z)$ will be deduced in Section~\ref{subsec:proof_operator_stable_moments} from the operator stability.
\begin{proposition}\label{prop:zeta_P_moments}
Let $0<p<2$ and $z\in \frac 12  \calD$.
\begin{enumerate}

\item[\textup{(1)}] If $\Re z_1<\frac 1p$, then $\E |(z_d-1)\zeta_P(z)|^p<\infty$.

\item[\textup{(2)}] If $\Re z_1>\frac 1p$, then $\E |(z_d-1)\zeta_P(z)|^p=\infty$ (unless $d=1$ and $z=1$).

\end{enumerate}
\end{proposition}

\subsection{Fluctuations of the partition function}
\label{subsec:fluctutations} First, we need to introduce several normalizing
sequences. For each $1\leq k\leq d$, let $\{u_{n,k}\}_{n\in\N}$ be a real
sequence satisfying
\begin{equation}\label{eq:u_n_k_tail}
N_{n,k} \sim \sqrt{2\pi} u_{n,k} \eee^{\frac 12 u_{n,k}^2}  , \quad n\to\infty.
\end{equation}
Equivalently, we can choose
\begin{equation}\label{eq:u_n_k_asympt}
u_{n,k}=\sqrt{2 \log N_{n,k}} - \frac{\log(4\pi \log N_{n,k})+o(1)}{2\sqrt{2 \log N_{n,k}}} \sim \sqrt{2 n \log \alpha_k} = \sigma_k\sqrt{na_k}.
\end{equation}
It is well known, see~\cite[Theorem~1.5.3]{leadbetter_book}, that if $\eta_1,\eta_2,\ldots$ are independent real standard Gaussian random variables, then
$$
u_{n,k} \left(\max_{i=1,\ldots, N_{n,k}} \eta_i -u_{n,k}\right) \todistr \eee^{-\eee^{-x}}.
$$
Let $\beta\in \C$ be located inside (but not on the boundary) of some phase $G^{d_1} F^{d_2} E^{d_3}$ and let $\sigma \geq 0$. For $1\leq k\leq d$, we define a sequence of functions $c_{n,k}(\beta)$ (which is needed to normalize the $k$-th level of the GREM) by
\begin{equation}\label{eq:def_c_nk}
c_{n,k}(\beta) =
\begin{cases}
\beta \sqrt {n a_k} \, u_{n,k}, & \text{if } \beta\in G_k,\\
\frac 12 \log N_{n,k}+ a_k \sigma^2n,  & \text{if } \beta\in F_k,\\
\log N_{n,k}  + \frac 12 a_k \beta^2n,  & \text{if } \beta\in E_k.
\end{cases}
\end{equation}
Then, define a normalizing function $c_n(\beta)$ by
\begin{equation}\label{eq:def_c_n_beta}
c_n(\beta)=c_{n,1}(\beta)+\ldots+ c_{n,d}(\beta).
\end{equation}
\begin{theorem}\label{theo:fluct}
Let $\beta\in G^{d_1} F^{d_2} E^{d_3}$ and let $\sigma \geq 0$. Then,
$$
\frac{\ZZZ_n(\beta)}{\eee^{c_n(\beta)}} \todistr
\begin{cases}
1, &\text{if } d_1=0 \text{ and } d_2=0,\\
N_{\C} (0,1), &\text{if } d_1=0 \text{ and } d_2>0,\\
\zeta_P(\frac{\beta}{\sigma_1},\ldots, \frac{\beta}{\sigma_{d_1}}),  &\text{if } d_1>0 \text{ and } d_2=0,\\
c S_{\sigma_1/\sigma},  &\text{if } d_1>0 \text{ and } d_2>0.
\end{cases}
$$
Here, $\zeta_P$ is the  Poisson cascade zeta function; $S_{\alpha}$ is the
rotationally symmetric, complex standard $\alpha$-stable random variable with
characteristic function $\E \eee^{i \Re (S_{\alpha} \bar z)}=\eee^{-|z|^{\alpha}}$,
$z\in\C$, where $\alpha\in (0,2)$; and $c$ is a constant.
\end{theorem}
\begin{proof}
We will establish stronger results below. The case $d_1=0$, $d_2=0$ follows from Theorem~\ref{theo:moment_S_n_3} below. The case $d_1=0$ and $d_2>0$ follows from Proposition~\ref{prop:clt_simple_1}. (For the asymptotics of the variance, see Proposition~\ref{prop:asympt_exp_variance_log_scale}).  The case $d_1>0$,  $d_2=0$ follows from Theorem~\ref{theo:functional_clt_1} below. Finally, the case $d_1>0$, $d_2>0$ follows from Theorem~\ref{theo:functional_clt} (with $t=0$) below.
\end{proof}
\begin{remark}
The assumption $\sigma \geq 0$  in Theorem~\ref{theo:fluct} can be removed if we define
$$
c_{n,k}(\beta) = (\sgn \sigma) \cdot  \beta \sqrt {n a_k} \, u_{n,k} \text{ for } \beta\in G_k.
$$
\end{remark}

\subsection{Functional limit theorems and local structure of zeros}
One may ask whether the partition function $\ZZZ_n(\beta)$ converges, after an appropriate rescaling (involving, if necessary, a rescaling of the variable $\beta$), to some limiting stochastic process. In this section, we state functional limit theorems of this type. Since weak convergence of random analytic functions implies weak convergence of point processes of zeros, see Proposition~\ref{prop:weak_conv_zeros} below, any functional limit theorem implies a result on the local structure of zeros of $\ZZZ_n(\beta)$.

\begin{figure}
\begin{tabular*}{\textwidth}{p{0.5\textwidth}p{0.5\textwidth}}
\includegraphics[width=0.5\textwidth]{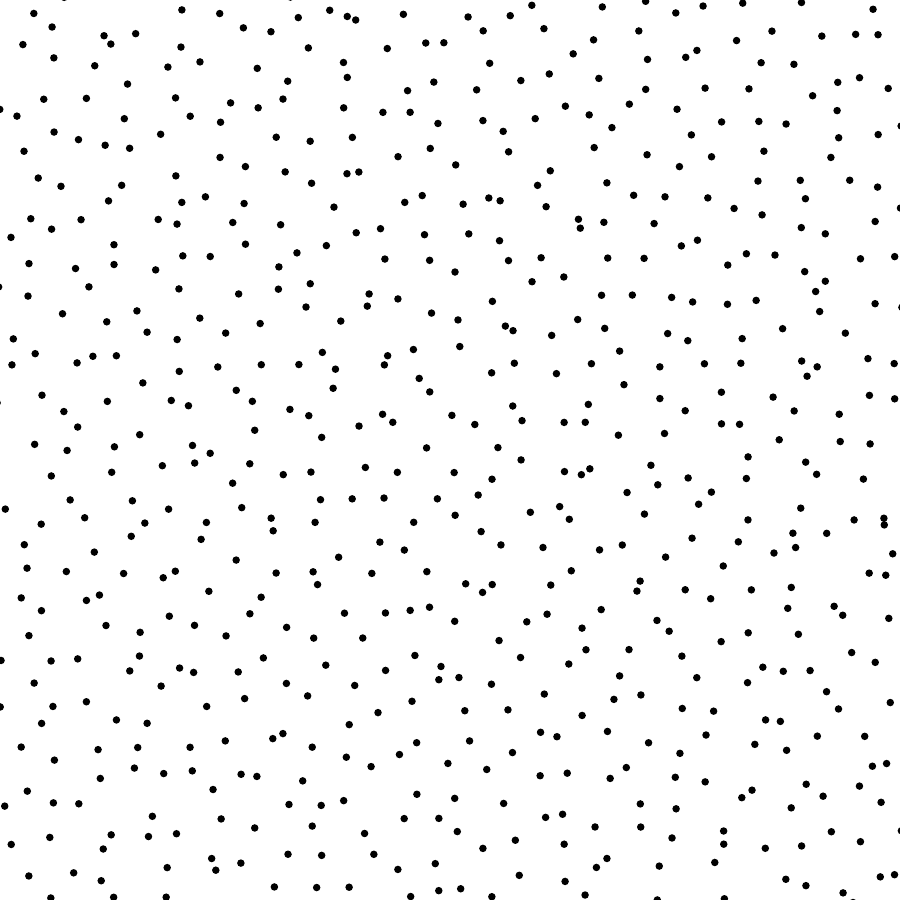}
&
\includegraphics[width=0.5\textwidth]{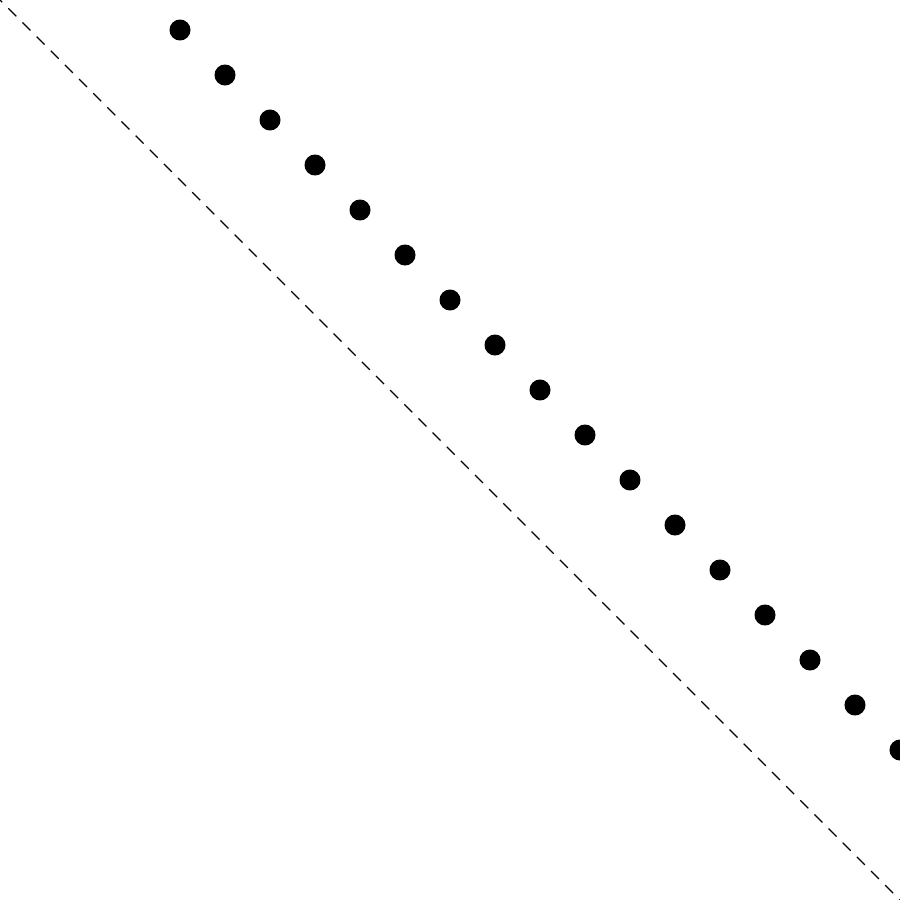}
\\
\begin{center}
{\small
Zeros of the plane Gaussian Analytic Function.
}
\end{center}
&
\begin{center}
{\small ``Curve of zeros'' seen locally. The dotted line is the boundary between the phases.}
\end{center}
\end{tabular*}
\caption{\small Point processes of zeros.}
\label{fig:plane_GAF}
\end{figure}

\subsubsection{Phase $E_1=E^d$}
The first result is a law of large numbers in the phase $E_1=E^d$.
\begin{theorem}\label{theo:moment_S_n_3}
The following convergence of random analytic functions holds weakly  on $\HHH(E_1)$:
\begin{equation}\label{eq:ZZZ_n_beta_weak_to_1_E1}
\frac{\ZZZ_n(\beta)}{\E \ZZZ_n(\beta)} \toweak 1.
\end{equation}
\end{theorem}
In the next two theorems, we will obtain more refined results by a ``cleverer'' choice of normalization. The first theorem deals with the domain $E_1 \cap \{ |\sigma| < \frac{\sigma_1}{2}\}$. In this case, the limiting fluctuations of $\ZZZ_n(\beta)$ are given by the \textit{plane Gaussian analytic function} $\XXX$; see~\cite{peres_etal_book,sodin_tsirelson}.
It is a random analytic function $\{\XXX(t)\colon t\in \C\}$ given by
\begin{equation}\label{eq:def_XXX_plane_GAF}
\XXX(t) = \eee^{-\frac{t^2}{2}} \, \sum_{k=0}^{\infty} N_k \frac{t^k}{\sqrt{k!}},
\end{equation}
where $N_1,N_2,\ldots\sim N_{\C} (0,1)$ are independent complex standard Gaussian random variables. The finite-dimensional distributions of $\XXX$ are multivariate complex Gaussian distributions and the second-order structure of $\XXX$ is given by
\begin{equation}\label{eq:XXX_covariance_plane_GAF}
\E \XXX(t)=0,
\;\;
\E [\XXX(t_1) \XXX(t_2)] =  0,
\;\;
\E [\XXX(t_1) \overline {\XXX(t_2)}] = \eee^{-\frac 12 (t_1-\bar t_2)^2},
\;\; t_1,t_2\in\C.
\end{equation}
The restriction of $\XXX$ to $\R$ is a stationary complex Gaussian process. The factor $\eee^{-t^2/2}$ in~\eqref{eq:def_XXX_plane_GAF} is chosen to simplify the statements of our results and is usually not used in the literature.   The set of complex zeros of $\XXX$ is a remarkable \textit{stationary} point process; see Figure~\ref{fig:plane_GAF}, left. The intensity of this point process is $\pi^{-1}$, that is for every Borel set $B\subset \C$ we have
$$
\E \left[\sum_{z\in B} \ind_{\XXX(z)=0}\right] = \frac 1 {\pi} \, \text{Leb} (B).
$$
For more information on the zeros of $\XXX$, we refer to~\cite{peres_etal_book,sodin_tsirelson}.

\vspace*{2mm}
We are ready to state the functional limit theorem in the domain $E_1 \cap \{ |\sigma| < \frac{\sigma_1}{2}\}$. Recall the definition of $c_{n,k}(\beta)$ from~\eqref{eq:def_c_nk} and define
\begin{equation}\label{eq:c_n_tilde_def}
\tilde c_n(\beta) = c_{n,2}(\beta) + \ldots+ c_{n,d}(\beta).
\end{equation}
\begin{theorem}\label{theo:functional_clt_EF_k0_first_statement}
Fix $\beta_*=\sigma_*+i\tau_*\in E_1 \cap \{ |\sigma| < \frac{\sigma_1}{2}\}$.
Then, the following convergence of random analytic functions holds weakly on $\HHH(\C)$:
\begin{equation}\label{eq:theo:functional_clt_EF_k0_first_state}
\left\{\frac{\ZZZ_n\left(\beta_*+\frac t {\sqrt n}\right)- \E \ZZZ_n\left(\beta_*+\frac t {\sqrt n}\right)}
{N_{n,1}^{\frac 12}\eee^{a_1 \left(\sigma_* + \frac{t}{\sqrt n}\right)^2 n} \,  \eee^{\tilde c_n\left(\beta_*+\frac{t}{\sqrt n}\right)}}\colon t\in\C\right\}
\toweak
\{\XXX(\sqrt{a_1} t)\colon t\in\C\},
\end{equation}
where $\{\XXX(t)\colon t\in\C\}$ is the plane Gaussian analytic function~\eqref{eq:def_XXX_plane_GAF}.
\end{theorem}

In the domain $E_1 \cap \{\sigma>\frac{\sigma_1}{2}\}$, the limiting fluctuations of $\ZZZ_n(\beta)$ are given by the Poisson zeta function.
\begin{theorem}\label{theo:functional_E1_small_domain}
The following convergence of random analytic functions holds weakly on $\HHH(E_1 \cap \{\sigma>\frac{\sigma_1}{2}\})$:
\begin{equation}\label{eq:theo:functional_E1_small_domain}
\frac{\ZZZ_n(\beta) - \E \ZZZ_n(\beta)} {\eee^{\beta \sqrt{na_1} u_{n,1} +\tilde c_n(\beta)} }
\toweak
\zeta_P\left(\frac{\beta}{\sigma_1}\right).
\end{equation}
\end{theorem}
\begin{remark}
By symmetry, see~\eqref{eq:symmetry1}, the following convergence of random analytic functions holds weakly on $\HHH(E_1 \cap \{\sigma < - \frac{\sigma_1}{2}\})$:
\begin{equation}\label{eq:theo:functional_E1_symmetry_first_state}
\frac{\ZZZ_n(\beta) - \E \ZZZ_n(\beta)} {\eee^{-\beta \sqrt{na_1} u_{n,1} + \tilde c_n(\beta)}}
 \toweak \zeta_P^{-}\left(-\frac{\beta}{\sigma_1}\right),
\end{equation}
where $\zeta_P^-$ is a copy of $\zeta_P$. In fact, one can even show that the functional limit theorem holds on the \textit{union} of both domains, namely $E_1 \cap \{|\sigma| > \frac{\sigma_1}{2}\}$, and that the limiting functions $\zeta_P$ and $\zeta_P^-$ are \textit{independent}; see Remark~\ref{rem:E_2_negative_beta_symmetry_explain} and also Remark~\ref{rem:G_d_1_E_d_3_symmetry} for explanation.
\end{remark}

It follows from the above results by an elementary calculation that in phase $E_1$ the fluctuations of $\ZZZ_n(\beta)$ around its expectation are of smaller order than the expectation. One can therefore expect that the function $\ZZZ_n$ has no zeros in $E_1$.  The next theorem makes this precise.
\begin{theorem}\label{theo:no_zeros_E1}
Let $K$ be a compact subset of $E_1$. Then, there exist $C=C(K)$ and $\eps=\eps(K)>0$ such that for all $n\in\N$,
$$
\P[\exists \beta\in K\colon \ZZZ_n(\beta)=0] < C \eee^{-\eps n}.
$$
\end{theorem}
\begin{corollary}\label{cor:no_zeros_E1}
The following weak convergence of point processes on $\NNN(E_1)$ holds:
$$
\Zeros \{\ZZZ_n(\beta) \colon \beta \in E_1\} \toweak \emptyset.
$$
Here,  $\emptyset$ denotes the empty point process on $E_1$.
\end{corollary}

\subsubsection{Phases of the form $G^{d_1}E^{d_3}$} In the next theorem, we
prove the functional convergence of the partition function $\ZZZ_n(\beta)$ in
the phases of the form $G^{d_1}E^{d_3}$, where $d_1,d_3\in \{0,\ldots,d\}$
satisfy $d_1+d_3=d$.   The limiting process is given in terms of the
$d_1$-variate Poisson cascade zeta function $\zeta_P$. Recall that $c_n(\beta)$
was defined in~\eqref{eq:def_c_n_beta}. For $1\leq l\leq d$, define
\begin{equation}\label{eq:T_l_beta}
T^{l} (\beta) = \left(\frac{\beta}{\sigma_1}, \ldots, \frac{\beta}{\sigma_{l}}\right)\in \C^l,
\;\;\;
T^0 (\beta) = \emptyset.
\end{equation}
\begin{theorem}\label{theo:functional_clt_1}
Fix some $d_1,d_3\in \{0,\ldots,d\}$ such that $d_1+d_3=d$.
The following convergence
of random analytic functions holds weakly on $\HHH(G^{d_1}E^{d_3}\cap \{\sigma>0\})$:
\begin{equation}\label{eq:theo:functional_clt_1}
\frac{\ZZZ_n(\beta)}{\eee^{c_n(\beta)}}  \toweak \zeta_P\left(T^{d_1}(\beta)\right).
\end{equation}
\end{theorem}
In particular, for $d_1=0$, the limiting process is $\zeta_P(\emptyset)=1$, and we recover Theorem~\ref{theo:moment_S_n_3}.
\begin{remark}\label{rem:G_d_1_E_d_3_symmetry}
Let $d_1\geq 1$. By symmetry, see~\eqref{eq:symmetry1}, a result similar to Theorem~\ref{theo:functional_clt_1}  holds in the domain $G^{d_1} E^{d_3} \cap \{\sigma < 0\}$. Namely, the following convergence of random analytic functions holds weakly on $\HHH(G^{d_1} E^{d_3} \cap \{\sigma < 0\})$:
\begin{equation}\label{eq:theo:functional_clt_1_symmetry}
\frac{\ZZZ_n(\beta)}{\eee^{c_n(-\beta)}}  \toweak \zeta_P^-\left(T^{d_1}(-\beta)\right),
\end{equation}
where $\zeta_P^-$ is a copy of $\zeta_P$. One can show that~\eqref{eq:theo:functional_clt_1} and~\eqref{eq:theo:functional_clt_1_symmetry} can be combined into a \textit{joint} convergence in the phase $G^{d_1} E^{d_3}$ and that the limiting functions $\zeta_P$ and $\zeta_P^-$ are \textit{independent}, for $d_1\geq 1$. We will not provide a complete proof of the independence, but let us explain the idea. The function $\zeta_P$ in~\eqref{eq:theo:functional_clt_1} appears as the contribution of the \textit{upper} extremal order statistics among the energies on the first $d_1$ levels of the GREM, whereas all other levels make a deterministic contribution equal to the expectation. The function $\zeta_P^-$ in~\eqref{eq:theo:functional_clt_1_symmetry} appears as the contribution of the \textit{lower} extremal order statistics on the first $d_1$ levels of the GREM. Since upper and lower extremal order statistics become independent in the large sample limit, the limiting functions $\zeta_P$ and $\zeta_P^-$ are independent.
\end{remark}

\begin{corollary}
Under the same conditions as in Theorem~\ref{theo:functional_clt_1}, the
following weak convergence of point processes on $\NNN(G^{d_1}E^{d_3}\cap \{\sigma>0\})$ holds:
$$
\Zeros \{\ZZZ_n(\beta)\}
\toweak
\Zeros \left\{\zeta_P\left(T^{d_1}(\beta)\right)\right\}.
$$
\end{corollary}
Note that the intensity of zeros in the limiting point process is $O(1)$ and hence these zeros do not appear in the limit in Theorem~\ref{theo:zeros_global}.  For $d_1=0$, the limiting point process of zeros is empty and we recover Corollary~\ref{cor:no_zeros_E1}.

\subsubsection{Phases with at least one fluctuation level}
Our next result is a functional limit theorem describing the limiting structure
of the stochastic process $\ZZZ_n(\beta)$ in an infinitesimal neighborhood of some $\beta_*=\sigma_*+i\tau_*\in G^{d_1}F^{d_2}E^{d_3}$, where $d_2\geq 1$.
\begin{theorem}\label{theo:functional_clt}
Fix some $d_1, d_2, d_3\in \{0,\ldots,d\}$ with $d_1+d_2+d_3=d$ and $d_2 \geq 1$. Also, fix some $\beta_*=\sigma_*+i\tau_*\in G^{d_1} F^{d_2} E^{d_3}$ such that $\sigma_*\geq 0$.  Then, for a suitable normalizing function $c_n(\beta_*;t)$ (which is quadratic in $t$), the following convergence of random analytic functions holds weakly on $\HHH(\C)$:
$$
\left\{\eee^{-c_n(\beta_*; t)} \ZZZ_n\left(\beta_*+\frac{t}{\sqrt n}\right)\colon t\in\C \right\}
\toweak
\left\{\sqrt{W}\, \XXX(\kappa t) \colon t\in\C\right\},
$$
where
\begin{enumerate}
\item[\textup{(1)}] $W = \zeta_P(2 T^{d_1} (\sigma_*))$  and $\zeta_P$ is the Poisson cascade zeta function with $d_1$ variables;
\item[\textup{(2)}] $\{\XXX(t)\colon t\in\C\}$ is the plane Gaussian analytic function~\eqref{eq:def_XXX_plane_GAF};
\item[\textup{(3)}] $\kappa^2=\sum_{k=1}^{d_2} a_{d_1+k}$ is the total variance of the GREM levels which are in the fluctuation phase;
\item[\textup{(4)}] the processes $\zeta_P$ and $\XXX$ are independent.
\end{enumerate}
\end{theorem}
The formula for $c_n(\beta_*; t)$ will be given in~\eqref{eq:def_c_nk_t_fluct_level}, \eqref{eq:c_n_beta_*_aux_def1} below. If $d_1=0$ (i.e.,\ there are no levels in the glassy phase), then the limit
is the Gaussian analytic function $\XXX(\kappa t)$ since we have $\zeta_P(\emptyset)=1$.
In the case $d_1\neq 0$, the limiting process is a Gaussian process rescaled by the square root of
an independent real $\frac{\sigma_1}{2\sigma_*}$-stable random
variable $W = \zeta_P(2 T^{d_1}(\sigma_*))$ with skewness parameter $+1$. Such a process is itself complex $\frac{\sigma_1}{\sigma_*}$-stable with complex isotropic margins.  Processes of this type are called subgaussian;
see~\cite{samorodnitsky_taqqu_book}.
\begin{corollary}
Under the same conditions as in Theorem~\ref{theo:functional_clt}, the following convergence of the point processes of zeros holds weakly on $\NNN(\C)$:
$$
\Zeros \left\{\ZZZ_n\left(\beta_*+\frac t {\sqrt n}\right)\colon t\in\C\right\} \toweak \Zeros \left\{\XXX(\kappa t)\colon t\in\C\right\}.
$$
\end{corollary}


\subsubsection{Curves of zeros: Beak shaped boundaries}\label{subsubsec:Lines_of_zeros_beak_shaped}
In phase $G^{l-1} E^{d-l+1}$ the fluctuations of $\ZZZ_n(\beta)$ are given by an $(l-1)$-variate Poisson cascade zeta function, whereas in phase $G^{l} E^{d-l}$ the fluctuations are given by an $l$-variate  zeta function. On the boundary between these two phases, under an appropriate scaling, \textit{both} functions become visible in the limit.
\begin{theorem}\label{theo:functional_clt_GE_beak_boundary}
Fix some $1\leq l\leq d$ and some $\beta_*=\sigma_*+i\tau_*\in \C$ such that $\sigma_* > \frac{\sigma_l}{2}$, $\tau_*>0$ and $\sigma_*+\tau_*=\sigma_l$.
Then, there exist a complex sequence $d_{n,l}=O(\log n)$ and a sequence of functions $h_{n,l}(t)$ (which are quadratic functions in $t$) such that weakly on $\HHH(\C)$ it holds that
\begin{equation*}
\left\{\eee^{-h_{n,l}(t)}\ZZZ_n \left(\beta_* + \frac{d_{n,l} + t}{a_l (\beta_* - \sigma_l) n} \right) \colon t\in\C\right\}
\toweak
\left\{\eee^{t}\zeta^{(l-1)} + \zeta^{(l)} \colon t\in\C\right\}.
\end{equation*}
Here, $(\zeta^{(l-1)}, \zeta^{(l)})$ is a random vector given by
\begin{equation}\label{eq:zeta_1_zeta_2_beak_bound}
(\zeta^{(l-1)}, \zeta^{(l)})
=
\left(\zeta_P\left(T^{l-1}(\beta_*)\right), \zeta_P\left(T^l (\beta_*)\right)\right).
\end{equation}
In~\eqref{eq:zeta_1_zeta_2_beak_bound}, both $\zeta^{(l-1)}$ and $\zeta^{(l)}$ are based on the same Poisson cascade point process.
\end{theorem}
We will provide exact expressions for $d_{n,l}$ and $h_{n,l}(t)$ in~\eqref{eq:d_n_l_def} an~\eqref{eq:h_n_l_t_def}, below. Theorem~\ref{theo:functional_clt_GE_beak_boundary} allows us to clarify the local structure of the line of zeros on the beak shaped boundary between the phases $G^{l-1}E^{d-l+1}$ and $G^l E^{d-l}$.
\begin{corollary}
Under the same conditions as in Theorem~\ref{theo:functional_clt_GE_beak_boundary}, the following convergence of point processes holds weakly on $\NNN(\C)$:
$$
\Zeros\left\{\ZZZ_n \left(\beta_* +  \frac{d_{n,l} + t}{a_l (\beta_* - \sigma_l) n} \right) \colon t\in\C\right\}
\toweak
\sum_{k\in \Z} \delta\left(\log \left(-\frac{\zeta^{(l-1)}}{\zeta^{(l)}}\right) + 2\pi i k\right).
$$
\end{corollary}
It follows that the zeros of $\ZZZ_n(\beta)$ in a neighborhood of $\beta_*$ look locally like \textit{equally spaced} points on a line parallel to the boundary between $G^{l-1}E^{d-l+1}$ and $G^l E^{d-l}$; see Figure~\ref{fig:plane_GAF}, right. The spacing between neighboring zeros is
$$
\frac {\sqrt{2} \pi}{ a_l \tau_*} \cdot \frac 1n + o\left(\frac 1n\right).
$$
This agrees with what one expects from the definition of the measure $\Xi^{EG}_l$; see Section~\ref{subsec:zeros_global}.  From the formula for $d_{n,l}$, see~\eqref{eq:d_n_l_def}, it can be seen that the zeros are located \textit{outside} the phase $E_l$, the distance to the boundary being of order $\text{const}\cdot \frac{\log n} n$.

\subsubsection{Curves of zeros: Arc shaped boundaries}\label{subsubsec:Lines_of_zeros_arc_shaped}
In the next theorem, we describe the local structure of the partition function $\ZZZ_n(\beta)$
in an infinitesimal neighborhood of some $\beta_*=\sigma_*+i\tau_*$ located on the boundary separating the phases $G^{d_1}F^{d_2}E^{d_3}$ and $G^{d_1} F^{d_2-1}E^{d_3+1}$, where $d_2\geq 1$. We assume that
\begin{equation}\label{eq:tech448}
\frac{\sigma_{d_1}}{2} < \sigma_* < \frac{\sigma_{d_1+1}}{2}, \;\; \tau_*>0,\;\; \sigma_*^2 + \tau_*^2 = \frac{\sigma_{d_1+d_2}^2}{2}.
\end{equation}
\begin{theorem}\label{theo:functional_clt_line of_zeros_d2_geq2}
Fix some $d_1, d_2, d_3\in \{0,\ldots,d\}$ with $d_1+d_2+d_3=d$ and $d_2\geq 2$. Also, fix some $\beta_*=\sigma_*+i\tau_*\in\C$ satisfying~\eqref{eq:tech448}.
Then, for suitable normalizing functions $f_n(\beta_*; t)$ (which are linear in $t$), the following convergence of random analytic functions holds weakly on $\HHH(\C)$:
$$
\left\{\eee^{-f_n(\beta_*;t)} \ZZZ_n\left(\beta_*+\frac{t}{n}\right)\colon t\in\C \right\}
\toweak
\left\{ \sqrt{W} \, (\eee^{\lambda' t} N'+ \eee^{\lambda'' t} N'') \colon t\in\C\right\},
$$
where
\begin{enumerate}
\item[\textup{(1)}] $W = \zeta_P (2 T^{d_1} (\sigma_*))$ and $\zeta_P$ is the Poisson cascade zeta function with $d_1$ variables;

\item[\textup{(2)}] 
$N',N''\sim N_{\C}(0,1)$ are independent complex standard normal random
variables;

\item[\textup{(3)}] $\lambda', \lambda''$ are constants given in Remark~\ref{rem:lambda_prime} below;

\item[\textup{(4)}] the random variable $W$ and the random vector $(N',N'')$ are independent.  
\end{enumerate}
\end{theorem}
\begin{remark}\label{rem:lambda_prime}
Define the ``partial variances'' $A_{l,m}=a_{l}+\ldots+a_m$ for $1\leq l\leq m\leq d$ and let $A_{l,m}=0$ if $l>m$.
The constants $\lambda'$ and $\lambda''$ are given by
\begin{align*}
\lambda'  &= 2\sigma_* A_{d_1+1,d_1+d_2} + \beta_* A_{d_1+d_2+1,d},\\
\lambda'' &= 2\sigma_* A_{d_1+1,d_1+d_2-1} + \beta_* A_{d_1+d_2,d}.
\end{align*}
Note that $\lambda'-\lambda''=\bar \beta a_{d_1+d_2}$. A  formula  for the normalizing function $f_n(\beta_*; t)$ will be provided in~\eqref{eq:f_n_beta_star_t} below.
\end{remark}
\begin{corollary}\label{cor:theo:functional_clt_line of_zeros_d2_geq2}
Under the same conditions as in Theorem~\ref{theo:functional_clt_line of_zeros_d2_geq2}, we have the following weak convergence of point processes on $\NNN(\C)$:
$$
\Zeros \left\{\ZZZ_n\left(\beta_*+\frac tn \right)\colon t \in \C\right\}
\toweak
\sum_{k\in\Z} \delta\left( \frac 1 {\bar
\beta a_{d_1+d_2}} \left( \log \left( - \frac{N''}{N'}\right)+ 2\pi i k\right)\right).
$$
\end{corollary}
In the case $d_2=1$ we have a slightly different result.
\begin{theorem}\label{theo:functional_clt_line of_zeros_d2_eq1}
Fix some $d_1, d_2, d_3\in \{0,\ldots,d\}$ with $d_1+d_2+d_3=d$ and $d_2=1$. Also, fix some $\beta_*=\sigma_*+i\tau_*\in\C$ satisfying~\eqref{eq:tech448}.
Then, for suitable normalizing functions $f_n(\beta_*; t)$ (which are linear in $t$), the following convergence of random analytic functions holds weakly on $\HHH(\C)$:
$$
\left\{\eee^{-f_n(\beta_*;t)} \ZZZ_n\left(\beta_*+\frac{t}{n}\right)\colon t\in\C \right\}
\toweak
\left\{\eee^{\lambda' t} \sqrt W N + \eee^{\lambda'' t} \zeta^{(d_1)} \colon t\in\C\right\},
$$
where
\begin{enumerate}

\item[\textup{(1)}] $W = \zeta_P (2 T^{d_1} (\sigma_*))$ and $\zeta^{(d_1)} = \zeta_P(T^{d_1}(\beta_*))$, where  in both cases the  zeta function $\zeta_P$ is based on the same Poisson cascade point process;

\item[\textup{(2)}] $N\sim N_{\C}(0,1)$ is a complex standard normal random variable;

\item[\textup{(3)}] $\lambda'$ and $\lambda''$ are constants given in Remark~\ref{rem:lambda_prime};

\item[\textup{(4)}] the random vector $(W, \zeta^{(d_1)})$ and the random variable $N$ are independent.
\end{enumerate}
\end{theorem}
An explicit  formula  for $f_n(\beta_*; t)$ will be given in~\eqref{eq:f_n_beta_star_t} below.
\begin{corollary}\label{cor:theo:functional_clt_line of_zeros_d2_eq1}
Under the same conditions as in Theorem~\ref{theo:functional_clt_line of_zeros_d2_eq1}, we have the following weak convergence of point processes on $\NNN(\C)$:
$$
\Zeros \left\{\ZZZ_n\left(\beta_*+\frac tn \right)\colon \beta \in \C\right\}
\toweak
\sum_{k\in\Z} \delta\left( \frac 1 {\bar \beta_* a_{d_1+1}} \left( \log \left( - \frac{\zeta^{(d_1)}}{\sqrt W N}\right)+ 2\pi i k\right)\right).
$$
\end{corollary}
Both in Corollary~\ref{cor:theo:functional_clt_line of_zeros_d2_geq2} and Corollary~\ref{cor:theo:functional_clt_line of_zeros_d2_eq1} the zeros of $\ZZZ_n(\beta)$ in a neighborhood of $\beta_*$ look locally like equally spaced points, the spacing being
$$
\frac{2\pi}{a_{d_1+d_2} |\beta_*|} \cdot \frac 1n + o\left(\frac 1n\right).
$$
This agrees with what we expect from the definition of the measure $\Xi_{d_1+d_2}^{EF}$; see Section~\ref{subsec:zeros_global}.

\subsubsection{Fluctuations on the vertical half-line boundaries}
Let us finally state a theorem on the fluctuations of $\ZZZ_n(\beta)$ for $\beta$ on the boundary between $F_l$ and $G_l$, for some $1\leq l\leq d$. This theorem is obtained by adjoining $l-1$ glassy phase levels to Theorem~\ref{theo:clt_boundary}.
\begin{theorem}\label{theo:clt_boundary_spin_glass}
Let $\beta\in \C$ be such that $\sigma = \frac{\sigma_l}{2}$ and $\tau>\frac{\sigma_l}{2}$, for some $1\leq l\leq d$. Then, for a suitable normalizing sequence $r_n(\beta)$,
$$
\frac{\ZZZ_n(\beta)}{\eee^{r_n(\beta)}}
\todistr
\frac 1 {\sqrt 2}
\sqrt{\zeta_P\left(2 T^{l-1} (\sigma) \right)}\, N,
$$
where $N\sim N_{\C}(0,1)$ is independent of $\zeta_P$.
\end{theorem}
An exact expression for $r_n(\beta)$ will be provided in~\eqref{eq:m_n_beta_clt_bound_with_spin_glass} below.
At the ``triple points''  (i.e., at points, where the phases $E_l$, $F_l$ and $G_l$ meet), the result takes the following form.
\begin{theorem}\label{theo:clt_triple_point}
Let $\beta\in \C$ be such that $\sigma = \tau = \frac{\sigma_l}{2}$, for some $1\leq l\leq d$. Then, for a suitable normalizing sequence $r_n(\beta)$,
$$
\frac{\ZZZ_n(\beta)}{\eee^{r_n(\beta)}}
\todistr
\frac 1 {\sqrt 2}
\left(\sqrt{\zeta_P\left(2  T^{l-1} (\sigma) \right)}\, N + \zeta_P(T^{l-1}(\beta))\right),
$$
where $N\sim N_{\C}(0,1)$ is independent from the zeta functions and both zeta functions are based on the same Poisson cascade point process.
\end{theorem}

\subsection{Passing to continuum hierarchies}
\label{sec:crem}
In the GREM with $d$ levels, there are
$d$ (real temperature) phase transitions at inverse temperatures
$\beta = \sigma_1,\ldots,\sigma_d$, whereas more interesting spin glass models like the
Sherrington--Kirkpatrick model are known (or conjectured) to exhibit a ``continuum
of freezing phase transitions'' or the so-called \textit{full replica symmetry breaking}. It has
been suggested by~\citet{Derrida_Gardner1} that it is possible to consider the
limit of the GREM as $d$, the number of levels, goes to $\infty$.
\citet{bovier_kurkova2} defined the continuum limit of the GREM,  the Continuous
Random Energy Model (CREM), and computed its free energy at real $\beta$. In
this section, we will show heuristically how to pass to the continuum hierarchy limit of the
GREM in the complex $\beta$ case; see Figure~\ref{fig:crem}. It should be
stressed that the arguments in this section are not entirely rigorous.


\begin{figure}
\begin{tabular*}{\textwidth}{p{0.5\textwidth}p{0.5\textwidth}}
\includegraphics[width=0.5\textwidth]{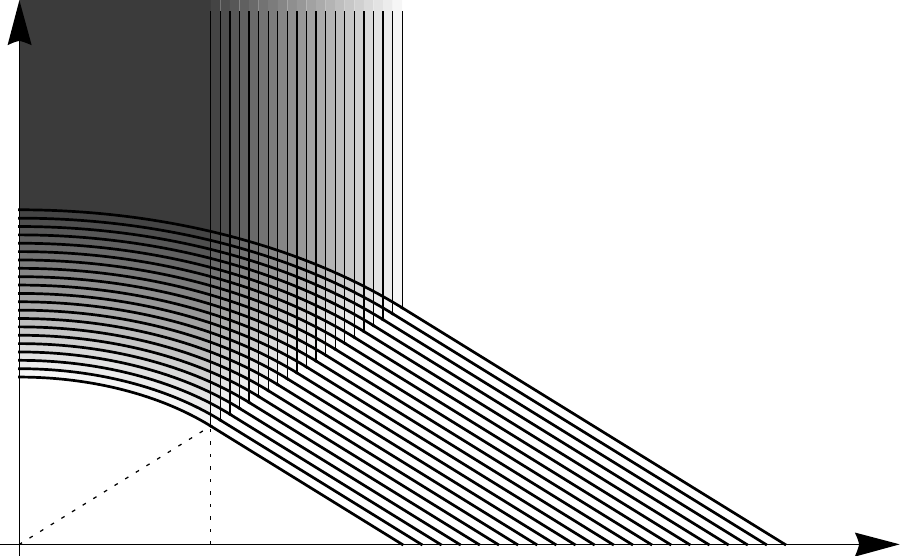}
&
\includegraphics[width=0.5\textwidth]{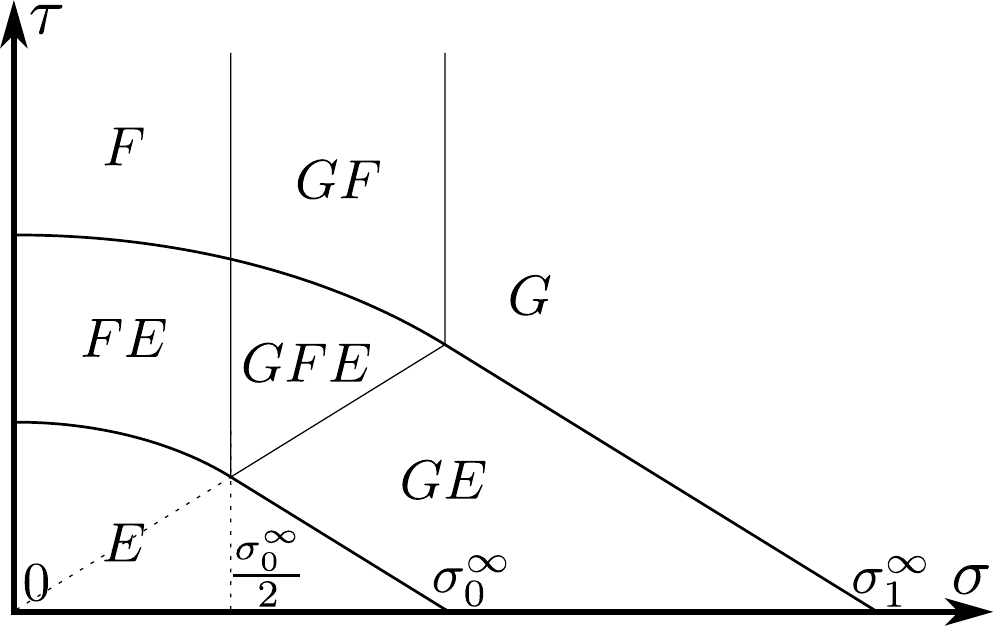}
\\
\begin{center}
{\small
Approximating CREM by GREM with many levels.
}
\end{center}
&
\begin{center}
{\small Phase diagram of the CREM.}
\end{center}
\end{tabular*}
\caption{\small Phase diagram of the CREM.}
\label{fig:crem}
\end{figure}

Let $A \colon [0,1]\to \R$ be an increasing, concave function with $A(0)=0$. Fix also some $\alpha>1$. Consider a GREM with $d$ levels whose parameters $(a_1,\ldots,a_d)$ and $(\alpha_1,\ldots,\alpha_d)$ are given by
\begin{equation}\label{eq:param_CREM}
a_1+\ldots+a_k = A\left(\frac kd\right), \quad \log \alpha_k = \frac 1d \log \alpha, \quad 1\leq k\leq d.
\end{equation}
The total number of energies in this GREM is $\alpha^{n+o(1)}$ and the variance of each energy is $A(1)n$.

Let us now pass to the large $d$ limit. Let $t\in [0,1]$. Then, it follows from~\eqref{eq:param_CREM} that the large $d$ limit of $d a_{[td]}$ is $A'(t)$. Hence, the large $d$ limit of the critical temperature $\sigma_{[td]}$ is
$$
\sigma_t^{\infty} = \sqrt{\frac{2\log \alpha}{A'(t)}}.
$$
The large $d$ limits of the domains $G_{[td]}$, $F_{[td]}$, $E_{[td]}$  are the domains
\begin{align*}
G_t^{\infty}
&=
\{\beta\in\C \colon 2|\sigma| >  \sigma_t^{\infty}, |\sigma|+|\tau| > \sigma_t^{\infty}\},
\\
F_t^{\infty}
&=
\{\beta \in \C \colon 2|\sigma| < \sigma_t^{\infty}, 2 (\sigma^2+\tau^2) > (\sigma_t^{\infty})^2\},
\\
E_t^{\infty}
&= \C\bsl \overline{G_t^{\infty}\cup F_t^{\infty}}.
\end{align*}
Recall that the complex plane phases of a GREM with $d$ levels were denoted by $G^{d_1} F^{d_2} E^{d_3}$, where the parameters $d_1,d_2,d_3 \in \N_0$ satisfy $d_1+d_2+d_3=d$.
Instead of $d_1,d_2,d_3$, in the continuum limit we have three parameters $\gamma_1,\gamma_2,\gamma_3\in [0,1]$ which are the large $d$ limits of $\frac{d_1}{d}, \frac{d_2}{d}, \frac{d_3}{d}$ and hence satisfy $\gamma_1+\gamma_2+\gamma_3=1$. To find the formula for $\gamma_1,\gamma_2,\gamma_3\in [0,1]$ note that in the $d$-level GREM,
$$
d_1=\max\{k\geq 0\colon \beta\in G_k\}, \;\;   d_1+d_2+1 = \max\{k\geq 0 \colon \beta\in E_k\}.
$$
Passing to the large $d$ limit, we obtain
$$
\gamma_1=\sup\{t\in [0,1] \colon \beta \in G_t^{\infty}\}, \;\; \gamma_1+\gamma_2 = \sup\{t\in [0,1] \colon \beta\in E_t^{\infty}\}.
$$
For the $d$-level GREM, Theorem~\ref{theo:free_energy} states that the log-partition function is
$
p(\beta)=p_G(\beta)+p_F(\beta)+p_E(\beta),
$
where $p_G(\beta)$, $p_F(\beta)$, $p_E(\beta)$ are the contributions of spin glass, fluctuation and expectation levels given by
\begin{align}
p_G(\beta) &= |\sigma| \sum_{k=1}^{d_1}  \sqrt{2a_k\log \alpha_k},\\
p_F(\beta) &= \sum_{k=d_1+1}^{d_1+d_2} \left(\frac 12 \log \alpha_k + a_k \sigma^2\right),\\
p_E(\beta) &= \sum_{k=d_1+d_2+1}^{d} \left(\log \alpha_k + \frac 12 a_k (\sigma^2-\tau^2)\right).
\end{align}
Replacing Riemann sums by Riemann integrals, we obtain that in the large $d$ limit, the log-partition function is
\begin{equation}\label{eq:CREM_0}
p^{\infty}(\beta)=p_G^{\infty}(\beta)+p_F^{\infty}(\beta)+p_E^{\infty}(\beta),
\end{equation}
where
\begin{align}
p_G^{\infty}(\beta) &= |\sigma| \sqrt{2\log \alpha}  \int_{0}^{\gamma_1} \sqrt{A'(t)}\dd t ,\label{eq:CREM_1}\\
p_F^{\infty}(\beta) &= \frac {\gamma_2}2  \log \alpha + (A(\gamma_1+\gamma_2)-A(\gamma_1))\sigma^2,\label{eq:CREM_2}\\
p_E^{\infty}(\beta) &= \gamma_3 \log \alpha + \frac 12 (\sigma^2-\tau^2) (A(1)- A(\gamma_1+\gamma_2)). \label{eq:CREM_3}
\end{align}
If $\beta$ is real, then $\gamma_1$ is the solution of $\sigma_{\gamma_1}^{\infty}=\sigma$, $\gamma_2=0$, $\gamma_3=1-\gamma_1$ and the log-partition function is given by
\begin{equation}\label{eq:CREM_real}
p^{\infty}(\beta) = |\sigma| \sqrt{2\log \alpha}  \int_{0}^{\gamma_1} \sqrt{A'(t)}\dd t + (1-\gamma_1) \log \alpha + \frac {\sigma^2} 2 (A(1)- A(\gamma_1)).
\end{equation}
This formula is known, see~\cite[Theorem~3.3]{bovier_kurkova2} (where the second term seems to be missing) and~\cite[Theorem~4.2]{bovier_kurkova_review} (where all terms are present).

In the continuum limit of the GREM, there are seven  phases which we denote by
$$
GFE, GF, FE, GE, G, F, E;
$$
see Figure~\ref{fig:crem}. In such a phase, for every letter which is not in the name of this phase, the corresponding $\gamma_i$ must vanish.  For example, the phase $FE$ is characterized by the conditions $\gamma_1=0$, $\gamma_2\neq 0$, $\gamma_3\neq 0$.

It should be stressed that we do not have a rigorous proof
that~\eqref{eq:CREM_0}, \eqref{eq:CREM_1}, \eqref{eq:CREM_2}, \eqref{eq:CREM_3}
apply to the CREM as defined in~\cite{bovier_kurkova2}. In the real $\beta$
case, \citet{bovier_kurkova2} use were able  to sandwich a CREM between two
close GREM's which allowed them to derive~\eqref{eq:CREM_real} rigorously using
Gaussian comparison inequalities. This method does not seem to work in the
complex $\beta$ case because we cannot apply the comparison inequalities.

The Branching Random Walk and the Gaussian Multiplicative Chaos can be seen as
special (or limiting) cases of the CREM with $A(t)=t$. In this case, $\sigma_t^{\infty} \equiv
1$ which means that we have only the phases $E,F,G$ as in the REM, see~\cite{derrida_evans_speer,lacoin_rhodes_vargas,madaule_rhodes_vargas,madaule_rhodes_vargas1}.

\subsection{Further extensions of the model}
Similarly to the setup of~\cite[Section~2.3]{kabluchko_klimovsky}, one can consider a complex
GREM with arbitrary correlations between the real
and imaginary parts of the random exponents. That is, given correlation parameters $\rho_1,\ldots,\rho_d \in [-1,1]$, consider a Gaussian random field $\{Y_{\eps}\colon \eps\in \SSS_n\}$ having the same distribution as  $\{X_\eps\colon \eps\in \SSS_n\}$, see~\eqref{eq:gaussian_field_def},  and satisfying
\begin{equation} \label{eq:corr}
\Cov (X_{\eps},Y_{\eta})
=
 \sum_{k=1}^{l(\eps,\eta)} \rho_k a_k
,
\qquad
\eps,\eta \in \SSS_n,
\end{equation}
where $l(\eps, \eta)=\min\{k\in \N\colon \eps_k\neq \eta_k\}-1$. Along the lines of the present paper, one can study the partition
function
\begin{align}
\label{eq:partition_function_corr}
\hat{\ZZZ}_n(\beta)=\sum_{\eps\in \SSS_n} \eee^{\sqrt n (\sigma X_{\eps} + i \tau Y_{\eps})},
\qquad
\beta=(\sigma,\tau)\in \R^2.
\end{align}
It seems that Theorems~\ref{theo:free_energy} and~\ref{theo:zeros_global} need
no changes even if we substitute partition function \eqref{eq:ZZZ_n_beta_def}
with the one from \eqref{eq:partition_function_corr}. The more refined results
on fluctuations such as Theorem~\ref{theo:fluct}, however, need appropriate
modifications; see~\cite{kabluchko_klimovsky} for the case $d=1$.

\subsection{Structure of the proofs} The remaining part of the paper is devoted
to proofs. In order to obtain the fluctuations of $\ZZZ_n(\beta)$, we will use
the following approach. We will write the partition function $\ZZZ_n(\beta)$ as
$$
\ZZZ_n(\beta) =
\sum_{k=1}^{N_{n,1}} \eee^{\beta \sqrt{na_1} \xi_k}
\tilde \ZZZ_{n,k}(\beta),
$$
where $\eee^{\beta \sqrt{na_1} \xi_k}$ are the contributions of the first GREM
level, and $\tilde \ZZZ_{n,k}(\beta)$ are the contributions of the remaining
$d-1$ levels which are given by
$$
\tilde \ZZZ_{n,k}(\beta) = \sum_{\eps_2=1}^{N_{n,2}} \ldots \sum_{\eps_d=1}^{N_{n,d}} \eee^{\beta \sqrt n (\sqrt{a_2} \xi_{k\eps_2} + \ldots + \sqrt{a_d} \xi_{k\eps_2\ldots\eps_d})}.
$$
This provides a representation of $\ZZZ_n(\beta)$ as a sum of independent random
variables (in a triangular scheme), and the powerful theory of summation of
independent random variables and vectors~\cite{petrov_book,meerschaert_book} can
be used. A similar approach was used in the case of the REM at real temperature
by~\citet{bovier_kurkova_loewe}.  The main question is what are the properties
of random variables $\eee^{\beta \sqrt{na_1}\xi_k}$ and $\tilde
\ZZZ_{n,k}(\beta)$. Depending on the behavior of the contributions of the first
level, we will distinguish between two cases: the \textit{Gaussian case} and the
\textit{Poissonian case}.


\vspace*{2mm} \noindent \textsc{Gaussian case: $|\sigma| < \frac{\sigma_1}{2}$.}
The main feature of this case is that it is possible to verify the
\textit{Lindeberg condition} for the random variables $\eee^{\beta \sqrt{na_1}
\xi_k}$. As a consequence, the fluctuations of $\ZZZ_n(\beta)$ turn out to be
Gaussian. The Gaussian case includes the sets $F^{d_2} E^{d-d_2}$ with $1\leq
d_2\leq d$, the set $E_1\cap \{|\sigma| < \frac{\sigma_1}{2}\}$, and the
boundaries between these sets.  Note that only a part of  the phase $E_1=E^d$ is
 included in the Gaussian case. The Gaussian case will be analyzed in
Sections~\ref{sec:proof_clt}, \ref{sec:cov}, \ref{sec:func_clt_sigma_small}. In
Section~\ref{sec:proof_clt_boundary}, we analyze the case $|\sigma| =
\frac{\sigma_1}{2}$. Although the Lindeberg condition is not satisfied in this
case, we will verify some weaker conditions which ensure that the limiting
fluctuations of $\ZZZ_n(\beta)$ are still Gaussian.

\vspace*{2mm} \noindent \textsc{Poissonian case: $|\sigma| >
\frac{\sigma_1}{2}$.} In this case, the random variables $\eee^{\beta
\sqrt{na_1} \xi_k}$ \textit{do not} satisfy the Lindeberg condition. Instead, it
turns out that the contribution of the first level comes from  the
\textit{extremal order statistics} among the energies on the first level. The
limiting fluctuations of $\ZZZ_n(\beta)$ in the Poissonian case will be
described in terms of a Poisson cascade zeta function $\zeta_P$. Main results on
this function will be proved in Section~\ref{sec:zeta_P}.  The Poissonian  case
includes all phases in which the first level is in the glassy (G) phase, the
set $E_1\cap \{|\sigma|> \frac{\sigma_1}{2}\}$, and the boundaries between these
sets. Section~\ref{sec:first_level_GREM} contains some preliminary results on
the first level of the GREM.  Results on the fluctuations of $\ZZZ_n(\beta)$ in
phases of the form $G^{d_1} E^{d-d_1}$ with $1\leq d_1 \leq d$, as well as in
the set $E_1\cap \{|\sigma|> \frac{\sigma_1}{2}\}$, will be proved in
Section~\ref{sec:func_CLT_GE} after an essential part of the work has been done
in Section~\ref{sec:moments_GE}. Section~\ref{sec:func_CLT_GE_boundary} deals
with boundaries separating  phases of the form $G^{d_1}E^{d-d_1}$.  Results on
the fluctuations of $\ZZZ_n(\beta)$ in phases involving at least one level in
fluctuation (F) phase will be proved in
Section~\ref{sec:adjoin_spin_glass_to_F}.

\vspace*{2mm} Let us finally make a remark on the contributions of the levels
$2,\ldots,d$. Since $\tilde \ZZZ_{n,k}(\beta)$  has the same structure as
$\ZZZ_n(\beta)$, it is natural to use induction over $d$, the number of levels
of the GREM. Then, the induction assumption provides information about $\tilde
\ZZZ_{n,k}(\beta)$. For technical reasons, we will frequently need to obtain
estimates on the moments of $\tilde \ZZZ_{n,k}(\beta)$. To prove such estimates,
we also use induction. It is useful to keep in mind the following principle: the
moment properties of $\tilde \ZZZ_{n,k}(\beta)$ are usually better than those of
$\ZZZ_n(\beta)$. The reason for this is the standing
assumption~\eqref{eq:convexity}.

\vspace*{2mm} Only after the fluctuations of $\ZZZ_n(\beta)$ at every complex
$\beta$ have been identified, we will be able to prove
Theorem~\ref{theo:free_energy} (on the limiting log-partition function) and
Theorem~\ref{theo:zeros_global} (on the global distribution of complex zeros).
One may ask whether it is possible to prove Theorem~\ref{theo:free_energy}
directly, without computing the fluctuations of $\ZZZ_n(\beta)$. As we explained
in Section~\ref{subsec:free_energy}, it is not difficult to \textit{guess} the
formula for the limiting log-partition function. However, we do not know any
rigorous \textit{proof} of this formula which avoids the computation of the
fluctuations of $\ZZZ_n(\beta)$. The main difficulty is that in order to obtain
a lower estimate on $|\ZZZ_n(\beta)|$ we need to control the possible
cancellations among the terms in the definition of $\ZZZ_n(\beta)$, a problem
which does not appear in the real temperature case. Our way to control the
cancellations is to find the limiting distribution and to show that it has no
atom at zero.

\section{Auxiliary results} In this section, we collect a number of mostly
well-known results which will be frequently used in the sequel. The reader may
skip this section and return to it later if necessary.

\subsection{Inequalities for the moments of random variables} In our proofs, we
will often need estimates for the moments of random variables. In this section,
we collect such estimates. For example, we will frequently use Lyapunov's
inequality: For every real or complex random variable $X$ and every $0<s\leq t$
it holds that
\begin{equation}\label{eq:lyapunov_ineq}
(\E |X|^s)^{1/s} \leq (\E |X|^t)^{1/t}.
\end{equation}
For arbitrary (deterministic) numbers $x_1,\ldots,x_m\in\C$ and $p>0$, it holds that
\begin{equation}\label{eq:jensen_ineq}
|x_1+\ldots+x_m|^p \leq \max(1, m^{p-1}) \sum_{i=1}^m |x_i|^p.
\end{equation}
In the case $p\geq 1$, this follows from Jensen's inequality, whereas  in the case $0<p\leq 1$ the inequality is easy to prove by induction.
\begin{lemma}\label{lem:centered_moment_ineq}
For $p\geq 0$, and any complex-valued random variable $Z$,
\begin{equation}\label{eq:lem:centered_moment_ineq}
\E |Z - \E Z|^p \leq \max(1, 2^{p-1}) (\E |Z|^p + |\E Z|^p).
\end{equation}
For $p\geq 1$ we even have $\E |Z - \E Z|^p \leq 2^p \E |Z|^p$.
\end{lemma}
\begin{proof}
Inequality~\eqref{eq:lem:centered_moment_ineq} follows from~\eqref{eq:jensen_ineq}. For $p\geq 1$ we have $|\E Z|^p \leq (\E|Z|)^p \leq \E |Z|^p$ by~\eqref{eq:lyapunov_ineq}. This proves the second statement of the lemma.
\end{proof}
The next proposition (see~\cite{von_bahr_esseen} or~\cite[Chapter~2.3]{petrov_book}) is an immediate corollary of~\eqref{eq:jensen_ineq}.
\begin{proposition}\label{prop:ineq_moment_p_01}
Let $\eta_1,\ldots,\eta_n$ be arbitrary (not necessarily independent) real or complex random variables with finite $p$-th absolute moment, where $0< p\leq 1$. Then,
\begin{align}
\label{eq:ineq_moment_p_01}
\E |\eta_1+\ldots+\eta_n|^p \le \sum_{k=1}^n \E |\eta_k|^p.
\end{align}
\end{proposition}
Von Bahr and Esseen~\cite{von_bahr_esseen} showed that up to a multiplicative
constant, inequality \eqref{eq:ineq_moment_p_01} remains true for $1\leq p\leq 2$, if we additionally
assume that the variables are independent and centered.
\begin{proposition}[von Bahr--Esseen inequality~\cite{von_bahr_esseen}]\label{prop:von_bahr_esseen}
Let $\eta_1,\ldots,\eta_n$ be centered, independent real or complex random
variables with finite $p$-th absolute moment, where $1\leq p\leq 2$. Then,
\begin{align}
\label{eq:ineq_moment_p_02}
\E |\eta_1+\ldots+\eta_n|^p \leq 2 \sum_{k=1}^n \E |\eta_k|^p.
\end{align}
\end{proposition}
We need a similar estimate for not necessarily centered random variables.
\begin{proposition}\label{prop:von_bahr_esseen_non_centered}
Let $\eta_1,\ldots,\eta_n$ be independent real or complex random variables with finite $p$-th absolute moment, where $1\leq p\leq 2$. Let also $m_k=\E \eta_k$ and $m=m_1+\ldots+m_n$. Then,
$$
\E |\eta_1+\ldots+\eta_n|^p \leq
2^{2p-1}\sum_{k=1}^n \E |\eta_k|^p + 2^{p-1} |m|^p.
$$
\end{proposition}

\begin{proof}
The random variables  $\tilde \eta_k:=\eta_k-m_k$ are centered. Using Jensen's inequality~\eqref{eq:jensen_ineq}  and applying the von Bahr--Esseen inequality to $\tilde \eta_k$, we obtain
\begin{align*}
\E |\eta_1+\ldots+\eta_n|^p
\leq 2^{p-1} \E |\tilde \eta_1+\ldots+\tilde \eta_n|^p + 2^{p-1} |m|^p
\leq 2^{p} \sum_{k=1}^n \E |\tilde\eta_k|^p+ 2^{p-1}|m|^p.
\end{align*}
The proof is completed by noting that $\E |\tilde\eta_k|^p \leq  2^p \E |\eta_k|^p$ by Lemma~\ref{lem:centered_moment_ineq}.
\end{proof}
For $p=2$, the von Bahr--Esseen inequality is trivially satisfied by the
additivity property of the variance, however, for $p>2$, it is, in general, not
valid. Instead, we have the following result.
\begin{proposition}[Rosenthal inequality~\cite{rosenthal}]\label{prop:rosenthal}
Let $\eta_1,\ldots,\eta_n$ be centered, independent real or complex random variables with finite $p$-th absolute moment, where $p\geq 2$. Then,
$$
\E |\eta_1+\ldots+\eta_n|^p \leq K_p \max \left\{\sum_{k=1}^n \E |\eta_k|^p, \sum_{k=1}^n (\E |\eta_k|^2 )^{p/2}  \right\},
$$
where  $K_p$ is a constant depending only on $p$ (and not depending on $n$ or on the distribution of  the $\eta_k$'s).
\end{proposition}
We need a version of this inequality which is valid for not necessarily centered random variables.
\begin{proposition}\label{prop:rosenthal_non_centered}
Let $\eta_1,\ldots,\eta_n$ be independent real or complex random variables with finite $p$-th absolute moment, where $p\geq 2$. Let also $m_k=\E \eta_k$ and $m=m_1+\ldots+m_n$. Then,
$$
\E |\eta_1+\ldots+\eta_n|^p \leq K_p' \max \left\{\sum_{k=1}^n \E |\eta_k|^p, \left(\sum_{k=1}^n \E |\eta_k|^2 \right)^{p/2}  \right\}
+ K_p' |m|^p,
$$
where  $K_p'$ is a constant depending only on $p$ (and not depending on $n$ or on the distribution of  the $\eta_k$'s).
\end{proposition}
\begin{proof}
The random variables  $\tilde \eta_k:=\eta_k-m_k$ are centered.
Using Jensen's inequality~\eqref{eq:jensen_ineq}  and applying the Rosenthal inequality to $\tilde \eta_k$, we obtain
\begin{align*}
\E |\eta_1+\ldots+\eta_n|^p
&\leq 2^{p-1} \E |\tilde \eta_1+\ldots+\tilde \eta_n|^p + 2^{p-1} |m|^p\\
&\leq 2^{p-1} K_p \max\left\{\sum_{k=1}^n \E |\tilde\eta_k|^p, \left(\sum_{k=1}^n \E |\tilde\eta_k|^2 \right)^{p/2} \right\}+ 2^{p-1}|m|^p.
\end{align*}
To complete the proof, note that $\E |\tilde\eta_k|^p \leq 2^p \E |\eta_k|^p$ by Lemma~\ref{lem:centered_moment_ineq} and that $\E |\tilde \eta_k|^2 = \E |\eta_k|^2 -|m_k|^2 \leq \E |\eta_k|^2$.
\end{proof}

\subsection{Truncated exponential moments of the normal distribution}
In this section we recall several well-known properties of the Gaussian distribution.
Let $\xi\sim N_{\R}(0,1)$ be a real standard normal random variable.  Let
$$
\varphi(t)=\eee^{-\frac 12 t^2},
\quad
\Phi(x)=\frac{1}{\sqrt{2\pi}} \int_{-\infty}^x \eee^{-\frac 12 t^2}\dd t,
\quad
\bar \Phi(x)=1-\Phi(x)
$$
be the density, the distribution function, and the tail function of $\xi$, respectively.
\begin{lemma}[Mills ratio inequality]\label{lem:mills_ratio_ineq}
For every $x>0$, $\bar \Phi(x) \leq \frac 1x \varphi(x)$.
\end{lemma}
\begin{proof}
Using the definition of $\bar \Phi$ and introducing the variable $s:=t-x$ we obtain
$$
\bar \Phi(x)
=  \frac {1}{\sqrt {2\pi}} \int_{x}^{\infty} \eee^{-\frac{t^2} 2} \dd t
=  \varphi(x) \int_{0}^{\infty} \eee^{-\frac{s^2}{2}-xs}\dd s
\leq
\varphi(x) \int_{0}^{\infty} \eee^{-xs}\dd s
=
\frac{\varphi(x)}{x}.
$$
This is the desired inequality.
\end{proof}
The next lemma follows from a simple change of variables; see \cite[Lemma~3.3]{kabluchko_klimovsky}.
\begin{lemma}\label{lem:exp_moment_gauss_eq}
Let $\xi\sim N_{\R}(0,1)$. For every $a\in \R$, $w\in\C$,
\begin{enumerate}
\item[\textup{(1)}] $\E [\eee^{w\xi}\ind_{\xi>a}] = \eee^{\frac{w^2}{2}} \bar \Phi(a-w)$.
\item[\textup{(2)}] $\E [\eee^{w\xi}\ind_{\xi<a}] = \eee^{\frac{w^2}{2}} \Phi(a-w)$.
\end{enumerate}
\end{lemma}

\begin{lemma}\label{lem:exp_moment_gauss_ineq}
Let $\xi\sim N_{\R}(0,1)$ and $a,w\in\R$.
\begin{enumerate}
\item[\textup{(1)}] For $a>w$,
$
\E [\eee^{w(\xi-a)}\ind_{\xi>a}] \leq \frac{1}{\sqrt{2\pi}(a-w)} \eee^{-a^2/2}.
$
\item[\textup{(2)}] For $a<w$,
$
\E [\eee^{w(\xi-a)}\ind_{\xi<a}] \leq \frac{1}{\sqrt{2\pi}(w-a)} \eee^{-a^2/2}.
$
\end{enumerate}
\end{lemma}
\begin{proof}[Proof of~\textup{(1)}]
By Lemma~\ref{lem:exp_moment_gauss_eq}, we have
$
\E [\eee^{w(\xi-a)}\ind_{\xi>a}] = \eee^{\frac{w^2}{2}-aw} \bar \Phi(a-w).
$
Applying Lemma~\ref{lem:mills_ratio_ineq} to the right-hand side, we obtain the desired inequality. The proof of~(2) is analogous.
\end{proof}

The function $\Phi$ admits an analytic continuation to the entire complex plane. The following lemma gives the well-known complex plane asymptotics of $\Phi$. It is standard, see, e.g., \citep[Eq.~7.1.23 on p.~298]{abramowitz_stegun} for~\eqref{eq:Phi_asympt_complex} and~\citep[Chapter IV, Problem 189]{polya_szegoe_book} for~\eqref{eq:Phi_asympt_complex1}. We will sketch the proof, since the lemma will be crucial when establishing the beak shaped form of the phases $G^{d_1} E^{d_3}$.
\begin{lemma}\label{lem:Phi_asympt_complex}
Fix some $\eps>0$. The following asymptotics hold uniformly in the region specified below as $|z|\to\infty$:
\begin{equation}\label{eq:Phi_asympt_complex}
\Phi(z)=
\begin{cases}
-\frac{1+o(1)}{\sqrt{2\pi}z} \eee^{-\frac{z^2}2},
& \text{ if } |\arg z|>\frac{\pi}{4}+\eps,\\
1-\frac{1+o(1)}{\sqrt{2\pi}z} \eee^{-\frac{z^2}2},
& \text{ if } |\arg z|<\frac{3\pi}{4}-\eps.
\end{cases}
\end{equation}
In particular,
\begin{equation}\label{eq:Phi_asympt_complex1}
\Phi(z) \to 
\begin{cases}
1, &\text{ if } |\arg z|<\frac{\pi}{4}-\eps,\\
0, &\text{ if } |\arg z|>\frac{3\pi}{4}+\eps,\\
\infty, &\text{ if } \frac{\pi}{4} + \eps < |\arg z| < \frac{3\pi}{4}-\eps.\\
\end{cases}
\end{equation}
\end{lemma}
\begin{remark}
We take the principal value of the argument, ranging in $(-\pi,\pi]$ and having
a jump discontinuity on the negative half-axis. In the domain $\frac{\pi}{4}+\eps <|\arg
z|<\frac{3\pi}{4}-\eps$ both asymptotics in~\eqref{eq:Phi_asympt_complex} can be
applied and give the same result.
\end{remark}
\begin{proof}[Proof of Lemma~\ref{lem:Phi_asympt_complex}]
We prove the second case of~\eqref{eq:Phi_asympt_complex}. The analytic continuation of the function $\bar \Phi(z)=1-\Phi(z)$ is given by
$$
\bar \Phi(z)= \frac 1 {\sqrt {2\pi}}\int_{\gamma_z} \eee^{-\frac{s^2}{2}}\dd s,
$$
where, for the time being, $\gamma_z$ is the horizontal ray connecting $z$ to $i\Im z + \infty$. However, since the function $\eee^{-s^2/2}$ converges to $0$ exponentially fast for $|\arg s|<\frac{\pi}{4}$, we can rotate $\gamma_z$ by any angle $\theta\in (-\frac{\pi}{4}, \frac {\pi}4)$ without changing the integral. Let us agree to choose $\theta$ in the following way:
$$
\theta=
\begin{cases}
0, &\text{if } |\arg z| < \frac{\pi}{2}-\eps,\\
-\frac{\pi}{4} + \frac {\eps}{2}, &\text{if } \arg z \in (\frac{\pi}{4} + \eps , \frac{3\pi}{4}-\eps)\\
\frac{\pi}{4}  - \frac {\eps}{2}, &\text{if } \arg z \in (-\frac{3\pi}{4} + \eps, -\frac{\pi}{4}-\eps).
\end{cases}
$$
Note that the domain $|\arg z|<\frac{3\pi}{4}-\eps$ is completely covered (with overlaps) by these $3$ cases.  We parametrize the contour $\gamma_z$ as follows:
$$
s = \gamma_z(t) = z +\eee^{i\theta} t/|z|,\;\;\; t\geq 0.
$$
Then,  the integral for $\bar \Phi(z)$ takes the form
$$
\bar \Phi(z)
=
\left(\frac{1}{\sqrt{2\pi} z} \eee^{-\frac{z^2}{2}}\right) \int_{0}^{\infty} \omega(z) \eee^{-\omega(z) t - \frac 12 \eee^{2i\theta} \frac{t^2}{|z|^2}} \dd t
=:\left(\frac{1}{\sqrt{2\pi} z} \eee^{-\frac{z^2}{2}}\right) I(z),
$$
where $\omega(z) = \eee^{i\theta} z/|z|$. The above choice of $\theta$ guarantees that $|\arg \omega(z)|<\frac{\pi}{2}-\frac{\eps}{2}$.  It is an elementary exercise to show that under this restriction,  the integral $I(z)$ converges uniformly to $1$ as $|z|\to\infty$. This proves the second case of~\eqref{eq:Phi_asympt_complex}.

The first case of~\eqref{eq:Phi_asympt_complex} follows from the second case and the formula $\Phi(z)=\bar \Phi(-z)$. To prove~\eqref{eq:Phi_asympt_complex1}, use~\eqref{eq:Phi_asympt_complex} and note that $\eee^{-z^2/2}$ uniformly converges to $0$ as $|z|\to\infty$ in such a way that $|\arg z| < \frac{\pi}{4}-\eps$ or $|\arg z| > \frac{3\pi}{4}+\eps$.
\end{proof}

\subsection{Results on weak convergence}
Let $D\subset \C$ be a connected open set. A family of random continuous or analytic functions on $D$ is called tight if every sequence from this family contains a weakly convergent subsequence. Criteria for tightness in the space $C(D)$ are well-known;
see~\cite[Theorems~8.2, 8.3]{billingsley_book}.  These criteria simplify considerably if we are dealing with \textit{analytic} (rather than merely continuous) functions. The next lemma can be found in~\cite[the remark after Lemma~2.6]{shirai}; see also~\cite[Lemma~4.2]{kabluchko_klimovsky} for a slightly weaker result.
\begin{proposition}\label{prop:tightness_random_analytic}
Let $Z_1,Z_2,\ldots$ be a sequence of random analytic functions on a connected open set $D\subset \C$. Assume that there is a $p>0$ and a locally integrable function $F\colon D\to \R$ such that
$$
\sup_{n\in\N} \E |Z_n(\beta)|^p < F(\beta) \text{ for all } \beta\in D.
$$
Then, the sequence $Z_1,Z_2,\ldots$ is tight on $\HHH(D)$.
\end{proposition}

The next proposition, see~\cite[Lemma~4.3]{kabluchko_klimovsky} or~\cite[Proposition~2.3]{shirai}, states that the weak convergence of random analytic functions implies the weak convergence of the corresponding point processes of zeros. It is essential that the functions are analytic, not merely continuous.
\begin{proposition}\label{prop:weak_conv_zeros}
Let $Z_1,Z_2,\ldots$ be a sequence of random analytic functions on $D$ converging to some random analytic function $Z$ weakly on $\HHH(D)$. Assume that $Z$ is not identically zero, with probability $1$. Then, the following convergence of point processes holds weakly on $\NNN(D)$:
$$
\Zeros \{Z_n(\beta)\colon  \beta\in D\} \toweak \Zeros\{Z(\beta)\colon \beta\in D\}.
$$
\end{proposition}

For the next proposition,  we refer to \cite[Proposition~14.6]{kallenberg_book}.
\begin{proposition}\label{prop:weak_conv_if_on_compact_sets}
A sequence of random continuous functions $Z_1,Z_2,\ldots$ converges weakly to some random continuous function $Z$ on $C(D)$ if and only if for every compact set $K\subset D$ the restriction of $Z_n$ to $K$ converges to the restriction of $Z$ to $K$ weakly on $C(K)$.
\end{proposition}

The next lemma is standard; see Theorem~4.2 on~p.~25 in~\cite{billingsley_book}.
\begin{lemma}\label{lem:interchange_limits}
For every $n\in \N\cup\{\infty\}$, let $\bS_n$ and $\bS_{n,T}$, $T\in \N$, be random elements defined on a probability space $(\Omega_n,\calA_n, \P_n)$ and taking values in a separable metric space $(A,\rho)$. Assume that
\begin{enumerate}

\item[\textup{(1)}] For every $T\in\N$,  $\bS_{n,T}$ converges weakly to $\bS_{\infty,T}$, as $n\to\infty$.

\item[\textup{(2)}] For every $\eps>0$, we have
$
\lim_{T\to\infty} \limsup_{n\to\infty} \P_n[\rho(\bS_{n},\bS_{n,T})>\eps]
=0.
$

\item[\textup{(3)}] For every $\eps>0$, we have
$
\lim_{T\to\infty} \P_{\infty} [\rho(\bS_{\infty},\bS_{\infty, T})>\eps]
=
0.
$

\end{enumerate}
Then, $\bS_{n}$ converges weakly to  $\bS_{\infty}$, as $n\to\infty$.
\end{lemma}
\begin{remark}
Lemma~\ref{lem:interchange_limits} is illustrated by the following diagram:
\newarrow{Implies} ===={=>}
$$
\begin{diagram}
&\bS_{n,T} &\rTo^{\mbox{\tiny $w$}}_{\mbox{\tiny $n\to\infty$}}  &\bS_{\infty, T}\\
&\dImplies^{\mbox{\tiny $P$}}_{\mbox{\tiny $\overset{T,n}{\underset{\infty}\Downarrow}$}}    &
&\dTo^{\mbox{\tiny $\overset{T}{\underset{\infty}\downarrow}$}}_{\mbox{\tiny $P$}}              \\
&\bS_n &\rDotsto^{\mbox{\tiny $w$}}_{\mbox{\tiny $n\to\infty$}}   &\bS_{\infty}
\end{diagram}
$$
We will apply Lemma~\ref{lem:interchange_limits} many times in the proofs of functional limit theorems. In our context, $\bS_n$ will be a normalized version of the partition function $\ZZZ_n$, whereas $\bS_{n,T}$ will be a truncated version of $\bS_n$, with $T$ being a truncation parameter.
\end{remark}

The next lemma will be used in the proof of Theorem~\ref{theo:functional_clt_GE_beak_boundary}.
\begin{lemma}\label{lem:weak_conv_infinitesimal_neighborhoof}
Let $Z_1,Z_2\ldots$ be a sequence of random continuous functions on $\C$ converging weakly to a random continuous function $Z$ on $\C$. Let $\beta_*\in \C$ be fixed and let $\beta_n\in\C$ and $q_n\in \C$ be sequences such that $\lim_{n\to\infty} \beta_n=\beta_*$ and $\lim_{n\to\infty} q_n = 0$. Then, weakly on $C(\C)$ it holds that
$$
\{Z_n(\beta_n + q_n t)\colon t\in \C\}\toweak \{Z(\beta_*)\colon t\in\C\}.
$$
\end{lemma}
\begin{proof}
Define the  mappings $\psi_n, \psi \colon C(\C)\to C(\C)$ by
$$
\psi_n(f)(t) = f(\beta_n + q_n t),\;\; \psi(f)(t) = f(\beta_*)\;\;\text{ for } f\in C(\C),\;\; t\in\C.
$$
If $f_n\in C(\C)$ is a sequence converging to $f\in C(\C)$ locally uniformly, then it is easy to check that $\psi_n(f_n)$ converges to $\psi(f)$ locally uniformly. By the continuous mapping theorem, see Theorem~3.27 in~\cite{kallenberg_book}, we obtain that $\psi_n(Z_n)$ converges to $\psi(Z)$ weakly on $C(\C)$. This proves the lemma.
\end{proof}
\subsection{Central limit theorems for triangular arrays of random vectors}
We will often use classical results on limiting distributions for sums of independent random vectors. Specifically, to prove central limit theorems in the case $|\sigma|<\frac{\sigma_1}{2}$ we will need Lyapunov's central limit theorem. Let $k_n\in \N$ be a sequence such that $\lim_{n\to\infty} k_n=\infty$.
\begin{theorem}
\label{theo:clt_lyapunov_vectors}
For every $n\in\N$, let $\{\bZ_{n,k}\colon 1\leq k\leq k_n\}$ be independent $\R^m$-valued random vectors written as  $\bZ_{n,k}=\{\bZ_{n,k}(i)\}_{i=1}^m$.   Let
$
\bS_n^*=\sum_{k=1}^{k_n} (\bZ_{n,k}-\E \bZ_{n,k}).
$
Assume that 
\begin{enumerate}

\item[\textup{(1)}] The covariance matrix of $\bS_n^*$ converges as $n\to\infty$ to some matrix $\Sigma$.

\item[\textup{(2)}] The Lyapunov condition is satisfied: For some $\delta>0$,
\begin{equation}\label{eq:lyapunov_clt_cond}
\lim_{n\to\infty} \sum_{k=1}^{k_n} \E |\bZ_{n,k}(i)|^{2+\delta}=0 \text{ for all } 1\leq i\leq m.
\end{equation}

\end{enumerate}
Then, $\bS_n^*$ converges in distribution to a mean zero Gaussian distribution on $\R^m$ with mean $0$ and covariance matrix $\Sigma$.
\end{theorem}

\begin{remark}
Usually, the Lyapunov condition is stated in the form
\begin{equation}\label{eq:lyapunov_clt_cond_1}
\lim_{n\to\infty} \sum_{k=1}^{k_n} \E |\bZ_{n,k} - \E \bZ_{n,k}|^{2+\delta}=0 \text{ for all } 1\leq i\leq m.
\end{equation}
However, it is easy to see using Lemma~\ref{lem:centered_moment_ineq} that~\eqref{eq:lyapunov_clt_cond} implies~\eqref{eq:lyapunov_clt_cond_1}.
\end{remark}

The following result, see~\cite[Theorem~3.2.2]{meerschaert_book},  is somewhat more general than the Lyapunov (and even Lindeberg) central limit theorem. We will need it in the case $|\sigma|=\frac{\sigma_1}{2}$. We denote by $|\cdot|$ the Euclidean norm in $\R^m$ and by $\Cov Z$ the covariance matrix of an $\R^m$-valued  random vector $Z$.
\begin{theorem}\label{theo:rvaceva}
For every $n\in\N$, let $\{\bZ_{n,k}\colon 1\leq k\leq k_n\}$ be independent $\R^m$-valued random
 vectors. Assume that the following conditions hold:
\begin{enumerate}
\item[\textup{(1)}] \label{cond:gned1} For every $\eps>0$, $\lim_{n\to\infty} \sum_{k=1}^{k_n} \P[|\bZ_{n,k}|>\eps]=0$.
\item[\textup{(2)}] \label{cond:gned2} For some positive semidefinite matrix $\Sigma$,
$$
\Sigma
=
\lim_{\eps\downarrow 0} \limsup_{n\to\infty} \sum_{k=1}^{k_n} \Cov [\bZ_{n,k} \ind_{|\bZ_{n,k}|<\eps}]
=
\lim_{\eps\downarrow 0} \liminf_{n\to\infty} \sum_{k=1}^{k_n} \Cov [\bZ_{n,k} \ind_{|\bZ_{n,k}|<\eps}].
$$
\end{enumerate}
Then, the random vector $\bS_n:=\sum_{k=1}^{k_n} (\bZ_{n,k}-\E [\bZ_{n,k}\ind_{|\bZ_{n,k}|<R}])$ converges weakly to a mean zero  Gaussian distribution on $\R^m$ with covariance matrix $\Sigma$.
Here, $R>0$ is arbitrary.
\end{theorem}

\section{Proof of the central limit theorem in the strip $|\sigma| < \frac {\sigma_1}2$}\label{sec:proof_clt}

\subsection{Proof of Theorem~\ref{theo:clt}}
Let $\beta=\sigma+i\tau\in\C\bsl\{0\}$ be such that $|\sigma|<\frac {\sigma_1}2$. For a complex-valued random variable $Z$, the variance is defined by $\Var Z = \E|Z|^2- |\E Z|^2$. Our aim is to prove Theorem~\ref{theo:clt} which states that
\begin{equation}\label{eq:CLT_restate}
\frac{\ZZZ_n(\beta)-\E \ZZZ_n(\beta)}{\sqrt{\Var \ZZZ_n(\beta)}} \todistr
\begin{cases}
N_{\C}(0,1), & \text{if } \tau\neq 0,\\
N_{\R}(0,1), & \text{if } \tau = 0.
\end{cases}
\end{equation}

The idea of the proof is to split $\ZZZ_n(\beta)$ into the sum of the contributions of the first level multiplied by the contributions of all other levels. We can write
$$
\ZZZ_n(\beta)-\E \ZZZ_n(\beta) = \sum_{k=1}^{N_{n,1}} W_{n,k}^*,
$$
where for every $n\in\N$, $\{W_{n,k}^*\colon 1\leq k\leq N_{n,1}\}$ are i.i.d.\ random variables defined by
\begin{equation}\label{eq:def_W_k_n}
W_{n,k}^*= X_{n,k} Y_{n,k} - \E [X_{n,k} Y_{n,k}].
\end{equation}
Here, $X_{n,k}$ (the contributions of the first level) and $Y_{n,k}$ (the contributions of the remaining levels) are given by
\begin{equation}\label{eq:def_X_k_n_Y_k_n}
X_{n,k}
=
\eee^{\beta \sqrt {n a_1}\,   \xi_{k}},
\quad
Y_{n,k}
=
\sum_{\eps_2=1}^{N_{n,2}}\ldots \sum_{\eps_d=1}^{N_{n,d}} \eee^{\beta \sqrt n (\sqrt a_2\, \xi_{k\eps_2}+\ldots+\sqrt{a_d}\,\xi_{k\eps_2\ldots\eps_d})}.
\end{equation}
Note that for every $k$, the random variable $Y_{n,k}$ has the same structure as $\ZZZ_n(\beta)$
but with $d-1$ instead of $d$ levels.  Also, note that both families $\{X_{n,k}\colon
1\leq k\leq N_{n,1}\}$ and $\{Y_{n,k}\colon 1\leq k\leq N_{n,1}\}$ consist of i.i.d.\
random variables, and that there is no dependence between these families.

Let $z_n^2= \Var \ZZZ_n(\beta)$. Our aim is to show that the random variable $\sum_{k=1}^{N_{n,1}} z_n^{-1} W_{n,k}^*$ converges in distribution to a standard normal random variable (which may be real or complex depending on whether $\tau=0$ or $\tau \neq 0$). We will show that the triangular array
$$
\{z_n^{-1} W_{n,k}^*\colon 1\leq k\leq N_{n,1}, n\in\N\}
$$
satisfies the conditions of the Lyapunov central limit theorem; see Theorem~\ref{theo:clt_lyapunov_vectors}.  We view each
$W_{n,k}^*$ as a two-dimensional random vector $(\Re W_{n,k}^*, \Im W_{n,k}^*)$. To
simplify the notation, let $(W_n^*,X_n,Y_n)$ be random variables having the same
(joint) law as any of the $(W_{n,k}^*, X_{n,k}, Y_{n,k})$. Note that $\E W_n^*=0$,
by~\eqref{eq:def_W_k_n}.


\vspace*{2mm}
\noindent
\textsc{Step 1.} In the first step, we will compute the asymptotics of the covariance matrix of the vector $W_n^*$ given by
$$
\Cov W_n^*=
\begin{pmatrix}
\E [(\Re W_n^*)^2] & \E [\Re W_n^* \, \Im W_n^*]\\
\E [\Re W_n^* \, \Im W_n^*] & \E [(\Im W_n^*)^2]
\end{pmatrix}.
$$  Namely, we will show that
\begin{align}
\lim_{n\to\infty} N_{n,1} z_n^{-2} \Cov W_n^*  &=
\frac 12 \cdot
\begin{pmatrix}
1&0\\
0 & 1
\end{pmatrix},
&\text{if } &\tau\neq 0, \label{eq:CLT_cov_tau_neq_0}\\
\lim_{n\to\infty} N_{n,1} z_n^{-2} \Cov W_n^*  &=
\begin{pmatrix}
1&0\\
0 & 0
\end{pmatrix},
&\text{if } &\tau= 0. \label{eq:CLT_cov_tau_eq_0}
\end{align}
\begin{lemma}\label{lem:CLT_E_W_n_2_abs}
For $\tau\neq 0$, $\E [W_n^*\phantom{}^2] = o(\E |W_n^*|^2)$.
\end{lemma}
\begin{proof}
We have, due to the independence of $X_n$ and $Y_n$,
\begin{align}
\E[ W_n^*\phantom{}^2] &= \E X_n^2 \, \E Y_n^2 -(\E X_n \, \E Y_n)^2, \label{eq:CLT_E_W_n}
\\
\E |W_n^*|^2 &= \E |X_n|^2 \, \E |Y_n|^2 -|\E X_n|^2 \, |\E Y_n|^2. \label{eq:CLT_E_W_n_abs}
\end{align}
We will show that  the term $\E |X_n|^2 \, \E |Y_n|^2$ asymptotically dominates all other terms in these equalities. We have
$$
\E X_n^2= \E \eee^{2\beta \sqrt {n a_1} \xi} = \eee^{2\beta^2 a_1 n},
\quad
\E |X_n|^2= \E \eee^{2\sigma \sqrt {n a_1} \xi} = \eee^{2\sigma^2 a_1 n}.
$$
Since $\Re (\beta^2) = \sigma^2-\tau^2 < \sigma^2$ for $\tau\neq 0$, we have $\E X_n^2=o(\E |X_n|^2)$.
Similarly,
\begin{equation}\label{eq:CLT_E_X_n_square}
|\E X_n|^2= |\eee^{\frac 12 \beta^2 a_1 n}|^2 = \eee^{(\sigma^2-\tau^2)a_1 n} = o(\E |X_n|^2) \text{ for } \beta\neq 0.
\end{equation}
Also, we have the inequalities
$|\E Y_n^2|\leq \E |Y_n|^2$ and $|(\E Y_n)^2|=|\E Y_n|^2\leq \E |Y_n|^2$.
Inserting all these results into~\eqref{eq:CLT_E_W_n} and~\eqref{eq:CLT_E_W_n_abs} yields  $\E [W_n^*\phantom{}^2] = o(\E |W_n^*|^2)$.
\end{proof}

Recall that $z_n^2 = \E  |\sum_{k=1}^{N_{n,1}} W_{n,k}^*|^2$ and $\E W_{n,k}^*=0$. Hence,
\begin{equation}\label{eq:CLT_v_n_2}
z_n^2 = N_{n,1} \E |W_n^*|^2 = N_{n,1} (\E (\Re W_n^*)^2 + \E (\Im W_n^*)^2).
\end{equation}
In the case $\tau=0$, we have $\Im W_n^*=0$ which immediately yields~\eqref{eq:CLT_cov_tau_eq_0}. In the case $\tau\neq 0$, we have by Lemma~\ref{lem:CLT_E_W_n_2_abs},
\begin{align}
\E (\Re W_n^*)^2 - \E (\Im W_n^*)^2 &= \Re \E [W_n^*\phantom{}^2] = o(\E |W_n^*|^2), \label{eq:CLT_Re_E_w_n_2}\\
2 \, \E (\Re W_n^* \, \Im W_n^*) &=  \Im \E [W_n^*\phantom{}^2] = o(\E |W_n^*|^2), \label{eq:CLT_Im_E_w_n_2}
\end{align}
From~\eqref{eq:CLT_v_n_2}, \eqref{eq:CLT_Re_E_w_n_2}, \eqref{eq:CLT_Im_E_w_n_2}, we obtain that~\eqref{eq:CLT_cov_tau_neq_0} holds in the case $\tau\neq 0$.

As a byproduct of Lemma~\ref{lem:CLT_E_W_n_2_abs}, see~\eqref{eq:CLT_E_W_n_abs} and~\eqref{eq:CLT_E_X_n_square}, we proved the following
\begin{lemma}\label{lem:proof_CLT_Var_asympt}
For $\beta\in \C\bsl \{0\}$,
$
\E |W_n^*|^2\sim \E |X_n|^2 \, \E |Y_n|^2
$
and
$$
z_n^2 = \Var \ZZZ_n(\beta)\sim N_{n,1} \, \E |X_n|^2 \, \E |Y_n|^2.
$$
\end{lemma}

\vspace*{2mm}
\noindent
\textsc{Step 2.}
We will now verify the Lyapunov condition. Choose some $2 < p < \frac{\log \alpha_1}{a_1\sigma^2}$. This is possible by the assumption $2|\sigma|<\sigma_1$. We will verify that
\begin{equation}\label{eq:proof_clt_lyapunov}
N_{n,1} \E |W_n^*|^p =o(z_n^p).
\end{equation}
In view of the inequality $|x+y|^p\leq 2^{p-1} (|x|^p+|y|^p)$, it suffices to verify that
\begin{enumerate}
\item [(L1)] $N_{n,1} \E |X_n|^p \E |Y_n|^p = o(N_{n,1}^{p/2} (\E |X_n|^2)^{p/2} (\E |Y_n|^2)^{p/2})$.
\item [(L2)] $\alpha_1^n |\E X_n|^p |\E Y_n|^p = o(N_{n,1}^{p/2} (\E |X_n|^2)^{p/2} (\E |Y_n|^2)^{p/2})$.
\end{enumerate}
Since~(L1) implies~(L2) by the Jensen inequality, we need to verify~(L1) only.
This will be done in Lemmas~\ref{lem:proof_CLT_est_moment_X_n}
and~\ref{lem:proof_CLT_est_moment_Y_n}, below.
\begin{lemma}\label{lem:proof_CLT_est_moment_X_n}
Let $2 < p < \frac{\log \alpha_1}{a_1\sigma^2}$. Then, $N_{n,1} \E |X_n|^p =o(N_{n,1}^{p/2} (\E |X_n|^2)^{p/2})$.
\end{lemma}
\begin{proof}
We have $\E |X_n|^p = e^{\frac  12 \sigma^2 p^2 a_1 n}$ and the lemma follows immediately from the inequality
$$
\log \alpha_1 +\frac 12 \sigma^2 p^2 a_1 < \frac 12 p \log \alpha_1 +\sigma^2 p a_1
$$
which, in turn, is a consequence of the assumption $2 < p < \frac{\log \alpha_1}{a_1\sigma^2}$.
\end{proof}

\begin{lemma}\label{lem:proof_CLT_est_moment_Y_n}
Let $2 < p < \frac{\log \alpha_2}{a_2\sigma^2}$. Then, $\E |Y_n|^p = O((\E |Y_n|^2)^{p/2})$.
\end{lemma}
\begin{proof}
The proof is by induction over $d$. For $d=1$ we have $Y_n=1$ and the statement
is true. Suppose that the inequality is true in the setting of $d$ levels. We need to verify
it for $d+1$ levels. However, the analogue of $Y_n$ for $d+1$ levels is $\ZZZ_n(\beta)$.  That is, we need to show that
\begin{equation}\label{eq:ZZZ_n_moment_est}
\E |\ZZZ_n(\beta)|^p <C  (\E |\ZZZ_n(\beta)|^2)^{p/2} \quad \text{for } 2 < p < \frac{\log \alpha_1}{a_1\sigma^2}.
\end{equation}
To this end, we apply Proposition~\ref{prop:rosenthal_non_centered} to the random variables $\eta_{k}=X_{k,n}Y_{k,n}$:
$$
\E |\ZZZ_n(\beta)|^p \leq K_p' \max\{ N_{n,1} \E |X_n|^p \E |Y_n|^p, \left(N_{n,1} \E |X_nY_n|^2\right)^{p/2}\}
+ |\E Z_n(\beta)|^p.
$$
To complete the proof of~\eqref{eq:ZZZ_n_moment_est}, we need to show that
\begin{enumerate}
\item [(A1)] $N_{n,1} \E |X_n|^p \E |Y_n|^p \leq C (\E |\ZZZ_n(\beta)|^2)^{p/2}$.
\item [(A2)] $N_{n,1} \E |X_n|^2 \E |Y_n|^2 \leq C \E |\ZZZ_n(\beta)|^2$.
\item [(A3)] $|\E \ZZZ_n(\beta)|^p \leq C (\E |\ZZZ_n(\beta)|^2)^{p/2}$.
\end{enumerate}
Condition~(A1) follows from the induction assumption together with Lemma~\ref{lem:proof_CLT_est_moment_X_n}.
Condition~(A3) is an immediate consequence of Jensen's inequality (since $p>2$). The left-hand side of~(A2) is asymptotic to $\Var \ZZZ_n(\beta)$ by Lemma~\ref{lem:proof_CLT_Var_asympt}, which proves Condition~(A2).
\end{proof}

\section{Proof of the central limit theorem for $|\sigma| = \frac {\sigma_1}2$}\label{sec:proof_clt_boundary}
\subsection{Proof of Theorem~\ref{theo:clt_boundary}}
Let $\beta=\beta(n)=\sigma(n)+i\tau$ be such that $\tau\in\R$ is constant and $\sigma=\sigma(n)$ depends on $n$. Assume  that for some $u\in\R$,
\begin{equation}\label{eq:sigma_boundary_CLT_restate}
\sigma(n) = \frac{\sigma_1}{2} - \frac{u}{2\sqrt{na_1}} + o\left(\frac 1 {\sqrt n} \right).
\end{equation}
Our aim is to prove Theorem~\ref{theo:clt_boundary} which states that
\begin{equation}\label{eq:CLT_boundary_restate}
\frac{\ZZZ_n(\beta)-\E \ZZZ_n(\beta)}{\sqrt{\Var \ZZZ_n(\beta)}} \todistr
\begin{cases}
N_{\C}(0,\Phi(u)), & \text{if } \tau \neq  0,\\
N_{\R}(0,\Phi(u)), & \text{if } \tau = 0.
\end{cases}
\end{equation}

\vspace*{2mm}
\noindent
\textsc{Step 0.}
We start by introducing some notation. As in Section~\ref{sec:proof_clt}, we represent  $\ZZZ_n(\beta)$ as a sum of the contributions of the first level multiplied by the contributions of all other levels:
$$
\ZZZ_n(\beta) = \sum_{k=1}^{N_{n,1}} W_{n,k}, \quad  W_{n,k}= X_{n,k} Y_{n,k},
$$
where $X_{n,k}=\eee^{\beta \sqrt{na_1}\xi_k}$ (the contributions of the first level) and $Y_{n,k}$ (the contributions of the remaining levels) are defined in the same way as in~\eqref{eq:def_X_k_n_Y_k_n}. Define
\begin{equation}\label{eq:def_Xnk_Ynk_clt_bound}
X_{n,k}'=\frac{X_{n,k}}{\sqrt{\E |X_{n,k}|^2}} = \eee^{\beta\sqrt{na_1} \xi_k - \sigma^2 na_1},
\quad
Y_{n,k}'=\frac{Y_{n,k}}{\sqrt{\E |Y_{n,k}|^2}}.
\end{equation}
We write $X_n = \eee^{\beta \sqrt{na_1} \xi}, Y_n, X_n', Y_n', W_n$ for random variables having the same distribution as $X_{n,k},Y_{n,k}, X_{n,k}',Y_{n,k}',W_{n,k}$. Note that by Lemma~\ref{lem:proof_CLT_Var_asympt} (which holds locally uniformly in $\beta\in \C\bsl\{0\}$),
\begin{equation}\label{eq:z_n_variance_def_asympt}
z_n^2:= \Var \ZZZ_n(\beta) \sim N_{n,1} \E |X_n|^2\, \E |Y_n|^2.
\end{equation}
As we will see later, the conditions of the Lyapunov (and even Lindeberg)
central limit theorem are not satisfied in the boundary case. Instead, we will
use Theorem~\ref{theo:rvaceva}. In Steps~1--5 below, we will verify the
conditions of Theorem~\ref{theo:rvaceva}  for the array $\{z_n^{-1} W_{n,k} \colon
1\leq k\leq N_{n,1}\}$. The proof will be completed in Step 6.

\vspace*{2mm}
\noindent
\textsc{Step 1.} We will verify the second condition of Theorem~\ref{theo:rvaceva} by showing that  for every $\eps>0$,
\begin{align}
\lim_{n\to\infty} N_{n,1}  \Cov \left(z_n^{-1}W_n \ind_{|z_n^{-1}W_n|<\eps}\right)  &=
\frac {\Phi(u)} 2 \cdot
\begin{pmatrix}
1&0\\
0 & 1
\end{pmatrix},
&\text{ if } &\tau\neq 0, \label{eq:CLT_bound_cov_tau_eq_0}\\
\lim_{n\to\infty}  N_{n,1}  \Cov \left(z_n^{-1}W_n \ind_{|z_n^{-1}W_n|<\eps}\right) &=
\Phi(u) \cdot \begin{pmatrix}
1&0\\
0 & 0
\end{pmatrix},
&\text{ if } &\tau= 0. \label{eq:CLT_bound_cov_tau_neq_0}
\end{align}
Here, we consider $W_n$ as a vector with values in $\C\equiv \R^2$ and $\Cov$ denotes the covariance matrix. To prove~\eqref{eq:CLT_bound_cov_tau_eq_0} and~\eqref{eq:CLT_bound_cov_tau_neq_0}, it suffices to show that
\begin{align}
&\lim_{n\to\infty} N_{n,1} \E \left[|z_n^{-1} W_n|^2 \ind_{|z_n^{-1} W_n|<\eps}\right]=\Phi(u), &\text{for } &\tau\in\R, \label{eq:clt_bound_show1}\\
&\lim_{n\to\infty} N_{n,1} \E \left[(z_n^{-1} W_n)^2 \ind_{|z_n^{-1} W_n|<\eps}\right]=0, & \text{for } &\tau\neq 0,\label{eq:clt_bound_show2}\\
&\lim_{n\to\infty} \sqrt {N_{n,1}}\, \E \left[|z_n^{-1} W_n| \ind_{|z_n^{-1} W_n|<\eps}\right]=0, & \text{for } &\tau\in\R. \label{eq:clt_bound_show3}
\end{align}
We will prove~\eqref{eq:clt_bound_show1}, \eqref{eq:clt_bound_show2}, \eqref{eq:clt_bound_show3} in Steps 2 and 3. Note in passing that~\eqref{eq:clt_bound_show1} shows that the Lindeberg condition is not satisfied. (For the Lindeberg condition to hold, the limit in~\eqref{eq:clt_bound_show1} should be $1$). This is why we are using Theorem~\ref{theo:rvaceva}.

\vspace*{2mm}
\noindent
\textsc{Step 2.}
We prove~\eqref{eq:clt_bound_show1} and~\eqref{eq:clt_bound_show2}. In view of~\eqref{eq:z_n_variance_def_asympt}, it is sufficient to show that for every $\eps>0$,
\begin{align}
&\lim_{n\to\infty} \E \left[\left|X_n'Y_n'\right|^2 \ind_{|X_n'Y_n'|<\eps \sqrt{N_{n,1}}}\right]=\Phi(u), &\text{for } &\tau\in\R, \label{eq:clt_bound_tech1_0}\\
&\lim_{n\to\infty} \E \left[(X_n'Y_n')^2 \ind_{|X_n'Y_n'|<\eps \sqrt{N_{n,1}}}\right]=0, & \text{for } &\tau\neq 0. \label{eq:clt_bound_tech1_0_tilde}
\end{align}
Conditioning on $Y_n'=y\in\C$, using the total expectation formula and introducing the notation
\begin{align}
f_n(y)&=|y|^2\, \E \left[|X_n'|^2 \ind_{|X_n'|< \eps |y|^{-1} \sqrt {N_{n,1}}}\right],\label{eq:clt_bound_def_fn}\\
\tilde f_n (y)&=y^2\, \E \left[ (X_n')^2 \ind_{|X_n'|< \eps |y|^{-1} \sqrt {N_{n,1}}}\right], \label{eq:clt_bound_def_fn_tilde}
\end{align}
we can write~\eqref{eq:clt_bound_tech1_0} and~\eqref{eq:clt_bound_tech1_0_tilde} as
\begin{align}
&\lim_{n\to\infty} \E f_n(Y_n')=\Phi(u), &\text{for } \tau\in\R&, \label{eq:clt_bound_tech1}\\
&\lim_{n\to\infty} \E \tilde f_n(Y_n')=0, &\text{for } \tau\neq 0&. \label{eq:clt_bound_tech1_tilde}
\end{align}
The proof of~\eqref{eq:clt_bound_tech1} and~\eqref{eq:clt_bound_tech1_tilde} follows from Steps 2A, 2B, 2C below.

\vspace*{2mm}
\noindent
\textsc{Step 2A.} We show that for every $A>1$,
\begin{align}
&\lim_{n\to\infty} \E \left[f_n(Y_n') \ind_{|Y_n'|\in [A^{-1},A]} \right] = \Phi(u) \, \E \left[ |Y_n'|^2 \ind_{|Y_n'|\in [A^{-1}, A]}\right], &\text{for } \tau\in\R,& \\
&\lim_{n\to\infty} \E \left[ \tilde f_n (Y_n') \ind_{|Y_n'|\in [A^{-1},A]} \right] = 0, &\text{for } \tau\neq 0.&
\end{align}
It suffices to show that uniformly in $|y|\in [A^{-1}, A]$,
\begin{align}
&\lim_{n\to\infty} f_n(y) = |y|^2 \Phi(u),  &\text{for } \tau\in\R&,\\
&\lim_{n\to\infty} \tilde f_n(y) = 0, &\text{for } \tau\neq 0&.
\end{align}
Equivalently, we need to show that uniformly in $c\in [A^{-1},A]$ (where $A>1$ may be different now),
\begin{align}
&\lim_{n\to\infty}  \E \left[|X_n'|^2 \ind_{|X_n'|< c \sqrt{N_{n,1}}} \right] = \Phi(u), &\text{for } \tau\in\R,& \label{eq:clt_bound_tech3} \\
&\lim_{n\to\infty}  \E \left[ (X_n')^2 \ind_{|X_n'|< c \sqrt{N_{n,1}}} \right] = 0, &\text{for } \tau\neq 0.& \label{eq:clt_bound_tech3_tilde}
\end{align}
The inequality $|X_n'|<c \sqrt{N_{n,1}}$ is equivalent to $\xi<a_n$, where
\begin{equation}\label{eq:clt_bound_def_an}
a_n=\frac{\frac 12 \log N_{n,1} + \log c + \sigma^2 na_1}{\sigma \sqrt{na_1}}.
\end{equation}

\vspace*{2mm}
\noindent
\textit{Proof of~\eqref{eq:clt_bound_tech3}.} By definition of $X_n'$, see~\eqref{eq:def_Xnk_Ynk_clt_bound}, and Lemma~\ref{lem:exp_moment_gauss_eq} we have
$$
\E \left[|X_n'|^2 \ind_{|X_n'|< c \sqrt{N_{n,1}}} \right] = \E [\eee^{2\sigma \sqrt{na_1} \xi - 2\sigma^2 na_1} \ind_{\xi<a_n}] = \Phi(a_n-2\sigma \sqrt{na_1}).
$$
Using~\eqref{eq:clt_bound_def_an}, \eqref{eq:asympt_N_nk}, \eqref{eq:sigma_boundary_CLT_restate}, we obtain
\begin{equation}\label{eq:clt_bound_asympt_an}
a_n-2\sigma \sqrt{na_1}
=
\frac{na_1 (\frac 14 \sigma_1^2 - \sigma^2) + o(\sqrt n) + \log c}{\sigma \sqrt{na_1}}
=
u+o(1).
\end{equation}
This holds uniformly in $c\in [A^{-1}, A]$. We arrive at~\eqref{eq:clt_bound_tech3}.

\vspace*{2mm}
\noindent
\textit{Proof of~\eqref{eq:clt_bound_tech3_tilde}.}
By definition of $X_n'$, see~\eqref{eq:def_Xnk_Ynk_clt_bound}, and Lemma~\ref{lem:exp_moment_gauss_eq} we have
$$
\E \left[(X_n')^2 \ind_{|X_n'|< c \sqrt{N_{n,1}}} \right] = \E [\eee^{2 \beta \sqrt{na_1} \xi - 2\sigma^2 na_1} \ind_{\xi<a_n}] = \eee^{2(\beta^2-\sigma^2)na_1} \Phi(a_n-2\beta \sqrt{na_1}).
$$
Now, we have $\Re (\beta^2 - \sigma^2)=-\tau^2<0$ since $\tau\neq 0$. For the same reason,  we have $\lim_{n\to\infty} \Im (a_n-2\beta \sqrt{na_1}) = \pm \infty$ (depending on the sign of $\tau$) and it follows from \eqref{eq:clt_bound_asympt_an} and Lemma~\ref{lem:Phi_asympt_complex} that
$$
\lim_{n\to\infty} \Phi(a_n-2\beta \sqrt{na_1})=0 \text{ or } 1.
$$
This implies~\eqref{eq:clt_bound_tech3_tilde}.

\vspace*{2mm}
\noindent
\textsc{Step 2B.} We show that
\begin{equation}\label{eq:clt_bound_tech4}
\lim_{A\to\infty} \limsup_{n\to \infty} \E [f_n(Y_n') \ind_{|Y_n'|<A^{-1}}]
=
\lim_{A\to\infty} \limsup_{n\to \infty} |\E [\tilde f_n(Y_n') \ind_{|Y_n'|<A^{-1}}]|
 = 0.
\end{equation}
Clearly, $|\tilde f_n(y)|\leq f_n(y)\leq |y|^2$,  since $\E |X_n'|^2=1$ by definition. It follows that $\E [f_n(Y_n') \ind_{|Y_n'|<A^{-1}}] \leq A^{-2}$ and similarly with $\tilde f_n$ instead of $f_n$. This implies~\eqref{eq:clt_bound_tech4}.

\vspace*{2mm}
\noindent
\textsc{Step 2C.}  We show that
\begin{equation}\label{eq:clt_bound_tech5}
\lim_{A\to\infty} \limsup_{n\to \infty} \E [f_n(Y_n') \ind_{|Y_n'|>A}]
=
\lim_{A\to\infty} \limsup_{n\to \infty} |\E [\tilde f_n(Y_n') \ind_{|Y_n'|>A}]|
=
0.
\end{equation}
Note that $Y_n$ is the analogue of $\ZZZ_n(\beta)$ with $d-1$ levels. Since the
smallest inverse critical temperature for $Y_n$ is $\sigma_2$ and
$\sigma<\frac{\sigma_2}{2}$, there is $p=2+\delta>2$ such that $\E
|Y_n'|^{2+\delta} < C$ for all $n\in\N$; see~\eqref{eq:ZZZ_n_moment_est}. From
$|\tilde f_n(y)|\leq f_n(y)\leq |y|^2$, it follows that for all $n\in\N$,
$$
|\E [\tilde f_n(Y_n') \ind_{|Y_n'|>A}]|
\leq
\E [f_n(Y_n') \ind_{|Y_n'|>A}] \leq  \E [|Y_n'|^2 \ind_{|Y_n'|>A}] \leq  A^{-\delta} \E |Y_n'|^{2+\delta} < C A^{-\delta}.
$$
This implies~\eqref{eq:clt_bound_tech5}.

\vspace*{2mm}
\noindent
\textsc{Step 3.}
We prove~\eqref{eq:clt_bound_show3}. By~\eqref{eq:z_n_variance_def_asympt}, it suffices to show that for every $\eps>0$,
\begin{equation}\label{eq:clt_bound_tech7}
\lim_{n\to\infty} \E \left[|X_n'Y_n'| \ind_{|X_n'Y_n'| < \eps \sqrt {N_{n,1}}}\right] = 0.
\end{equation}
Take some $A>1$. We can consider the cases $|Y_n'|<A^{-1}$ and $|Y_n'|\geq A^{-1}$ separately to obtain the estimate
\begin{align*}
\E \left[|X_n'Y_n'| \ind_{|X_n'Y_n'| < \eps \sqrt {N_{n,1}}}\right]
&\leq
A^{-1} \E |X_n'| + \E \left[|X_n'Y_n'| \ind_{|X_n'| < A \eps \sqrt {N_{n,1}}}\right]\\
&\leq
A^{-1} +  \E \left[|X_n'| \ind_{|X_n'| < A \eps \sqrt {N_{n,1}}}\right],
\end{align*}
where in the second line we have used that $\E |X_n'|\leq 1$ and $\E |Y_n'|\leq 1$ since $\E |X_n'|^2=\E |Y_n'|^2=1$ by definition. Regarding the expectation on the right-hand side we obtain, by the definition of $X_n'$ and Lemma~\ref{lem:exp_moment_gauss_eq},
$$
\E \left[|X_n'| \ind_{|X_n'| < A \eps \sqrt {N_{n,1}}}\right]
=
\E [\eee^{\sigma \sqrt{na_1} \xi - \sigma^2 na_1} \ind_{\xi <a_n}]
=
\eee^{-\frac 12 \sigma^2 na_1} \Phi(a_n-\sigma \sqrt {na_1}).
$$
Here, we defined $a_n$ as in~\eqref{eq:clt_bound_def_an} with $c= A\eps$. Since
we can estimate $\Phi$ by $1$, the right-hand side converges to $0$, as
$n\to\infty$.  Combining everything together and letting $A\to\infty$, we
obtain~\eqref{eq:clt_bound_tech7}.

\vspace*{2mm}
\noindent
\textsc{Step 4.}
We show that, for every $\eps>0$,
\begin{equation}\label{eq:clt_bound_show4}
\lim_{n\to\infty} N_{n,1} \E \left[|z_n^{-1} W_n| \ind_{|z_n^{-1} W_n|>\eps}\right]=0.
\end{equation}
This statement will be needed to replace the truncated expectation by the usual
one in Theorem~\ref{theo:rvaceva}. In view
of~\eqref{eq:z_n_variance_def_asympt}, it suffices to show that
\begin{equation}\label{eq:clt_bound_show4_0}
\lim_{n\to\infty} \sqrt{N_{n,1}}\, \E \left[|X_n'Y_n'| \ind_{|X_n'Y_n'| > \eps \sqrt {N_{n,1}}}\right] = 0.
\end{equation}
Introducing the function
$$
g_n(y) = |y| \sqrt{N_{n,1}}\, \E \left[|X_n'|\ind_{|X_n'|> \eps |y|^{-1} \sqrt {N_{n,1}}}\right],
$$
where $y\in\C$, we can rewrite~\eqref{eq:clt_bound_show4_0} in the following form:
\begin{equation}\label{eq:clt_bound_tech2}
\lim_{n\to\infty} \E g_n(Y_n')=0.
\end{equation}
The proof of~\eqref{eq:clt_bound_tech2} will be provided in Steps 4A and 4B, below.

\vspace*{2mm}
\noindent
\textsc{Step 4A.} Fix $\delta\in (\frac 12,1)$. We will show that
\begin{equation}\label{eq:clt_bound_tech8}
\lim_{n\to\infty} \E \left[g_n(Y_n') \ind_{|Y_n'| < \eps \eee^{\delta \sigma^2 n a_1}}\right] = 0. 
\end{equation}
Let $y\in\C$ be such that $|y| < \eps \eee^{\delta \sigma^2 n a_1}$. Defining $a_n$ as in~\eqref{eq:clt_bound_def_an} with $c=\eps |y|^{-1}$, we can write
$$
g_n(y)
=
|y|\sqrt{N_{n,1}}\, \eee^{-\sigma^2 na_1} \E [\eee^{\sigma \sqrt {na_1} \xi} \ind_{\xi>a_n}]
$$
Arguing as in~\eqref{eq:clt_bound_asympt_an}, we have $a_n - \sigma \sqrt{na_1} > \eta \sqrt n$, for some $\eta>0$ and all sufficiently large $n\in \N$. Here, we used that $\delta<1$. Hence, using Lemma~\ref{lem:exp_moment_gauss_ineq}, Part~1, inserting the value of $a_n$, see~\eqref{eq:clt_bound_def_an}, and doing elementary transformations, we arrive at
\begin{align*}
g_n(y)
&\leq
\frac{C |y|}{\sqrt n}  \sqrt{N_{n,1}}\, \eee^{-\sigma^2 na_1} \eee^{\sigma \sqrt{na_1} a_n - \frac 12 a_n^2}\\
&\leq
\frac{C|y|}{\sqrt n} \eee^{  -\frac {1}{2\sigma^2 na_1} \left ( \left( \frac 12 \log N_{n,1} - \sigma^2 na_1\right)^2 + \log^2 c + (\log N_{n,1}) (\log c)\right)}\\
&\leq
\frac{C}{\sqrt n} \eps^{-\frac{\log N_{n,1}}{2\sigma^2 na_1}} |y|^{1 + \frac{\log N_{n,1}}{2\sigma^2 na_1}} ,
\end{align*}
where in order to obtain the last inequality we used the non-negativity of the squares. By~\eqref{eq:asympt_N_nk} and~\eqref{eq:sigma_boundary_CLT_restate}, we have $\frac 12 \log N_{n,1} = \sigma^2 na_1+ O(\sqrt n)$. Take some $p>2$.  For sufficiently large $n$, we obtain the estimate
$$
\E \left[g_n(Y_n') \ind_{|Y_n'| < \eps \eee^{\delta \sigma^2 n a_1}}\right]
\leq
\frac {C(\eps)} {\sqrt n}\,  \E |Y_n'\vee 1|^p
\leq
\frac {C(\eps)} {\sqrt n}\,  (\E |Y_n'|^p+1)
.
$$
By Lemma~\ref{lem:proof_CLT_est_moment_Y_n}, the expectation on the right-hand side is bounded by a constant not depending on $n$ if provided that $p$ is sufficiently close to $2$.
This completes the proof of~\eqref{eq:clt_bound_tech8}.

\vspace*{2mm}
\noindent
\textsc{Step 4B.}
In this step, we show that
\begin{equation}\label{eq:clt_bound_tech11}
\lim_{n\to\infty} \E \left[g_n(Y_n') \ind_{|Y_n'| \geq  \eps \eee^{\delta \sigma^2 n a_1}}\right]=0.
\end{equation}
Using the definition of the function $g_n$ and the inequality $\E|X_n'|\leq 1$, we obtain the estimate $g_n(y)\leq |y|  \sqrt{N_{n,1}}$ for all $y\in \C$.  Hence,
$$
\E \left[g_n(Y_n') \ind_{|Y_n'| \geq  \eps \eee^{\delta \sigma^2 n a_1}}\right]
\leq
\sqrt {N_{n,1}}\, \E \left[|Y_n'| \ind_{|Y_n'| \geq  \eps \eee^{\delta \sigma^2 n a_1}}\right].
$$
Using the fact that $\E |Y_n'|^2=1$ and $\frac 12 \log N_{n,1} = \sigma^2 na_1+ O(\sqrt n)$, see~\eqref{eq:asympt_N_nk} and~\eqref{eq:sigma_boundary_CLT_restate}, we obtain that
$$
\E \left[g_n(Y_n') \ind_{|Y_n'| \geq  \eps \eee^{\delta \sigma^2 n a_1}}\right]
\leq
\eps^{-1} \eee^{-\delta \sigma^2 na_1} \cdot  \eee^{\frac 12 \sigma^2 na_1+O(\sqrt n)},
$$
which goes to $0$ as $n\to\infty$ since $\delta>\frac 12$. This proves~\eqref{eq:clt_bound_tech11}.

\vspace*{2mm}
\noindent
\textsc{Step 5.} It follows immediately from Step 4 that, for every $\eps>0$,
$$
\lim_{n\to\infty} N_{n,1} \P[|z_n^{-1} W_n| >  \eps] = 0.
$$
This verifies the first condition of Theorem~\ref{theo:rvaceva}.

\vspace*{2mm} \noindent \textsc{Step 6.} After we have verified the conditions
of Theorem~\ref{theo:rvaceva} for the array $\{z_n^{-1}W_{n,k}\colon 1\leq k\leq
N_{n,1}\}$, we can complete the proof of Theorem~\ref{theo:clt_boundary} as
follows. By Theorem~\ref{theo:rvaceva}, we have
\begin{equation}\label{eq:clt_bound_tech10}
z_n^{-1} \sum_{k=1}^{N_{n,1}} (W_{n,k} - \E [W_{n,k}\ind_{|z_n^{-1} W_{n,k}|<1}]) \todistr
\begin{cases}
N_{\C}(0,\Phi(u)), & \text{if } \tau \neq  0,\\
N_{\R}(0,\Phi(u)), & \text{if } \tau = 0.
\end{cases}
\end{equation}
Note that the covariance structure of the limiting distribution has been
computed in~\eqref{eq:CLT_bound_cov_tau_eq_0}
and~\eqref{eq:CLT_bound_cov_tau_neq_0}. By Step~4,  we can replace the truncated
expectation $\E [W_{n,k}\ind_{|z_n^{-1} W_{n,k}|<1}]$
in~\eqref{eq:clt_bound_tech10} by the usual expectation $\E W_{n,k}$. Recalling
that $\ZZZ_n(\beta)=\sum_{k=1}^{N_{n,1}} W_{n,k}$, we arrive
at~\eqref{eq:CLT_boundary_restate}.

\section{Covariance structure of the partition function}\label{sec:cov}
In this section, we prove asymptotic results on the covariance function of the random field $\ZZZ_n(\beta)$.
In particular, we prove Proposition~\ref{prop:asympt_exp_variance_log_scale}.

\subsection{The variance of the partition function} \label{sec:var}
Fix some $\beta\in\C$. Recall that for $0\leq l\leq d$ we defined
\begin{equation}\label{eq:def_b_l}
b_l = \log \alpha + 2\sigma^2 a + \sum_{m=l+1}^{d} (\log \alpha_m - |\beta|^2 a_m).
\end{equation}
We show that
\begin{equation}\label{eq:var_ZZZ_n_beta_restate}
\Var \ZZZ_n(\beta)
\sim
\begin{cases}
\eee^{b_k n},
&\text{if } \frac{\sigma_k} {\sqrt 2} < |\beta| < \frac{\sigma_{k+1}} {\sqrt 2}, \;\; 1\leq k\leq d,\\
\eee^{b_1 n},
&\text{if } 0 <  |\beta| <  \frac{\sigma_1} {\sqrt 2}.
\end{cases}
\end{equation}
The boundary case $|\beta|=\frac{\sigma_{k}}
{\sqrt 2}$, $1\leq k\leq d$,  will be considered in Remarks~\ref{rem:var_bound} and~\ref{rem:var_bound_k1} below.

\begin{proof}[Proof of~\eqref{eq:var_ZZZ_n_beta_restate}]
Recall that $a=a_1+\ldots+a_d$ is the variance of $X_{\eps}$, $\eps\in \SSS_n$. Define also the ``partial variances'' $A_{l,m}=a_{l}+\ldots+a_m$ for $1\leq l\leq m\leq d$. Let $A_{l,m}=0$ if $l>m$. Our aim is to compute the asymptotics of
$$
\Var \ZZZ_n(\beta)=\E |\ZZZ_n(\beta)|^2- |\E \ZZZ_n(\beta)|^2.
$$
The subsequent estimates will be locally uniform in $\beta$.

\vspace*{2mm}
\noindent
\textsc{Step 1.}
Let us compute $\E |\ZZZ_n(\beta)|^2$ first.  Fix some path $\eta\in \SSS_n$ in the GREM tree, for example the ``left-most'' one $\eta=(1,\ldots, 1)$. Then,
\begin{equation}\label{eq:E_Z_n_Z_n_quer_1}
\E |\ZZZ_n(\beta)|^2 = \E [\ZZZ_n(\beta) \overline {\ZZZ_n(\beta)}]
=
N_n \sum_{\eps\in \SSS_n} \E \eee^{\sqrt n (\beta  X_{\eta} +\bar {\beta} X_{\eps} )}
=
\sum_{l=0}^d B_{n,l},
\end{equation}
where in $B_{n,l}$ we restrict the sum to the paths $\eps\in \SSS_n$ having exactly $l$ common edges with $\eta$. That is,
\begin{align}
B_{n,l}
&=
N_n \sum_{\eps\in \SSS_n\colon l(\eta,\eps)=l} \E \eee^{\sqrt n (\beta  X_{\eta} +\bar {\beta} X_{\eps} )}\label{eq:B_l}\\
&=
N_n \cdot (N_{n, l+1}-1) \cdot N_{n, l+2} \ldots \cdot N_{n, d} \cdot \eee^{2\sigma^2 A_{1,l} n} \eee^{(\sigma^2-\tau^2)A_{l+1,d}n}. \notag
\end{align}
Here, $l(\eta, \eps)=\min\{k\in \N\colon \eps_k\neq 1\}-1$ denotes the number of edges which are common to $\eps$ and $\eta$ and we used the fact that
$$
\Var [\beta  X_{\eta} +\bar {\beta} X_{\eps}] = 4\sigma^2 A_{1,l} + 2(\sigma^2-\tau^2)A_{l+1,d}.
$$
Recall that $\alpha=\alpha_1\ldots\alpha_d$ and $N_n\sim \alpha^{n}$. It follows that for every $0\leq l\leq d$,
\begin{equation}\label{eq:lim_log_B_n_k}
B_{n,l}
\sim
\exp\left\{n\left(\log \alpha + 2\sigma^2 a + \sum_{m=l+1}^{d} (\log \alpha_m - |\beta|^2 a_m)\right)\right\}
\sim
\eee^{b_l n}.
\end{equation}
Assume now that $\frac{\sigma_{k}}{\sqrt 2} < |\beta| < \frac{\sigma_{k+1}}{\sqrt 2}$, for some $0\leq k\leq d$. (Recall that $\sigma_0=0$ and $\sigma_{d+1}=+\infty$).
Then, $b_k$ is  strictly larger than all $b_l$'s with $l\neq k$.
This is easily seen by noting that $\log \alpha_m - |\beta|^2 a_m$ is negative for $m\leq k$ and  positive for $m\geq k+1$; see~\eqref{eq:beta_k_def}.
Hence, we obtain from~\eqref{eq:E_Z_n_Z_n_quer_1} that
\begin{equation}\label{eq:E_Z_n_Z_n_quer_sim}
\E |\ZZZ_n(\beta)|^2 \sim B_{n,k} \sim \eee^{b_k n},  \quad \text{ for } \quad \frac{\sigma_{k}}{\sqrt 2} < |\beta| < \frac{\sigma_{k+1}}{\sqrt 2}, \quad 0\leq k\leq d.
\end{equation}

\vspace*{2mm}
\noindent
\textsc{Step 2.}
We now show that
\begin{equation}\label{eq:E_Z_n_fluct_vers_expect}
|\E \ZZZ_n(\beta)|^2 = o(\E |\ZZZ_n(\beta)|^2),
\quad \text{ for } \quad \frac{\sigma_{k}}{\sqrt 2} < |\beta| < \frac{\sigma_{k+1}}{\sqrt 2}, \quad 1\leq k\leq d.
\end{equation}
Note that the case $k=0$ is excluded. By Proposition~\ref{prop:asympt_expect} we have
\begin{equation}\label{eq:var_proof_1}
\lim_{n\to\infty} \frac 1n \log  |\E \ZZZ_n(\beta)|^2 =  2\log \alpha + (\sigma^2-\tau^2)a=b_0.
\end{equation}
On the other hand, it follows from the assumption $k\neq 0$ that we have $b_0<b_k$. Hence,
\begin{equation} \label{eq:var_proof_2}
\lim_{n\to\infty}\frac 1n \log  \E |\ZZZ_n(\beta)|^2
=
b_k
>
b_0.
\end{equation}
Combining~\eqref{eq:var_proof_1} and~\eqref{eq:var_proof_2}, we obtain
that~\eqref{eq:E_Z_n_fluct_vers_expect} holds.

\vspace*{2mm}
\noindent
\textsc{Step 3.}
From~\eqref{eq:E_Z_n_Z_n_quer_sim} and~\eqref{eq:E_Z_n_fluct_vers_expect} we
obtain that
\begin{equation}\label{eq:Z_n_var}
\Var \ZZZ_n(\beta) \sim B_{n,k} \sim \eee^{b_k n},  \quad \text{ for } \quad \frac{\sigma_{k}}{\sqrt 2} < |\beta| < \frac{\sigma_{k+1}}{\sqrt 2}, \quad 1\leq k\leq d.
\end{equation}

\vspace*{2mm}
\noindent
\textsc{Step 4.}
Let us finally prove that
\begin{equation}\label{eq:var_tech11}
\Var \ZZZ_n(\beta) \sim B_{n,1} \sim \eee^{b_1 n}, \quad \text{ for } \quad 0 < |\beta| < \frac{\sigma_{1}}{\sqrt 2}.
\end{equation}
Note that the variance is asymptotic to  $B_{n,1}$, not $B_{n,0}$. Of course, we have $\E |\ZZZ_n(\beta)|^2 \sim B_{n,0}$ by~\eqref{eq:E_Z_n_Z_n_quer_sim}, but we will show that the term $B_{n,0}$ cancels in the formula for the variance.
Namely, for $B_{n,0}':=B_{n,0} -  |\E \ZZZ_n(\beta)|^2$ we have
\begin{align}
B_{n,0}'
&= N_n\, (N_{n,1}-1) N_{n,2} \cdot \ldots \cdot N_{n,d} \, \eee^{(\sigma^2-\tau^2) a n} - N_n^2 \eee^{(\sigma^2-\tau^2) a n} \label{eq:def_B_n_0_prime}\\
&= -N_n^2 N_{n,1}^{-1}\, \eee^{(\sigma^2-\tau^2) a  n }. \notag
\end{align}
It follows that
$$
b_0'
:=
\lim_{n\to\infty} \frac 1n \log |B_{n,0}'|
=2\log \alpha-\log \alpha_1 + (\sigma^2-\tau^2) a.
$$
It follows from $|\beta| < \frac{\sigma_{1}}{\sqrt 2}$ that $b_0>b_1>\ldots>b_d$. Therefore, since $\beta\neq 0$,
\begin{equation}\label{eq:asympt_var_proof_b_1_largest}
b_1 = 2\log \alpha + (\sigma^2-\tau^2) a -(\log \alpha_1- |\beta|^2 a_1) > \max\{b_0',b_2,\ldots,b_d\}.
\end{equation}
It follows that the term $B_{n,1}$ has larger order than $B_{n,0}', B_{n,2}, \ldots, B_{n,d}$. Hence,
$$
\Var \ZZZ_n(\beta)
=
\E |\ZZZ_n(\beta)|^2- |\E \ZZZ_n(\beta)|^2
=
B_{n,0}' + \sum_{l=1}^d B_{n,l} \sim B_{n,1}.
$$
Therefore, \eqref{eq:var_tech11} holds.
\end{proof}

\subsection{Local covariance structure inside the rings}\label{subsec:cov} In
Sections~\ref{subsec:cov} and~\ref{subsec:cov_boundary}, we look at the
covariance of the partition function $\ZZZ_n(\beta)$ in a window of
infinitesimal  size near some fixed point $\beta_* \in \C$. We show that
$\ZZZ_n(\beta)$ possesses some nontrivial limiting covariance structure inside
this window. As in Section~\ref{sec:var}, there are phase transitions on the
circles $|\beta_*|=\frac {\sigma_k}{\sqrt 2}$, $1\leq k\leq d$; see
Figure~\ref{fig:clt_simple}, left.

Fix some $\beta_*=\sigma_* + i\tau_*\in\C$.  Define normalizing functions $g_{n,1}(\beta_*;t),\ldots,g_{n,d}(\beta_*;t)$, where $t\in\C$,  by
\begin{equation}\label{eq:def_d_nk_t}
g_{n,l}(\beta_*;t) =
\begin{cases}
\frac 12 \log N_{n,l}+ a_l (\sqrt n \sigma_*+ t)^2,  & \text{if } |\beta_*| > \frac{\sigma_l}{\sqrt 2},\\
\log N_{n,l}  + \frac 12 a_l (\sqrt n \beta_* + t)^2,  & \text{if } |\beta_*| < \frac{\sigma_l}{\sqrt 2}.
\end{cases}
\end{equation}
Let
\begin{equation}\label{eq:g_n_beta_star_t_def}
g_n(\beta_*; t) = g_{n,1}(\beta_*; t)+\ldots+g_{n,d}(\beta_*;t).
\end{equation}
Define
stochastic processes $\{Z_n(t)\colon t\in\C\}$ and $\{Z_n^*(t)\colon t\in\C\}$
by
$$
Z_n(t)= \eee^{-g_n(\beta_*; t)} \ZZZ_n\left(\beta_*+\frac{t}{\sqrt n}\right),
\quad
Z_n^*(t) = Z_n(t) - \E Z_n(t).
$$
\begin{proposition}\label{prop:asympt_cov}
Let $\beta_*\in\C\bsl \R$ be such that for some $1\leq k\leq d$,
\begin{equation}\label{eq:sigma_leq_beta_leq_sigma}
\frac{\sigma_{k}}{\sqrt 2} < |\beta_*| < \frac{\sigma_{k+1}}{\sqrt 2}.
\end{equation}
Then, for every $t_1,t_2,t\in \C$,
\begin{align}
&\lim_{n\to\infty} \E [ Z_n(t_1) \overline{Z_n(t_2)} ] = \lim_{n\to\infty} \E [ Z_n^*(t_1) \overline{Z_n^*(t_2)} ] = \eee^{-\frac 12 (a_1+\ldots+a_k)(t_1-\bar t_2)^2}, \label{eq:cov_1}\\
&\lim_{n\to\infty} \E [ Z_n(t_1) Z_n(t_2) ] = \lim_{n\to\infty} \E [ Z_n^*(t_1) Z_n^*(t_2) ] = 0, \label{eq:cov_2}\\
&\lim_{n\to\infty} \E Z_n(t)  = 0. \label{eq:cov_proof_3}
\end{align}
\end{proposition}
\begin{proof}
To prove the proposition, we need to prove~\eqref{eq:cov_proof_3} and to show that
\begin{align}
&\lim_{n\to\infty} \E [ Z_n(t_1) \overline{Z_n(t_2)} ] = \eee^{-\frac 12(a_1+\ldots+a_k) (t_1-\bar t_2)^2}, \label{eq:cov_proof_1}\\
&\lim_{n\to\infty} \E [ Z_n(t_1) Z_n(t_2) ] = 0.\label{eq:cov_proof_2}
\end{align}
\textit{Proof of~\eqref{eq:cov_proof_1}.} Recall that $\eta=(1,\ldots,1)$ is the left-most path in the GREM tree.  Writing $Z_n(t_1)$ and $Z_n(t_2)$ as sums, see~\eqref{eq:ZZZ_n_beta_def}, and taking the products we obtain
\begin{align*}
\E [ Z_n(t_1) \overline{Z_n(t_2)} ]
&=
\eee^{-g_n(\beta_*;t_1)-\overline {g_n(\beta_*;t_2)}} N_n \sum_{\eps\in \SSS_n} \E \eee^{(\sqrt n \beta_* +t_1)X_{\eta} + ( \sqrt n \bar \beta_* +\bar t_2)X_{\eps}}\\
&=
\eee^{-g_n(\beta_*;t_1)-\overline {g_n(\beta_*;t_2)}} \sum_{l=0}^d B_{n,l}(t_1,t_2),
\end{align*}
where in $B_{n,l}(t_1,t_2)$ we take the sum over all paths $\eps \in \SSS_n$ in the GREM tree having exactly $l$ edges in common with $\eta$, that is
\begin{align}
B_{n,l}(t_1,t_2)
&= N_n  (N_{n,l+1}-1) N_{n,l+2} \ldots N_{n,d} \cdot \E\left[\eee^{(\sqrt n\beta_* + t_1)X_{\eta}+ (\sqrt n\bar \beta_* + \bar t_2)X_{\eps}} \right] \label{eq:def_B_n_l_t1_t2}\\
&\sim N_n N_{n,l+1}\ldots N_{n,d} \cdot  \eee^{\frac 12 (2\sqrt n \sigma_*+ t_1 + \bar t_2)^2 A_{1,l}} \eee^{\frac 12 ((\sqrt n \beta_* + t_1)^2+(\sqrt n \bar\beta_* + \bar t_2)^2)A_{l+1,d}}. \notag
\end{align}
Note that $B_{n,l}(t_1,t_2)$ differs from $B_{n,l}$ by a factor $\eee^{O(\sqrt n)}$; see~\eqref{eq:B_l}. By the same argument as in Section~\ref{sec:var}, condition~\eqref{eq:sigma_leq_beta_leq_sigma} implies that   $B_{n,l}(t_1,t_2)=o(B_{n,k}(t_1,t_2))$ for all $l\neq k$.  It follows that
$$
\E [ Z_n(t_1) \overline{Z_n(t_2)}]
\sim
\eee^{-g_n(\beta_*;t_1)-\overline {g_n(\beta_*;t_2)}} B_{n,k}(t_1,t_2)
\sim
\eee^{-\frac 12 (a_1+\ldots+a_k) (t_1-\bar t_2)^2},
$$
where the last step follows by a simple calculation; see~\eqref{eq:def_d_nk_t}.

\vspace*{2mm}
\noindent
\textit{Proof of~\eqref{eq:cov_proof_2}.}  We have
\begin{align*}
\E [ Z_n(t_1) Z_n(t_2) ]
&=
\eee^{-g_n(\beta_*;t_1)- g_n(\beta_*;t_2)} N_n \sum_{\eps\in \SSS_n} \E \eee^{(\sqrt n \beta_* +t_1)X_{\eta} + (\sqrt n \beta_* + t_2)X_{\eps}}\\
&=
\eee^{-g_n(\beta_*;t_1)- g_n(\beta_*;t_2)} \sum_{l=0}^d C_{n,l}(t_1,t_2),
\end{align*}
where in $C_{n,l}(t_1,t_2)$ we take the sum over all paths $\eps\in\SSS_n$ having exactly $l$ edges in common with $\eta$, that is
\begin{align}
C_{n,l}(t_1,t_2)
&= N_n (N_{n,l+1}-1)N_{n,l+2} \ldots N_{n,d} \cdot \E\left[\eee^{(\sqrt n\beta_* + t_1)X_{\eta}+ (\sqrt n \beta_* + t_2)X_{\eps}} \right]\label{eq:def_C_n_l_t1_t2}\\
&\sim  N_n N_{n,l+1}\ldots N_{n,d} \cdot  \eee^{\frac 12 (2\sqrt n \beta_* + t_1 + t_2)^2 A_{1,l}} \eee^{\frac 12 ((\sqrt n \beta_* + t_1)^2+(\sqrt n \beta_* + t_2)^2) A_{l+1,d}}. \notag
\end{align}
Since $\Re (\beta_*^2) < \sigma_*^2$ by the assumption $\beta_*\notin\R$, we see
that $C_{n,l}(t_1,t_2)=o(B_{n,l}(t_1,t_2))$ for every $1\leq l\leq d$. For
$l=0$, we have $A_{1,l}=0$ and a weaker estimate
$C_{n,l}(t_1,t_2)=O(B_{n,l}(t_1,t_2))$. In the proof of~\eqref{eq:cov_proof_1},
we have shown that $B_{n,l}(t_1,t_2)=o(B_{n,k}(t_1,t_2))$, for all $l\neq k$.
Since the value $l=0$ is not optimal (by the assumption $k\geq 1$), we obtain
that $C_{n,l}(t_1,t_2)=o(B_{n,k}(t_1,t_2))$ for all $0\leq l\leq d$.  This,
together with the result of~\eqref{eq:cov_proof_1},
yields~\eqref{eq:cov_proof_2}.

\vspace*{2mm}
\noindent
\textit{Proof of~\eqref{eq:cov_proof_3}.}
By Proposition~\ref{prop:asympt_expect}, we have
$$
\E Z_n(t)
=
N_n \eee^{-g_n(\beta_*;t)} \eee^{\frac 12 (\sqrt n \beta_* + t)^2 a}
=
\prod_{l=1}^{d}(N_{n,l} \eee^{-g_{n,l}(\beta   _*;t) + \frac 12 n\beta_*^2 a_l +O(\sqrt n)}).
$$
It is easy to check using~\eqref{eq:def_d_nk_t} that, for $l\leq k$, the
corresponding term in the product is $O(e^{-\eps n})$, for some $\eps>0$,
whereas for $l> k$ it is $e^{O(\sqrt n)}$ (in fact, $1$). Since $k\geq 1$, we
have at least one term of the form $O(e^{-\eps n})$. It follows that  the
product converges to $0$.
\end{proof}

Proposition~\ref{prop:asympt_cov} is not valid in the case $k=0$. For this case, we need a slightly different normalization.
Fix some $\beta_*=\sigma_*+i\tau_*\in\C$.  Define $\hat g_{n}(\beta_*;t)$, a modification of $g_{n}(\beta_*;t)$,  by
\begin{equation}\label{eq:def_d_nk_t_hat}
\hat g_{n}(\beta_*;t) =
\left(\frac 12 \log N_{n,1}+ a_1 (\sqrt n \sigma_*+ t)^2\right) +
\sum_{l=2}^d \left(\log N_{n,l}  + \frac 12 a_l (\sqrt n \beta_* + t)^2\right).
\end{equation}
Note that $\hat g_n(\beta_*; t)$ differs from $g_n(\beta_*;t)$ just by the way the first level is normalized.  In the case $k=0$ we define the
stochastic processes $\{Z_n(t)\colon t\in\C\}$ and $\{Z_n^*(t)\colon t\in\C\}$
by
$$
Z_n(t)= \eee^{-\hat g_n(\beta_*; t)} \ZZZ_n\left(\beta_*+\frac{t}{\sqrt n}\right),
\quad
Z_n^*(t) = Z_n(t) - \E Z_n(t).
$$
\begin{proposition}\label{prop:asympt_cov_k0}
Let $\beta_*\in\C\bsl \R$ be such that  $|\beta_*| < \frac{\sigma_{1}}{\sqrt 2}$.
Then, for every $t_1,t_2,t\in \C$,
\begin{align}
&\lim_{n\to\infty} \E [ Z_n^*(t_1) \overline{Z_n^*(t_2)} ] = \eee^{-\frac 12 a_1 (t_1-\bar t_2)^2}, \label{eq:cov_1_k0}\\
&\lim_{n\to\infty} \E [ Z_n^*(t_1) Z_n^*(t_2) ] = 0, \label{eq:cov_2_k0}\\
&\lim_{n\to\infty} \E Z_n(t)  = \infty. \label{eq:cov_proof_3_k0}
\end{align}
\end{proposition}

\begin{remark}
It follows from~\eqref{eq:cov_proof_3_k0} that~\eqref{eq:cov_1_k0} and~\eqref{eq:cov_2_k0} are not valid with $Z_n^*$ replaced by $Z_n$.
\end{remark}

\begin{proof}[Proof of Proposition~\ref{prop:asympt_cov_k0}]
\textit{Proof of~\eqref{eq:cov_1_k0}.}
In the same way as in the proof of Proposition~\ref{prop:asympt_cov}, Eq.~\eqref{eq:cov_proof_1},  we obtain that
\begin{align*}
\E [ Z_n(t_1) \overline{Z_n(t_2)} ]
&=
\eee^{-\hat g_n(\beta_*;t_1)-\overline {\hat g_n(\beta_*;t_2)}} \sum_{l=0}^d B_{n,l}(t_1,t_2) \\
&\sim
\eee^{-\hat g_n(\beta_*;t_1)-\overline {\hat g_n(\beta_*;t_2)}} B_{n,0}(t_1,t_2),
\end{align*}
where $B_{n,l}(t_1,t_2)$ is the same as in that proof. However, we will show that the term $B_{n,0}(t_1,t_2)$ cancels almost completely in the expression
\begin{align*}
\E [ Z_n^*(t_1) \overline{Z_n^*(t_2)}]
&=
\E [ Z_n(t_1) \overline{Z_n(t_2)}] - \E[Z_n(t_1)] \overline{\E [Z_n(t_2)]}\\
&=
\eee^{-\hat g_n(\beta_*;t_1)-\overline {\hat g_n(\beta_*;t_2)}}
\left(B_{n,0}'(t_1,t_2)  + \sum_{l=1}^d B_{n,l}(t_1,t_2)\right),
\end{align*}
where
\begin{align*}
B_{n,0}'(t_1,t_2)
=  B_{n,0}(t_1,t_2) - \E \ZZZ_n\left(\beta_*+ \frac {t_1}{\sqrt n}\right) \overline{\E \ZZZ_n\left(\beta_*+ \frac {t_2}{\sqrt n}\right)}.
\end{align*}
Recalling  the formula for $B_{n,0}(t_1,t_2)$, see~\eqref{eq:def_B_n_l_t1_t2}, and using Proposition~\ref{prop:asympt_expect}, we obtain that
\begin{align}
B_{n,0}'(t_1,t_2)
&=
\left(N_n (N_{n,1}-1)N_{n,2} \ldots N_{n,d} - N_n^2\right) \cdot \eee^{\frac 12 (\sqrt n \beta_* +t_1)^2a+\frac 12 (\sqrt n \bar \beta_* +\bar t_2)^2a } \label{eq:B_n_0_prime_t1_t2}\\
&=
-N_n^2 N_{n,1}^{-1} \eee^{\frac 12 (\sqrt n \beta_* +t_1)^2a+\frac 12 ( \sqrt n \bar \beta_* +\bar t_2)^2a}. \notag
\end{align}
Note that with $B_{n,l}$ as in~\eqref{eq:B_l} and $B_{n,0}'$ as in~\eqref{eq:def_B_n_0_prime},
$$
B_{n,l}(t_1,t_2) = B_{n,l} \eee^{O(\sqrt n)} , \quad 0\leq l\leq d, \quad\quad B_{n,0}'(t_1,t_2)=B_{n,0}'\eee^{O(\sqrt n)}.
$$
Hence, by the argument from Section~\ref{sec:var}, see~\eqref{eq:asympt_var_proof_b_1_largest}, we have
$$
B_{n,l}(t_1,t_2)=o(B_{n,1}(t_1,t_2)),\quad 2\leq l\leq d, \quad\quad
B_{n,0}'(t_1,t_2)=o(B_{n,1}(t_1,t_2)).
$$
It follows that
$$
\E [ Z_n^*(t_1) \overline{Z_n^*(t_2)}]
\sim
\eee^{-\hat g_n(\beta_*;t_1)-\overline {\hat g_n(\beta_*;t_2)}} B_{n,1}(t_1,t_2)
\sim
\eee^{-\frac 12 a_1 (t_1-\bar t_2)^2},
$$
where the last step follows by a simple calculation; see~\eqref{eq:def_B_n_l_t1_t2} and~\eqref{eq:def_d_nk_t_hat}.

\vspace*{2mm}
\noindent
\textit{Proof of~\eqref{eq:cov_2_k0}.}  In the same way as in the proof of Proposition~\ref{prop:asympt_cov}, Eq.~\eqref{eq:cov_proof_2}, we have
\begin{align*}
\E [ Z_n(t_1) Z_n(t_2) ]
=
\eee^{-\hat g_n(\beta_*;t_1)- \hat g_n(\beta_*;t_2)} \sum_{l=0}^d C_{n,l}(t_1,t_2),
\end{align*}
where $C_{n,l}(t_1,t_2)$ is the same as in that proof. Hence,
\begin{align*}
\E [ Z_n^*(t_1) Z_n^*(t_2)]
&=
\E [ Z_n(t_1) Z_n(t_2)] - \E[Z_n(t_1)] \E [Z_n(t_2)]\\
&=
\eee^{-\hat g_n(\beta_*;t_1)- \hat g_n(\beta_*;t_2)}
\left(C_{n,0}'(t_1,t_2)  + \sum_{l=1}^d C_{n,l}(t_1,t_2)\right),
\end{align*}
where
$$
C_{n,0}'(t_1,t_2)
=
C_{n,0}(t_1,t_2) - \E \ZZZ_n\left(\beta_*+ \frac {t_1}{\sqrt n}\right) \E \ZZZ_n\left(\beta_*+ \frac {t_2}{\sqrt n}\right).
$$
Using the formula for $C_{n,0}(t_1,t_2)$, see~\eqref{eq:def_C_n_l_t1_t2},  and Proposition~\ref{prop:asympt_expect}, we obtain that
\begin{align}
C_{n,0}'(t_1,t_2)
=
-N_n^2 N_{n,1}^{-1} \eee^{\frac 12 (\beta_* \sqrt n+t_1)^2a+\frac 12 (\beta_* \sqrt n + t_2)^2a }. \label{eq:C_n_0_prime_t1_t2}
\end{align}
Note that $\Re (\beta_*^2) < \sigma_*^2$ by the assumption $\beta_*\notin\R$. In the same way as in the proof of Proposition~\ref{prop:asympt_cov}, we get that $C_{n,l}(t_1,t_2)=o(B_{n,l}(t_1,t_2))$ for every $0\leq l \leq d$ (note that these terms differ from $C_{n,l}$ and $B_{n,l}$ by a factor of $\eee^{O(\sqrt n)}$).  Additionally,  $C_{n,0}'(t_1,t_2)=o(B_{n,0}'(t_1,t_2))$, compare~\eqref{eq:B_n_0_prime_t1_t2} and~\eqref{eq:C_n_0_prime_t1_t2}.  It follows that
$$
\E [ Z_n^*(t_1) Z_n^*(t_2)] = o(\E [ Z_n^*(t_1) \overline{Z_n^*(t_2)}]) =o(1),
$$
where the last step is by~\eqref{eq:cov_1_k0}. This establishes~\eqref{eq:cov_2_k0}.

\vspace*{2mm}
\noindent
\textit{Proof of~\eqref{eq:cov_proof_3_k0}.}
By Proposition~\ref{prop:asympt_expect}, we have
$$
\E Z_n(t)
=
N_n \eee^{-\hat g_n(\beta_*;t)} \eee^{\frac 12 (\sqrt n \beta_* + t)^2 a}
=
\eee^{\frac 12 \log N_{n,1} + \frac 12 (\sqrt n \beta_* + t)^2 a_1 - (\sqrt n \sigma_*+ t)^2 a_1}
$$
The right-hand side goes to $\infty$ by~\eqref{eq:asympt_N_nk} and the assumption $|\beta_*|< \frac{\sigma_1}{\sqrt 2}$.
\end{proof}

\begin{remark}\label{rem:local_cov_real_beta}
In Propositions~\ref{prop:asympt_cov} and~\ref{prop:asympt_cov_k0}, we did not consider the case of real $\beta_*$. If  $\beta_*\in \R$, then the expressions for the limits of $\E[Z_n(t_1) \overline {Z_n(t_2)}]$ and $\E[Z_n^*(t_1) \overline {Z_n^*(t_2)}]$ remain the same, but the limits of $\E [Z_n(t_1) Z_n(t_2)]$ and $\E [Z_n^*(t_1) Z_n^*(t_2)]$ change, namely we have
$$
\E [Z_n(t_1) Z_n(t_2)] = \E[Z_n(t_1) \overline {Z_n(\bar t_2)}],
\qquad
\E [Z_n^*(t_1) Z_n^*(t_2)] = \E[Z_n^*(t_1) \overline {Z_n^*(\bar t_2)}].
$$
The expressions for $\E Z_n(t)$  remain the same.
\end{remark}

\subsection{Local covariance structure on the boundary circles}\label{subsec:cov_boundary} 
In this section, we compute the local covariance structure of the partition function $\ZZZ_n(\beta)$ in a small window around    $\beta_*\in\C$ such that $|\beta_*|=\frac {\sigma_k}{\sqrt 2}$, for some $1\leq k\leq d$. As we shall see, in order to obtain a non-trivial covariance function in the limit, we have to choose the linear size of the window to be of order $\frac 1n$. The results of this section will be needed to prove Theorems~\ref{theo:functional_clt_line of_zeros_d2_geq2} and~\ref{theo:functional_clt_line of_zeros_d2_eq1} which describe the structure of the arc shaped ``curves of zeros''.

Take some $\beta_*=\sigma_*+i\tau_*\in\C$.
Similarly to~\eqref{eq:def_b_l}, we define, for $0\leq l\leq d$,
\begin{equation}\label{eq:def_b_l_2}
b_l=\log \alpha + 2\sigma_*^2 a + \sum_{m=l+1}^d (\log \alpha_m - |\beta_*|^2 a_m ).
\end{equation}
\begin{proposition}\label{prop:asympt_cov_bound}
Let $\beta_*\in \C\bsl \R$ be such that $|\beta_*|=\frac {\sigma_k}{\sqrt 2}$ for some $2\leq k\leq d$.  Define stochastic processes $\{Z_n(t)\colon t\in\C\}$ and $\{Z_n^*(t)\colon t\in\C\}$ by
$$
Z_n(t)= \eee^{-\frac 12 b_k n} \ZZZ_n\left(\beta_*+\frac{t}{n}\right),
\quad
Z_n^*(t) = Z_n(t) - \E Z_n(t).
$$
Then, for every $t_1,t_2,t\in \C$,
\begin{align}
&\lim_{n\to\infty} \E [ Z_n(t_1) \overline{Z_n(t_2)} ] = \lim_{n\to\infty} \E [ Z_n^*(t_1) \overline{Z_n^*(t_2)} ] = \eee^{t_1\lambda_k + \bar t_2 \bar \lambda_k} + \eee^{t_1\lambda_{k-1} + \bar t_2 \bar \lambda_{k-1}},\label{eq:cov_1_bound}\\
&\lim_{n\to\infty} \E [ Z_n(t_1) Z_n(t_2) ] = \lim_{n\to\infty} \E [ Z_n^*(t_1) Z_n^*(t_2) ] = 0, \label{eq:cov_2_bound}\\
&\lim_{n\to\infty} \E Z_n(t)  = 0, \label{eq:cov_proof_3_boundary}
\end{align}
where $\lambda_l=2\sigma_*A_{1,l} + \beta_* A_{l+1,d}$, $1\leq l\leq d$.
\end{proposition}
\begin{remark}\label{rem:var_bound}
In particular, we obtain that under the assumptions of Proposition~\ref{prop:asympt_cov_bound},
$$
\Var \ZZZ_n\left(\beta_* + \frac tn \right) \sim  \eee^{b_k n} (\eee^{2 \Re (\lambda_k t)} + \eee^{2 \Re (\lambda_{k-1} t)}).
$$
\end{remark}
\begin{proof}
To prove the proposition, we need to prove~\eqref{eq:cov_proof_3_boundary} and to show that
\begin{align}
&\lim_{n\to\infty} \E [ Z_n(t_1) \overline{Z_n(t_2)} ] = \eee^{t_1\lambda_k + \bar t_2 \bar \lambda_k} + \eee^{t_1\lambda_{k-1} + \bar t_2 \bar \lambda_{k-1}}, \label{eq:cov_proof_1_boundary}\\
&\lim_{n\to\infty} \E [ Z_n(t_1) Z_n(t_2) ] = 0.\label{eq:cov_proof_2_boundary}
\end{align}
\textit{Proof of~\eqref{eq:cov_proof_1_boundary}.} Let $\eta\in \SSS_n$ be some fixed path in the GREM tree, say $\eta=(1,\ldots,1)$. We have
\begin{align*}
\E [ Z_n(t_1) \overline{Z_n(t_2)} ]
&=
\eee^{-n b_k} N_n \sum_{\eps\in \SSS_n} \E \eee^{  \sqrt n (\beta_* + \frac{t_1} n)X_{\eta} + \sqrt n(\bar \beta_*  +\frac{\bar t_2} n) X_{\eps}}\\
&=
\eee^{-n b_k} \sum_{l=0}^d D_{n,l}(t_1,t_2),
\end{align*}
where in $D_{n,l}(t_1,t_2)$ we take the sum over all paths $\eps\in \SSS_n$ in the GREM tree having exactly $l$ edges in common with $\eta$, that is
\begin{align*}
D_{n,l}(t_1,t_2)
&= N_n (N_{n,l+1}-1) N_{n,l+2} \ldots N_{n,d} \cdot  \E \eee^{  \sqrt n (\beta_* + \frac{t_1} n)X_{\eta} + \sqrt n(\bar \beta_*  +\frac{\bar t_2} n) X_{\eps}} \\
&\sim  N_n N_{n,l+1}\ldots N_{n,d} \cdot  \eee^{\frac n 2 (2 \sigma_*+ \frac 1n (t_1 + \bar t_2))^2 A_{1,l}} \eee^{\frac n2 ((\beta_* + \frac  {t_1}n)^2+(\bar\beta_* + \frac{\bar t_2}{n})^2)A_{l+1,d}}\\
&\sim \eee^{n b_l} \eee^{2\sigma_*(t_1+\bar t_2)A_{1,l}} \eee^{(\beta_* t_1 + \bar \beta_* \bar t_2) A_{l+1,d}}.
\end{align*}
The last step follows from~\eqref{eq:asympt_N_nk} and~\eqref{eq:def_b_l_2}. From the condition $|\beta_*|=\frac{\sigma_k}{\sqrt 2}$, it follows that $b_k=b_{k-1}$ and that $b_l<b_k$ for $l\notin \{k, k-1\}$. This means that only the terms $D_{n,k}(t_1,t_2)$ and $D_{n,k-1}(t_1,t_2)$ are asymptotically relevant. Since
$$
D_{n,k}(t_1,t_2) + D_{n,k-1}(t_1,t_2) \sim  \eee^{nb_k} (\eee^{t_1\lambda_k + \bar t_2 \bar \lambda_k} + \eee^{t_1\lambda_{k-1} + \bar t_2 \bar \lambda_{k-1}}),
$$
we arrive at~\eqref{eq:cov_proof_1_boundary}.

\vspace*{2mm}
\noindent
\textit{Proof of~\eqref{eq:cov_proof_2_boundary}.}  We have
\begin{align*}
\E [ Z_n(t_1) Z_n(t_2)]
&=
\eee^{-n b_k} N_n \sum_{\eps\in \SSS_n} \E \eee^{  \sqrt n (\beta_* + \frac{t_1} n)X_{\eta} + \sqrt n(\beta_*  +\frac{t_2} n) X_{\eps}}\\
&=
\eee^{-n b_k} \sum_{l=0}^d E_{n,l}(t_1,t_2),
\end{align*}
where in $E_{n,l}(t_1,t_2)$ we take the sum over all paths $\eps\in \SSS_n$ having exactly $l$ edges in common with $\eta$, that is
\begin{align*}
E_{n,l}(t_1,t_2)
&= N_n (N_{n,l+1}-1)N_{n,l+2}\ldots N_{n,d} \cdot \E \eee^{  \sqrt n (\beta_* + \frac{t_1} n)X_{\eta} + \sqrt n(\beta_*  +\frac{t_2} n) X_{\eps}} \\
&\sim  N_n N_{n,l+1}\ldots N_{n,d}\cdot  \eee^{\frac n 2 (2 \beta_* + \frac 1n (t_1 + t_2))^2 A_{1,l}} \eee^{\frac n2 ((\beta_* + \frac  {t_1}n)^2+(\beta_* + \frac{t_2}{n})^2)A_{l+1,d}}.
\end{align*}
Since $\Re(\beta_*^2)<\sigma_*^2$ by the assumption $\beta_*\notin \R$, we have
$
E_{n,l}(t_1,t_2)=o(D_{n,l}(t_1,t_2))
$
and hence $E_{n,l}(t_1,t_2)=o(\eee^{n b_k})$ for all $1\leq l\leq d$. For $l=0$, this argumentation does not work since $A_{1,0}=0$.  Instead, for $l=0$, we have
$$
E_{n,0}(t_1,t_2) = N_n^2 \eee^{n \beta_*^2 a +O(1)} = o(\eee^{n b_k}),
$$
where the last step holds since $b_k > b_0 = 2\log \alpha +(\sigma_*^2-\tau_*^2)a$ for $2\leq k\leq d$; see~\eqref{eq:def_b_l_2}. Note that this argument does fails for $k=1$. (Which is the reason why we excluded the case $k=1$ in Proposition~\ref{prop:asympt_cov_bound}). Summarizing, we have shown that $E_{n,l}(t_1,t_2)=o(\eee^{n b_k})$ for all $0\leq l\leq d$. The proof of~\eqref{eq:cov_proof_2_boundary} is complete.

\vspace*{2mm}
\noindent
\textit{Proof of~\eqref{eq:cov_proof_3_boundary}.} We have
$$
\E Z_n(t) = N_n e^{-\frac 12 b_k n} \eee^{\frac 12 n (\beta_* + \frac tn )^2 a}
=\eee^{\frac n2 \sum_{m=1}^k (\log \alpha_m - |\beta_*|^2 a_m) + O(1)}.
$$
If $2\leq k\leq d$, then the sum in the exponent is strictly negative and~\eqref{eq:cov_proof_3_boundary} follows. Note that for  $k=1$ the sum contains just one term and this term is $0$. This is another reason why Proposition~\ref{prop:asympt_cov_bound} is not valid for $k=1$.
\end{proof}

The next proposition covers the case $k=1$ which was left open in Proposition~\ref{prop:asympt_cov_bound}.
\begin{proposition}\label{prop:asympt_cov_bound_k1}
Let $\beta_*\in \C\bsl \R$ be such that $|\beta_*|=\frac {\sigma_1}{\sqrt 2}$.  Define the stochastic processes $\{Z_n(t)\colon t\in\C\}$ and $\{Z_n^*(t)\colon t\in\C\}$ by
$$
Z_n(t)=
N_n^{-1} \eee^{-\frac 12 \beta_*^2 a n} \ZZZ_n\left(\beta_*+\frac{t}{n}\right),
\quad
Z_n^*(t) = Z_n(t) - \E Z_n(t).
$$
Then, for every $t_1,t_2,t\in \C$,
\begin{align}
&\lim_{n\to\infty} \E [ Z_n^*(t_1) \overline{Z_n^*(t_2)} ] = \eee^{t_1\lambda_1 + \bar t_2 \bar \lambda_1}, \label{eq:cov_1_bound_k1}\\
&\lim_{n\to\infty} \E [ Z_n^*(t_1) Z_n^*(t_2) ] = 0, \label{eq:cov_2_bound_k1} \\
&\lim_{n\to\infty} \E Z_n(t) = \eee^{\beta_* t a} = \eee^{\lambda_0 t}, \label{eq:cov_3_bound_k1}
\end{align}
where $\lambda_1=2\sigma_*a_1 + \beta_* A_{2,d}$ and $\lambda_0=\beta_* a$.
\end{proposition}
\begin{remark}\label{rem:var_bound_k1}
In particular, we obtain that under the assumptions of Proposition~\ref{prop:asympt_cov_bound_k1},
$$
\Var \ZZZ_n\left(\beta_* + \frac tn \right) \sim   \eee^{b_1 n} \eee^{2 \Re (\lambda_1 t)}.
$$
To see this, note that by the assumption $|\beta_*|=\frac{\sigma_1}{\sqrt 2}$ and~\eqref{eq:asympt_N_nk}  we have
\begin{equation}\label{eq:var_bound_k1_proof_b0_b1}
b_0 = b_1 = 2\log \alpha + (\sigma_*^2-\tau_*^2)a,
\quad
|N_n \eee^{\frac 12 \beta_*^2 an}|\sim \eee^{\frac 12 b_1 n}.
\end{equation}
\end{remark}
\begin{proof}
To prove the proposition, we need to establish~\eqref{eq:cov_3_bound_k1} and to show that
\begin{align}
&\lim_{n\to\infty} \E [ Z_n(t_1) \overline{Z_n(t_2)} ] = \eee^{t_1\lambda_1 + \bar t_2 \bar \lambda_1} + \eee^{(\beta_* t_1 + \bar \beta_* \bar t_2) a},\label{eq:cov_1_bound_k1_proof}\\
&\lim_{n\to\infty} \E [ Z_n(t_1) Z_n(t_2) ] = \eee^{\beta_* (t_1 + t_2) a}. \label{eq:cov_2_bound_k1_proof}
\end{align}
\textit{Proof of~\eqref{eq:cov_3_bound_k1}.}  We have
$$
\E Z_n(t) = \eee^{-\frac 12 \beta_*^2 a n}  \eee^{\frac 12 n (\beta_* + \frac tn )^2 a}
=\eee^{\beta_* t a  + o(1)}.
$$
Note for future reference that the convergence is locally uniform in $t$.

\vspace*{2mm}
\noindent
\textit{Proof of~\eqref{eq:cov_1_bound_k1_proof}.}
Recall~\eqref{eq:var_bound_k1_proof_b0_b1}.
In the same way as in the proof of~\eqref{eq:cov_proof_1_boundary}, we obtain that
\begin{align*}
\E [ Z_n(t_1) \overline{Z_n(t_2)} ]
\sim
\eee^{-n b_1} \sum_{l=0}^d D_{n,l}(t_1,t_2),
\end{align*}
where  $D_{n,l}(t_1,t_2)$ satisfies
$$
D_{n,l}(t_1,t_2)
\sim \eee^{n b_l} \eee^{2\sigma_*(t_1+\bar t_2)A_{1,l}} \eee^{(\beta_* t_1 + \bar \beta_* \bar t_2) A_{l+1,d}}.
$$
From the assumption $|\beta_*|=\frac{\sigma_1}{\sqrt 2}$, it follows that $b_0=b_1$ and $b_l<b_1$ for $2\leq l\leq d$. Thus, only the terms $D_{n,0}(t_1,t_2)$  and $D_{n,1}(t_1,t_2)$ are asymptotically relevant.
We have
$$
D_{n,1}(t_1,t_2) + D_{n,0}(t_1,t_2) \sim  \eee^{nb_1} (\eee^{t_1\lambda_1 + \bar t_2 \bar \lambda_1} + \eee^{(\beta_* t_1 + \bar \beta_* \bar t_2) a}).
$$
This yields~\eqref{eq:cov_1_bound_k1_proof}.

\vspace*{2mm}
\noindent
\textit{Proof of~\eqref{eq:cov_2_bound_k1_proof}.} In the same way as in the proof of~\eqref{eq:cov_proof_2_boundary}, we have
\begin{align*}
\E [ Z_n(t_1) Z_n(t_2)]
=
N_n^{-2} \eee^{-\beta_*^2 a n} \sum_{l=0}^d E_{n,l}(t_1,t_2),
\end{align*}
where $E_{n,l}(t_1,t_2)$ satisfies
\begin{align*}
E_{n,l}(t_1,t_2)
\sim N_n N_{n,l+1}\cdot \ldots \cdot N_{n,d} \cdot  \eee^{\frac n 2 (2 \beta_* + \frac 1n (t_1 + t_2))^2 A_{1,l}} \eee^{\frac n2 ((\beta_* + \frac  {t_1}n)^2+(\beta_* + \frac{t_2}{n})^2)A_{l+1,d}}.
\end{align*}
Since $\Re(\beta_*^2)<\sigma_*^2$ by the assumption $\beta_*\notin \R$, we have that $E_{n,l}(t_1,t_2)=o(D_{n,l}(t_1,t_2))$ for all $1\leq l\leq d$. However, for $l=0$, we have $A_{1,l}=0$ and
$$
E_{n,0}(t_1,t_2)
\sim
N_n^2 \eee^{\frac n2 ((\beta_* + \frac  {t_1}n)^2+(\beta_* + \frac{t_2}{n})^2)a}
\sim N_n^2 \eee^{\beta_*^2 an}\eee^{\beta_* (t_1+t_2) a}.
$$
This yields~\eqref{eq:cov_2_bound_k1_proof}.
\end{proof}

\begin{remark}\label{rem:local_cov_real_beta_bound_case}
In Propositions~\ref{prop:asympt_cov_bound} and~\ref{prop:asympt_cov_bound_k1}, we left open  the case of real $\beta_*$. If $\beta_*\in \R$, then the same considerations as in Remark~\ref{rem:local_cov_real_beta} apply.
\end{remark}

\section{Functional central limit theorems for $|\sigma|<\frac {\sigma_1}{2}$} \label{sec:func_clt_sigma_small}
\subsection{Statements of functional central limit theorems}  \label{subsec:func_clt}

\begin{figure}
\includegraphics[width=0.49\textwidth]{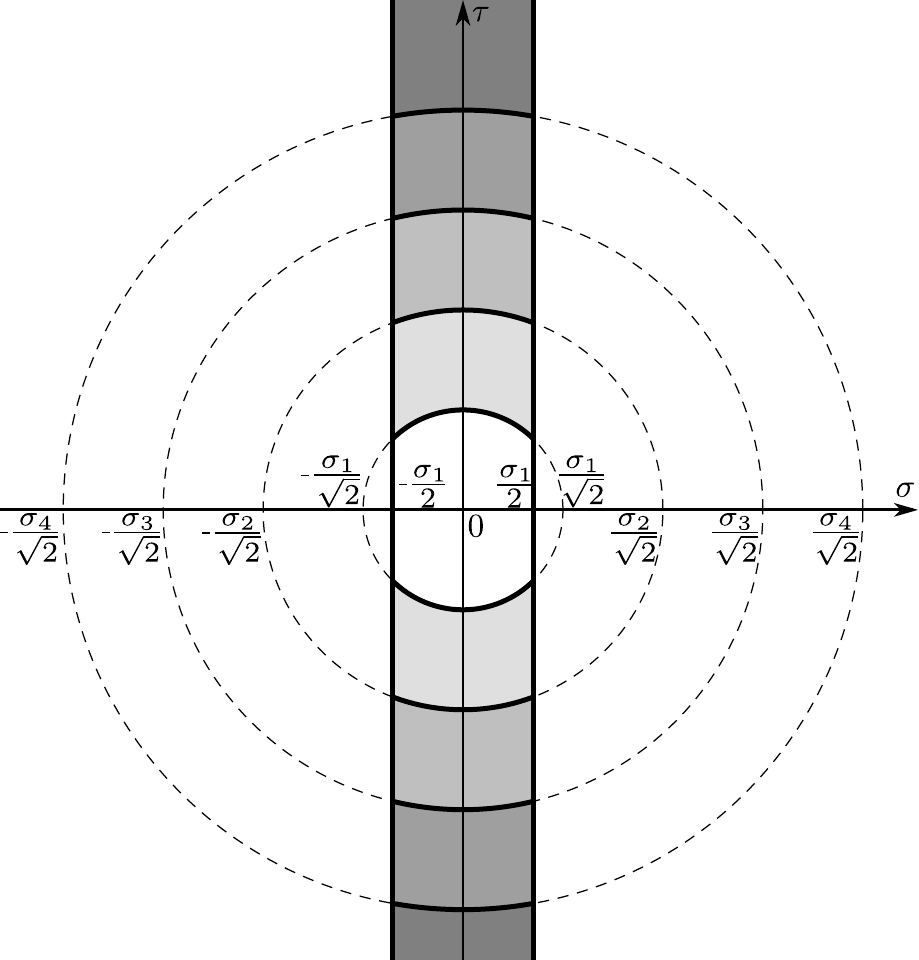}
 \caption
{\small  Domains in which the functional central limit theorems are valid.
}
\label{fig:phases_func_clt}
\end{figure}

We use the notation $\beta_*=\sigma_*+i\tau_*\in\C$. The next theorem is a functional central limit theorem in phase  $F^k E^{d-k}$, for $1\leq k\leq d$. It is a particular case and the first step in the proof of Theorem~\ref{theo:functional_clt}.  Recall that $g_n(\beta_*;t)$ was defined in~\eqref{eq:def_d_nk_t} and~\eqref{eq:g_n_beta_star_t_def}, Section~\ref{subsec:cov}.
\begin{theorem}\label{theo:functional_clt_EF}
Let $\beta_*\in\C$ be such that $|\sigma_*|<\frac {\sigma_1}2$ and  $\frac{\sigma_k}{\sqrt 2} < |\beta_*| < \frac{\sigma_{k+1}}{\sqrt 2}$ for some $1\leq k\leq d$. (These requirements are equivalent to $\beta_*\in F^k E^{d-k}$).
Then, the following convergence holds weakly on $\HHH(\C)$:
\begin{equation}\label{eq:theo:functional_clt_EF}
\left\{\eee^{-g_n(\beta_*;t)} \ZZZ_n\left(\beta_*+\frac t {\sqrt n}\right)\colon t\in\C\right\}
\toweak
\{\XXX(\sqrt{A_{1,k}} t)\colon t\in\C\},
\end{equation}
where $\{\XXX(t)\colon t\in\C\}$ is the plane Gaussian analytic function~\eqref{eq:def_XXX_plane_GAF} and $A_{1,k}=a_1+\ldots+a_k$.
\end{theorem}
Theorem~\ref{theo:functional_clt_EF} is not valid in the case $k=0$. The reason is that for $k=0$ the expectation is larger than the fluctuations and so, an additional centering is needed to extract the fluctuations.  For $k=0$, we have the following result (which is a restatement of Theorem~\ref{theo:functional_clt_EF_k0_first_statement}).  Recall the definition of $\hat g_n(\beta_*;t)$ from~\eqref{eq:def_d_nk_t_hat}. Recall also that $\ZZZ_n^*(\beta_*)= \ZZZ_n(\beta_*)-\E \ZZZ_n(\beta_*)$.
\begin{theorem}\label{theo:functional_clt_EF_k0}
Let $\beta_*\in\C$ be such that $|\sigma_*|<\frac {\sigma_1}2$ and  $|\beta_*| < \frac{\sigma_{1}}{\sqrt 2}$. (This implies but is not equivalent to $\beta_* \in E^d$).
Then, the following convergence holds weakly on $\HHH(\C)$:
\begin{equation}\label{eq:theo:functional_clt_EF_k0}
\left\{\eee^{-\hat g_n(\beta_*;t)} \ZZZ_n^*\left(\beta_*+\frac t {\sqrt n}\right)\colon t\in\C\right\}
\toweak
\{\XXX(\sqrt{a_1} t)\colon t\in\C\},
\end{equation}
where $\{\XXX(t):t\in\C\}$ is the plane Gaussian analytic function~\eqref{eq:def_XXX_plane_GAF}.
\end{theorem}

Note that we can replace $\ZZZ_n$ by $\ZZZ_n^*$ in Theorem~\ref{theo:functional_clt_EF}, but we cannot replace $\ZZZ_n^*$ by $\ZZZ_n$ in Theorem~\ref{theo:functional_clt_EF_k0}.

Next, we are going to state a functional limit theorem for the boundary between
the phases $F^k E^{d-k}$ and $F^{k-1} E^{d-k+1}$, for $2\leq k\leq d$. But first
let us explain the idea. If we look at $\ZZZ_n$ locally at scale $1/{\sqrt n}$
in phases $F^k E^{d-k}$ and $F^{k-1} E^{d-k+1}$, we see essentially the Gaussian
analytic functions $\XXX(\sqrt{A_{1,k}} t)$ and $\XXX(\sqrt{A_{1,k-1}} t)$. In
fact, it is convenient to think of $\ZZZ_n$ as of a weighted sum of all such
Gaussian analytic functions over all $k$. However, the weights are such that in
any phase  \textit{just one} Gaussian analytic function is dominating and all
other functions are not visible in the limit. Now, if we look at $\ZZZ_n$ near
the boundary of $F^k E^{d-k}$ and $F^{k-1} E^{d-k+1}$, we see \textit{two}
Gaussian analytic functions \textit{simultaneously}. It turns out that the right
scale to look at in the boundary case is $1/n$  (which is smaller than $1/{\sqrt
n}$). Hence,  in fact, we see not two Gaussian analytic functions but rather just
two Gaussian random variables, $N'$ and $N''$, with some weights. Here is the
exact statement.
\begin{theorem}\label{theo:functional_clt_EF_boundary}
Let $\beta_*\in \C$ be such that $|\sigma_*| < \frac {\sigma_1}2$ and $|\beta_*| = \frac{\sigma_{k}}{\sqrt 2}$ for some $2\leq k\leq d$.
Then, the following convergence holds weakly on $\HHH(\C)$:
\begin{equation}\label{eq:theo:functional_clt_EF_boundary}
\left\{\eee^{-\frac 12 b_k n} \ZZZ_n\left(\beta_*+\frac{t}{n}\right): t\in\C\right\}
\toweak
\{\eee^{\lambda_k t}N' + \eee^{\lambda_{k-1} t}N''\colon t\in\C\},
\end{equation}
where $N', N''\sim N_{\C}(0,1)$ are independent and $\lambda_l=2\sigma_*A_{1,l} + \beta_* A_{l+1,d}$, $1\leq l\leq d$.
\end{theorem}
On the boundary between $F^1 E^{d-1}$ and $E^d$, we have a slightly different functional central limit theorem. The reason is that in phase $F^1 E^{d-1}$ the partition function $\ZZZ_n$ looks locally like a Gaussian analytic function, whereas in phase $E^d$ (where the expectation dominates) the partition function looks locally like the expectation (plus Gaussian fluctuations which have smaller order of magnitude than the expectation).  So, on the boundary between these two phases, $\ZZZ_n$ looks locally like a weighted sum of a Gaussian random variable $N$ and a constant.
\begin{theorem}\label{theo:functional_clt_EF_boundary_k1}
Let $\beta_*\in \C$ be such that $|\sigma_*|<\frac {\sigma_1}2$ and $|\beta_*| = \frac{\sigma_{1}}{\sqrt 2}$.
Then, the following convergence holds weakly on $\HHH(\C)$:
\begin{equation}\label{eq:theo:functional_clt_EF_boundary_k1}
\left\{ N_n^{-1} \eee^{-\frac 1 2 \beta_*^2  a n}\ZZZ_n\left(\beta_*+\frac{t}{n}\right)\colon t\in\C\right\}
\toweak
\{\eee^{\lambda_1 t} N + \eee^{\lambda_0 t}\colon t\in\C\},
\end{equation}
where $N\sim N_{\C}(0,1)$ and $\lambda_1 = 2\sigma_*a_1 + \beta_* A_{2,d}$, $\lambda_0=\beta_* a$.
\end{theorem}

\subsection{Proofs of functional central limit theorems}
All four theorems stated in Section~\ref{subsec:func_clt} will be deduced from the following general result.
\begin{proposition}\label{prop:functional_CLT_general}
Fix some $\beta_*\in \C$ such that $|\sigma_*|<\frac {\sigma_1}{2}$. Assume that $c_n \colon \C\to \C\bsl \{0\}$ are deterministic analytic functions and  $q_n\in \C$ is a deterministic sequence such that the process $Z_n^*(t):=c_n^{-1}(t)\ZZZ_n^*(\beta_*+ q_n t)$
has the property that for all $t_1,t_2,t\in\C$,
\begin{align}
\lim_{n\to\infty} \E [ Z_n^*(t_1) \overline{Z_n^*(t_2)} ] &= \E [ Z_{\infty}^*(t_1) \overline{Z_{\infty}^*(t_2)} ], \label{eq:func_clt_ass1}\\
\lim_{n\to\infty} \E [ Z_n^*(t_1) Z_n^*(t_2) ] &= \E [ Z_{\infty}^*(t_1) Z_{\infty}^*(t_2)] = 0,   \label{eq:func_clt_ass2}\\
\Var  Z_n^*(t)  &\leq F(t),   \label{eq:func_clt_ass3}
\end{align}
where $\{Z_{\infty}^*(t)\colon t\in\C\}$ is a zero-mean complex Gaussian process with sample paths in $\HHH(\C)$ and $F\colon\C\to\R$ is a locally bounded function.  Then, weakly on $\HHH(\C)$ it holds that
\begin{equation}\label{eq:func_clt_abstr}
\left\{ Z_n^*(t)\colon t\in\C\right\}
\toweak
\{Z_{\infty}^*(t)\colon t\in\C\}.
\end{equation}
\end{proposition}
\begin{proof}[Proof of Theorems~\ref{theo:functional_clt_EF}, \ref{theo:functional_clt_EF_k0}, \ref{theo:functional_clt_EF_boundary}, \ref{theo:functional_clt_EF_boundary_k1}]
In Propositions~\ref{prop:asympt_cov}, \ref{prop:asympt_cov_k0}, \ref{prop:asympt_cov_bound}, \ref{prop:asympt_cov_bound_k1}, we have shown that assumptions~\eqref{eq:func_clt_ass1} and~\eqref{eq:func_clt_ass2} are fulfilled with
\begin{enumerate}
\item $q_n = 1/\sqrt n$, $c_n(t) = \eee^{g_n(\beta_*;t)}$, $Z_{\infty}^*(t)=\XXX(\sqrt{A_{1, k}} t)$ in  Proposition~\ref{prop:asympt_cov}.
\item $q_n = 1/\sqrt n$, $c_n(t) = \eee^{\hat g_n(\beta_*;t)}$, $Z_{\infty}^*(t)=\XXX(\sqrt{a_1} t)$ in  Proposition~\ref{prop:asympt_cov_k0}.
\item $q_n = 1/n$, $c_n(t) = \eee^{\frac 12 b_k n}$, $Z_{\infty}^*(t)=\eee^{t\lambda_k}N' + \eee^{t\lambda_{k-1}}N''$ in Proposition~\ref{prop:asympt_cov_bound}.
\item $q_n = 1/n$, $c_n(t) = N_n\eee^{\frac 12 \beta_*^2 a n}$, $Z_{\infty}^*(t)=\eee^{t\lambda_1} N$ in Proposition~\ref{prop:asympt_cov_bound_k1}.
    \end{enumerate}
Condition~\eqref{eq:func_clt_ass3} is satisfied because the statements of Propositions~\ref{prop:asympt_cov}, \ref{prop:asympt_cov_k0}, \ref{prop:asympt_cov_bound}, \ref{prop:asympt_cov_bound_k1} hold locally uniformly in $t_1,t_2\in\C$, as it is easy to see from the proofs. Applying Proposition~\ref{prop:functional_CLT_general} we obtain Theorems~\ref{theo:functional_clt_EF}, \ref{theo:functional_clt_EF_k0}, \ref{theo:functional_clt_EF_boundary}, \ref{theo:functional_clt_EF_boundary_k1}. In fact, in the fourth case, we obtain that
\begin{equation}\label{eq:theo:functional_clt_EF_boundary_k1_star}
\left\{ N_n^{-1} \eee^{-\frac 1 2 \beta_*^2  a n} \ZZZ_n^*\left(\beta_*+\frac{t}{n}\right)\colon t\in\C\right\}
\toweak
\{\eee^{\lambda_1 t} N : t\in\C\}.
\end{equation}
However, in Proposition~\ref{prop:asympt_cov_bound_k1}, Eq.~\eqref{eq:cov_3_bound_k1}, we have shown that $N_n^{-1}\eee^{-\frac 1 2 \beta_*^2  a n} \E \ZZZ_n(\beta_*+\frac{t}{n})$ converges to  $\eee^{\beta_* a t}$ locally uniformly in $t$ and hence, in $\HHH(\C)$; see the proof of~\eqref{eq:cov_3_bound_k1}. Together with~\eqref{eq:theo:functional_clt_EF_boundary_k1_star}, this yields~\eqref{eq:theo:functional_clt_EF_boundary_k1}.
\end{proof}
\begin{proof}[Proof of Proposition~\ref{prop:functional_CLT_general}]
We have the representation
$$
Z_n^*(t)= \sum_{k=1}^{N_{n,1}} V_{n,k}^*(t),
$$
where $\{V_{n,k}^*(t)\colon t\in \C\}$ is a stochastic process defined by $V_{n,k}^*(t)= V_{n,k}(t)-\E V_{n,k}(t)$ and
$$
V_{n,k}(t)= c_n^{-1}(t)\, \eee^{\sqrt{na_1} (\beta_*  + q_n t)  \xi_{k}} \sum_{\eps_2=1}^{N_{n,2}}\ldots \sum_{\eps_d=1}^{N_{n,d}} \eee^{\sqrt n(\beta_* + q_n t)(\sqrt a_2 \xi_{k\eps_2}+\ldots+\sqrt{a_d}\xi_{k\eps_2\ldots\eps_d})}.
$$
Note that for every $n\in\N$ the processes $\{V_{n,k}^*(t)\colon t\in \C\}$, $1\leq k\leq N_{n,1}$, are  independent by the definition of the GREM.

First, we show that~\eqref{eq:func_clt_abstr} holds in the sense of weak convergence of finite-dimensional distributions.  Pick some $t_1,\ldots,t_r\in\C$. We show that the random vector $\bS_n^*:=\{Z_n^*(t_i)\}_{i=1}^r$ converges to $\bS_{\infty}^*:=\{Z_{\infty}^*(t_i)\}_{i=1}^r$ in distribution. We consider these $r$-dimensional complex random vectors as $2r$-dimensional real random vectors. To prove that $\bS_n^*\to \bS_{\infty}^*$ in distribution, we will verify the conditions of Lyapunov's Theorem~\ref{theo:clt_lyapunov_vectors}.   By~\eqref{eq:func_clt_ass1} and~\eqref{eq:func_clt_ass2}, the covariance matrix of $\bS_n^*$ converges to the covariance matrix of $\bS_{\infty}^*$. This verifies the first condition of Theorem~\ref{theo:clt_lyapunov_vectors}. It remains to verify the Lyapunov condition: For some $p=2+\delta>2$,
\begin{equation}\label{eq:lyapunov_fclt_verification}
\lim_{n\to\infty} N_{n,1} \E |V_{n,1}^*(t_i)|^{p}=0, \quad 1\leq i \leq r.
\end{equation}
Fix some $1\leq i \leq r$. The random variable
$$
\tilde V_{n}^* := \frac{c_n(t_i)}{\sqrt{\Var \ZZZ_n(\beta_*+ q_n t_i)}} V_{n,1}^*(t_i)
$$
has the same distribution as the random variable $z_n^{-1} W_n^*$ in Section~\ref{sec:proof_clt}  and hence, by~\eqref{eq:proof_clt_lyapunov}, we obtain that $N_{n,1} \E |\tilde V_{n}^*|^{p}$ converges to $0$ as $n\to\infty$ provided that $\delta>0$ is sufficiently small. Note that we have to insert $\beta_* + q_n t_i$ instead of $\beta$ in~\eqref{eq:proof_clt_lyapunov}  but this causes no problems since~\eqref{eq:proof_clt_lyapunov} holds locally uniformly in the domain  $|\sigma| < \sigma_1/\sqrt {2p}$.
On the other hand, by~\eqref{eq:func_clt_ass3} we have the estimate $|V_{n,1}^*(t_i)| < C |\tilde V_{n}^*|$.  This completes the verification of~\eqref{eq:lyapunov_fclt_verification}.

Thus, we can apply Theorem~\ref{theo:clt_lyapunov_vectors} to obtain that $\bS_n^*\to \bS_{\infty}^*$ in distribution. This means that the process $\{Z_n^*(t)\colon t\in\C\}$ converges to $\{Z_{\infty}^*(t)\colon t\in\C\}$ in the sense of finite-dimensional distributions.
The fact that the sequence of processes  $\{Z_n^*(t)\colon t\in\C\}$, $n\in\N$, is tight in $\HHH(\C)$ follows from~\eqref{eq:func_clt_ass3} and Proposition~\ref{prop:tightness_random_analytic}.
\end{proof}


\section{Meromorphic continuation of the Poisson cascade zeta function}\label{sec:zeta_P}

\subsection{Uniform absolute convergence on compact sets: Proof of Theorem~\ref{theo:zeta_abs_conv}}
The first na\"ive attempt to prove Theorem~\ref{theo:zeta_abs_conv} would be to try to
demonstrate the absolute convergence of the integral
$$
\int_{0}^{\infty}\ldots
\int_{0}^{\infty} x_1^{-z_1}\ldots x_d^{-z_d} \dd x_1\ldots \dd x_d
$$
for $z\in\calD$. In view of Lemma~\ref{lem:poisson_cascade_moments}, this would imply that $\E\zeta_P(z)<\infty$. However, this
integral diverges because of the singularity which emerges if one of the
variables $x_1,\ldots,x_d$ is close to $0$. Following the method
of~\cite{bovier_kurkova1}, we will therefore introduce a subset $F_{\gamma}(a)$
of $\R_+^d=(0,\infty)^d$ in which all variables $x_1,\ldots,x_d$ are well-separated from $0$. Then, we will show that the integral over this set converges. Over the complement of
this set, the zeta sum can be reduced to a finite number of zeta functions of
smaller dimension, and the induction can be applied.
For $d=1$, Theorem~\ref{theo:zeta_abs_conv} follows from the fact that $\lim_{n\to\infty} \frac 1n P_n = 1$ a.s.\ by the law of large numbers.
Henceforth, we assume that $d\geq 2$.

\vspace*{2mm}
\noindent
\textsc{Step 1.}
Fix some parameters $\gamma_1>\ldots>\gamma_d>0$. Let the variables
$x=(x_1,\ldots,x_d)$ and $y=(y_1,\ldots,y_d)$ take values in $\R_+^d$ and be
connected by the relations
\begin{equation}\label{eq:def_y_k}
y_1=x_1^{\gamma_1}, \quad  y_2=x_1^{\gamma_1}x_2^{\gamma_2},\quad \ldots, \quad y_d=x_1^{\gamma_1}\ldots x_d^{\gamma_d}.
\end{equation}
The inverse transformation is given by
\begin{equation}\label{eq:def_y_k_inverse}
x_1=y_1^{1/\gamma_1},\quad x_2=\left(\frac{y_2}{y_1}\right)^{1/\gamma_2},\quad \ldots, \quad x_d=\left(\frac{y_d}{y_{d-1}}\right)^{1/\gamma_d}.
\end{equation}
We will often write $\dd x$ and $\dd y$ for $\dd x_1\ldots \dd x_d$ and $\dd y_1\ldots \dd y_d$.
Consider, for $a>0$, the set
\begin{equation}\label{eq:def_F}
F_{\gamma}(a)=\{(x_1,\ldots,x_d)\in \R_+^d\colon y_1\geq a, \ldots,y_d\geq a\}.
\end{equation}

\vspace*{2mm}
\noindent
\textsc{Step 2.}
Let $K\subset \calD$ be a compact set.
Consider a domain
\begin{equation}\label{eq:def_calD_gamma}
\calD_{\gamma}= \left\{(z_1,\ldots,z_d)\in \C^d\colon \frac{\Re z_1-1}{\gamma_1}>\ldots >\frac{\Re z_d-1}{\gamma_d}>0\right\}\subset \calD.
\end{equation}
We can find $\gamma_1>\ldots>\gamma_d>0$ such that $K\subset \calD_{\gamma}$, just take all $\gamma_i$'s to be sufficiently close to $1$. Moreover, it follows from~\eqref{eq:def_calD_gamma} that we can find an $\eps>0$ such that for all $z\in K$,
\begin{equation}\label{eq:ineq_Re_z_k}
\frac{\Re z_2}{\gamma_2} - \frac{\Re z_1}{\gamma_1} < \frac{1}{\gamma_2}-\frac{1}{\gamma_1}-\eps,
\quad
\ldots,
\quad
\frac{\Re z_d}{\gamma_d} - \frac{\Re z_{d-1}}{\gamma_{d-1}} < \frac{1}{\gamma_{d}}-\frac{1}{\gamma_{d-1}}-\eps
\end{equation}
and
\begin{equation}\label{eq:ineq_Re_z_k_1}
\frac{\Re z_d}{\gamma_d} > \frac{1}{\gamma_{d}}+\eps,
\end{equation}

\vspace*{2mm}
\noindent
\textsc{Step 3.}
Let the set $F=F_{\gamma}(1)$ be as in~\eqref{eq:def_F}. Let $x\in F$. Then, for all $z\in K$,
\begin{align*}
|x_1^{-z_1}\ldots x_d^{-z_d}|
&=
x_1^{\gamma_1 \left(-\frac{\Re z_1}{\gamma_1}\right)}\ldots x_d^{\gamma_d  \left(-\frac{\Re z_d}{\gamma_d}\right)}
=
y_d^{-\frac{\Re z_d}{\gamma_d}}y_{d-1}^{\frac{\Re z_d}{\gamma_d}-\frac{\Re z_{d-1}}{\gamma_{d-1}}}
\ldots y_1^{\frac{\Re z_2}{\gamma_2}-\frac{\Re z_{1}}{\gamma_{1}}}.
\end{align*}
Since $y_1\geq 1,\ldots,y_d\geq 1$ for $x\in F$,  we obtain, by~\eqref{eq:ineq_Re_z_k} and~\eqref{eq:ineq_Re_z_k_1},
\begin{equation}\label{eq:varphi_est}
\varphi(x_1,\ldots,x_d):=\sup_{z\in K} |x_1^{-z_1}\ldots x_d^{-z_d}|
\leq
y_d^{-\frac 1 {\gamma_d}-\eps} y_{d-1}^{\frac 1{\gamma_d}-\frac 1{\gamma_{d-1}}-\eps}\ldots y_1^{\frac{1}{\gamma_2}-\frac 1 {\gamma_1}-\eps}.
\end{equation}
Recall that $\Pi$ is the Poisson cascade point process from Section~\ref{subsec:zeta_P}. To prove Theorem~\ref{theo:zeta_abs_conv}, we need to show that
\begin{equation}\label{eq:zeta_P_abs_conv_proof}
\sum_{x\in\Pi} \varphi(x) < +\infty \text{ a.s.}
\end{equation}
Note that~\eqref{eq:zeta_P_abs_conv_proof} is satisfied for $d=1$ since $\lim_{n\to\infty} \frac 1n P_n = 1$ a.s.\ by the law of large numbers.  We can make the induction assumption that~\eqref{eq:zeta_P_abs_conv_proof} holds in dimensions $1,\ldots,d-1$.  The proof of~\eqref{eq:zeta_P_abs_conv_proof} will be complete after we have shown that in dimension $d$,
\begin{align}
S:=\sum_{x\in\Pi \cap F} \varphi(x)<\infty \text{ a.s.}
\;\;\; \text{ and } \;\;\;
R:=\sum_{x\in\Pi \bsl F} \varphi(x)<\infty \text{ a.s.} \label{eq:proof_abs_conv}
\end{align}

\vspace*{2mm}
\noindent
\textsc{Step 4.} We prove that $S<\infty$ a.s. The  Jacobian of the transformation $(x_1,\ldots, x_d)\mapsto (y_1,\ldots,y_d)$, see~\eqref{eq:def_y_k}, is given by $\frac {\dd y}{\dd x} = \gamma_1\ldots \gamma_d \frac{y_1\ldots y_d}{x_1\ldots x_d}$. Using this transformation, the estimate~\eqref{eq:varphi_est}, and~\eqref{eq:def_y_k_inverse}, we obtain that
\begin{align*}
\int_{F}\varphi(x)\dd x
&\leq
\int_{1}^{\infty}\ldots \int_{1}^{\infty}
y_d^{-\frac 1 {\gamma_d}-\eps} y_{d-1}^{\frac 1{\gamma_d}-\frac 1{\gamma_{d-1}}-\eps}\ldots y_1^{\frac{1}{\gamma_2}-\frac 1 {\gamma_1}-\eps} \dd x_1\ldots \dd x_d\\
&\leq
\frac {1}{\gamma_1\ldots \gamma_d} \int_{1}^{\infty}\ldots \int_{1}^{\infty} y_d^{-1-\eps}\ldots y_{1}^{-1-\eps} \dd y_1\ldots \dd y_d.
\end{align*}
The integral on the right-hand side is finite. Hence, $\int_{F}\varphi(x)\dd x<\infty$. Since the intensity of the point process $\Pi$ is the Lebesgue measure, see Lemma~\ref{lem:poisson_cascade_moments}, we obtain that $\E S<\infty$. Hence, $S$ is finite a.s.

\vspace*{2mm}
\noindent
\textsc{Step 5.} We prove that $R<\infty$ a.s. The idea is to reduce $R$ to a finite number of zeta functions of dimension which is smaller than $d$. For $m=1,\ldots,d$, define a set
$$
A_m=\{(x_1,\ldots,x_m)\in \R_+^m\colon y_1\geq 1,\ldots, y_{m-1} \geq 1, y_m<1\}.
$$
Let $E_m$ be the random set consisting of those indices $(\eps_1,\ldots,\eps_m)\in\N^m$ for which  the point $(P_{\eps_1}, P_{\eps_1\eps_2},\ldots, P_{\eps_1 \ldots \eps_m})$ is in $A_m$. We will show that $E_m$ is finite with probability $1$. In fact, we will even show that the expected number of elements in $E_m$ is finite.
Using the transformation $(x_1,\ldots, x_m)\mapsto (y_1,\ldots,y_m)$, see~\eqref{eq:def_y_k}, we obtain
\begin{multline*}
\int_{A_m} \dd x_1\ldots \dd x_m
\\
=\frac{1}{\gamma_1\ldots \gamma_m}
\int_{(1,\infty)^{m-1}}
\int_0^1 y_1^{\frac 1 {\gamma_1} - \frac 1 {\gamma_2} - 1}\ldots  y_{m-1}^{\frac 1 {\gamma_{m-1}} - \frac 1 {\gamma_m} - 1} y_m^{\frac{1}{\gamma_m}-1}\dd y_1 \ldots \dd y_m.
\end{multline*}
The integral on the right-hand side is finite because $\gamma_1>\ldots>\gamma_m>0$, thus proving that $E_m$ is finite a.s.
For every point $x\in \R_{+}^d\bsl F$, there exists a unique $m=1,\ldots,d$ such that the projection of $x$ onto the first $m$ coordinates belongs to $A_m$. Hence, we have
\begin{equation}\label{eq:zeta_abs_conv_tech1}
R
=
\sum_{x\in\Pi \bsl F} \varphi(x)
=
\sum_{m=1}^d \sum_{(\eps_1,\ldots,\eps_m)\in E_m} R_{\eps_1\ldots\eps_m},
\end{equation}
where $R_{\eps_1\ldots \eps_m}$ is a random variable defined by
$$
R_{\eps_1\ldots\eps_m}=\sum_{\eps_{m+1},\ldots,\eps_d\in \N} \varphi(P_{\eps_1},\ldots, P_{\eps_1\ldots\eps_d}).
$$
Since, as we have shown, the sum on the right hand side of~\eqref{eq:zeta_abs_conv_tech1}
involves a finite number of summands a.s., the proof of the a.s.\ finiteness of $R$ would be
complete if we could show that $R_{\eps_1,\ldots,\eps_m}$ is finite a.s.  Let $K_m$ be
the projection of the compact set $K\subset \calD\subset \R_{+}^d$ onto the last $d-m$
coordinates.
There is a constant $C_{\eps_1,\ldots,\eps_m}$ such that
$$
R_{\eps_1\ldots\eps_m}
\leq
C_{\eps_1,\ldots,\eps_m}
\sum_{\eps_{m+1},\ldots,\eps_d\in \N}
\max_{(z_{m+1},\ldots, z_d) \in K_m} |P_{\eps_1\ldots \eps_{m+1}}^{-z_{m+1}}\ldots P_{\eps_1\ldots\eps_d}^{-z_d}|.
$$
The sum on the right-hand side has the same structure as the sum $\sum_{x\in
\Pi}\varphi(x)$, but in dimension $d-m$. Therefore, by the induction hypothesis,
this sum is finite a.s.,\ thus proving $R_{\eps_1\ldots \eps_m}$ is finite a.s. Hence, $R<\infty$ a.s.

\subsection{Meromorphic continuation of $\zeta_P$: Proof of Theorem~\ref{theo:zeta_mero_cont}}
This section is devoted to the proof of Theorem~\ref{theo:zeta_mero_cont}. The main step will be done in Proposition~\ref{prop:zeta_*_mero_cont}. We continue to use the notation of the previous section.
\begin{proposition}\label{prop:int_I}
For $z=(z_1,\ldots,z_d)\in \calD_{\gamma}$, we have
\begin{equation}\label{eq:def_I_z_a}
I_{\gamma}(z;a)
:=
\int_{F_{\gamma}(a)} x_1^{-z_1}\ldots x_d^{-z_d} \dd x
=
\frac {a^{\frac{1-z_1}{\gamma_1}}}{\gamma_1\ldots\gamma_d}
\left(\prod_{k=1}^{d-1} \frac 1 {\frac{z_k-1}{\gamma_k} - \frac{z_{k+1}-1}{\gamma_{k+1}}}\right) \frac{\gamma_d}{z_d-1}.
\end{equation}
The integral in~\eqref{eq:def_I_z_a} converges absolutely for $z\in \calD_{\gamma}$.
\end{proposition}
\begin{proof}
We use induction on $d$. For $d=1$, the identity reduces to the integral
$$
\int_{a^{1/\gamma_1}}^{\infty} x_1^{-z_1}\dd x_1 = \frac{1}{z_1-1} a^{\frac{1-z_1}{\gamma_1}}.
$$
Assume that the identity~\eqref{eq:def_I_z_a} is true for $d-1$ variables. We prove that it holds for $d$ variables. We can write the conditions $y_1\geq a,\ldots,y_d \geq a$ in the following form:
$$
x_1 \geq a^{1/\gamma_1} \text{ and } x_2^{\gamma_2}\geq a x_1^{-\gamma_1}, \ldots, x_2^{\gamma_2}\ldots x_d^{\gamma_d} \geq  ax_1^{-\gamma_1}.
$$
Note that the conditions on $x_2,\ldots,x_d$ are of the same form as in $F_{\gamma}(a)$, but with $d-1$ variables and with $ax_1^{-\gamma_1}$ instead of $a$. Therefore, we define a set
\begin{equation}\label{eq:tilde_F_def}
F_{\tilde \gamma}(a) = \{(x_2,\ldots,x_d)\in \R_{+}^{d-1}\colon x_2^{\gamma_2} \geq a,\ldots, x_2^{\gamma_2}\ldots x_d^{\gamma_d} \geq a\}.
\end{equation}
Using Fubini's theorem and then applying the induction assumption to the integral over the variables $x_2,\ldots,x_d$, we obtain
\begin{align*}
I_{\gamma}(z;a)
&=
\int_{a^{1/\gamma_1}}^{\infty} \left( x_1^{-z_1} \int_{F_{\tilde \gamma}(ax_1^{-\gamma_1})} x_2^{-z_2}\ldots x_d^{-z_d} \dd x_2\ldots \dd x_d\right) \dd x_1\\
&=
\frac {a^{\frac{1-z_2}{\gamma_2}}}{\gamma_2\ldots\gamma_d}
\left(\prod_{k=2}^{d-1} \frac{1}{\frac{z_k-1}{\gamma_k} - \frac{z_{k+1}-1}{\gamma_{k+1}}} \right) \frac{\gamma_d}{z_d-1} \int_{a^{1/\gamma_1}}^{\infty} x_1^{-z_1-\gamma_1 \frac{1-z_2}{\gamma_2}} \dd x_1.
\end{align*}
Evaluation of the integral yields the desired formula~\eqref{eq:def_I_z_a}. Note that the integral converges since $\Re(z_1+\gamma_1 \frac{1-z_2}{\gamma_2})>1$ by the assumption $z\in\calD_{\gamma}$.
\end{proof}

\begin{proposition}\label{prop:zeta_*_mero_cont}
Let $\Pi$ be the Poisson cascade point process defined in Section~\ref{subsec:zeta_P}. Fix $a>0$ and $\gamma_1>\ldots>\gamma_d>0$.  With probability $1$, the function
\begin{equation}\label{eq:def_zeta_P_a}
\zeta^*_P(z_1,\ldots,z_d;a):= \sum_{x\in \Pi \cap F_{\gamma}(a)} x_1^{-z_1}\ldots x_d^{-z_d} - I_{\gamma}(z_1,\ldots,z_d;a),
\end{equation}
defined originally on $\calD_{\gamma}$, has a meromorphic continuation to the following larger domain:
$$
\frac 12 \calD_{\gamma}=\{z\in\C^d\colon 2z\in \calD_{\gamma}\}.
$$
The function $f_{\gamma}(z;a):=(z_d-1)\zeta^*_P(z;a)$ is a.s.\ analytic on $\frac 12 \calD_{\gamma}$.  For every $z\in \frac 12 \calD_{\gamma}$,
\begin{equation}\label{eq:exp_var_zeta_star}
\E f_{\gamma}(z;a)=0,
\quad
\Var f_{\gamma}(z;a)= a^{\frac{1-2\Re z_1}{\gamma_1}} \Var f_{\gamma}(z;1)<\infty.
\end{equation}
\end{proposition}
\begin{proof}
We use induction over the number of levels $d$. For $d=1$, the proposition has
been established in Theorem~2.6 of~\cite{kabluchko_klimovsky}; see
also~\eqref{eq:zeta_P_anal_cont_d_1}.   Take some $d\geq 2$ and assume that the
statement of the proposition, including~\eqref{eq:exp_var_zeta_star}, holds in
dimensions $1,\ldots,d-1$. We prove that it holds in dimension $d$. The idea is
to represent the $d$-variate function $\zeta_P^*$ as a sum of the terms
$P_k^{-z_1}$ multiplied by independent copies of the $(d-1)$-variate function
$\zeta_P^*$.

\vspace*{2mm}
\noindent
\textsc{Step 1: Notation.} 
Take some $T>a$. Define $F_{\gamma}(a,T)=F_{\gamma}(a)\cap \{y_1\leq  T^{\gamma_1}\}$, a truncated version of the set $F_{\gamma}(a)$, by
$$
F_{\gamma}(a,T)=\{(x_1,\ldots,x_d)\in \R_{+}^d\colon a\leq y_1\leq T^{\gamma_1}, y_2\geq a, \ldots, y_d\geq a \}.
$$
Consider also $I_{\gamma}(z;a,T)$, a truncated version of the integral $I_{\gamma}(z;a)$:
$$
I_{\gamma}(z;a,T)=\int_{F_{\gamma}(a,T)} x_1^{-z_1}\ldots x_d^{-z_d} \dd x.
$$
By Proposition~\ref{prop:int_I}, the integral defining $I_{\gamma}(z; a,T)$ converges absolutely  for $z\in \calD_{\gamma}$ and hence, defines an analytic function of $z$ on $\calD_{\gamma}$. An exact formula for $I_{\gamma}(z; a,T)$ will be provided later; see~\eqref{eq:I_gamma_z_a_T_integral}.  By Theorem~\ref{theo:zeta_abs_conv}, the following expression defines a random function of $z$ which is with probability $1$ analytic on  $\calD_{\gamma}$:
\begin{equation}\label{eq:def_zeta_P_a_T}
\zeta_P^*(z; a,T):= \sum_{x\in \Pi \cap F_{\gamma}(a,T)} x_1^{-z_1}\ldots x_d^{-z_d} - I_{\gamma}(z;a,T).
\end{equation}

\vspace*{2mm}
\noindent
\textsc{Step 2: Meromorphic continuation of $\zeta_P^*(z; a,T)$.} Write $\tilde x=(x_2,\ldots,x_d)$, $\tilde z=(z_2,\ldots,z_d)$,  etc. For  $\eps_1\in\N$ let $\tilde \Pi_{\eps_1}$ be the point process on $\R^{d-1}$ given by
$$
\tilde \Pi_{\eps_1}=\sum_{\tilde \eps =(\eps_2,\ldots,\eps_d)\in\N^{d-1}} \delta(P_{\eps_1\eps_2},\ldots, P_{\eps_1\ldots\eps_d}).
$$
In the definition of $\Pi$, the $(d-1)$-dimensional point process $\tilde \Pi_{\eps_1}$ is  ``attached'' to the point $P_{\eps_1}$.
Define the random functions $\tilde \zeta^*_1(z;a), \tilde \zeta^*_2(z;a),\ldots$ by
\begin{equation}\label{eq:def_zeta_star_k}
\tilde \zeta^*_{k}(z;a) = \sum_{\tilde x \in \tilde \Pi_{k} \cap F_{\tilde \gamma}(a P_{k}^{-\gamma_1})} x_2^{-z_2}\ldots x_d^{-z_d} - I_{\tilde \gamma}(\tilde z; a P_{k}^{-\gamma_1}),
\quad z\in \calD_{\gamma}.
\end{equation}
Here, $F_{\tilde \gamma}(a)\subset \R_{+}^{d-1}$ is the set defined in~\eqref{eq:tilde_F_def} and by Proposition~\ref{prop:int_I},
\begin{equation}\label{eq:tilde_I_def}
I_{\tilde \gamma}(\tilde z; a) =
\int_{F_{\tilde \gamma}(a)} x_2^{-z_2}\ldots x_d^{-z_d} \dd x
=
\frac {a^{\frac{1-z_2}{\gamma_2}}}{\gamma_2\ldots\gamma_d}
\left(\prod_{k=2}^{d-1} \frac 1 {\frac{z_k-1}{\gamma_k} - \frac{z_{k+1}-1}{\gamma_{k+1}}}\right) \frac{\gamma_d}{z_d-1}.
\end{equation}
Let $\calA$ be the $\sigma$-algebra generated by $P_1,P_2,\ldots$. Note that conditionally on $\calA$, the random functions $\tilde \zeta^*_1(z;a), \tilde \zeta^*_2(z;a),\ldots$ are independent. Also,
\begin{equation}\label{eq:zeta^*_k_eqdistr}
\tilde \zeta^*_k(z;a) \eqdistr \zeta_P^*(z_2,\ldots,z_d; a P_k^{-\gamma_1}) \text{ conditionally on } \calA.
\end{equation}
Due to the absolute convergence of the series in~\eqref{eq:def_zeta_P_a_T} for $z\in \calD_{\gamma}$, we can change the order of summation and write
\begin{equation}\label{eq:zeta_P^*_decomposition}
f_{\gamma}(z; a,T)
:=
(z_d-1) \zeta_P^*(z; a,T)
= S_1(z;a,T)+S_2(z;a,T),
\end{equation}
where
\begin{align}
S_1(z; a,T) &= (z_d-1) \sum_{a\leq P_k^{\gamma_1}\leq T^{\gamma_1}} P_k^{-z_1}   \tilde \zeta^*_k(z;a), \label{eq:S_1_def}\\
S_2(z; a,T) &= (z_d-1)\left(\sum_{a\leq P_k^{\gamma_1}\leq T^{\gamma_1}} P_k^{-z_1} I_{\tilde \gamma}(\tilde z; a P_k^{-\gamma_1}) - I_{\gamma}(z; a,T)\right).\label{eq:S_2_def}
\end{align}
This representation is valid for $z\in \calD_{\gamma}$. However, by the
induction assumption and~\eqref{eq:zeta^*_k_eqdistr}, the function $\tilde
f_{k}(z;a) := (z_d-1)\tilde \zeta^*_k(z;a)$ (and hence, the function
$S_1(z; a, T)$) has an analytic continuation to $\frac 12 \calD_{\gamma}$, with probability
$1$.  Concerning the analytic continuation of $S_2(z; a,T)$, note that although
the integral defining $I_{\tilde \gamma}(\tilde z; a)$ in~\eqref{eq:tilde_I_def}
may diverge for $z\notin\calD_{\gamma}$, the expression on the right-hand side
of~\eqref{eq:tilde_I_def}, multiplied by $z_d-1$, is an analytic function of $z\in \frac 12 \calD_{\gamma}$. The
following formula, which is valid for $z\in\calD_{\gamma}$,
\begin{equation}\label{eq:I_gamma_z_a_T_integral}
I_{\gamma}(z;a,T)
=
\int_{a^{1/\gamma_1}}^{T} x_1^{-z_1} I_{\tilde \gamma} (\tilde z; ax_1^{-\gamma_1})  \dd x_1
=
I_{\tilde \gamma} (\tilde z; a) \int_{a^{1/\gamma_1}}^{T} x_1^{-z_1- \gamma_1\frac{1-z_2}{\gamma_2}}   \dd x_1
\end{equation}
yields an analytic continuation of $(z_d-1)I_{\gamma}(z;a,T)$ to the domain $\frac 12
\calD_{\gamma}$.  Therefore, $S_2(z; a,T)$ has analytic
continuation to $\frac 12 \calD_{\gamma}$. Hence, the function $f_{\gamma}(z;
a,T)$ defined in~\eqref{eq:zeta_P^*_decomposition} has an analytic continuation to
$\frac 12 \calD_{\gamma}$, with probability $1$.

\vspace*{2mm}
\noindent
\textsc{Step~3: Expectation and variance of $S_1(z; a,T)$ and $S_2(z; a,T)$.} By the result of Step~2, we  can view $f_{\gamma}(z;
a,T)=(z_d-1)\zeta_P^*(z; a,T)$ as a random element with values in the space $\HHH(\frac 12 \calD_{\gamma})$. We
will now prove that the limit  of $f_{\gamma}(z; a,T)$ (in the sense of $\HHH(\frac 12 \calD_{\gamma})$
and as $T\to\infty$)  exists a.s. Since for $z\in\calD_{\gamma}$ this limit
coincides with $f_{\gamma}(z; a) =(z_d-1) \zeta^*_P(z;a)$, we get the
desired meromorphic continuation of $\zeta_P^*(z;a)$.

We need to compute the first two moments of $S_1(z; a,T)$ and $S_2(z; a,T)$ for $z\in \frac 12 \calD_{\gamma}$.
Introduce the functions $\tilde f_{k}(z;a) = (z_d-1)\tilde \zeta^*_k(z; a)$,
$k\in\N$. Recall that $\calA$ is the $\sigma$-algebra generated by the Poisson
process $P_1,P_2,\ldots$.  By~\eqref{eq:zeta^*_k_eqdistr} and  the induction
assumption~\eqref{eq:exp_var_zeta_star}, we have
\begin{equation}\label{eq:zeta_cont_tech1}
\E[\tilde f_k(z;a)| \calA] = 0, \quad
\Var [\tilde f_k(z;a) | \calA] =  (aP_k^{-\gamma_1})^{\frac{1-2\Re z_2}{\gamma_2}}  \Var f_{\tilde \gamma}(\tilde z;1) \;\;\;\text{ a.s.}
\end{equation}
Using~\eqref{eq:S_1_def}, \eqref{eq:zeta_cont_tech1} and then the total expectation formula $\E S= \E [\E[S|\calA]]$, we obtain that
\begin{equation} \label{eq:S_1_S_2_expect_1}
\E [S_1(z; a, T)| \calA]=0 \text{ a.s.}, \;\;\; \E S_1(z; a, T)=0.
\end{equation}
Using~\eqref{eq:S_2_def}, \eqref{eq:I_gamma_z_a_T_integral}, and the fact that $S_2(z; a,T)$ is $\calA$-measurable, we obtain that
\begin{equation} \label{eq:S_1_S_2_expect_2}
\E S_2(z; a, T)=\E [S_1(z; a, T)\overline{S_2(z; a, T)}]=0.
\end{equation}

We now compute the variance of $S_1(z; a,T)$. Using~\eqref{eq:S_1_def} and the scaling property of the variance in~\eqref{eq:zeta_cont_tech1}, we obtain
\begin{align*}
\Var [S_1(z; a,T)| \calA ]
&=
\sum_{a<P_k^{\gamma_1}<T^{\gamma_1}} P_k^{-2\Re z_1} \Var[\tilde f_k (z; a)|\calA]\\
&=
a^{\frac{1-2\Re z_2}{\gamma_2}} \Var f_{\tilde \gamma} (\tilde z; 1)  \sum_{a<P_k^{\gamma_1}<T^{\gamma_1}} P_k^{-2\Re z_1-\gamma_1 \frac{1-2\Re z_2}{\gamma_2}}.
\end{align*}
Using the formula for the total variance $\Var S = \E \Var [S|\calA] + \Var \E [S|\calA]$ and noting that the second term in it vanishes by~\eqref{eq:S_1_S_2_expect_1}, we get
$$
\Var S_1(z; a,T) =  a^{\frac{1-2\Re z_2}{\gamma_2}} \Var f_{\tilde \gamma} (\tilde z; 1) \int_{a^{1/\gamma_1}}^{T} x_1^{-2\Re z_1-\gamma_1 \frac{1-2\Re z_2}{\gamma_2}}\dd x_1.
$$
Using the definition of $\calD_{\gamma}$, see~\eqref{eq:def_calD_gamma}, it is easy to check that
\begin{equation}\label{eq:2_Re_z_1}
2\Re z_1+\gamma_1 \frac{1-2\Re z_2}{\gamma_2}>1 \text{ for } z\in \frac 12 \calD_{\gamma}.
\end{equation}
Hence, the integral converges as $T\to\infty$. We obtain
\begin{equation}\label{eq:S_1_var_lim}
\lim_{T\to+\infty} \Var S_1(z; a,T) =
\gamma_1^{-1} a^{\frac{1-2\Re z_1}{\gamma_1}}
\frac {\Var f_{\tilde \gamma} (\tilde z; 1)} {\frac {1- 2\Re z_2}{\gamma_2} - \frac {1- 2\Re z_1}{\gamma_1}}.
\end{equation}

We compute the variance of $S_2(z; a,T)$. Since $P_1,P_2,\ldots$ form a
homogeneous Poisson process with intensity $1$, the variance of the linear statistic $\sum_{k\in\N}
\varphi(P_k)$ is given by $\int_0^{\infty} |\varphi(x_1)|^2 \dd x_1$.
Using~\eqref{eq:S_2_def} and the scaling property of $I_{\tilde \gamma}(\tilde
z; a)$ following from~\eqref{eq:tilde_I_def}, we obtain
\begin{align*}
\Var S_2(z; a,T)
&=
 |z_d-1|^2 \int_{a^{1/\gamma_1}}^{T} x_1^{-2\Re z_1} |I_{\tilde \gamma}(\tilde z; a x_1^{-\gamma_1})|^2 \dd x_1\\
&=
a^{\frac{2-2\Re z_2}{\gamma_2}} |(z_d-1)I_{\tilde \gamma}(\tilde z; 1)|^2 \int_{a^{1/\gamma_1}}^{T} x_1^{-2\Re z_1-\gamma_1 \frac{2-2\Re z_2}{\gamma_2}}\dd x_1.
\end{align*}
By~\eqref{eq:2_Re_z_1}, the integral converges as $T\to+\infty$.  We have
\begin{equation}\label{eq:S_2_var_lim}
\lim_{T\to+\infty} \Var S_2(z; a,T) = \gamma_1^{-1}
a^{\frac{1-2\Re z_1}{\gamma_1}}
\frac {|(z_d-1)I_{\tilde \gamma}(\tilde z; 1)|^2} {\frac {2- 2\Re z_2}{\gamma_2} - \frac {2- 2\Re z_1}{\gamma_1}}.
\end{equation}

\vspace*{2mm}
\noindent
\textsc{Step 4: Meromorphic continuation of $\zeta_P^*(z; a)$.}
We are in position to complete the proof of Proposition~\ref{prop:zeta_*_mero_cont}. Fix an arbitrary $a>0$ and some
compact set $K\subset \frac 12 \calD_{\gamma}$. Consider $S_1(z; a,T)$ and
$S_2(z; a,T)$ as stochastic processes indexed by $T\in\N$ and taking values in
the Banach space $C(K)$ of continuous functions on $K$. By the properties of the
Poisson process and~\eqref{eq:S_1_def}, \eqref{eq:S_2_def}, both processes have
independent  (but not identically distributed) increments. Also, both processes
have zero mean by~\eqref{eq:S_1_S_2_expect_1} and~\eqref{eq:S_1_S_2_expect_2}.
Hence, for every $z\in K$, both $\{S_1(z;a,T)\}_{T\in\N}$ and
$\{S_2(z;a,T)\}_{T\in\N}$ are martingales. By~\eqref{eq:S_1_var_lim}
and~\eqref{eq:S_2_var_lim}, both martingales are bounded in $L^2$ and hence, for
every $z\in K$ the sequences $S_1(z; a,T)$ and $S_2(z; a,T)$ converge as
$T\to\infty$ to some random variables, a.s.\ and in $L^2$. Hence, both sequences
$S_1(z;a,T)$ and $S_2(z; a,T)$ (viewed as sequences of stochastic processes on
$K$) converge as $T\to\infty$ to some limiting stochastic processes $S_1(z;a)$
and $S_2(z;a)$, in the sense of finite-dimensional distributions. In fact, both
sequences are tight by Proposition~\ref{prop:tightness_random_analytic} with $p=2$. The assumptions of Proposition~\ref{prop:tightness_random_analytic} are fulfilled since $\Var S_1(z; a,T)$ and $\Var S_2(z; a,T)$ are increasing in $T$ and can be bounded by the limits given in~\eqref{eq:S_1_var_lim} and~\eqref{eq:S_2_var_lim}.  Hence, both sequences $S_1(z; a,T)$ and $S_2(z;
a,T)$ converge as $T\to\infty$ weakly on $C(K)$. We will show that in fact, they converge even a.s. A classical  theorem of L\'evy states that
partial sums of  independent (not necessarily identically distributed)
$\R$-valued random variables converge weakly if and only if they converge a.s.
\citet{ito_nisio} extended this result to variables with values in a Banach
space. Recalling that the sequences $\{S_1(z;a,T)\}_{T\in\N}$ and $\{S_2(z;
a,T)\}_{T\in\N}$ have independent increments,  we obtain that both $S_1(z; a,T)$
and $S_2(z; a,T)$ (considered as $C(K)$-valued random variables) converge a.s.\
as $T\to\infty$. Since the uniform limit of analytic functions is analytic, we
obtain the desired analytic continuation of $f_{\gamma}(z;a)$. It follows
from~\eqref{eq:zeta_P^*_decomposition} and~\eqref{eq:S_1_S_2_expect_1} that $\E
f_{\gamma}(z;a)=0$. The scaling property of the variance
in~\eqref{eq:exp_var_zeta_star} follows from~\eqref{eq:zeta_P^*_decomposition},
\eqref{eq:S_1_S_2_expect_2}, \eqref{eq:S_1_var_lim}, \eqref{eq:S_2_var_lim}.
\end{proof}

\begin{proof}[Completing the proof of Theorem~\ref{theo:zeta_mero_cont}]
We use induction over $d$. For $d=1$, the statement has been established
in~\cite{kabluchko_klimovsky}; see also~\eqref{eq:zeta_P_anal_cont_d_1}.  Take some $d\geq 2$. Assume that the statement is
valid in dimensions $1,\ldots, d-1$. Our aim is to prove that it holds in
dimension $d$. Fix some $a>0$ and $\gamma_1>\ldots>\gamma_d$. Consider a domain $F_{\gamma}(a)$ as
in~\eqref{eq:def_F}. Recalling~\eqref{eq:def_zeta_P_a}, we have, for every $z\in
\calD_{\gamma}$, a representation
\begin{equation}\label{eq:zeta_decomp_three_terms}
(z_d-1)\zeta_P(z)= (z_d-1) \zeta_P^*(z; a) + (z_d-1) I_{\gamma}(z; a) + (z_d-1) \zeta_P^{\diamond}(z; a),
\end{equation}
where
$$
\zeta_P^{\diamond}(z; a) = \sum_{x\in \Pi\bsl F_{\gamma}(a)}  x_1^{-z_1}\ldots x_d^{-z_d}.
$$
This representation is valid for $z\in\calD_{\gamma}$ since we can interchange
the order of summation in the definition of $\zeta_P$ by the absolute
convergence established in Theorem~\ref{theo:zeta_abs_conv}. However, by
Propositions~\ref{prop:zeta_*_mero_cont} and~\ref{prop:int_I}, the first two
terms on the right-hand side of~\eqref{eq:zeta_decomp_three_terms} have an
analytic continuation to $\frac 12 \calD_{\gamma}$ with probability $1$.

Let us now show that the function $(z_d-1) \zeta_P^{\diamond}(z;a)$ has an analytic continuation to $\frac 12 \calD_{\gamma}$, with probability $1$.  For concreteness, take $a=1$. Introduce the
sets $E_m$ as in the proof of Theorem~\ref{theo:zeta_abs_conv}, Step~5. For $z\in
\calD_{\gamma}$, we have a representation
\begin{equation}\label{eq:zeta_diamond_decomposition}
\zeta_P^{\diamond}(z; a)
=
\sum_{m=1}^d \sum_{(\eps_1,\ldots,\eps_m)\in E_m}
P_{\eps_1}^{-z_1}\ldots P_{\eps_1\ldots\eps_m}^{-z_m} \zeta_{\eps_1,\ldots,\eps_m}(z_{m+1},\ldots,z_d ),
\end{equation}
where
\begin{equation}\label{eq:zeta_eps_1_eps_m}
\zeta_{\eps_1,\ldots,\eps_m}(z_{m+1},\ldots,z_d) = \sum_{\eps_{m+1},\ldots,\eps_d\in\N} P_{\eps_1\ldots\eps_{m+1}}^{-z_{m+1}}\ldots P_{\eps_1\ldots\eps_d}^{-z_d}.
\end{equation}
The functions $(z_d-1)\zeta_{\eps_1,\ldots,\eps_m}(z_{m+1},\ldots,z_d)$ are
$(d-m)$-variate analogues of the function $(z_d-1)\zeta_P(z)$ and hence, with probability
$1$ admit an analytic continuation to $\frac 12 \calD_{\gamma}$ by the induction
assumption.  Since the sets $E_m$ are finite a.s.\ (as we have shown in the
proof of Theorem~\ref{theo:zeta_abs_conv}, Step~5), we obtain the a.s.\ existence of an
analytic continuation of $(z_d-1)\zeta_P^{\diamond}(z;a)$ to $\frac 12
\calD_{\gamma}$. The a.s.\ existence of the analytic continuation to $\frac 12
\calD_{\gamma}$ has been thus established for all three terms on the right-hand
side of~\eqref{eq:zeta_decomp_three_terms}. This yields the desired analytic
continuation of $(z_d-1)\zeta_P(z)$.
\end{proof}

\subsection{A recursive formula for $\zeta_P$} In this section,  we will prove a
formula allowing to represent the $d$-variate zeta function $\zeta_P$ as a
combination of infinitely many independent  copies of the $(d-1)$-variate
$\zeta_P$. Let $\tilde \calD$ be a $(d-1)$-dimensional analogue of the set
$\calD$, that is
$$
\tilde \calD=\{\tilde z= (z_2,\ldots, z_d)\in \C^{d-1} \colon \Re z_2>\ldots>\Re z_d>1\}.
$$
Define independent random analytic functions $\tilde \zeta_1 , \tilde \zeta_2,\ldots$ on $\tilde \calD$ by
\begin{equation}\label{eq:tilde_zeta_k}
\tilde \zeta_k(\tilde z) = \sum_{\tilde x \in \tilde \Pi_k} x_2^{-z_2}\ldots x_d^{-z_d},  \quad k\in\N, \quad \tilde z\in \tilde \calD.
\end{equation}
Recall that for $k\in \N$ we denote by $\tilde \Pi_k$ the $(d-1)$-dimensional point process  ``attached'' to the point $P_{k}$ in the definition of the Poisson cascade point process $\Pi$, that is
$$
\tilde \Pi_{k}=\sum_{\tilde \eps =(\eps_2,\ldots,\eps_d)\in\N^{d-1}} \delta(P_{k\eps_2},\ldots, P_{k\eps_2\ldots\eps_d}).
$$
For $z\in\calD$ (which implies that $\tilde z\in\tilde \calD$), we can interchange the order of summation in the definition of $\zeta_P$ due to the absolute convergence established in Theorem~\ref{theo:zeta_abs_conv}. Hence,
\begin{equation}\label{eq:zeta_recursion_trivial}
\zeta_P(z)=\sum_{k=1}^{\infty} P_k^{-z_1} \tilde \zeta_k(\tilde z).
\end{equation}
In the next proposition, we give an extension of \eqref{eq:zeta_recursion_trivial} to $\frac 12 \calD$.   Note that by Theorem~\ref{theo:zeta_mero_cont}, the functions $\tilde \zeta_k(\tilde z)$ (defined originally for $\tilde z\in\tilde \calD$) admit a meromorphic continuation to $\frac 12 \tilde \calD$, with probability $1$.
\begin{proposition}\label{prop:zeta_recursion}
Let $d\geq 2$. For $T\in \N$ define $\zeta_P(z;T)$, a random meromorphic function on $\frac 12 \calD$, by
\begin{equation}\label{eq:zeta_P_T_def}
\zeta_P(z;T) = \sum_{P_k\leq T} P_k^{-z_1} \tilde \zeta_k(\tilde z).
\end{equation}
Then, with probability $1$,
\begin{equation}\label{eq:zeta_P_recursive_anal_cont}
(z_d-1)\zeta_P(z;T) \toT (z_d-1) \zeta_P(z) \text{ on } \HHH\left(\frac 12 \calD\right).
\end{equation}
\end{proposition}
\begin{remark}\label{rem:recursion_zeta_d_1}
It is important to stress that if we take $d=1$ and interpret $\tilde \zeta_k(\tilde z)$ as $1$, then~\eqref{eq:zeta_P_recursive_anal_cont} does \textit{not} hold since the series $\sum_{k=1}^{\infty} P_{k}^{-z_1}$ converges in the half-plane $\Re z_1>1$ only. In order to obtain an analogue of~\eqref{eq:zeta_P_recursive_anal_cont} for $d=1$, one needs to subtract a regularizing term; see~\eqref{eq:zeta_P_anal_cont_d_1}.
Somewhat surprisingly, in the case $d\geq 2$,  it is not necessary to subtract any regularizing terms from~\eqref{eq:zeta_P_T_def}. The reason is that for $d\geq 2$ the random variables $\tilde \zeta_k(\tilde z)$ are non-degenerate and it is known that multiplying the terms of a series by non-degenerate random variables may improve its convergence properties.
\end{remark}
\begin{proof}[Proof of Proposition~\ref{prop:zeta_recursion}]
First of all, it has been already observed above that~\eqref{eq:zeta_P_recursive_anal_cont} is valid for $z \in  \calD$ by interchanging the order of summation.  We will prove that the left-hand side of~\eqref{eq:zeta_P_recursive_anal_cont} converges as $T\to\infty$ to \textit{some} random analytic function in $\HHH(\frac 12 \calD)$, with probability $1$. The fact that the limiting function coincides with the right-hand side of~\eqref{eq:zeta_P_recursive_anal_cont}, follows then by the uniqueness of the analytic continuation.

\vspace*{2mm}
\noindent
\textsc{Step 1.}
Fix some $\gamma_1>\ldots>\gamma_d>0$. For $z\in \calD_{\gamma}$, we can interchange the order of summation in the definition of $\zeta_P(z;T)$ and hence, we have a representation
\begin{equation}\label{eq:zeta_P_z_T_S1S2S3}
(z_d-1) \zeta_P(z;T) = S_1(z;1,T) + \bar S_2(z; 1, T) + S_3(z; 1, T),
\end{equation}
where $S_1(z; 1,T)$, $\bar S_2(z; 1, T)$ and $S_3(z;1,T)$ are given by
\begin{align*}
S_1(z; 1,T)
&=
(z_d-1) \sum_{1\leq P_k\leq T} P_k^{-z_1}   \tilde \zeta^*_k(z;1),\\
\bar S_2(z; 1, T)
&= (z_d-1) I_{\tilde \gamma} (\tilde z; 1) 
\sum_{P_k\leq T} P_{k}^{-z_1-\gamma_1 \frac{1-z_2}{\gamma_2}},\\
S_3(z;1,T)
&= (z_d-1) \sum_{x\in \Pi \bsl F_{\gamma}(1 ,T)} x_1^{-z_1}\ldots x_d^{-z_d}\ind_{x_1\leq T}.
\end{align*}
Note that $S_1(z;1,T)$  is defined as in~\eqref{eq:S_1_def}.

\vspace*{2mm}
\noindent
\textsc{Step 2.}
We have already shown in the proof of Theorem~\ref{theo:zeta_mero_cont} that the function $S_1(z;1,T)$  admits an analytic continuation to $\frac 12 \calD_{\gamma}$ and that $S_1(z;1,T)$ converges, as $T\to\infty$, to a limiting random analytic function  in  $\HHH(\frac 12 \calD_{\gamma})$ with probability $1$.

\vspace*{2mm}
\noindent
\textsc{Step 3.}
Let us consider $\bar S_2(z;1,T)$ next.  Recall that the function $(z_d-1)I_{\tilde \gamma} (\tilde z; a)$, defined as the right-hand side in~\eqref{eq:tilde_I_def} (but not as the integral there!), is an analytic function for $z\in \frac 12 \calD_{\gamma}$. This yields an analytic continuation of $\bar S_2(z;1,T)$ to $\frac 12 \calD_{\gamma}$. Let us prove the convergence of $\bar S_2(z;1,T)$, as $T\to\infty$.
 By the definition of $\calD_{\gamma}$, see~\eqref{eq:def_calD_gamma},  and the inequality $\gamma_1>\gamma_2>0$, we have that, for every $z\in \frac 12 \calD_{\gamma}$,
$$
\Re \left(z_1+\gamma_1 \frac{1-z_2}{\gamma_2}\right) > 1.
$$
Since $\lim_{n\to\infty} \frac 1n P_n=1$ a.s.\ by the law of large numbers, we obtain that $\bar S_2(z; 1, T)$ converges in $\HHH(\frac 12 \calD_{\gamma})$  with probability $1$, as $T\to\infty$.


\vspace*{2mm} \noindent \textsc{Step 4.} We will complete the proof by showing
that $S_3(z;1,T)$ admits an analytic continuation to $\frac 12 \calD_{\gamma}$
and converges to $(z_d-1)\zeta_P^{\diamond}(z; 1)$ in $\HHH(\frac 12
\calD_{\gamma})$ as $T\to\infty$, with probability $1$.
Recall~\eqref{eq:def_y_k} and~\eqref{eq:def_y_k_inverse}. For $m=1,\ldots,d$ and
$T\in \N\cup\{\infty\} $, define a set
$$
A_m(T)=\{(x_1,\ldots,x_m)\in \R_{+}^m \colon 1\leq y_1 < T^{\gamma_1}, y_2\geq 1, \ldots, y_{m-1} \geq 1, y_m<1\}.
$$
We interpret $A_1(T)$ as $(0,1)$. Let $E_m(T)$ be the random set consisting of those indices $(\eps_1,\ldots,\eps_m)\in\N^m$ for which  the point $(P_{\eps_1}, P_{\eps_1\eps_2},\ldots, P_{\eps_1 \ldots \eps_m})$ is in $A_m(T)$. In the proof of Theorem~\ref{theo:zeta_abs_conv}, Step 5, we have shown that $E_m=E_m(\infty)$ (and hence, $E_m(T)$) is finite with probability $1$.
Now, for $z\in \calD_{\gamma}$, we have a representation
$$
S_3(z; 1,T)
=
(z_d-1)\sum_{m=1}^d \sum_{(\eps_1,\ldots,\eps_m)\in E_m(T)}
P_{\eps_1}^{-z_1}\ldots P_{\eps_1\ldots\eps_m}^{-z_m} \zeta_{\eps_1,\ldots,\eps_m}(z_{m+1},\ldots,z_d ),
$$
where $\zeta_{\eps_1,\ldots,\eps_m}(z_{m+1},\ldots,z_d)$ is defined as in~\eqref{eq:zeta_eps_1_eps_m}. This provides an analytic continuation of $S_3(z; 1,T)$ to $\frac 12 \calD_{\gamma}$.  Since the $\cup_{T\in\N} E_m(T)= E_m$ and $E_m$ is a.s.\ finite, we must have $E_m(T)=E_m$ for sufficiently large $T$, a.s. Hence, $S_3(z; 1,T)$ coincides with $(z_d-1)\zeta_P^{\diamond}(z; 1)$ for sufficiently large $T$, a.s. This establishes the required statement.
\end{proof}

\subsection{Proof that $\zeta_P(z)$ has no atoms}
The next proposition will be needed in the proof of Theorem~\ref{theo:free_energy}.
\begin{proposition}\label{prop:zeta_P_no_atoms}
For every $z\in \frac 12 \calD$, the random variable $(z_d-1)\zeta_P(z)$ has no atoms in $\C$ except for $d=1$, $z=1$.
\end{proposition}
\begin{remark}
For $d=1$, $z=z_1=1$, we have $(z-1)\zeta_P(z)=1$, where the value is understood by continuity.
\end{remark}
\begin{proof}[Proof of Proposition~\ref{prop:zeta_P_no_atoms}]
We follow the method used in the proof of Lemma~3.10 in~\cite{kabluchko_klimovsky}, where the case $d=1$ has been established. [Note that in this lemma the assumption $\beta\neq 1$ is missing].  We may assume that $d\geq 2$ and that the proposition has already been established in the setting of $d-1$ levels. 
Fix some $z\in \frac 12 \calD$. Given a random variable $Y$ with values in $\C$ write
$$
Q(Y)=Q(Y;0)=\sup_{w\in \C} \P[Y=w]
$$
for the supremum over the probabilities of the atoms of $Y$. This is a special case of the concentration function; see~\eqref{eq:def_Q_concentration_func} below. Let $U$ and $V$ be independent random variables with values in $\C$.  It is easy to check that
\begin{enumerate}
\item [(Q1)] $Q(U+V)\leq Q(U)$; see also~\eqref{eq:Q_ineq} below.
\item [(Q2)] If $U$ has no atoms and $V$ has no atom at $0$, then $UV$ has no atoms.
\end{enumerate}
By Proposition~\ref{prop:zeta_recursion}, for every $T\in \N$ we have a representation $(z_d-1)\zeta_P(z) = U(T)+V(T)$, where
$$
U(T)=(z_d-1)\zeta_P(z; T) = \sum_{k=1}^{\infty} P_k^{-z_1} \ind_{P_k\in [0,T]} (z_d-1)\tilde \zeta_k(\tilde z)
$$
and the random variables $U(T)$  and $V(T)$ are independent.  We will prove that
$Q(U(T))\leq \eee^{-T}$.  Then,  by Property~Q1, we would have
$Q((z_d-1)\zeta_P(z))\leq \eee^{-T}$ for every $T\in \N$. This would imply the
statement of the proposition by letting $T\to\infty$.  For $m\in\N_0$, let
$A_m(T)$ be the event which occurs if the interval $[0,T]$ contains exactly $m$
points of the form $P_i$, $i\in\N$. Note that $\P[A_0(T)]=\eee^{-T}$. By the
formula of the total probability, for every $w\in\C$,
$$
\P[U(T) = w] \leq \eee^{-T} + \sum_{m=1}^{\infty} \P[U(T)=w|A_m(T)].
$$
Conditioned on $A_m(T)$, the random variables $P_1,\ldots,P_m$ have the same joint law as the increasing order statistics of the i.i.d.\ random variables $\eta_1,\ldots,\eta_m$ distributed uniformly on $[0,T]$.  Therefore, by Property~Q1, for every $m\in\N$,
$$
\P[U(T)=w|A_m(T)] \leq Q \left( \sum_{k=1}^{m} \eta_k^{-z_1} (z_d-1) \tilde \zeta_k(\tilde z) \right)
\leq Q(\eta_1^{-z_1} (z_d-1)\tilde \zeta_1(\tilde z))=0.
$$
The last step follows by Property~Q2 from the fact that $\eta_1^{-z_1}$ has no atoms and $(z_d-1)\tilde \zeta_1(\tilde z)$ has no atom $0$ by the induction assumption.
\end{proof}

A random vector with values in $\R^m$ is called \textit{full} if its distribution is not concentrated on some proper affine subspace of $\R^m$. The next proposition will be needed in the proof of Proposition~\ref{prop:zeta_P_moments}.
\begin{proposition}\label{prop:zeta_P_full}
If $z\in \frac 12 \calD\backslash \R^d$,  then $(z_d-1)\zeta_P(z)$ is full in $\C\equiv \R^2$. 
\end{proposition}
\begin{proof}
%
Let  $U$ and $V$ be independent random variables with values in $\C\equiv \R^2$. The following statements are easy to verify:
\begin{enumerate}
\item [(F1)] If $U$ is full, then $U+V$ is full.
\item [(F2)] If $U$ is full and $V$ has no atom at $0$, then $UV$ is full.
\item [(F3)] If $A$ is an event with $\P[A]>0$ and the conditional distribution of $U$ given $A$ is full, then the unconditional distribution of $U$ is full as well.
\end{enumerate}

Assume first that $z_1\notin \R$. We have a representation  $(z_d-1)\zeta_P(z) =
U+V$, where $U$ and $V$ are independent random variables and $U =
\sum_{k=1}^{\infty} P_k^{-z_1} \ind_{P_k\leq 1} (z_d-1) \tilde \zeta_k( \tilde
z)$. Indeed, for $d=1$ this (with $\tilde \zeta_k(\tilde z)=1$) follows
from~\eqref{eq:zeta_P_anal_cont_d_1}, whereas for $d\geq 2$ the statement
follows from Proposition~\ref{prop:zeta_recursion}. We will show that $U$ is
full. By Property~F1, this would imply that $U+V$ is full as well. Let
$A=\{P_1\leq 1<P_2\}$ be the event which occurs if the interval $[0,1]$ contains
exactly $1$ point of the form $P_i$, $i\in\N$. The probability of $A$ is equal
to $\eee^{-1}$ and hence is strictly positive. Since $z_1\notin \R$, the
conditional law of $P_{1}^{-z_1}$ given $A$ is full.   Also, $(z_d-1)\tilde
\zeta_1(\tilde z)$ has no atom at $0$ by Proposition~\ref{prop:zeta_P_no_atoms}.
 By Property~F2, the conditional law of $P_{1}^{-z_1}(z_d-1)\tilde
\zeta_1(\tilde z)$ given $A$ is full. On the event $A$ this random variable is
equal to $U$.  By Property~F3 the law of $U$ is full. By Property~F1 the law of
$U+V=(z_d-1)\zeta_P(z)$ is full as well. This completes the proof in the case
$z_1\notin \R$.

Assume, by induction, that the statement of Proposition~\ref{prop:zeta_P_full}
holds in the setting of $d-1$ levels. Note that the basis of induction (that is,
the case $d=1$, $\Re z_1>\frac 12$, $z_1\notin \R$) has been verified above. We
prove that the proposition holds in the setting of $d$ levels. We may assume
that $z_1\in \R$ since we have already considered the case $z_1\notin \R$.  We
use the same notation for $U,V,A$ as above. By Property~F1 it suffices to show
that $U$ is full.  Since at least one coordinate among $z_2,\ldots,z_d$ is not
real, the law of $(z_d-1)\tilde \zeta_1(\tilde z)$ is full by the induction
assumption. This random variable is independent from $A$, so its law remains
full conditionally on $A$. Also, conditionally on $A$, the random variable
$P_1^{-z_1}$ has no atom at $0$. By Property F2, the law of
$P_{1}^{-z_1}(z_d-1)\tilde \zeta_1(\tilde z)$, conditionally on $A$, is full.
Since this law coincides with the conditional law of $U$ given $A$, we obtain
that the unconditional law of $U$ is full by Property~F3.
\end{proof}

Let us finally mention a lemma which will be useful in Section~\ref{sec:free_energy_global_zeros_proofs}.
\begin{lemma}\label{lem:no_atoms_linear_function}
Let $N$ be a random variable with values in $\C$ and having no atoms. Also, let $(S_1,S_2)$ be a random vector with values in $\C^2$ which is independent of $N$ and such that $S_1$ has no atom at $0$. Then, the random variable $S_1N+S_2$ has no atoms.
\end{lemma}
\begin{proof}
Let $\mu$ be the distribution of $(S_1,S_2)$ on $\C^2$. Then, $\mu(\{0\}\times\C)=0$. By the convolution formula, for every $w\in\C$ we have
$$
\P[S_1N+S_2=w] = \int_{(\C \bsl \{0\})\times \C)} \P[s_1 N +s_2=w] \mu (\dd s_1, \dd s_2).
$$
To complete the proof, note that since $N$ has no atoms, we have $\P[s_1 N +s_2=w]=0$, for every $s_1\in \C\bsl \{0\}$ and $s_2\in \C$.
\end{proof}

\subsection{Operator stability and moments of $\zeta_P$: Proof of Propositions~\ref{prop:zeta_P_operator_stable} and~\ref{prop:zeta_P_moments}} \label{subsec:proof_operator_stable_moments}
\begin{proof}[Proof of Proposition~\ref{prop:zeta_P_operator_stable}]
The idea of the proof is the same as in Proposition~1.4.1 of~\cite{samorodnitsky_taqqu_book}. Let $d\geq 2$. By Proposition~\ref{prop:zeta_recursion}, every function $(z_d-1)\zeta_P^{(j)}(z)$ can be represented as an a.s.\ limit of $(z_d-1)\zeta_P^{(j)}(z;T)$ as $T\to\infty$, where
$$
(z_d-1)\zeta_P^{(j)}(z;T) = \sum_{P_{k,j}\leq T} P_{k,j}^{-z_1} (z_d-1) \tilde \zeta_{k,j}(\tilde z).
$$
Here, $\sum_{k=1}^{\infty} \delta(P_{k,j})$, $1\leq j\leq m$, are independent unit intensity Poisson point processes on $(0,\infty)$ and independently, $(z_d-1) \tilde \zeta_{k,j}(\tilde z)$, $k\in\N$, $1\leq j\leq m$, are independent copies of the random analytic function $(z_d-1)\zeta_P(\tilde z)$.   Then,
\begin{align*}
\sum_{j=1}^m (z_d-1)\zeta_P^{(j)}(z;T)
& =  m^{z_1}\sum_{j=1}^m \sum_{m P_{k,j}\leq m T} (mP_{k,j})^{-z_1} (z_d-1)\tilde \zeta_{k,j}(\tilde z)\\
&\eqdistr
m^{z_1} (z_d-1)\zeta_P(z;mT)
\end{align*}
because $\sum_{k=1}^{\infty} \sum_{j=1}^m \delta(mP_{k,j})$ is a Poisson point process on $(0,\infty)$ with intensity $1$. Letting $T\to\infty$ yields the required statement. For $d=1$, the proof is similar.
\end{proof}

\begin{proof}[Proof of Proposition~\ref{prop:zeta_P_moments}]
If $z_k\in\R$, for all $1\leq k\leq d$, then the distribution of $(z_d-1)\zeta_P(z)$ is real stable with index $\alpha:= \frac 1{z_1}$. It is also non-degenerate by Proposition~\ref{prop:zeta_P_no_atoms}, except where $d=1$, $z=1$. It is well-known that a non-degenerate $\alpha$-stable distributions has finite absolute moments of order $p<\alpha$ and that the absolute moments of order $p\geq \alpha$ are infinite; see~\cite[Property~1.2.16]{samorodnitsky_taqqu_book}.

Henceforth, we may assume that $z\in \frac 12 \calD \bsl \R^d$. Consider the map $w\mapsto m^{z_1} w$, $w\in\C$,  as a linear operator on $\C\equiv \R^2$. In the basis $\{1,i\}$, this operator can be written as $m^{B}$, where $B$ is the matrix
$$
B =
\begin{pmatrix}
&\Re z_1 & -\Im z_1\\
&\Im z_1 & \Re z_1
\end{pmatrix}
.
$$
By Proposition~\ref{prop:zeta_P_operator_stable}, the random variable $(z_d-1)\zeta_P(z)$ has operator stable distribution on $\C\equiv \R^2$ with exponent matrix $B$; see~\cite{meerschaert_book} for the definition of the exponent matrix. Moreover, this distribution is full in $\C\equiv \R^2$ by Proposition~\ref{prop:zeta_P_full}.  The spectrum  of $B$ is $\spec B= \{z_1, \bar z_1\}$. It is known that a full operator stable law has finite moments of all orders  $p\in (0,1/\Lambda)$, where  $\Lambda := \max \{\Re \lambda \colon \lambda\in \spec B\}$; see Theorem~3 in~\cite{hudson_etal}.  Also, it is known that the moments of order $p> 1/\Lambda$  are infinite provided that $\Lambda>\frac 12$; see Theorem~4 in~\cite{hudson_etal}.  In our case, we have $\Lambda=\Re z_1>\frac 12$.
\end{proof}

\section{The first level of the GREM}\label{sec:first_level_GREM}

In this section, we collect some results on the \textit{first level} of the GREM. These results will be used to obtain the fluctuations of $\ZZZ_n(\beta)$ in the Poissonian case $|\sigma|>\frac{\sigma_1}{2}$. (Though, let us stress that we do \textit{do not} assume this condition to hold throughout this section.

\subsection{Convergence to the Poisson process}
Recall that the first level of the GREM is labeled by the i.i.d.\ real standard Gaussian random variables $\{\xi_k\colon 1\leq k \leq N_{n,1}\}$. It turns out that if the inverse temperature $\beta\in\C$ is such that $|\sigma| > \frac{\sigma_1}{2}$,  the main contribution of the first level to the partition function $\ZZZ_n(\beta)$ comes from the \textit{extremal order statistics} among the $\xi_k$'s. (Upper order statistics for $\sigma>0$ and lower order statistics for $\sigma<0$).  It is well-known from the standard extreme-value theory~\cite[Theorem~1.5.3]{leadbetter_book} that the appropriately normalized upper order statistics of the $\xi_k$'s converge, as $n\to\infty$, to the Poisson point process with intensity $\eee^{-u}\dd u$ on $\R$. Namely, weakly on $\NNN(\R)$ it holds that
\begin{equation}\label{eq:extreme_order_stats_PPP}
\sum_{k=1}^{N_{n,1}} \delta\left(u_{n,1}(\xi_k -u_{n,1})\right) \toweak \text{PPP}(e^{-u}\dd u).
\end{equation}
Here, the normalizing sequence $u_{n,1}$ is as in~\eqref{eq:u_n_k_tail} and~\eqref{eq:u_n_k_asympt}.
For our purposes, it will be convenient to introduce a transformation of the energies at the first level  which, as we will show in Lemma~\ref{lem:P_n_k_to_Poisson}, maps the upper order statistics of the  $\xi_k$'s to approximately a \textit{homogeneous} Poisson point process on $(0,\infty)$. This transformation will be frequently used in our proofs.  Define random variables $\{P_{n,k} \colon n\in\N, 1\leq k\leq N_{n,1}\}$ by
\begin{equation}\label{eq:P_n_k_aux}
P_{n,k} = \eee^{-\sigma_1 \sqrt {n a_1}\, (\xi_k - u_{n,1})}.
\end{equation}
Note that, for every $n\in\N$, the random variables $\{P_{n,k}\colon 1\leq k\leq N_{n,1}\}$ are i.i.d.\


\begin{lemma}\label{lem:exp_P_n_k_z}
For every $z\in\C$ and $0<A<B$,
$$
\lim_{n\to\infty} N_{n,1} \E [P_{n,k}^{-z}\ind_{A\leq P_{n,k}\leq B}] = \int_A^B y^{-z}\dd y.
$$
If, additionally, $\Re z<1$, then the formula continues to hold even for $0\leq A < B$.
\end{lemma}
\begin{proof}
Introduce the real variables $x$ and $y$ such that $y(x)=\eee^{-\sigma_1 \sqrt{na_1} \, (x-u_{n,1})}$ and, consequently, $x(y)=u_{n,1}-\frac{\log y}{\sigma_1 \sqrt{na_1}}$. The transformation $x\mapsto y$ is a monotone decreasing bijection between $\R$ and $(0,\infty)$.  Recalling that $P_{n,k}=y(\xi_k)$, where $\xi_k\sim N_{\R}(0,1)$, we obtain
\begin{align*}
N_{n,1} \E [P_{n,k}^{-z}\ind_{A\leq P_{n,k}\leq B}]
&=
N_{n,1} \frac{1}{\sqrt{2\pi}} \int_{y^{-1}(B)}^{y^{-1}(A)} (y(x))^{-z} \eee^{-\frac 12 x^2} \dd x\\
&=
N_{n,1} \frac{1}{\sqrt{2\pi}} \int_{A}^{B} y^{-z} \eee^{-\frac 12 (u_{n,1}-\frac{\log y}{\sigma_1 \sqrt{na_1}})^2} \frac{\dd y}{y\sigma_1\sqrt{na_1}}\\
&=
N_{n,1} \frac{\eee^{-\frac 12 u_{n,1}^2}}{\sqrt{2\pi}\sigma_1\sqrt{na_1}}  \int_{A}^{B} y^{-z-1}\eee^{u_{n,1} \frac{\log y}{\sigma_1\sqrt{na_1}}} \eee^{-\frac 12 \frac{(\log y)^2}{\sigma_1^2 na_1}}  \dd y.
\end{align*}
Using the fact that $u_{n,1}\sim \sigma_1 \sqrt{na_1}$, see~\eqref{eq:u_n_k_asympt},  the relation between $u_{n,1}$ and $N_{n,1}$, see~\eqref{eq:u_n_k_tail}, and the dominated convergence theorem (note that the integrand can be estimated by $y^{-\Re z\pm \eps}$, for sufficiently large $n$), we obtain that the right-hand side converges to $\int_A^B y^{-z}\dd y$.
\end{proof}

The next lemma is just a reformulation of~\eqref{eq:extreme_order_stats_PPP}.
\begin{lemma}\label{lem:P_n_k_to_Poisson}
Let $\sum_{k=1}^{\infty} \delta(P_k)$ be a unit intensity Poisson point process on $(0,\infty)$. The following convergence of point processes holds weakly on $\NNN([0,\infty))$:
$$
\sum_{k=1}^{N_{n,1}} \delta(P_{n,k}) \toweak \sum_{k=1}^{\infty} \delta(P_k).
$$
\end{lemma}
\begin{proof}
Let $\mu_n$ be the probability distribution of $P_{n,k}$. Then, Lemma~\ref{lem:exp_P_n_k_z} with $z=0$ implies that the measure $N_{n,1}\mu_n$ converges vaguely to the Lebesgue measure on $[0,\infty)$. Since the variables $\{P_{n,k}\colon 1\leq k\leq N_{n,1}\}$ are i.i.d.,\ this yields the desired weak convergence by~\cite[Proposition~3.21]{resnick_book}.
\end{proof}

\subsection{Asymptotics for the truncated moments of $P_{n,k}$}
In this section, we compute the limits
$$
\lim_{n\to\infty} N_{n,1}\E [P_{n,k}^{-\frac{\beta}{\sigma_1}} \ind_{P_{n,k}> 1}]
\text{ and }
\lim_{n\to\infty} N_{n,1} \E [P_{n,k}^{-\frac{\beta}{\sigma_1}} \ind_{P_{n,k}<1}].
$$
As we have shown in Lemma~\ref{lem:exp_P_n_k_z} (with $z=0$), the probability distribution $\mu_n$ of $P_{n,k}$, multiplied by $N_{n,1}$, converges vaguely to the Lebesgue measure on $(0,\infty)$.  One could therefore try to proceed as follows:
$$
\lim_{n\to\infty} N_{n,1} \E [P_{n,k}^{-\frac{\beta}{\sigma_1}} \ind_{P_{n,k}> 1}]
=
\lim_{n\to\infty} N_{n,1} \int_{1}^{\infty} p^{-\frac{\beta}{\sigma_1}} \mu_n(\dd p)
=
\int_{1}^{\infty} p^{-\frac{\beta}{\sigma_1}} \dd p
=\frac{\sigma_1}{\beta - \sigma_1}.
$$
This approach works in the half-plane $\sigma>\sigma_1$ since under this
condition the integral $\int_{1}^{\infty} p^{-\beta/\sigma_1} \dd p$ is
convergent. However, as we will show in the next lemma, the above formula is
valid in a domain which is \textit{strictly larger} than the half-plane
$\sigma>\sigma_1$. This fact is crucial because, as we will see later, it is
responsible for the beak shaped form of the boundary between the phases $G_k$
and $E_k$.

\begin{figure}
\begin{tabular*}{\textwidth}{p{0.5\textwidth}p{0.5\textwidth}}
\begin{center}
\includegraphics[width=0.3\textwidth]{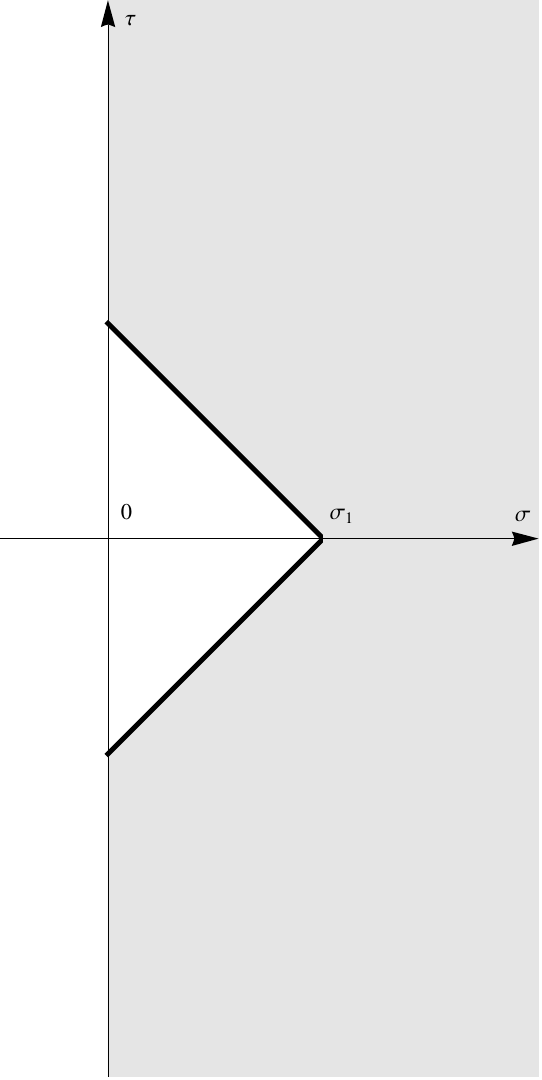}
\end{center}
&
\begin{center}
\includegraphics[width=0.3\textwidth]{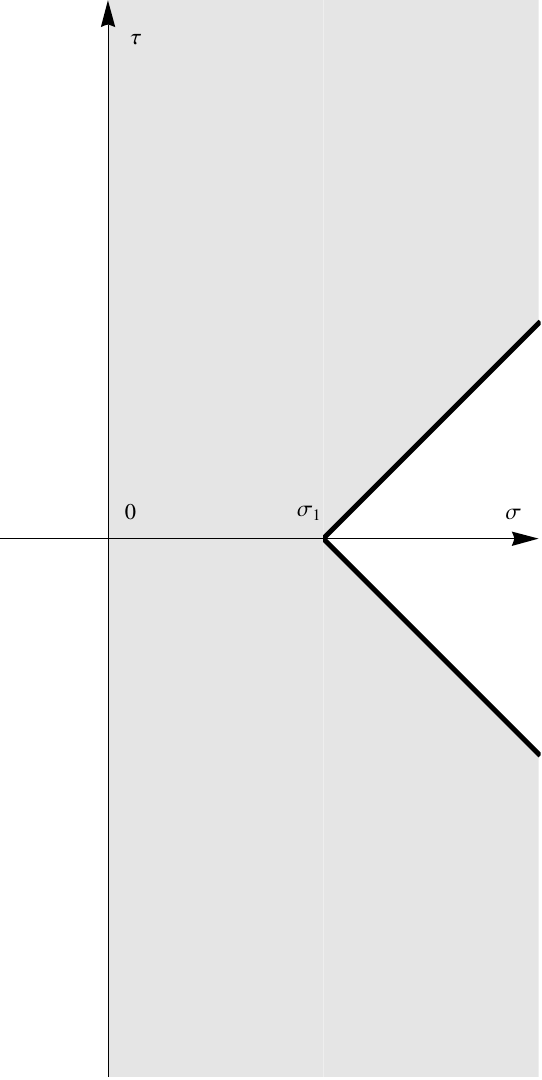}
\end{center}
\\
\begin{center}
{\small Domain in which Lemma~\ref{lem:P_n_k_moment_kljuv_1} is valid.}
\end{center}
&
\begin{center}
{\small Domain in which Lemma~\ref{lem:P_n_k_moment_kljuv_2} is valid.}
\end{center}
\end{tabular*}
\begin{center}
\end{center}
\caption{\small Asymptotics for the truncated moments of $P_{n,k}$}
\label{fig:truncated_moments}
\end{figure}

\begin{lemma}\label{lem:P_n_k_moment_kljuv_1}
Let $K$ be a compact subset of $\{\beta\in \C\colon \sigma>0, \sigma + |\tau| > \sigma_1\}$; see Figure~\ref{fig:truncated_moments}, left. Then, uniformly in $\beta\in K$ we have
$$
\lim_{n\to\infty} N_{n,1} \E [P_{n,k}^{-\frac{\beta}{\sigma_1}} \ind_{P_{n,k}> 1}]
=
\frac{\sigma_1}{\beta-\sigma_1}.
$$
\end{lemma}
\begin{proof}
Let $\xi\sim N_{\R}(0,1)$ be real standard normal variable. By the definition of $P_{n,k}$, see~\eqref{eq:P_n_k_aux}, and by Lemma~\ref{lem:exp_moment_gauss_eq}, Part~2, we have
\begin{align*}
N_{n,1} \E [P_{n,k}^{-\frac{\beta}{\sigma_1}} \ind_{P_{n,k}> 1}]
&=
N_{n,1} \E [\eee^{\beta \sqrt{na_1} (\xi-u_{n,1}) } \ind_{\xi < u_{n,1}}]\\
&=
N_{n,1} \eee^{-\beta \sqrt{na_1} u_{n,1}} \eee^{\frac 12 \beta^2 na_1} \Phi(u_{n,1} - \beta \sqrt{na_1}).
\end{align*}
In the first step, we have used that $\sigma>0$. Since $u_{n,1}\sim \sigma_1\sqrt{na_1}$ by~\eqref{eq:u_n_k_asympt}, we have
$$
\Re (u_{n,1} - \beta \sqrt{na_1}) \sim (\sigma_1-\sigma) \sqrt{na_1},
\quad
\Im (u_{n,1} - \beta \sqrt{na_1}) \sim -\tau \sqrt{na_1}.
$$
It follows from the assumption that $K$ is a compact subset of $\{\sigma + |\tau| > \sigma_1\}$ that we can find $\eps>0$ such that, for all $\beta\in K$ and for all sufficiently large $n\in \N$,
$$
u_{n,1} - \beta \sqrt{na_1} \in \left\{z\in\C\colon |\arg z| >\frac {\pi}{4} + \eps \right\}.
$$
Hence, we can apply Lemma~\ref{lem:Phi_asympt_complex} to obtain that
\begin{align*}
N_{n,1} \E [P_{n,k}^{-\frac{\beta}{\sigma_1}} \ind_{P_{n,k}> 1}]
&\sim
- N_{n,1} \eee^{-\beta \sqrt{na_1}u_{n,1}} \eee^{\frac 12 \beta^2 na_1} \frac{\eee^{-\frac 12 (u_{n,1} - \beta \sqrt{na_1})^2}}{\sqrt{2\pi}(u_{n,1} - \beta \sqrt{na_1})}  \\
&=
\frac{N_{n,1}}{\sqrt{2\pi} u_{n,1}  \eee^{\frac 12 u_{n,1}^2}}  \cdot  \frac{u_{n,1}}{\beta\sqrt{na_1} - u_{n,1}}
\end{align*}
The right-hand side converges to $\frac{\sigma_1}{\beta - \sigma_1}$ by~\eqref{eq:u_n_k_tail} and~\eqref{eq:u_n_k_asympt}.
\end{proof}

\begin{lemma}\label{lem:P_n_k_moment_kljuv_2}
Let $K$ be a compact subset of $\{\beta\in \C\colon \sigma>0, \sigma - |\tau| < \sigma_1\}$; see Figure~\ref{fig:truncated_moments}, right.  Then, uniformly in $\beta\in K$,
$$
\lim_{n\to\infty} N_{n,1} \E [P_{n,k}^{-\frac{\beta}{\sigma_1}} \ind_{P_{n,k} < 1}]
=
- \frac{\sigma_1}{\beta-\sigma_1}.
$$
\end{lemma}
\begin{proof}
The proof is similar to the proof of Lemma~\ref{lem:P_n_k_moment_kljuv_1}.
By the definition of $P_{n,k}$, see~\eqref{eq:P_n_k_aux}, and by Lemma~\ref{lem:exp_moment_gauss_eq}, Part~1, we have
\begin{align*}
N_{n,1} \E [P_{n,k}^{-\frac{\beta}{\sigma_1}} \ind_{P_{n,k} < 1}]
&=
N_{n,1} \E [\eee^{\beta \sqrt{na_1} (\xi-u_{n,1}) } \ind_{\xi > u_{n,1}}]\\
&=
N_{n,1} \eee^{-\beta \sqrt{na_1}u_{n,1}} \eee^{\frac 12 \beta^2 na_1} \bar \Phi(u_{n,1} - \beta \sqrt{na_1})\\
&=
N_{n,1} \eee^{-\beta \sqrt{na_1}u_{n,1}} \eee^{\frac 12 \beta^2 na_1} \Phi(\beta \sqrt{na_1} - u_{n,1}),
\end{align*}
where in the first equality we used that $\sigma>0$ and in the last step we used that $\bar \Phi(z)=\Phi(-z)$.
Since $u_{n,1}\sim \sigma_1\sqrt{na_1}$ by~\eqref{eq:u_n_k_asympt}, we have
$$
\Re (\beta \sqrt{na_1} - u_{n,1}) \sim (\sigma-\sigma_1) \sqrt{na_1},
\quad
\Im (\beta \sqrt{na_1} - u_{n,1}) \sim  \tau \sqrt{na_1}.
$$
It follows from the assumption that $K$ is a compact subset of $\{\sigma - |\tau| < \sigma_1\}$ that there is $\eps>0$ such that for all $\beta\in K$ and for all sufficiently large $n\in \N$,
$$
\beta \sqrt{na_1} - u_{n,1} \in \left\{z\in\C\colon |\arg z| > \frac {\pi}{4} + \eps \right\}.
$$
Hence, we can apply Lemma~\ref{lem:Phi_asympt_complex} to obtain that
\begin{align*}
N_{n,1} \E [P_{n,k}^{-\frac{\beta}{\sigma_1}} \ind_{P_{n,k} < 1}]
&\sim
- N_{n,1} \eee^{-\beta \sqrt{na_1}u_{n,1}} \eee^{\frac 12 \beta^2 na_1} \frac{\eee^{-\frac 12 (\beta \sqrt{na_1} - u_{n,1})^2}}{\sqrt{2\pi}(\beta \sqrt{na_1} - u_{n,1})}  \\
&=
-\frac{N_{n,1}}{\sqrt{2\pi} u_{n,1}  \eee^{\frac 12 u_{n,1}^2}}  \cdot  \frac{u_{n,1}}{\beta\sqrt{na_1} - u_{n,1}}
\end{align*}
The right-hand side converges to $-\frac{\sigma_1}{\beta - \sigma_1}$ by~\eqref{eq:u_n_k_tail} and~\eqref{eq:u_n_k_asympt}.
\end{proof}

\subsection{Estimates for the truncated moments of $P_{n,k}$}
In the next lemmata, we prove some estimates on the truncated moments of $P_{n,k}$.
\begin{lemma}\label{lem:P_n_k_moment_E1_part1}
Let $K$ be a compact  subset of  $\{\beta\in\C\colon 0\leq \sigma < \sigma_1\}$.
Then, there exists a constant $C=C(K)>0$ such that for all  $\beta\in K$ and all $n\in\N$,
$$
N_{n,1} \E |P_{n,k}^{-\frac{\beta}{\sigma_1}} \ind_{P_{n,k}\leq 1}| \leq  C.
$$
\end{lemma}
\begin{proof} Let $\xi\sim N_{\R}(0,1)$. By definition of $P_{n,k}$, see~\eqref{eq:P_n_k_aux},
$$
N_{n,1} \E |P_{n,k}^{-\frac{\beta}{\sigma_1}} \ind_{P_{n,k}\leq 1}| = N_{n,1} \E [\eee^{\sigma  \sqrt{na_1} \, (\xi-u_{n,1}) } \ind_{\xi\geq u_{n,1}}].
$$
We are going to apply Lemma~\ref{lem:exp_moment_gauss_ineq}, Part~$1$,  with $w=\sigma \sqrt{na_1}$ and $a=u_{n,1} \sim \sigma_1 \sqrt{na_1}$. Since $K$ is a compact subset of $\{\sigma<\sigma_1\}$, there exist $n_0\in\N$, $\eps>0$ such that $a>w$ and moreover $a-w>\eps u_{n,1}$ for all $n>n_0$, $\beta\in K$.  By Lemma~\ref{lem:exp_moment_gauss_ineq}, Part~$1$, for all $n>n_0$ and $\beta\in K$,
$$
N_{n,1} \E [\eee^{\sigma \sqrt{na_1} \, (\xi-u_{n,1})} \ind_{\xi\geq u_{n,1}}]
\leq  \frac{CN_{n,1}\eee^{-\frac 12 u_{n,1}^2}}{u_{n,1} - \sigma \sqrt{na_1}}
\leq CN_{n,1} u_{n,1}^{-1} \eee^{-\frac 12 u_{n,1}^2}.
$$
By~\eqref{eq:u_n_k_tail}, the right-hand side is bounded by $C$. If necessary, we can enlarge $C$ so that the estimate holds for all $n\in\N$.
\end{proof}
\begin{lemma}\label{lem:P_n_k_moment_E1_part2a}
Let $K$ be a compact subset of  $\{\beta\in \C\colon \sigma > \sigma_1\}$. Then,  there exist  constants $C=C(K)>0$ and $\eps=\eps(K)>0$ such that for all $\beta\in K$, $T\geq 1$ and all sufficiently large $n\in\N$,
$$
N_{n,1} |\E (P_{n,k}^{-\frac{\beta}{\sigma_1}} \ind_{P_{n,k}> T})|\leq N_{n,1} \E |P_{n,k}^{-\frac{\beta}{\sigma_1}} \ind_{P_{n,k}> T}| < C T^{-\eps}.
$$
\end{lemma}
\begin{proof}
Let $\xi\sim N_{\R}(0,1)$.
By definition of $P_{n,k}$, see~\eqref{eq:P_n_k_aux},
$$
N_{n,1} \E |P_{n,k}^{-\frac{\beta}{\sigma_1}} \ind_{P_{n,k}> T}| = N_{n,1} \E \left[\eee^{\sigma \sqrt{na_1}(\xi- u_{n,1}) } \ind_{\xi < u_{n,1} - \frac{\log T}{\sigma_1\sqrt {na_1}}}\right].
$$
We are going to apply Lemma~\ref{lem:exp_moment_gauss_ineq}, Part~2, with $w=\sigma\sqrt{na_1}$ and $a=u_{n,1} - \frac{\log T}{\sigma_1\sqrt {na_1}}$. We have $a \leq u_{n,1}\sim \sigma_1\sqrt{na_1}$ and hence, for sufficiently large $n$, $a<w$ and moreover, $w-a>\eta u_{n,1}$ for some sufficiently small constant $\eta>0$.  Therefore, by Lemma~\ref{lem:exp_moment_gauss_ineq},
\begin{align*}
N_{n,1} \E |P_{n,k}^{-\frac{\beta}{\sigma_1}} \ind_{P_{n,k}> T}|
&\leq
\frac{CN_{n,1} \eee^{-\frac{\sigma}{\sigma_1} \log T}}{\sigma \sqrt{na_1} - u_{n,1} + \frac{\log T}{\sigma_1\sqrt {na_1}}} \exp\left\{-\frac 12 \left(u_{n,1} - \frac{\log T}{\sigma_1\sqrt {na_1}}\right)^2\right\}\\
&\leq
C T^{\left(\frac{u_{n,1}}{\sigma_1\sqrt{na_1}}-\frac{\sigma}{\sigma_1}\right)} \frac{N_{n,1}}{u_{n,1}} \exp\left\{-\frac 12 u_{n,1}^2 \right\}.
\end{align*}
Since $u_{n,1} \sim \sigma_1\sqrt{na_1}$ and $K$ is a compact subset of $\{\sigma>\sigma_1\}$, we can find $n_0\in \N$ and $\eps>0$ such that for all $\beta\in K$ and $n>n_0$,
$$
\frac{u_{n,1}}{\sigma_1\sqrt{na_1}}-\frac{\sigma}{\sigma_1} < -\eps.
$$
Recalling~\eqref{eq:u_n_k_tail} we obtain the required estimate.
\end{proof}

\begin{lemma}\label{lem:P_n_k_moment_E1_part2b}
Let $K$ be a compact subset of  $\{\beta\in \C\colon \sigma>0, \sigma + |\tau|>\sigma_1\}$. Then,  there exists a constant  $C=C(K)>0$ such that for all $\beta\in K$, $T\geq 1$ and all sufficiently large $n\in\N$,
$$
N_{n,1} |\E (P_{n,k}^{-\frac{\beta}{\sigma_1}} \ind_{P_{n,k}> T})| < C T.
$$
\end{lemma}
\begin{proof}
By the definition of $P_{n,k}$, see~\eqref{eq:P_n_k_aux}, and Lemma~\ref{lem:exp_moment_gauss_eq}, Part~2, we have
\begin{align*}
N_{n,1} \E (P_{n,k}^{-\frac{\beta}{\sigma_1}} \ind_{P_{n,k}> T})
&=
N_{n,1} \E [\eee^{\beta \sqrt{na_1} (\xi - u_{n,1})} \ind_{\xi < a_n(T)}]\\
&=
N_{n,1} \eee^{-\beta \sqrt{na_1} u_{n,1}} \eee^{\frac 12 \beta^2 na_1} \Phi(a_n(T) - \beta \sqrt{na_1}),
\end{align*}
where $a_n(T) = u_{n,1} - \frac{\log T}{\sigma_1\sqrt{na_1}}$.  Since $u_{n,1}\sim \sigma_1\sqrt{na_1}$ by~\eqref{eq:u_n_k_asympt}, we have
\begin{align*}
\Re (a_n(T) - \beta \sqrt{na_1}) &< u_{n,1} - \sigma \sqrt{na_1} = (\sigma_1-\sigma) \sqrt{na_1} +o(\sqrt n),\\
\Im (a_n(T) - \beta \sqrt{na_1}) &= -\tau \sqrt{na_1} +o(\sqrt n),
\end{align*}
where the $o$-term is uniform  in $T$. It follows from the assumption that $K$ is a compact subset of $\{\sigma + |\tau| > \sigma_1\}$ that there is $\eps>0$ such that for all $n\in\N$, $T \geq 1$, $\beta\in K$,
$$
a_n(T) - \beta \sqrt{na_1} \in \left\{z\in \C \colon |\arg z| > \frac{\pi}{4} + \eps\right\}.
$$
Hence, we can use Lemma~\ref{lem:Phi_asympt_complex} to obtain that uniformly in $\beta\in K$ and $T\geq 1$,
\begin{align*}
N_{n,1} \E (P_{n,k}^{-\frac{\beta}{\sigma_1}} \ind_{P_{n,k}> T})
&\sim
C N_{n,1} \eee^{-\beta \sqrt{na_1} u_{n,1}} \eee^{\frac 12 \beta^2 na_1}\cdot \frac{\eee^{-\frac 12 (a_n(T)-\beta\sqrt{na_1})^2}}{a_n(T) - \beta \sqrt{na_1}}\\
&= \frac{C N_{n,1} \eee^{-\frac 12 u_{n,1}^2}}{a_n(T) - \beta\sqrt {na_1}}  \cdot T^{\frac{u_{n,1} - \beta \sqrt{na_1}}{\sigma_1 \sqrt{na_1}}} \cdot \eee^{- \frac 12 \frac{(\log T)^2}{\sigma_1^2 na_1}}.
\end{align*}
We  are going to show that there is a sufficiently small $\delta>0$ such that, for all $T\geq 1$, $\beta\in K$ and all sufficiently large $n$,
\begin{align}
&\frac{u_{n,1} - \sigma \sqrt{na_1}}{\sigma_1 \sqrt{na_1}} \leq 1, \label{eq:tech322a}\\
&|a_n(T) - \beta\sqrt {na_1}| > \delta u_{n,1}. \label{eq:tech322b}
\end{align}
After~\eqref{eq:tech322a} and~\eqref{eq:tech322b} have been established, the proof of the lemma can be completed by recalling~\eqref{eq:u_n_k_tail}.

\vspace*{2mm}
\noindent
\textit{Proof of~\eqref{eq:tech322a}}. The left-hand side in~\eqref{eq:tech322a} converges to $1 - \frac{\sigma}{\sigma_1} < 1$, since $K$ is assumed to be a compact subset of $\{\sigma>0\}$.

\vspace*{2mm}
\noindent
\textit{Proof of~\eqref{eq:tech322b}}.
We can find a sufficiently small $\eps>0$ such that $K$ is contained in the union of the sets  $\{|\tau|> \eps\}$ and  $\{\sigma > \sigma_1+\eps\}$.

\vspace*{2mm}
\noindent
\textsc{Case 1:} $|\tau|>\eps$. For sufficiently large $n$, we have
$$
|a_n(T) - \beta \sqrt{na_1}| \geq   |\Im(a_n(T) - \beta \sqrt{na_1})| = |\tau| \sqrt{na_1} > \delta u_{n,1}.
$$

\vspace*{2mm}
\noindent
\textsc{Case 2:} $\sigma>\sigma_1 + \eps$. For large enough $n$ and all $T \geq 1$, we have
$$
\Re (a_N(T)) \leq u_{n,1} < \left(\sigma_1 + \frac {\eps}2\right)\sqrt{na_1} <  \left(\sigma-\frac{\eps} 2\right) \sqrt{na_1}.
$$
Hence,
\begin{align*}
|a_n(T) - \beta \sqrt{na_1}| \geq   |\Re(a_n(T) - \beta \sqrt{na_1})| > \frac {\eps}{2} \sqrt{na_1} >\delta u_{n,1}.
\end{align*}
This completes the proof of~\eqref{eq:tech322b}.
\end{proof}

\subsection{Adjoining the remaining levels}
The next lemma will be used when we adjoin a new Poissonian level to a GREM with $d-1$ levels. In this lemma, one should  think of $P_{n,k}$ as the contributions of the first level of the GREM and of $\bZ_{n,k}$ as the contributions of the remaining $d-1$ levels.
\begin{lemma}\label{lem:adjoin_level}
Let $P_{n,k}$ be as above, see~\eqref{eq:P_n_k_aux}, and independently,  for every $n \in\N$, let $\{\bZ_{n,k}\colon 1\leq k \leq N_{n,1}\}$ be i.i.d.\ $\C^r$-valued random vectors with $\bZ_{n,k}=\{\bZ_{n,k}(i)\}_{i=1}^r$. Assume that $\bZ_{n,k}$ converges in distribution to some random vector $\bZ=\{\bZ(i)\}_{i=1}^r$, as $n\to\infty$.  Let also $c_{n,1},\ldots,c_{n,r}\in\C$ be sequences such that  $c_i:=\lim_{n\to\infty} c_{n,i}\in \C$ exists, for all $1\leq i\leq r$.  Then, for every $T>0$, we have the following weak convergence of random vectors in $\C^r$:
\begin{equation}\label{eq:adjoin_level}
\left\{\sum_{k=1}^{N_{n,1}} P_{n,k}^{-c_{n,i}} \ind_{P_{n,k}\leq T} \bZ_{n,k}(i)\right\}_{i=1}^r
\todistr
\left\{\sum_{k=1}^{\infty} P_k^{-c_i} \ind_{P_{k}\leq T} \bZ_{k}(i)\right\}_{i=1}^r,
\end{equation}
where $\sum_{k=1}^{\infty} \delta(P_k)$ is a unit intensity Poisson point process on $[0,\infty)$ and, independently,  $\bZ_1,\bZ_2,\ldots$ are i.i.d.\ copies of $\bZ$.
\end{lemma}
\begin{proof}
\textsc{Step 1.} Denote the vector on the left-hand side of~\eqref{eq:adjoin_level} by $\bS_n=\{\bS_n(i)\}_{i=1}^r$ and the vector on the right-hand side of~\eqref{eq:adjoin_level} by $\bS=\{\bS(i)\}_{i=1}^r$. We have to show that, for every continuous bounded function $f\colon \C^r\to\R$, we have
$$
\lim_{n\to\infty} \E f(\bS_n) = \E f(\bS).
$$
Let $K_n$ be the number of points $P_{n,k}$, $1\leq k\leq N_{n,1}$, which satisfy $P_{n,k}\leq T$. Similarly, let $K$ be the number of points $P_k$, $k\in\N$, which satisfy $P_k\leq T$. By the total expectation formula, we need to show that
\begin{equation}\label{eq:adjoin_level_proof_weak_conv_def}
\lim_{n\to\infty} \sum_{l=0}^{\infty} \E [f(\bS_n) | K_n=l] \P[K_n=l]
=
\sum_{l=0}^{\infty} \E [f(\bS) | K=l] \P[K=l].
\end{equation}
The proof of~\eqref{eq:adjoin_level_proof_weak_conv_def} follows from Steps~2, 3, 4 below.

\vspace*{2mm}
\noindent
\textsc{Step 2.} 
By Lemma~\ref{lem:exp_P_n_k_z} (with $z=0$) and the Poisson limit theorem, for every $l\in \N_0$, we have
$$
\lim_{n\to\infty} \P[K_n=l]=\P[K=l].
$$

\vspace*{2mm}
\noindent
\textsc{Step 3.} We show that for every $l\in \N_0$,
\begin{equation}\label{eq:adjoin_level_proof3}
\lim_{n\to\infty} \E [f(\bS_n) | K_n=l] = \E [f(\bS) | K=l].
\end{equation}
Let $\mu_n$ be the distribution of $P_{n,k}$. Conditionally on $K_n=l$, those random variables $P_{n,k}$ that satisfy $P_{n,k}\leq T$ have the same joint distribution as the i.i.d.\ random variables $(Q_{n,1},\ldots, Q_{n,l})$ distributed on $[0,T]$ according to the measure $\mu_n(\cdot)/\mu_n([0,T])$. This distribution converges weakly  to the uniform distribution on $[0,T]$ by Lemma~\ref{lem:exp_P_n_k_z} (with $z=0$) and hence, $(Q_{n,1},\ldots, Q_{n,l})$ converges in distribution to i.i.d.\ random variables $(Q_1,\ldots,Q_l)$ distributed uniformly on $[0,T]$.  We have
$$
\bS_n| \{K_n=l\}
\eqdistr
\left\{ \sum_{j=1}^l Q_{n,j}^{-c_{n,i}} \bZ_{n,j}(i)\right\}_{i=1}^r
\todistr
\left\{ \sum_{j=1}^l Q_{j}^{-c_{i}} \bZ_{j}(i)\right\}_{i=1}^r
\eqdistr
\bS| \{K=l\},
$$
 This proves~\eqref{eq:adjoin_level_proof3}.

\vspace*{2mm}
\noindent
\textsc{Step 4.} To complete the proof of~\eqref{eq:adjoin_level_proof_weak_conv_def}, we need to show that
$$
\lim_{k\to\infty} \limsup_{n\to\infty} \sum_{l=k}^{\infty} \E [f(\bS_n) | K_n=l] \P[K_n=l]=0.
$$
However, this follows from the estimate
$$
\limsup_{n\to\infty} \sum_{l=k}^{\infty} \E [f(\bS_n) | K_n=l] \P[K_n=l] \leq C \lim_{n\to\infty} \P[K_n\geq k] = \P[K\geq k],
$$
where we used the boundedness of the function $f$ and Step~2.
\end{proof}

\section{Moment estimates in phases without fluctuation levels}\label{sec:moments_GE}
\subsection{Introduction and notation}
In this section, we obtain estimates for the moments of $\ZZZ_n(\beta)$ and some related processes in phases of the form $G^{d_1}E^{d-d_1}$, where $0\leq d_1\leq d$.  The main results of this section, Proposition~\ref{prop:moment_S_n} and Lemma~\ref{lem:S_n_T_minus_S_n_GE}, will be a crucial ingredient in the proofs of functional limit theorems in Section~\ref{sec:func_CLT_GE}.
Some of our most important moment estimates will be valid in the domain
\begin{align*}
\calO &= \left (E_2\cup \left(\bigcup_{d_1=2}^d G^{d_1}E^{d-d_1}\right) \right) \cap \left\{\sigma > 0\right\}.
\end{align*}
Note that the set $\calO$ is open. It \textit{does} include the beak shaped boundary between $E_1=E^d$ and $G^1 E^{d-1}$ but it \textit{does not} include the boundaries between $G^{d_1-1}E^{d-d_1+1}$ and $G^{d_1}E^{d-d_1}$ for $2\leq d_1\leq d$.

To state our results, we need to define $S_n(\beta)$ and $S_n^{\circ}(\beta)$, two normalized versions of the random partition function $\ZZZ_n(\beta)$. It turns out that $S_n^{\circ}(\beta)$ is the ``correct'' normalization the sense that $S_n^{\circ}(\beta)$ has non-trivial limiting fluctuations in $\calO\cap \{\sigma > \frac{\sigma_1}{2}\}$; see for example Theorem~\ref{theo:functional_E1} below.
First, we define a normalizing sequence
\begin{equation}\label{eq:c_n_beta_tilde}
\tilde c_n(\beta) = c_{n,2}(\beta) + \ldots + c_{n,d}(\beta),
\end{equation}
where, for $2\leq k \leq d$ and $\beta\in \calO$,
\begin{equation}\label{eq:c_n_k_beta_def_repetition}
c_{n,k}(\beta) =
\begin{cases}
\beta \sqrt{na_k} u_{n,k}, &\text{ if } \beta\in G_k,\\
\log N_{n,k} + \frac 12 a_k \beta^2 n, &\text{ if } \beta\in E_k.
\end{cases}
\end{equation}
Think of $\eee^{\tilde c_n(\beta)}$ as of the sequence needed to normalize the levels $2,\ldots,d$.
Let $\{S_n(\beta)\colon \beta\in \calO\}$ and $\{S_n^{\circ}(\beta)\colon \beta\in  \calO\}$ be  random analytic functions defined by
\begin{align}
S_n(\beta) &= \frac{\ZZZ_n(\beta)} {\eee^{\beta \sqrt{na_1} u_{n,1} + \tilde c_n(\beta)}},\label{eq:S_n_beta_def}\\
\quad
S_n^{\circ}(\beta) &= \frac{\ZZZ_n(\beta) - \eee^{\tilde c_n(\beta)} N_{n,1} \E [\eee^{\beta \sqrt{na_1} \xi } \ind_{\xi < u_{n,1}}] } {\eee^{\beta \sqrt{na_1} u_{n,1} + \tilde c_n(\beta)}}.\label{eq:S_n_beta_circ_def}
\end{align}
Note that in $S_n(\beta)$ the $k$-th level, for $2\leq k\leq d$, is normalized by the expectation if $\beta\in E_k$ or by the order of the maximal energy on this level if $\beta\in G_k$. The first level is always normalized by the order of the maximal energy, even in the case $\beta\in E_1\cap\{\sigma>\frac{\sigma_1}{2}\}$ (where normalization by expectation may seem more natural at a first sight). Note also that $S_n^{\circ}(\beta)$ differs from $S_n(\beta)$ by an additional additive normalization. In the sequel, we agree to mark by $\vphantom{S}^\circ$ random variables normalized by some sort of truncated expectation.


\subsection{Second moment estimate in $\{|\beta| < \frac{\sigma_1}{\sqrt 2}\}$}
We start with a simple second moment estimate for $\ZZZ_n(\beta)$.
\begin{proposition}\label{prop:moment_S_n_Var_E_1}
Let $K$ be a compact subset of the disk $\{|\beta| < \frac{\sigma_1}{\sqrt 2}\}$.  Then, there exist constants $C=C(K)>0$ and $\eps=\eps(K)>0$ such that for all $\beta\in K$ and all $n\in\N$,
$$
\E\left|\frac{\ZZZ_n(\beta)}{\E \ZZZ_n(\beta)} - 1 \right|^2 \leq C \eee^{-\eps n}.
$$
\end{proposition}
\begin{proof}
Using Propositions~\ref{prop:asympt_exp_variance_log_scale}, \ref{prop:asympt_expect} and then~\eqref{eq:asympt_N_nk} and~\eqref{eq:beta_k_def}, we obtain that uniformly in $K$,
$$
\E\left|\frac{\ZZZ_n(\beta)}{\E \ZZZ_n(\beta)} - 1 \right|^2
=
\frac{\Var \ZZZ_n(\beta)}{|\E \ZZZ_n(\beta)|^2}
\sim
N_{n,1}^{-1} \eee^{|\beta|^2 a_1 n}
=
\eee^{-(\frac 12 \sigma_1^2 - |\beta|^2) a_1 n +o(1)}.
$$
Since  $\frac 12 \sigma_1^2 - |\beta|^2$ admits a strictly positive uniform lower bound on $K$, we can estimate the right-hand side by $\eee^{-\eps n}$, for sufficiently large $n$. Choosing the constant $C$ large enough, we can achieve that the estimate holds for all $n\in\N$.
\end{proof}

\subsection{The main estimate  and its corollaries}
Unfortunately, the second moment estimate of Proposition~\ref{prop:moment_S_n_Var_E_1} is valid in a very small domain only. In order to obtain estimates for larger domains, we need to replace the second moment by the moment of order $p\in (0,2)$. The main result of this section  can be stated as follows.
\begin{proposition}\label{prop:moment_S_n}
Fix $p\in (0,2)$.
\begin{enumerate}
\item[\textup{(1)}] Let $K$ be a compact subset of $E_2 \cap \{\frac{\sigma_1}{2} < \sigma < \frac{\sigma_1}{p}\}$. Then, there exists a constant $C=C(K)>0$ such that for all $\beta\in K$ and all $n\in\N$,
\begin{equation}\label{eq:prop:moment_S_n_part1}
\E |S_n^{\circ}(\beta)|^p < C.
\end{equation}

\item[\textup{(2)}] Let $K$ be a compact subset of $G^{d_1}E^{d-d_1}\cap\{0\leq \sigma<\frac{\sigma_1}{p}\}$, where $1\leq d_1\leq d$. Then, there is a constant $C=C(K)>0$ such  that for all $\beta\in K$ and all $n\in\N$,
\begin{equation}\label{eq:prop:moment_S_n_part2}
\E |S_n(\beta)|^p \leq C, \quad \E |S_n^{\circ}(\beta)|^p \leq C,
\end{equation}
\end{enumerate}
\end{proposition}
From the first part of Proposition~\ref{prop:moment_S_n}, we can draw the following corollaries on the moments of $\ZZZ_n(\beta)$ in $E_1$.
\begin{corollary}\label{cor:moment_S_n_1}
Fix $p\in (0,2)$. Let $K$ be a compact subset of $E_1 \cap \{|\sigma|< \frac{\sigma_1}{p}\}$. Then, there exist $C=C(K)>0$ and $\eps=\eps(K)>0$ such that for all $\beta\in K$ and all $n\in\N$,
\begin{equation}\label{eq:cor:moment_S_n_1}
\E\left|\frac{\ZZZ_n(\beta)}{\E \ZZZ_n(\beta)} - 1 \right|^p < C \eee^{-\eps n}.
\end{equation}
\end{corollary}

\begin{corollary}\label{cor:moment_S_n_2}
Fix $p\in (0,2)$. Let $K$ be a compact subset of $E_1 \cap \{|\sigma|< \frac{\sigma_1}{p}\}$. Then, there exists $C=C(K)>0$  such that for all $\beta\in K$ and all $n\in\N$,
\begin{equation}\label{eq:cor:moment_S_n_2}
\E\left|\frac{\ZZZ_n(\beta)}{\E \ZZZ_n(\beta)}\right|^p < C.
\end{equation}
\end{corollary}
Before turning to the proof of Proposition~\ref{prop:moment_S_n}, we show how to deduce Corollaries~\ref{cor:moment_S_n_1} and~\ref{cor:moment_S_n_2} from Proposition~\ref{prop:moment_S_n}.

\begin{proof}[Proof of Corollary~\ref{cor:moment_S_n_1} given Proposition~\ref{prop:moment_S_n}]
By Proposition~\ref{prop:moment_S_n_Var_E_1}, together with Lyapunov's inequality~\eqref{eq:lyapunov_ineq}, the required estimate~\eqref{eq:cor:moment_S_n_1} holds in any compact subset of the disk $\{|\beta| < \frac{\sigma_1}{\sqrt {2}}\}$.

Therefore, in the rest of the proof  we may assume that $K$ is a compact subset  of $E_1 \cap \{\frac{\sigma_1}{2} < |\sigma| < \frac{\sigma_1}{p}\}$. By symmetry, see~\eqref{eq:symmetry1}, we can also assume that $K\subset \{\sigma\geq 0\}$.
Expressing $\ZZZ_n(\beta)$ in terms of $S_n^{\circ}(\beta)$, see~\eqref{eq:S_n_beta_circ_def}, we obtain
\begin{equation}\label{eq:cor:moment_S_n_1_proof1}
\E\left|\frac{\ZZZ_n(\beta)}{\E \ZZZ_n(\beta)} - 1 \right|^p
=
\E \left|\frac{S_n^{\circ}(\beta) \eee^{\beta \sqrt {na_1} u_{n,1}}} {\eee^{\frac 12 \beta^2 na_1} N_{n,1}} + \frac{\E [\eee^{\beta \sqrt{na_1} \xi } \ind_{\xi < u_{n,1}}] }{\eee^{\frac 12 \beta^2 na_1}}-1 \right|^p.
\end{equation}
Applying  Lemma~\ref{lem:exp_moment_gauss_eq}, we obtain
$$
\frac{\E [\eee^{\beta \sqrt{na_1} \xi } \ind_{\xi < u_{n,1}}] }{\eee^{\frac 12 \beta^2 na_1}}-1
=
-\frac{\E [\eee^{\beta \sqrt{na_1} \xi } \ind_{\xi > u_{n,1}}]} {\eee^{\frac 12 \beta^2 na_1}}
=
-\bar \Phi(u_{n,1}-\beta \sqrt{na_1}).
$$
By Jensen's inequality~\eqref{eq:jensen_ineq} applied to~\eqref{eq:cor:moment_S_n_1_proof1},
$$
\E\left|\frac{\ZZZ_n(\beta)}{\E \ZZZ_n(\beta)} - 1 \right|^p
\leq
C
\E |S_n^{\circ}(\beta)|^p  \left(\frac{\eee^{\sigma \sqrt {na_1} u_{n,1}}} {\eee^{\frac 12 (\sigma^2-\tau^2) na_1} N_{n,1}}\right)^p
+
C (\bar \Phi(u_{n,1}-\beta \sqrt{na_1}))^p.
$$
To complete the proof, we need to estimate the terms on the right-hand side. This will be done in $3$ steps.

\vspace*{2mm}
\noindent
\textsc{Step 1.} By Proposition~\ref{prop:moment_S_n}, Part~1, there is a constant $C=C(K)>0$ such that $\E |S_n^{\circ}(\beta)|^p < C$ for all $\beta\in K$ and all $n\in\N$.

\vspace*{2mm}
\noindent
\textsc{Step 2.} Recall that $u_{n,1}\sim \sigma_1 \sqrt{na_1}$ and $N_{n,1} = \eee^{\frac 12 \sigma_1^2 na_1 + o(1)}$; see~\eqref{eq:u_n_k_asympt} and~\eqref{eq:asympt_N_nk}.  It follows that
$$
\frac{\eee^{\sigma \sqrt {na_1} u_{n,1}}} {\eee^{\frac 12 (\sigma^2-\tau^2) na_1} N_{n,1}}
=
\eee^{-\frac 12 na_1((\sigma_1-\sigma)^2-\tau^2) +o(n)}
<\eee^{-\eps n},
$$
for suitable $\eps>0$ and all sufficiently large $n$. Here, we have used the fact that $(\sigma_1-\sigma)^2-\tau^2$ admits a strictly positive uniform  lower bound on $K$ since $K$ is a compact subset of $E_1$.

\vspace*{2mm}
\noindent
\textsc{Step 3.} Let $z_n = u_{n,1}-\beta \sqrt{na_1}$. Then, $\Re z_n\sim (\sigma_1-\sigma) \sqrt{na_1}$ and $\Im z_n \sim -\tau \sqrt{na_1}$. Since $K$  is a compact subset of $E_1$,  it follows that there is $\delta=\delta(K)>0$ such that $|\arg z_n| < \frac {\pi}{4}-\delta < \frac {3\pi}{4}-\delta$ for all sufficiently large $n$ and all $\beta\in K$.  By Lemma~\ref{lem:Phi_asympt_complex} (second line of~\eqref{eq:Phi_asympt_complex}),  we have
$$
|\bar \Phi(u_{n,1}-\beta \sqrt{na_1})| < \eee^{-\frac 12 na_1((\sigma_1-\sigma)^2-\tau^2) +o(n)} < \eee^{-\eps n},
$$
for sufficiently large $n$, where the final estimate uses the same argumentation as in Step~$2$.

\vspace*{2mm}
Combining the $3$ steps, we get the required estimate~\eqref{eq:cor:moment_S_n_1}, for sufficiently large $n$. By enlarging $C$, if necessary, we can achieve that it holds for all $n$.
\end{proof}

\begin{proof}[Proof of Corollary~\ref{cor:moment_S_n_2} given Corollary~\ref{cor:moment_S_n_1}]
By Jensen's inequality~\eqref{eq:jensen_ineq}, we have
$$
\E\left|\frac{\ZZZ_n(\beta)}{\E \ZZZ_n(\beta)}\right|^p \leq C \E\left|\frac{\ZZZ_n(\beta)}{\E \ZZZ_n(\beta)} - 1 \right|^p +C.
$$
The required estimate~\eqref{eq:cor:moment_S_n_2} follows from Corollary~\ref{cor:moment_S_n_1}.
\end{proof}

\subsection{Proof of Proposition~\ref{prop:moment_S_n}} We use induction over
$d$, the number of levels of the GREM.  Assume that
Proposition~\ref{prop:moment_S_n} (and hence, Corollaries~\ref{cor:moment_S_n_1}
and~\ref{cor:moment_S_n_2}) are valid for any GREM with $d-1$ levels.  Our aim
is to prove that Proposition~\ref{prop:moment_S_n} holds for a GREM with $d$
levels. The main idea is to consider the first level  separately, and to apply
the induction assumption to the remaining $d-1$ levels. All variables which
refer to these remaining levels will be marked by a tilde ``$\sim$''. For
example, we define an index set
\begin{equation}\label{eq:parameter_set_def_tilde}
\tilde \SSS_n=\{\tilde \eps=(\eps_2,\ldots,\eps_d)\in \N^d\colon 1\leq \eps_2 \leq N_{2,n}, \ldots, 1\leq \eps_d \leq N_{d,n}\}.
\end{equation}

Define the random variables $P_{n,k}$, $n\in\N$, $1\leq k\leq N_{n,1}$, (the normalized contributions of the first level of the GREM) and the random analytic functions $\{\tilde Z_{n,k}(\beta)\colon \beta \in \calO\}$, $n\in\N$, $1\leq k\leq N_{n,1}$, (the
normalized contributions of the remaining $d-1$ levels of the GREM) by
\begin{align}
P_{n,k} &= \eee^{-\sigma_1 \sqrt {n a_1} (\xi_k - u_{n,1})}, \label{eq:P_n_k_def}\\
\tilde Z_{n,k} (\beta) &= \eee^{-\tilde c_{n}(\beta)} \sum_{\tilde \eps \in \tilde \SSS_n} \eee^{\beta \sqrt n (\sqrt{a_2} \xi_{k\eps_2} + \ldots + \sqrt{a_d} \xi_{k\eps_2\ldots\eps_d})} \label{eq:Z_n_k_tilde_def}.
\end{align}
By the definition of the GREM, these random variables have the following properties, for every $n\in\N$:
\begin{enumerate}
\item $\tilde Z_{n,k}(\beta)$, $1\leq k\leq N_{n,1}$, is an i.i.d.\ collection of random processes.
\item $P_{n,k}$, $1\leq k\leq N_{n,1}$, is an i.i.d.\ collection of random variables.
\item These two collections are independent.
\end{enumerate}
The properties of $P_{n,k}$ have been studied in Section~\ref{sec:first_level_GREM}. It is useful to keep in mind that $\{P_{n,k}\colon 1\leq k\leq N_{n,1}\}$ is approximatively (for $n\to\infty$) a homogeneous Poisson point process on $(0,\infty)$; see Lemma~\ref{lem:P_n_k_to_Poisson}. Note also that, for $\beta\in E_2$, we have $\E \tilde Z_{n,k}(\beta)=1$, but, for general $\beta\in \calO$, this need not be true.

We have the following representations
\begin{align}
S_{n} (\beta)
&=
\sum_{k=1}^{N_{n,1}} P_{n,k}^{-\frac{\beta}{\sigma_1}} \tilde Z_{n,k}(\beta), \label{eq:S_n_rep}
\\
S_{n}^{\circ} (\beta)
&=
\sum_{k=1}^{N_{n,1}} P_{n,k}^{-\frac{\beta}{\sigma_1}} \tilde Z_{n,k}(\beta)
-
N_{n,1} \E [ P_{n,k}^{-\frac{\beta}{\sigma_1}} \ind_{1\leq P_{n,k}}]. \label{eq:S_n_circ_rep}
\end{align}
For $T\in \N$, define truncated versions of $S_n(\beta)$ and $S_n^{\circ}(\beta)$ by
\begin{align}
S_{n,T} (\beta)
&=
\sum_{k=1}^{N_{n,1}} P_{n,k}^{-\frac{\beta}{\sigma_1}} \ind_{P_{n,k}\leq T}  \tilde Z_{n,k}(\beta), \label{eq:S_n_T_def}
\\
S_{n,T}^{\circ} (\beta)
&=
\sum_{k=1}^{N_{n,1}} P_{n,k}^{-\frac{\beta}{\sigma_1}} \ind_{P_{n,k}\leq T}  \tilde Z_{n,k}(\beta)
-
N_{n,1} \E [ P_{n,k}^{-\frac{\beta}{\sigma_1}} \ind_{1\leq P_{n,k}\leq T}]. \label{eq:S_n_T_circ_def}
\end{align}
The random function $\tilde Z_{n,k}(\beta)$ is the $(d-1)$-level analogue of $\eee^{-c_n(\beta)} \ZZZ_n(\beta)$, where $c_n(\beta)$ is defined by~\eqref{eq:def_c_n_beta}.   Note that $\tilde Z_{n,k}(\beta)$ corresponds to a GREM with branching exponents $(\alpha_2,\ldots,\alpha_d)$ and variances $(a_2,\ldots,a_d)$. Since we assumed that Proposition~\ref{prop:moment_S_n} (and hence, Corollary~\ref{cor:moment_S_n_2}) is valid for any GREM with $d-1$ levels, we have the following induction assumption.  Fix some  $r\in (0,2)$.
\begin{itemize}
\item [(IND1)] Let $K'$ be a compact subset of $E_2\cap \{0\leq \sigma < \frac {\sigma_2}{r}\}$. Then, there exists a constant $C=C(K)>0$ such that for all $\beta\in K'$ and all $n\in\N$,
\begin{equation*}
\E |\tilde Z_{n,k}(\beta)|^r < C.
\end{equation*}
\item [(IND2)]
Let $K'$ be a compact subset of $G^{d_1}E^{d-d_1}\cap \{0\leq \sigma < \frac {\sigma_2}{r}\}$, for some $2\leq d_1\leq d$.  Then, there exists a constant $C=C(K)>0$ such that for all $\beta\in K'$ and all $n\in\N$,
\begin{equation*}
\E |\tilde Z_{n,k}(\beta)|^r < C.
\end{equation*}
\end{itemize}
Note that $\text{(IND1)}$ follows from Corollary~\ref{cor:moment_S_n_2} (which, as we have already shown, follows from Proposition~\ref{prop:moment_S_n}), whereas $\text{(IND2)}$ follows directly from Proposition~\ref{prop:moment_S_n}, Part~2.
Note that in the case $d=1$ (which is the basis of our induction), we have $\tilde Z_{n,k}(\beta)=1$ so that $\text{(IND1)}$ is valid while $\text{(IND2)}$ is empty. Equivalently, we can state $\text{(IND1)}$ and $\text{(IND2)}$ as follows:
\begin{itemize}
\item [(IND)] Fix some  $r\in (0,2)$. Let $K'$ be a compact subset of
$
\calO \cap \{0\leq \sigma < \frac {\sigma_2}{r}\}.
$
Then, there exists a constant $C=C(K)>0$ such that for all $\beta\in K'$ and all $n\in\N$,
\begin{equation}\label{eq:Z_n_k_tilde_moment_induction}
\E |\tilde Z_{n,k}(\beta)|^r < C.
\end{equation}
\end{itemize}

\vspace*{2mm}
\noindent
Our aim is to prove~\eqref{eq:prop:moment_S_n_part1} and~\eqref{eq:prop:moment_S_n_part2}. This will be done in $4$ steps.

\vspace*{2mm}
\noindent
\textsc{Step 1.} In this step, we estimate the moments of $S_{n,1}(\beta) = S_{n,1}^{\circ}(\beta)$.
\begin{lemma}\label{lem:S_n_1_moments}
Fix $p\in (0,2)$. Let $K$ be a compact subset of $\calO \cap \{0\leq \sigma < \frac{\sigma_1}{p}\}$. Then, there is a constant $C=C(K)>0$ such that for all $\beta\in K$ and all $n\in\N$,
\begin{equation}\label{eq:S_n_1_moment_est}
\E |S_{n,1}(\beta)|^p = \E |S_{n,1}^{\circ}(\beta)|^p < C.
\end{equation}
\end{lemma}
\begin{proof}
For future use, note the inequality, valid for $\beta\in K$,
\begin{equation}\label{eq:S_n_1_tech1}
N_{n,1} \E [P_{n,k}^{-\frac{\sigma p}{\sigma_1}}\ind_{P_{n,k} \leq 1}] \cdot \E |\tilde Z_{n,k}(\beta)|^p < C.
\end{equation}
Indeed, by Lemma~\ref{lem:P_n_k_moment_E1_part1} (recall that $K\subset \{\sigma p  < \sigma_1\}$) we can estimate  the first factor on the left-hand side by $C$. Also, by the induction assumption~\eqref{eq:Z_n_k_tilde_moment_induction} we have $\E  |\tilde Z_{n,k}(\beta)|^p\leq C$ (recall that $K\subset \calO$ and $K\subset \{0\leq \sigma < \frac{\sigma_1}{p}\}\subset  \{0\leq \sigma < \frac{\sigma_2}{p}\}$).

\vspace*{2mm}
\noindent
\textsc{Case 1:} $0< p \leq 1$.  Then, by Proposition~\ref{prop:ineq_moment_p_01} and~\eqref{eq:S_n_1_tech1},
$$
\E |S_{n,1}(\beta)|^p \leq N_{n,1} \E [P_{n,k}^{-\frac{\sigma p}{\sigma_1}}\ind_{P_{n,k} \leq 1}] \, \E |\tilde Z_{n,k}(\beta)|^p < C.
$$

\vspace*{2mm}
\noindent
\textsc{Case 2:} $1 \leq  p < 2$.
Then, by Proposition~\ref{prop:von_bahr_esseen_non_centered} and~\eqref{eq:S_n_1_tech1},
\begin{align*}
\E |S_{n,1}(\beta)|^p
&\leq
C N_{n,1}  \E [P_{n,k}^{-\frac{\sigma p}{\sigma_1}} \ind_{P_{n,k} \leq 1}] \, \E |\tilde Z_{n,k}(\beta)|^p + C|\E S_{n,1}(\beta)|^p\\
&<
C + C|\E S_{n,1}(\beta)|^p.
\end{align*}
We need to estimate $\E S_{n,1}(\beta)$. We have
$$
|\E S_{n,1}(\beta)| = N_{n,1}  | \E (P_{n,k}^{-\frac{\beta}{\sigma_1}} \ind_{P_{n,k} \leq 1})| \cdot |\E \tilde Z_{n,k}(\beta)|.
$$
By Lemma~\ref{lem:P_n_k_moment_E1_part1} (recall that  $K\subset \{\sigma  < \frac{\sigma_1}{p}\} \subset \{\sigma  < \sigma_1\}$, since $p\geq 1$), we can estimate  the first factor on the left-hand side by $C$. Also, by the induction assumption~\eqref{eq:Z_n_k_tilde_moment_induction} we have $\E  |\tilde Z_{n,k}(\beta)|^p \leq C$ (recall that $K\subset \calO$ and $K\subset \{0\leq \sigma < \frac{\sigma_1}{p}\}\subset  \{0\leq \sigma < \frac{\sigma_2}{p}\}$). By Lyapunov's inequality~\eqref{eq:lyapunov_ineq} (recall that $p\geq 1$), this implies that $|\E  \tilde Z_{n,k}(\beta)| \leq C$. It follows that $|\E S_{n,1}(\beta)| < C$.
\end{proof}

\vspace*{2mm}
\noindent
\textsc{Step 2.} The aim of this step is to obtain estimates for the $p$-th moments of $S_n(\beta)-S_{n,T}(\beta)$ and $S_n^{\circ}(\beta)-S_{n,T}^{\circ}(\beta)$.
The inequalities which will prove will be needed in Section~\ref{sec:adjoin_spin_glass_GE_only}. The main result of this step is Lemma~\ref{lem:S_n_T_minus_S_n_GE}.
\begin{lemma}\label{lem:E_S_n_minus_S_n_T}
Let $K$ be a compact subset of  $G^{d_1}E^{d-d_1}\cap\{0 \leq  \sigma<\sigma_2\}$, where $2\leq d_1\leq d$.  Then, there exist constants $C=C(K)>0$ and $\eps=\eps(K)>0$  such that for all $\beta\in K$, $T\in \N$, $n\in\N$,
\begin{equation}\label{eq:moment_S_n_T_tail2}
|\E (S_{n}(\beta) - S_{n,T}(\beta))| < C T \, \eee^{-\eps n}.
\end{equation}
\end{lemma}
\begin{remark}\label{rem:lem:E_S_n_minus_S_n_T}
In the case $d_1=1$, we will prove a weaker estimate $|\E (S_{n}(\beta) - S_{n,T}(\beta))| < C T$.
\end{remark}
\begin{proof}[Proof of Lemma~\ref{lem:E_S_n_minus_S_n_T} and Remark~\ref{rem:lem:E_S_n_minus_S_n_T}]
Let $K$ be a compact subset of $G^{d_1}E^{d-d_1}\cap\{0 \leq  \sigma<\sigma_2\}$, where $1\leq d_1\leq d$. The subsequent estimates are valid uniformly over $\beta\in K$.  We have $K\subset \{\sigma>\frac{\sigma_1}{2}, \sigma + |\tau| > \sigma_1\}$, that is the first level of the GREM is in the glassy phase.
Thus, we can apply Lemma~\ref{lem:P_n_k_moment_E1_part2b} to obtain
$$
|\E (S_{n}(\beta) - S_{n,T}(\beta))| = N_{n,1} |\E P_{n,k}^{-\frac{\beta}{\sigma_1}}\ind_{P_{n,k}\geq T}|\, |\E \tilde Z_{n,k}(\beta)|
\leq
CT\, |\E \tilde Z_{n,k}(\beta)|.
$$
We have to estimate $\E \tilde Z_{n,k}(\beta)$. By definition of $\tilde Z_{n,k}(\beta)$, see~\eqref{eq:Z_n_k_tilde_def}, we have
\begin{equation}\label{eq:E_Z_n_k_tilde_computation}
\E \tilde Z_{n,k}(\beta) =  \prod_{l=2}^d  \left(\eee^{-c_{n,l}(\beta)} N_{n,l} e^{\frac 12 \beta^2 a_l n}\right).
\end{equation}
Note that the formula for $c_{n,l}(\beta)$ depends on whether $\beta\in G_l$ or $\beta\in E_l$; see~\eqref{eq:c_n_k_beta_def_repetition}.

\vspace*{2mm}
\noindent
\textsc{Case 1:} $d_1=1$. Then, $\beta\in E_l$ for $2\leq l \leq d$. With other words, the levels $2,\ldots,d$ are normalized by expectation; see~\eqref{eq:c_n_k_beta_def_repetition}.  Hence, all terms in~\eqref{eq:E_Z_n_k_tilde_computation} are equal to $1$ and $\E \tilde Z_{n,k}(\beta)=1$. This proves Remark~\ref{rem:lem:E_S_n_minus_S_n_T}.

\vspace*{2mm}
\noindent
\textsc{Case 2:} $2\leq d_1\leq d$. For $d_1 < l \leq d$, we have $\beta\in E_l$ and hence, the corresponding factor in~\eqref{eq:E_Z_n_k_tilde_computation} is equal to $1$; see~\eqref{eq:def_c_nk}. However, there is at least one $l$  with  $2\leq l\leq d_1$. For such $l$, we have $\beta\in G_l$ and~\eqref{eq:def_c_nk} yields
$$
|\eee^{-c_{n,l}(\beta)} N_{n,l} \eee^{\frac 12 \beta^2  a_l n}|
=
\eee^{-\sigma \sqrt {n a_l} \, u_{n,l} + \log N_{n,l} + \frac 12 (\sigma^2-\tau^2)  a_l n}
=
\eee^{\frac 12 na_l ((\sigma_l-\sigma)^2 - \tau^2) + o(n)}.
$$
In the last step, we used~\eqref{eq:asympt_N_nk} and~\eqref{eq:u_n_k_asympt}. By the assumption of the lemma, we have $\sigma<\sigma_2\leq \sigma_l$. Also, it follows from $\beta\in G_l$ that $0<\sigma_l-\sigma<|\tau|$. So, the expression $(\sigma_l-\sigma)^2-\tau^2$ admits a strictly negative upper bound $-\eps$ on $K$.  Hence, for every $2\leq l \leq d_1$ (and there is at least one such $l$) we can estimate the right-hand side by $C\eee^{-\eps n}$. It follows that, for $2\leq d_1\leq d$, we have the estimate
$$
|\E \tilde Z_{n,k}(\beta)| < C \eee^{-\eps n}.
$$
This completes the proof of~\eqref{eq:moment_S_n_T_tail2}.
\end{proof}

\begin{lemma}\label{lem:S_n_T_minus_S_n_GE}
Fix $p\in (0,2)$.
\begin{enumerate}

\item[\textup{(1)}] Let $K$ be a compact subset of $E_2 \cap \{\frac {\sigma_1}{2} < \sigma < \frac{\sigma_2}{p}\}$.  There exist constants $C=C(K)>0$ and $\eps=\eps(K)>0$ such that for all $\beta\in K$, $T\in \N$, $n\in\N$:
\begin{align}
& \E |S_{n}^{\circ}(\beta) - S_{n,T}^{\circ}(\beta)|^p \leq C T^{-\eps}.  \label{eq:lem:S_n_T_minus_S_n_case1}
\end{align}

\item[\textup{(2)}] Let $K$ be a compact subset of $G^{d_1}E^{d-d_1}\cap \{0\leq \sigma<\frac{\sigma_2}{p}\}$, where $2\leq d_1\leq d$.  There exist constants $C=C(K)>0$ and $\eps=\eps(K)>0$ such that for all $\beta\in K$, $T\in \N$, $n\in\N$:
\begin{align}
& \E |S_{n}(\beta) - S_{n,T}(\beta)|^p \leq C T^{-\eps} + CT^2 \eee^{-\eps n}.  \label{eq:lem:S_n_T_minus_S_n_case2}
\end{align}
\end{enumerate}
\end{lemma}
\begin{remark}
Under the assumptions of Part~2, there exists a constant  $C=C(K)>0$ such that for all $\beta\in K$, $T\in \N$, $n\in\N$ we have
\begin{equation} \label{eq:lem:S_n_T_minus_S_n_case2_rem}
\E |S_{n}^{\circ}(\beta) - S_{n,T}^{\circ}(\beta)|^p \leq C T^p + CT^2  \eee^{-\eps n}.
\end{equation}
To see this, note that by Lemma~\ref{lem:P_n_k_moment_E1_part2b},
\begin{equation}\label{eq:rem_tech1}
|(S_{n}^{\circ}(\beta) - S_{n,T}^{\circ}(\beta)) - (S_{n}(\beta) - S_{n,T}(\beta))|
=
N_{n,1} |\E (P_{n,k}^{-\frac{\beta}{\sigma_1}} \ind_{P_{n,k}>T})| < C T.
\end{equation}
Combining~\eqref{eq:lem:S_n_T_minus_S_n_case2} and~\eqref{eq:rem_tech1} and using Jensen's inequality~\eqref{eq:jensen_ineq}, we obtain~\eqref{eq:lem:S_n_T_minus_S_n_case2_rem}.
\end{remark}

\begin{proof}[Proof of Lemma~\ref{lem:S_n_T_minus_S_n_GE}]
We prove both parts of the lemma simultaneously. Write $D=E_2 \cap \{\frac {\sigma_1}{2} < \sigma < \frac{\sigma_2}{p}\}$ in the setting of Part~1 and $D=G^{d_1}E^{d-d_1}\cap \{0\leq \sigma < \frac{\sigma_2}{p}\}$ in the setting of Part~2.
Consider some $\beta_*=\sigma_*+i\tau_*\in K$.  So,  $\beta_*\in D$ and $\frac{\sigma_1}{2}<\sigma_*<\frac{\sigma_2} p$. Hence, there exists a closed disk $U\subset D$ centered at $\beta_*$ and a number $q$ such that for all $\beta=\sigma+ i \tau \in U$,
\begin{equation}\label{eq:q_choice}
\max\left\{p,\frac{\sigma_1}{\sigma}\right\} < q < \min\left\{\frac{\sigma_2}{\sigma}, 2\right\}. 
\end{equation}
In~\eqref{eq:lem:S_n_T_minus_S_n_case1} and~\eqref{eq:lem:S_n_T_minus_S_n_case2}, it suffices to provide estimates for the $q$-th moment instead of the $p$-th moment since by Lyapunov's inequality~\eqref{eq:lyapunov_ineq} we have (recalling that $p<q$)
\begin{align*}
\E |S_{n}(\beta) - S_{n,T}(\beta)|^p &\leq \left(\E |S_{n}(\beta) - S_{n,T}(\beta)|^q\right)^{p/q},
\\
\E |S_{n}^{\circ}(\beta) - S_{n,T}^{\circ}(\beta)|^p &\leq \left(\E |S_{n}^{\circ}(\beta) - S_{n,T}^{\circ}(\beta)|^q\right)^{p/q}
\end{align*}
and since by Jensen's inequality~\eqref{eq:jensen_ineq},
\begin{equation}\label{eq:tech_888}
(T^{-\eps} + T^2 \eee^{-\eps n})^{p/q} \leq T^{-\eps p/q} + T^{2p/q} \eee^{-\eps p n/q}
\leq
T^{-\eps p/q} + T^{2} \eee^{-\eps p n/q}.
\end{equation}
Also, note that it suffices to prove the required estimates~\eqref{eq:lem:S_n_T_minus_S_n_case1} and~\eqref{eq:lem:S_n_T_minus_S_n_case2} for $\beta\in U$ since $K$, being a compact set, can be covered by finitely many $U$'s. In the sequel, we always take  $\beta\in U$.

For future use, note that there exist $C=C(K)>0$, $\eps=\eps(K)>0$ such that for all $\beta\in U$, $T\in \N$, $n\in\N$,
\begin{equation}\label{eq:CLT_tech1}
N_{n,1} \E [P_{n,k}^{-\frac{\sigma q}{\sigma_1}}\ind_{P_{n,k}>T}] \, \E  |\tilde Z_{n,k}(\beta)|^q < C T^{-\eps}.
\end{equation}
Indeed, by Lemma~\ref{lem:P_n_k_moment_E1_part2a} (recall that $U\subset \{\sigma q  > \sigma_1\}$), we can estimate the first factor on the left-hand side by $C T^{-\eps}$. Besides, by the induction assumption~\eqref{eq:Z_n_k_tilde_moment_induction},  we have $\E  |\tilde Z_{n,k}(\beta)|^q\leq C$ (recall that $U\subset \calO$ and  $U\subset \{0\leq \sigma < \frac{\sigma_2}{q}\}$).

\vspace*{2mm}
\noindent
\textsc{Part 1.}
Assume that we are in the setting of Part~1 of Lemma~\ref{lem:S_n_T_minus_S_n_GE}. We prove~\eqref{eq:lem:S_n_T_minus_S_n_case1}.

\vspace*{2mm}
\noindent
\textsc{Case 1: $0<q\leq 1$.}
Using Lemma~\ref{lem:centered_moment_ineq},  we obtain
\begin{equation}\label{eq:GE_proof_tech3}
\E |S_{n}^{\circ}(\beta) - S_{n,T}^{\circ}(\beta)|^q
\leq C \E |S_{n}(\beta) - S_{n,T}(\beta)|^q
+ C |N_{n,1}\E(P_{n,k}^{-\frac{\beta}{\sigma_1}}\ind_{P_{n,k}>T})|^q.
\end{equation}
By Proposition~\ref{prop:ineq_moment_p_01} (which is applicable in the case $0<q\leq 1$) and  by~\eqref{eq:CLT_tech1},
\begin{align*}
\E |S_{n}(\beta) - S_{n,T}(\beta)|^q
&\leq
C N_{n,1} \E [P_{n,k}^{-\frac{\sigma q}{\sigma_1}} \ind_{P_{n,k}>T}]  \E |\tilde Z_{n,k}(\beta)|^q
\leq
CT^{-\eps}.
\end{align*}
The second term on the right-hand side of~\eqref{eq:GE_proof_tech3} can also be estimated by $CT^{-\eps}$.  This is because we can apply Lemma~\ref{lem:P_n_k_moment_E1_part2a} since $U\subset\{\sigma q>\sigma_1\}\subset \{\sigma>\sigma_1\}$ by~\eqref{eq:q_choice} and the assumption $q\leq 1$.

\vspace*{2mm}
\noindent
\textsc{Case 2: $1\leq q <2$.}
Recall that $\E \tilde Z_{n,k}(\beta)=1$ in the setting of Part~1. It follows that we can write
$$
S_{n}^{\circ}(\beta) - S_{n,T}^{\circ}(\beta)
=
\sum_{k=1}^{N_{n,1}} (P_{n,k}^{-\frac{\beta}{\sigma_1}} \ind_{P_{n,k}>T}  \tilde Z_{n,k}(\beta) - \E[P_{n,k}^{-\frac{\beta}{\sigma_1}}\ind_{P_{n,k}>T}\tilde Z_{n,k}(\beta)]).
$$
Note that the summands on the right-hand side have zero mean. By Proposition~\ref{prop:von_bahr_esseen} (which is applicable in the case $1\leq q<2$) and by Lemma~\ref{lem:centered_moment_ineq} (where we use that $q\geq 1$), we have
\begin{align*}
\E |S_{n}^{\circ}(\beta) - S_{n,T}^{\circ}(\beta)|^q
&\leq
C N_{n,1} \E \left|P_{n,k}^{-\frac{\beta}{\sigma_1}} \ind_{P_{n,k}>T}  \tilde Z_{n,k}(\beta) - \E[P_{n,k}^{-\frac{\beta}{\sigma_1}}\ind_{P_{n,k}>T}\tilde Z_{n,k}(\beta)]\right|^q\\
&\leq
C N_{n,1} \E [P_{n,k}^{-\frac{\sigma q}{\sigma_1}} \ind_{P_{n,k}>T}]  \E |\tilde Z_{n,k}(\beta)|^q.
\end{align*}
The right-hand side can be estimated by $CT^{-\eps}$ by~\eqref{eq:CLT_tech1}.

\vspace*{2mm}
\noindent
\textsc{Part 2.}
Assume that we are in the setting of Part~2 of Lemma~\ref{lem:S_n_T_minus_S_n_GE}. We prove~\eqref{eq:lem:S_n_T_minus_S_n_case2}. Note that $\E \tilde Z_{n,k}(\beta)$ is not necessarily equal to $1$ in the setting of Part~2. (In fact, only in phase $G^1E^{d-1}$ we have $\E \tilde Z_{n,k}(\beta)=1$).

\vspace*{2mm}
\noindent
\textsc{Case 1: $0<q\leq 1$.}
By Proposition~\ref{prop:ineq_moment_p_01} and~\eqref{eq:CLT_tech1}, we obtain that
\begin{equation}\label{eq:lem:S_n_T_minus_S_n_GE1_better}
\E |S_{n}(\beta) - S_{n,T}(\beta)|^q
\leq
N_{n,1} \E [P_{n,k}^{-\frac{\sigma q}{\sigma_1}}\ind_{P_{n,k}>T}] \, \E  |\tilde Z_{n,k}(\beta)|^q
\leq
C T^{-\eps}.
\end{equation}

\vspace*{2mm} \noindent \textsc{Case 2: $1\leq q <  2$.} By
Proposition~\ref{prop:von_bahr_esseen_non_centered} (which is applicable in the
case $1\leq q < 2$), we obtain that
$$
\E |S_{n}(\beta) - S_{n,T}(\beta)|^q \leq C N_{n,1} \E [P_{n,k}^{-\frac{\sigma q}{\sigma_1}}\ind_{P_{n,k}>T}] \, \E |\tilde Z_{n,k}(\beta)|^q + C |\E (S_{n}(\beta) - S_{n,T}(\beta))|^q.
$$
The first summand on the right-hand side can be estimated by $CT^{-\eps}$ by~\eqref{eq:CLT_tech1}. The second summand  can be estimated by $CT^q \eee^{-\eps n} < CT^2 \eee^{-\eps n}$ by Lemma~\ref{lem:E_S_n_minus_S_n_T}. Let us show that the assumptions of this lemma are satisfied. We have $U\subset G^{d_1}E^{d-d_1}$ with $2\leq d_1\leq d$. The assumption $q\geq 1$, together with~\eqref{eq:q_choice}, implies that $U\subset \{0\leq \sigma < \frac{\sigma_2}{q}\} \subset \{0\leq \sigma  < \sigma_2\}$, so that we can indeed apply Lemma~\ref{lem:E_S_n_minus_S_n_T}.


\vspace*{2mm}
We have thus proved the estimates~\eqref{eq:lem:S_n_T_minus_S_n_case1} and~\eqref{eq:lem:S_n_T_minus_S_n_case2} for $\beta\in U$. By compactness we can cover $K$ by a finite number of $U$'s. This completes the proof of Lemma~\ref{lem:S_n_T_minus_S_n_GE}.
\end{proof}

\vspace*{2mm}
\noindent
\textsc{Step 3.}
We are now ready to complete the proof of Proposition~\ref{prop:moment_S_n}.

\vspace*{2mm}
\noindent
\textsc{Part 1.}
Let $K$ be a compact subset of $E_2 \cap \{\frac{\sigma_1}{2} < \sigma < \frac{\sigma_1}{p}\}$. In Lemmas~\ref{lem:S_n_1_moments} and~\ref{lem:S_n_T_minus_S_n_GE}, Part~1, we proved the estimates $\E |S_{n,1}^{\circ}(\beta)|^p < C$ and $\E |S_n^{\circ}(\beta) - S_{n,1}^{\circ}(\beta)|^p < C$. Using Jensen's inequality~\eqref{eq:jensen_ineq} we obtain that $\E |S_n^{\circ}(\beta)|^p < C$.

\vspace*{2mm}
\noindent
\textsc{Part 2.}
Let $K$ be a compact subset of $G^{d_1}E^{d-d_1}\cap\{0\leq \sigma < \frac{\sigma_1}{p}\}$, where $1\leq d_1\leq d$.
Note that
\begin{equation}\label{eq:tech1}
|S_n(\beta)- S_n^{\circ}(\beta)| = N_{n,1} |\E (P_{n,k}^{-\frac{\beta}{\sigma_1}} \ind_{1\leq P_{n,k}})| < C,
\end{equation}
where the last step follows from Lemma~\ref{lem:P_n_k_moment_E1_part2b}.

\vspace*{2mm} \noindent \textsc{Case 1:} $d_1=1$.  Note that $G^1 E^{d-1}$ is a
subset of $E_2$. We already established in Part~1 of
Proposition~\ref{prop:moment_S_n} that $\E |S_n^{\circ}(\beta)|^p < C$.
From~\eqref{eq:tech1},  we obtain that $\E |S_n(\beta)|^p < C$.

\vspace*{2mm} \noindent \textsc{Case 2:} $2\leq d_1\leq d$. In
Lemmas~\ref{lem:S_n_1_moments} and~\ref{lem:S_n_T_minus_S_n_GE}, Part~2, we
proved the estimates $\E |S_{n,1}(\beta)|^p < C$ and $\E |S_n(\beta) -
S_{n,1}(\beta)|^p < C$. Using Jensen's inequality~\eqref{eq:jensen_ineq}, we
obtain that $\E |S_n(\beta)|^p < C$. It follows from~\eqref{eq:tech1} that $\E
|S_n^{\circ}(\beta)|^p < C$, thus completing the proof of
Proposition~\ref{prop:moment_S_n}.


\section{Functional limit theorems in phases without fluctuation levels}\label{sec:func_CLT_GE}
In this section, we prove functional limit theorems in phases of the form $G^{d_1}E^{d-d_1}$, where $0\leq d_1\leq d$. The proofs are based on the results of Section~\ref{sec:moments_GE}.

\subsection{Law of large numbers and absence of zeros in $E_1$}
We prove Theorems~\ref{theo:moment_S_n_3} and~\ref{theo:no_zeros_E1}.

\begin{proof}[Proof of Theorem~\ref{theo:moment_S_n_3}]
We have to show that weakly on $\HHH(E_1)$,
\begin{equation}\label{eq:ZZZ_n_beta_weak_to_1_E1_restate}
\frac{\ZZZ_n(\beta)}{\E \ZZZ_n(\beta)} \toweak 1.
\end{equation}
For every fixed $\beta\in E_1$, the random variable $\ZZZ_n(\beta)/\E \ZZZ_n(\beta)$ converges to $1$ a.s.\ by Corollary~\ref{cor:moment_S_n_1} and the Borel--Cantelli lemma. Hence, \eqref{eq:ZZZ_n_beta_weak_to_1_E1_restate} holds in the sense of finite-dimensional distributions. The tightness follows from Proposition~\ref{prop:tightness_random_analytic} and Corollary~\ref{cor:moment_S_n_2}.
\end{proof}

It is now easy to deduce Corollary~\ref{cor:no_zeros_E1}. Indeed, applying Proposition~\ref{prop:weak_conv_zeros} to~\eqref{eq:ZZZ_n_beta_weak_to_1_E1_restate}, yields the desired weak convergence of $\Zeros\{\ZZZ_n(\beta)\colon\beta\in E_1\}$ to the empty point process. To prove Theorem~\ref{theo:no_zeros_E1}, we need a more refined argument.

\begin{proof}[Proof of Theorem~\ref{theo:no_zeros_E1}]
Let $K$ be a compact subset of $E_1$.  We have to prove that the probability that $\ZZZ_n(\beta)$ has at least one zero in $K$ can be estimated by $C\eee^{-\eps n}$.  Let $\Gamma$ be a closed differentiable contour enclosing $K$ and located in $E_1$. By the same argumentation as in Section~4.3 of~\cite{kabluchko_klimovsky}, we have
$$
\P[\exists \beta\in K \colon \ZZZ_n(\beta) = 0 ] \leq C \oint_{\Gamma} \E\left|\frac{\ZZZ_n(\beta)}{\E\ZZZ_n(\beta)} - 1 \right| |\dd \beta|.
$$
Using Corollary~\ref{cor:moment_S_n_1} with $p=1$, we obtain that there are $C=C(\Gamma)>0$ and $\eps=\eps(\Gamma)>0$ such that for every $\beta\in \Gamma$ and every $n\in\N$
$$
\E\left|\frac{\ZZZ_n(\beta)}{\E\ZZZ_n(\beta)} - 1 \right| \leq C \eee^{-\eps n}.
$$
This yields the desired estimate.
\end{proof}

\subsection{Functional limit theorem in $E_2\cap \{|\sigma|>\frac{\sigma_1}{2}\}$}
The fluctuations of the random function $\ZZZ_n(\beta)$ in the domain $\{|\sigma|<\frac{\sigma_1}{2}\}$ have been identified in Theorem~\ref{theo:clt} and in Section~\ref{sec:func_clt_sigma_small}.  In this section, we identify the fluctuations of $\ZZZ_n(\beta)$ in the domain $E_2\cap \{\sigma>\frac{\sigma_1}{2}\}$.
\begin{theorem}\label{theo:functional_E1}
Let $D=E_2 \cap \{\sigma>\frac{\sigma_1}{2}\}$. The following convergence of random analytic functions holds weakly on $\HHH(D)$:
\begin{equation}\label{eq:theo:functional_E1}
S_n^{\circ}(\beta) = \frac{\ZZZ_n(\beta) - \eee^{\tilde c_n(\beta)} N_{n,1} \E [\eee^{\beta \sqrt{na_1} \xi } \ind_{\xi < u_{n,1}}] } {\eee^{\beta \sqrt{na_1} u_{n,1} + \tilde c_n(\beta)}}
\toweak
\zeta_P\left(\frac{\beta}{\sigma_1}\right) - \frac{\sigma_1}{\beta-\sigma_1}.
\end{equation}
\end{theorem}
\begin{remark}\label{rem:E_2_negative_beta_symmetry_explain}
By symmetry, see~\eqref{eq:symmetry1}, a similar result holds for the domain $E_2 \cap \{\sigma < - \frac{\sigma_1}{2}\}$. Namely, the following convergence of random analytic functions holds weakly on $\HHH(E_2 \cap \{\sigma < - \frac{\sigma_1}{2}\})$:
\begin{equation}\label{eq:theo:functional_E1_symmetry}
\frac{\ZZZ_n(\beta) - \eee^{\tilde c_n(-\beta)} N_{n,1} \E [\eee^{-\beta \sqrt{na_1} \xi } \ind_{\xi < u_{n,1}}] } {\eee^{-\beta \sqrt{na_1} u_{n,1} + \tilde c_n(-\beta)}}
 \toweak \zeta_P^{-}\left(-\frac{\beta}{\sigma_1}\right) + \frac{\sigma_1}{\beta + \sigma_1},
\end{equation}
where $\zeta_P^-$ is a copy of $\zeta_P$. In fact, one can even show that~\eqref{eq:theo:functional_E1} and~\eqref{eq:theo:functional_E1_symmetry} can be combined into a \textit{joint} convergence on the domain $E_2 \cap \{|\sigma| > \frac{\sigma_1}{2}\}$ and that the limiting functions $\zeta_P$ and $\zeta_P^-$ are \textit{independent}. We will not provide a complete proof of the independence, but let us explain the idea. The function $\zeta_P$ in~\eqref{eq:theo:functional_E1} appears as the contribution of the \textit{upper} extremal order statistics  of the first GREM level. The function $\zeta_P^-$ in~\eqref{eq:theo:functional_E1_symmetry} appears as the contribution of the \textit{lower} extremal order statistics of the first GREM level. Since upper and lower extremal order statistics become independent in the large sample limit, we have the independence of $\zeta_P$ and $\zeta_P^-$.
\end{remark}

Let us stress that the domain $E_2 \cap \{\sigma>\frac{\sigma_1}{2}\}$ on which Theorem~\ref{theo:functional_E1} is valid includes the domain $E_1\cap \{\sigma>\frac{\sigma_1}2\}$, the domain $G^1E^{d-1}\cap \{\sigma>0\}$, as well as the beak shaped boundary between these two domains. Restricting Theorem~\ref{theo:functional_E1} to these smaller domains, we obtain two important corollaries. The first corollary is a restatement of Theorem~\ref{theo:functional_E1_small_domain}.
\begin{corollary}\label{theo:functional_E1_cor1}
The following convergence of random analytic functions holds weakly on $\HHH(E_1 \cap \{\sigma>\frac{\sigma_1}{2}\})$:
\begin{equation}\label{eq:theo:functional_E1_cor1}
\frac{\ZZZ_n(\beta) - \E \ZZZ_n(\beta)} {\eee^{\beta \sqrt{na_1} u_{n,1} + \tilde c_n(\beta)}}
\toweak
\zeta_P\left(\frac{\beta}{\sigma_1}\right).
\end{equation}
\end{corollary}
\begin{corollary}\label{theo:functional_E1_cor2}
The following convergence of random analytic functions holds weakly on $\HHH(G^1 E^{d-1}\cap \{\sigma>0\})$:
\begin{equation}\label{eq:theo:functional_E1_cor2}
\frac{\ZZZ_n(\beta)}{\eee^{c_n(\beta)}}
=
\frac{\ZZZ_n(\beta)} {\eee^{\beta \sqrt{na_1} u_{n,1} + \tilde c_n(\beta)}}
\toweak
\zeta_P\left(\frac{\beta}{\sigma_1}\right).
\end{equation}
\end{corollary}
\begin{proof}[Proof of Corollaries~\ref{theo:functional_E1_cor1} and~\ref{theo:functional_E1_cor2} given Theorem~\ref{theo:functional_E1}]
By Lemmas~\ref{lem:P_n_k_moment_kljuv_1} and~\ref{lem:P_n_k_moment_kljuv_2}, locally uniformly on the specified domain, we have
\begin{align*}
\eee^{ - \beta \sqrt{na_1} u_{n,1}} N_{n,1} \E [\eee^{\beta \sqrt{na_1} \xi } \ind_{\xi < u_{n,1}}]
&=
\frac{\sigma_1}{\beta-\sigma_1}, &\text{ if }& \sigma + |\tau| > \sigma_1, \sigma>0\\
\eee^{ - \beta \sqrt{na_1} u_{n,1}} N_{n,1} \E [\eee^{\beta \sqrt{na_1} \xi } \ind_{\xi > u_{n,1}}]
&=
-\frac{\sigma_1}{\beta-\sigma_1}, &\text{ if }& \sigma + |\tau| < \sigma_1, \sigma>0.
\end{align*}
Inserting this into~\eqref{eq:theo:functional_E1}, we immediately obtain~\eqref{eq:theo:functional_E1_cor2} and~\eqref{eq:theo:functional_E1_cor1} .
\end{proof}

\begin{proof}[Proof of Theorem~\ref{theo:functional_E1}]
First, we show that~\eqref{eq:theo:functional_E1} holds in the sense of weak convergence of finite-dimensional distributions. Fix some $\beta_1,\ldots,\beta_r\in D$. We continue to use the notation from Section~\ref{sec:moments_GE}.
We are going to prove that the random vector $\bS_n^{\circ}: =  \{S_{n}^{\circ} (\beta_i)\}_{i=1}^r$ converges in distribution to $\bS_{\infty}^{\circ} = \{S_{\infty}^{\circ} (\beta_i)\}_{i=1}^r$, where
\begin{align*}
S_{n}^{\circ} (\beta)
&=
\sum_{k=1}^{N_{n,1}} P_{n,k}^{-\frac{\beta}{\sigma_1}}\tilde Z_{n,k}(\beta) - N_{n,1} \E [P_{n,k}^{-\frac{\beta}{\sigma_1}}\ind_{1\leq P_{n,k}}],\\
S_{\infty}^{\circ}(\beta)
&=
\zeta_P\left(\frac{\beta}{\sigma_1}\right) -\frac{\sigma_1}{\beta-\sigma_1}.
\end{align*}
Note that this definition of $S_n^{\circ}(\beta)$ is equivalent to the old ones; see~\eqref{eq:S_n_beta_circ_def}, \eqref{eq:S_n_circ_rep}. We will verify the conditions of Lemma~\ref{lem:interchange_limits} for the random vectors $\bS_{n,T}^{\circ}:= \{S_{n,T}^{\circ} (\beta_i)\}_{i=1}^r$ and $\bS_{\infty, T}^{\circ}(\beta):= \{S_{\infty, T}^{\circ}(\beta_i)\}_{i=1}^r$, where $T\in\N$ is a truncation parameter and
\begin{align*}
S_{n,T}^{\circ} (\beta)
&=
\sum_{k=1}^{N_{n,1}} P_{n,k}^{-\frac{\beta}{\sigma_1}} \ind_{P_{n,k}\leq T}  \tilde Z_{n,k}(\beta)
-
N_{n,1} \E [P_{n,k}^{-\frac{\beta}{\sigma_1}}\ind_{1\leq P_{n,k}\leq T}]
,\\
S_{\infty, T}^{\circ}(\beta)
&=
\sum_{k=1}^{\infty} P_k^{-\frac{\beta}{\sigma_1}} \ind_{P_{k}\leq T}
-
\int_1^T t^{-\frac{\beta}{\sigma_1}}\dd t
.
\end{align*}
The three conditions of Lemma~\ref{lem:interchange_limits} will be verified in three steps.

\vspace*{2mm}
\noindent
\textsc{Step 1.} We prove that $\bS_{n,T}^{\circ} \todistr \bS_{\infty,T}^{\circ}$ for every fixed $T\in\N$. By Lemma~\ref{lem:exp_P_n_k_z}, we have the convergence of regularizing terms: for every $\beta\in\C$,
\begin{equation}\label{eq:theo:functional_E_1_proof_reg_terms}
\lim_{n\to\infty} N_{n,1} \E [P_{n,k}^{-\frac{\beta}{\sigma_1}}\ind_{1\leq P_{n,k}\leq T}] = \int_1^T t^{-\frac{\beta}{\sigma_1}}\dd t.
\end{equation}
Since $\tilde Z_{n,k}(\beta)$, as defined in~\eqref{eq:Z_n_k_tilde_def}, is a $(d-1)$-level analogue of $\ZZZ_n(\beta)/\E \ZZZ_n(\beta)$ for $\beta\in E_2$, we have that  by Theorem~\ref{theo:moment_S_n_3}, the random variable $\tilde Z_{n,k}(\beta)$ converges in distribution to $1$, for every $\beta\in E_2$.  In particular, the random vector $\tilde \bZ_{n,k}:=\{\tilde Z_{n,k}(\beta_i)\}_{i=1}^r$ converges in distribution to the random vector $\tilde \bZ_k : = \{1\}_{i=1}^r$. By Lemma~\ref{lem:adjoin_level}, we obtain that
\begin{equation}\label{eq:conv_P_n_k_1_E1}
\left\{\sum_{k=1}^{N_{n,1}} P_{n,k}^{-\frac{\beta_i}{\sigma_1}} \ind_{P_{n,k}\leq T}  \tilde Z_{n,k}(\beta_i)\right\}_{i=1}^r
\todistr
\left\{ \sum_{k=1}^{\infty} P_k^{-\frac{\beta_i}{\sigma_1}} \ind_{P_{k}\leq T}\right\}_{i=1}^r.
\end{equation}
Combining~\eqref{eq:theo:functional_E_1_proof_reg_terms} and~\eqref{eq:conv_P_n_k_1_E1}, we obtain that $\bS_{n,T}^{\circ} \todistr \bS_{\infty,T}^{\circ}$.

\vspace*{2mm}
\noindent
\textsc{Step 2.} By \cite[Theorem~2.6]{kabluchko_klimovsky}, see also~\eqref{eq:zeta_P_anal_cont_d_1}, we have $\bS_{\infty, T}^{\circ} \todistrT \bS_{\infty}^{\circ}$.

\vspace*{2mm}
\noindent
\textsc{Step 3.} Let $p\in (0,2)$ be so close to $0$ that $\beta_1,\ldots,\beta_r\in E_2\cap\{\frac{\sigma_1}{2}< \sigma < \frac{\sigma_1}{p}\}$. To verify the second condition of Lemma~\ref{lem:interchange_limits} it suffices to prove that  for every $1\leq i\leq r$ we have
\begin{equation}\label{eq:proof_E1_step3}
\lim_{T\to\infty} \E |S_{n}^{\circ}(\beta_i) - S_{n,T}^{\circ}(\beta_i)|^p=0
\text{ uniformly in }n\in\N.
\end{equation}
However, this follows immediately from Lemma~\ref{lem:S_n_T_minus_S_n_GE}, Part~1.

\vspace*{2mm}
\noindent
\textsc{Step 4.}
Combining Steps~1, 2, 3 and applying Lemma~\ref{lem:interchange_limits}, we obtain that the random vector $\bS_n^{\circ}$ converges in distribution  to the random vector $\bS_{\infty}^{\circ}$. Hence,  \eqref{eq:theo:functional_E1} holds in the sense of weak convergence of finite-dimensional distributions. To complete the proof of Theorem~\ref{theo:functional_E1} we need  to show that the sequence of random functions $S_n^{\circ}(\beta)$ is tight on $\HHH(D)$. By Proposition~\ref{prop:weak_conv_if_on_compact_sets}, it suffices to show that the sequence $S_n^{\circ}(\beta)$ is tight on $\HHH(U)$, for arbitrary open set $U$ such that $\bar U\subset D$.
If $p>0$ is sufficiently small, then by Proposition~\ref{prop:moment_S_n}, Part~1, there is a constant $C$  such that $\E |S_n^{\circ}(\beta)|^p<C$ for all $\beta\in \bar U$ and all $n\in\N$.  Proposition~\ref{prop:tightness_random_analytic} implies that the sequence of random analytic functions $S_n^{\circ}(\beta)$ is tight on $\HHH(U)$, thus completing the proof of Theorem~\ref{theo:functional_E1}.
\end{proof}

\subsection{Functional limit theorem in phase $G^{d_1}E^{d-d_1}$} \label{sec:adjoin_spin_glass_GE_only}
In this section, we prove Theorem~\ref{theo:functional_clt_1}. Fix some $0\leq d_1\leq d$. Denote by $D$ the domain $G^{d_1}E^{d-d_1}\cap\{\sigma>0\}$. Our aim is to show that weakly on $\HHH(D)$,
\begin{equation}\label{eq:func_clt1_restate}
S_n(\beta)=
\frac{\ZZZ_n(\beta)}{\eee^{c_n(\beta)}}
=
\sum_{k=1}^{N_{n,1}} P_{n,k}^{-\frac{\beta}{\sigma_1}}\tilde Z_{n,k}(\beta)
\toweak
\zeta_P\left(\frac{\beta}{\sigma_1}, \ldots, \frac{\beta}{\sigma_{d_1}}\right).
\end{equation}
Note that in the case $d_1=0$ (that is, in the phase $E_1=E^d$), the convergence in~\eqref{eq:func_clt1_restate} (with the right-hand side interpreted as $1$) has been established in Theorem~\ref{theo:moment_S_n_3}. In the case $d_1=1$, we established~\eqref{eq:func_clt1_restate} in Corollary~\ref{theo:functional_E1_cor2}.

We will use induction over $d$, the number of levels in the GREM. In the case $d=1$ (which is the basis of induction) we have $d_1=0$ or $d_1=1$, so that~\eqref{eq:func_clt1_restate} has already been established.
We make the induction assumption that~\eqref{eq:func_clt1_restate} holds for any GREM with $d-1$ levels.  Our aim is to
prove that it holds for the GREM with $d$ levels. From now on, we may assume that $d_1\geq 2$, that is at least \textit{two} levels are in the glassy phase.
We will use the notation
$$
\beta^{\vartriangle}
=
\left(\frac{\beta}{\sigma_1},\ldots, \frac{\beta}{\sigma_{d_1}}\right)\in \C^{d_1},
\quad
\tilde \beta^{\vartriangle}
=
\left(\frac{\beta}{\sigma_2},\ldots, \frac{\beta}{\sigma_{d_1}}\right)\in \C^{d_1-1},
$$

First, we will show that~\eqref{eq:func_clt1_restate} holds in the sense of weak convergence of finite-dimensional distributions. Fix some $\beta_1,\ldots,\beta_r\in D$.
Our aim is to prove that the random vector $\bS_n: =  \{S_{n} (\beta_i)\}_{i=1}^r$ converges in distribution to $\bS_{\infty} = \{S_{\infty} (\beta_i)\}_{i=1}^r$, where
\begin{align*}
S_{n} (\beta)
=
\sum_{k=1}^{N_{n,1}} P_{n,k}^{-\frac{\beta}{\sigma_1}}\tilde Z_{n,k}(\beta),
\quad
S_{\infty}(\beta)
&=
\zeta_P(\beta^{\vartriangle}).
\end{align*}
This will be done by verifying the conditions of Lemma~\ref{lem:interchange_limits} for the random vectors $\bS_{n,T}:= \{S_{n,T} (\beta_i)\}_{i=1}^r$ and $\bS_{\infty, T}(\beta):= \{S_{\infty, T}(\beta_i)\}_{i=1}^r$, where $T\in\N$ is a truncation parameter  and
\begin{align*}
S_{n,T} (\beta)
=
\sum_{k=1}^{N_{n,1}} P_{n,k}^{-\frac{\beta}{\sigma_1}} \ind_{P_{n,k}\leq T}  \tilde Z_{n,k}(\beta),
\quad
S_{\infty, T}(\beta)
=
\sum_{k=1}^{\infty} P_k^{-\frac{\beta}{\sigma_1}} \ind_{P_{k}\leq T} \tilde \zeta_k(\tilde \beta^{\vartriangle}).
\end{align*}
Here, we denote by  $\{\tilde \zeta_k(\tilde \beta^{\vartriangle})\colon k\in\N \}$ i.i.d.\ random analytic functions on
$D$ with the same law as $\zeta_P (\tilde \beta^{\vartriangle})$.

\vspace*{2mm}
\noindent
\textsc{Step 1.} We prove that $\bS_{n,T} \todistr \bS_{\infty,T}$ for every fixed $T\in\N$.
The random function $\tilde Z_{n,k}(\beta)$ is an analogue of the random function $\eee^{-c_n(\beta)}\ZZZ_n(\beta)$ with $d-1$ levels. By the induction assumption, we have the following weak convergence on $\HHH(D)$:
$$
\tilde Z_{n,k}(\beta) \toweak  \zeta_P (\tilde \beta^{\vartriangle}),
$$
From Lemma~\ref{lem:adjoin_level}, it follows that
\begin{align}
\left\{\sum_{k=1}^{N_{n,1}} P_{n,k}^{-\frac{\beta_i}{\sigma_1}} \ind_{P_{n,k}\leq T}  \tilde Z_{n,k}(\beta_i)\right\}_{i=1}^r
&\todistr
\left\{\sum_{k=1}^{\infty} P_k^{-\frac{\beta_i}{\sigma_1}} \ind_{P_{k}\leq T} \tilde \zeta_k(\tilde \beta_i^{\vartriangle})\right\}_{i=1}^r. \label{eq:conv_P_n_k_1}
\end{align}
This yields the desired convergence.

\vspace*{2mm}
\noindent
\textsc{Step 2.} By Proposition~\ref{prop:zeta_recursion}, we have $\bS_{\infty, T} \todistrT \bS_{\infty}$ (recall that $d_1\geq 2$).

\vspace*{2mm}
\noindent
\textsc{Step 3.} Fix $\beta\in D$. Let $p>0$ be so small that $\sigma<\frac{\sigma_2}{p}$.  To verify the second condition of Lemma~\ref{lem:interchange_limits} it suffices to prove that
$$
\lim_{T\to\infty}  \limsup_{n\to\infty} \E |S_{n}(\beta) - S_{n,T}(\beta)|^p=0.
$$
However, this has already been established in Lemma~\ref{lem:S_n_T_minus_S_n_GE}, Part~2.

\vspace*{2mm}
\noindent
\textsc{Step 4.}
It follows from Steps~1, 2, 3 and Lemma~\ref{lem:interchange_limits}, that the random vector $\bS_n$ converges in distribution  to the random vector $\bS_{\infty}$. In other words, \eqref{eq:func_clt1_restate} holds in the sense of weak convergence of finite-dimensional distributions. To complete the proof of Theorem~\ref{theo:functional_clt_1} it remains  to show that the sequence of random functions $S_n(\beta)$ is tight on $\HHH(D)$.
By Proposition~\ref{prop:weak_conv_if_on_compact_sets}, it suffices to show that the sequence $S_n(\beta)$ is tight on $\HHH(U)$, for arbitrary open set $U$ such that $\bar U\subset D$.
If $p>0$ is sufficiently small, then by Proposition~\ref{prop:moment_S_n}, Part~2, there is a constant $C$  such that $\E |S_n(\beta)|^p<C$ for all $\beta\in \bar U$ and all $n\in\N$.  Proposition~\ref{prop:tightness_random_analytic} implies that the sequence of random analytic functions $S_n(\beta)$ is tight on $\HHH(U)$, thus completing the proof of Theorem~\ref{theo:functional_clt_1}.

\section{Functional limit theorems on beak shaped boundaries}\label{sec:func_CLT_GE_boundary}
In this section, we prove Theorem~\ref{theo:functional_clt_GE_beak_boundary}, a functional limit theorem for the partition function $\ZZZ_n(\beta)$ in an infinitesimal neighborhood of some $\beta_*$ located on the beak shaped boundary separating the phases $G^{l-1} E^{d-l+1}$ and $G^{l}E^{d-l}$, for $1\leq l \leq d$. The location and the size of the infinitesimal neighborhood are chosen to cover the ``line of zeros'' near the above mentioned boundary.

\subsection{Statement of the result and notation}
Fix some $1\leq l\leq d$. Let $\beta_*=\sigma_*+i\tau_*\in \C$ be such that
\begin{equation}\label{eq:beta_star_beak_boundary}
\sigma_* > \frac{\sigma_l}{2},\;\;\; \tau_*>0, \;\;\; \sigma_*+\tau_*=\sigma_l.
\end{equation}
These conditions imply that $\beta_*$ belongs to the boundary separating the phases $G^{l-1} E^{d-l+1}$ and $G^{l}E^{d-l}$.  First we need to introduce several normalizing sequences. Let $d_{n,l}$ be any complex sequence such that $|d_{n,l}|=O(\log n)$ and
\begin{equation}\label{eq:d_n_l_def}
d_{n,l} + \beta_* \, \frac{\log (4\pi n \log \alpha_l)}{2\sigma_l} - i na_l \tau_*^2\in 2\pi i \Z.
\end{equation}
Let $\beta_{n,l}(t)$ be a linear function of $t$ which is given by
\begin{equation}\label{eq:beta_n_l_s_def}
\beta_{n,l}(t) =
\beta_* + \eee^{-\frac{3\pi i}{4}} \cdot \frac {1} n   \cdot \frac{d_{n,l} + t}{\sqrt 2 a_l \tau_*}, \;\;\; t\in\C.
\end{equation}
Note that $\lim_{n\to\infty} \beta_{n,l}(t) =\beta_*$ for all $t\in\C$. Note also that $\Re d_{n,l}\sim -\frac{\sigma_*}{2\sigma_l}\log n$ is negative and hence, $\beta_{n,l}(t)$ is located \textit{outside} $E_l$ provided that $n$ is sufficiently large. The distance from $\beta_{n,l}(t)$ to the boundary of $E_l$ is asymptotic to $\text{const}\cdot \frac{\log n}{n}$.
Define a normalizing function
\begin{equation}\label{eq:h_n_l_t_def}
h_{n,l}(t) = \beta_{n,l}(t) \sum_{j=1}^l \sqrt{na_j}u_{n,j} + \sum_{j=l+1}^d \left(\log N_{n,j} + \frac 12 \beta^2_{n,l}(t) n a_{j}\right).
\end{equation}
We can restate Theorem~\ref{theo:functional_clt_GE_beak_boundary} as follows.
\begin{theorem}\label{theo:functional_clt_GE_beak_boundary_restate}
Fix some $1\leq l\leq d$ and some $\beta_*=\sigma_*+i\tau_*\in \C$ such that~\eqref{eq:beta_star_beak_boundary} holds.
Then, weakly on $\HHH(\C)$ it holds that
$$
\left\{\frac{\ZZZ_n (\beta_{n,l}(t))}{ \eee^{h_{n,l}(t)}} \colon t\in\C\right\}
\toweak
\{\eee^{t}\zeta^{(l-1)} + \zeta^{(l)}\colon t\in\C\}.
$$
Here, $(\zeta^{(l-1)}, \zeta^{(l)})$ is a  random vector given by
$$
(\zeta^{(l-1)}, \zeta^{(l)})
=
\left(\zeta_P\left(\frac{\beta_*}{\sigma_1} , \ldots, \frac{\beta_*}{\sigma_{l-1}}\right), \zeta_P\left(\frac{\beta_*}{\sigma_1} , \ldots, \frac{\beta_*}{\sigma_l}\right) \right),
$$
where both zeta functions are based on the same Poisson cascade point process.
\end{theorem}

The remaining part of Section~\ref{sec:func_CLT_GE_boundary} is devoted to the proof of Theorem~\ref{theo:functional_clt_GE_beak_boundary_restate}. We start by introducing the necessary notation.
Define the random variables $P_{n,k}$, $n\in\N$, $1\leq k\leq N_{n,1}$, (the normalized contributions of the first level of the GREM) and $\tilde Z_{n,k}(t)$, $n\in\N$, $1\leq k\leq N_{n,1}$, (the
normalized contributions of the remaining $d-1$ levels of the GREM) by
\begin{align}
P_{n,k} &= \eee^{-\sigma_1 \sqrt {n a_1} (\xi_k - u_{n,1})}, \label{eq:P_n_k_def_beak_bound}\\
\tilde Z_{n,k} (t) &=  \eee^{-\tilde h_{n,l}(t)} \sum_{\tilde \eps \in \tilde \SSS_n} \eee^{\beta_{n,l}(t) \sqrt n (\sqrt{a_2} \xi_{k\eps_2} + \ldots + \sqrt{a_d} \xi_{k\eps_2\ldots\eps_d})} \label{eq:Z_n_k_tilde_def_beak_bound},
\end{align}
where $\tilde \SSS_n$, the index set for the levels $2,\ldots,d$,  is defined as in~\eqref{eq:parameter_set_def_tilde} and
\begin{equation}\label{eq:h_n_l_t_tilde_def}
\tilde h_{n,l}(t) = \beta_{n,l}(t) \sum_{j=2}^l \sqrt{na_j}u_{n,j} + \sum_{j=l+1}^d \left(\log N_{n,j} + \frac 12 \beta^2_{n,l}(t) n a_{j}\right).
\end{equation}
By the definition of the GREM, these random variables have the following properties, for every $n\in\N$:
\begin{enumerate}
\item $\tilde Z_{n,k}(t)$, $1\leq k\leq N_{n,1}$, is an i.i.d.\ collection of random processes.
\item $P_{n,k}$, $1\leq k\leq N_{n,1}$, is an i.i.d.\ collection of random variables.
\item These two collections are independent.
\end{enumerate}
The properties of $P_{n,k}$ have been studied in Section~\ref{sec:first_level_GREM}.
We have the representation
\begin{equation}\label{eq:S_n_def_beak_bound}
S_n(t)
:= \frac{\ZZZ_n(\beta_{n,l}(t))}{\eee^{h_{n,l}(t)}}
= \sum_{k=1}^{N_{n,1}} P_{n,k}^{- \frac{\beta_{n,l}(t)}{\sigma_1}} \tilde Z_{n,k}(t).
\end{equation}
Introduce also a version of $S_n(t)$ centered by a truncated expectation:
\begin{equation}\label{eq:S_n_circ_def_beak_bound}
S_n^{\circ}(t)
= S_n(t)
- N_{n,1} \E \left[P_{n,k}^{- \frac{\beta_{n,l}(t)}{\sigma_1}}\ind_{1\leq P_{n,k}}\right]\, \E [\tilde Z_{n,k}(t)].
\end{equation}
For $T\in \N$, consider the truncated versions of $S_n(t)$ and $S_n^{\circ}(t)$ defined by
\begin{align}
S_{n,T} (t)
&=
\sum_{k=1}^{N_{n,1}} P_{n,k}^{-\frac{\beta_{n,l}(t)}{\sigma_1}}  \ind_{P_{n,k}\leq T}  \tilde Z_{n,k}(t), \label{eq:S_n_T_def_beak_bound}\\
S_{n,T}^{\circ} (t)
&=
S_{n,T} (t)
- N_{n,1} \E \left[P_{n,k}^{- \frac{\beta_{n,l}(t)}{\sigma_1}}\ind_{1\leq P_{n,k}\leq T}\right]\, \E [\tilde Z_{n,k}(t)].
\label{eq:S_n_T_circ_def_beak_bound}
\end{align}

\subsection{Basis of induction: $l=1$}\label{subsec:beak_bound_basis}
Our proof of Theorem~\ref{theo:functional_clt_GE_beak_boundary_restate} uses induction over $l$.  First, we show that Theorem~\ref{theo:functional_clt_GE_beak_boundary_restate} holds for $l=1$.
Fix some $\beta_*=\sigma_*+i\tau_*\in \C$ such that $\sigma_*>\frac{\sigma_1}{2}$, $\tau_*>0$, $\sigma_*+\tau_*=\sigma_1$.
We are going to show that weakly on $\HHH(\C)$ it holds that
\begin{equation}\label{eq:theo:functional_clt_GE_beak_boundary_restate_l1}
\left\{\frac{\ZZZ_n (\beta_{n,1}(t))}
{\eee^{h_{n,1}(t)}}
\colon t\in\C\right\} \toweak  \left\{\eee^{t} + \zeta_P\left(\frac{\beta_*}{\sigma_1}\right)\colon t\in\C\right\}.
\end{equation}
The main step in the proof of~\eqref{eq:theo:functional_clt_GE_beak_boundary_restate_l1} is the following result. Recall that $\tilde c_n(\beta)$ was defined in~\eqref{eq:c_n_beta_tilde} and~\eqref{eq:c_n_k_beta_def_repetition}.
\begin{proposition}\label{prop:theo:functional_E1_restate}
The following convergence of random analytic functions holds weakly on $\HHH(E_2 \cap \{\frac{\sigma_1}{2} < \sigma < \sigma_1\})$:
\begin{equation*}
\left\{\frac{\ZZZ_n(\beta)}{\eee^{\beta \sqrt{na_1} u_{n,1} + \tilde c_n(\beta) } } - N_{n,1} \eee^{\frac 12 \beta^2 na_1 - \beta \sqrt{na_1} u_{n,1}}\colon \beta\in\C\right\}
\toweak
\left\{\zeta_P\left(\frac{\beta}{\sigma_1}\right)\colon \beta\in\C\right\}.
\end{equation*}
\end{proposition}
\begin{proof}
We will use Theorem~\ref{theo:functional_E1}. It follows from the definition of $P_{n,k}$, see~\eqref{eq:P_n_k_def_beak_bound},  that
$$
\E [P_{n,k}^{-\frac{\beta}{\sigma_1}}] = \eee^{\frac 12 \beta^2 na_1 - \beta \sqrt{na_1} u_{n,1}}.
$$
Using this and then Lemma~\ref{lem:P_n_k_moment_kljuv_2},
we obtain that locally uniformly on $\{\sigma>0, \sigma - |\tau|<\sigma_1\}$,
\begin{align}
N_{n,1} \E [P_{n,k}^{-\frac{\beta}{\sigma_1}} \ind_{P_{n,k}>1}]
&= N_{n,1} \eee^{\frac 12 \beta^2 na_1 - \beta \sqrt{na_1} u_{n,1}}
 - N_{n,1} \E [P_{n,k}^{-\frac{\beta}{\sigma_1}} \ind_{P_{n,k}<1}]\label{eq:tech22_beak_bound}\\
&= N_{n,1} \eee^{\frac 12 \beta^2 na_1 - \beta \sqrt{na_1} u_{n,1}}
 +\frac{\sigma_1}{\beta - \sigma_1} + o(1). \notag
\end{align}
In particular, this holds locally uniformly on $E_2 \cap \{\frac{\sigma_1}{2} < \sigma < \sigma_1\}$. Inserting~\eqref{eq:tech22_beak_bound} into Theorem~\ref{theo:functional_E1} we obtain Proposition~\ref{prop:theo:functional_E1_restate}.
\end{proof}

\begin{lemma}\label{lem:beak_bound_expect_extremes_e_t}
If $\beta_{n,l}(t)$ is given by~\eqref{eq:beta_n_l_s_def}, with some $1\leq l\leq d$, then locally uniformly in $t\in\C$,
\begin{equation}\label{eq:beak_bound_expect_extremes_e_t}
\lim_{n\to\infty} N_{n,l} \eee^{\frac 12 \beta_{n,l}^2(t) na_l - \beta_{n,l}(t) \sqrt{na_l} u_{n,l}} = \eee^{t}.
\end{equation}
\end{lemma}
\begin{proof}
The proof is a straightforward but lengthy calculation.
Write $\beta_n=\beta_* + \frac {\delta_n} n$, where $\delta_n=O(\log n)$ is some complex sequence.  Using~\eqref{eq:asympt_N_nk} and~\eqref{eq:u_n_k_asympt} we have
\begin{align*}
\lefteqn{\log N_{n,l} + \frac {na_l} 2 \beta_n^2  - \beta_n \sqrt{na_l} u_{n,l}}\\
&=
n\log \alpha_l + \frac {na_l} 2\left(\beta_*^2 + 2\beta_* \frac {\delta_n}{n}\right)
- \beta_*\sqrt{na_l}\left( \sqrt{2n\log \alpha_l} - \frac{\log (4\pi n \log \alpha_l)}{2\sqrt{2n\log \alpha_l}}\right)\\
&\quad \quad \quad \quad -\frac {\delta_n}{n} \sqrt{na_l} \sqrt{2n\log \alpha_l} + o\left(1\right)\\
&=
\delta_n a_l (\beta_* - \sigma_l) -  n a_l i \tau_*^2 +  \beta_* \frac{\log (4\pi n \log \alpha_l)}{2\sigma_l} +o(1).
\end{align*}
In the last step, we used that $\sigma_l\sqrt{a_l} = \sqrt{2\log \alpha_l}$ and
$$
n\log \alpha_l+ \frac {na_l} 2   \beta_*^2 - \beta_*\sqrt{na_l} \sqrt{2n\log \alpha_l}
=
i na_l \tau_*(\sigma_* - \sigma_l)
=
-i na_l \tau_*^2.
$$
Let us now choose
$$
\delta_n:=\frac{d_{n,l} + t}{a_l (\beta_* - \sigma_l)} = \eee^{-\frac{3\pi i}{4}} \frac{d_{n,l} + t}{\sqrt 2 a_l \tau_*},
$$
where $d_{n,l}$ satisfies~\eqref{eq:d_n_l_def}.  Then, $\beta_n=\beta_{n,l}(t)$ and
$$
\lim_{n\to\infty} \exp\left(\log N_{n,l} + \frac {na_l} 2 \beta_n^2  - \beta_n \sqrt{na_l} u_{n,l}\right) = \eee^t.
$$
This completes the proof of~\eqref{eq:beak_bound_expect_extremes_e_t}.
\end{proof}

\begin{proof}[Proof of~\eqref{eq:theo:functional_clt_GE_beak_boundary_restate_l1}]
Taking $\beta=\beta_{n,1}(t)$ in Proposition~\ref{prop:theo:functional_E1_restate} and applying Lemma~\ref{lem:weak_conv_infinitesimal_neighborhoof}, we obtain that the process
$$
\left\{\frac{\ZZZ_n(\beta_{n,1}(t))}{\eee^{\beta_{n,1}(t) \sqrt{na_1} u_{n,1} + \tilde c_n(\beta_{n,1}(t)) } } - N_{n,1} \eee^{\frac 12 \beta_{n,1}^2(t) na_1 - \beta_{n,1}(t) \sqrt{na_1} u_{n,1}} \colon t\in \C\right\}
$$
converges weakly on $\HHH(\C)$ to the process
$$
\left\{\zeta_P\left(\frac{\beta_*}{\sigma_1}\right)\colon t\in \C\right\}.
$$
Note that the limit is  considered as a stochastic process indexed by $t\in\C$ (although it actually does not depend on $t$).  Applying Lemma~\ref{lem:beak_bound_expect_extremes_e_t}, we obtain~\eqref{eq:theo:functional_clt_GE_beak_boundary_restate_l1}.
\end{proof}

The next lemma provides a moment estimate valid for $\ZZZ_n(\beta_{n,l}(t))$  in the case $l=1$. It will serve as a basis of induction in the proof of Proposition~\ref{prop:moment_S_n_beak_bound}.
\begin{lemma}\label{lem:beak_bound_moments_l_1}
Fix $p\in (0,2)$ such that $p < \frac{\sigma_1}{\sigma_*}$.
Let $K$ be a compact subset of $\C$. Then, there exists a constant $C=C(K)>0$ such that for all $t \in K$ and all $n\in\N$,
$$
\E \left|\frac{\ZZZ_n(\beta_{n,1}(t))}{\eee^{h_{n,1}(t)}}\right|^p \leq C.
$$
\end{lemma}
\begin{proof}
Since $\beta_{n,1}(t)$ converges to $\beta_*\in E_2\cap \{\sigma>\frac{\sigma_1}2\}$ uniformly in $t\in K$, we can use Proposition~\ref{prop:moment_S_n}, Part~1, to obtain that
$$
\E \left|\frac{\ZZZ_n(\beta_{n,1}(t))}{\eee^{h_{n,1}(t)}} - N_{n,1} \E \left[P_{n,k}^{-\frac{\beta_{n,1}(t)}{\sigma_1}} \ind_{P_{n,k}>1}\right] \right|^p \leq C.
$$
We have already shown in~\eqref{eq:tech22_beak_bound} and Lemma~\ref{lem:beak_bound_expect_extremes_e_t} that uniformly in $t\in K$,
$$
N_{n,1} \E \left[P_{n,k}^{-\frac{\beta_{n,1}(t)}{\sigma_1}} \ind_{P_{n,k}>1}\right] \leq C.
$$
Using Jensen's inequality~\eqref{eq:jensen_ineq}, we obtain the statement.
\end{proof}

\subsection{Moment estimates}
In this section, we prove estimates for the moments of $S_n(t)$. The main results are Proposition~\ref{prop:moment_S_n_beak_bound} and Lemma~\ref{lem:S_n_T_minus_S_n_GE_beak_bound}.
\begin{proposition}\label{prop:moment_S_n_beak_bound}
Fix $p\in (0,2)$ such that $p < \frac{\sigma_1}{\sigma_*}$.
Let $K$ be a compact subset of $\C$. Then, there exists a constant $C=C(K)>0$ such that for all $t \in K$ and all $n\in\N$,
\begin{equation}\label{eq:prop:moment_S_n_beak_bound}
\E |S_n(t)|^p < C.
\end{equation}
\end{proposition}

The rest of the section is devoted to the proof of Proposition~\ref{prop:moment_S_n_beak_bound}.
We will use induction over $l$. Note that the case $l=1$ (which is the base of our induction) has been verified in Lemma~\ref{lem:beak_bound_moments_l_1}.
Let us take $l\geq 2$ and assume that  Proposition~\ref{prop:moment_S_n_beak_bound} holds for all smaller values of $l$.  The random function $\tilde Z_{n,k}(t)$ is the analogue of $S_n(t)$, with $d$ and $l$  reduced by $1$.  Thus, our induction assumption reads as follows.
\begin{itemize}
\item [(IND)] Fix some  $r\in (0,2)$ such that $r < \frac{\sigma_2}{\sigma_*}$. Let $K$ be a compact subset of $\C$.
Then, there exists a constant $C=C(K)>0$ such that for all $t\in K$ and all $n\in\N$,
\begin{equation}\label{eq:Z_n_k_tilde_moment_induction_beak_bound}
\E |\tilde Z_{n,k}(t)|^r < C.
\end{equation}
\end{itemize}

\vspace*{2mm}
\noindent
\textsc{Step 1.} In this step, we estimate the moments of $S_{n,1}(t)=S_{n,1}^{\circ}(t)$.
\begin{lemma}\label{lem:S_n_1_moments_beak_bound}
Fix $p\in (0,2)$ such that $p < \frac{\sigma_1}{\sigma_*}$. Let $K$ be a compact subset of $\C$. Then, there is a constant $C=C(K)>0$ such that for all $t\in K$ and all $n\in\N$,
\begin{equation}\label{eq:S_n_1_moment_est_beak_bound}
\E |S_{n,1}^{\circ}(t)|^p = \E |S_{n,1}(t)|^p  < C.
\end{equation}
\end{lemma}
\begin{proof}
For future use, note the inequality, valid uniformly for $t\in K$,
\begin{equation}\label{eq:S_n_1_tech1_beak_bound}
N_{n,1} \E \left[P_{n,k}^{-\frac{p \Re \beta_{n,l}(t)}{\sigma_1}}\ind_{P_{n,k} \leq 1}\right] \cdot \E |\tilde Z_{n,k}(t)|^p < C.
\end{equation}
Here is a proof of~\eqref{eq:S_n_1_tech1_beak_bound}. Note that $p \Re \beta_{n,l}(t)$ converges to $\sigma_*p < \sigma_1$ uniformly in $t\in K$.  By Lemma~\ref{lem:P_n_k_moment_E1_part1}, we can estimate  the first factor on the left-hand side of~\eqref{eq:S_n_1_tech1_beak_bound} by $C$. Also, by the induction assumption~\eqref{eq:Z_n_k_tilde_moment_induction_beak_bound} we have $\E  |\tilde Z_{n,k}(t)|^p\leq C$ (recall that we assume that $p < \frac{\sigma_1}{\sigma_*} < \frac{\sigma_2}{\sigma_*}$). This proves~\eqref{eq:S_n_1_tech1_beak_bound}.

\vspace*{2mm}
\noindent
\textsc{Case 1:} $0< p \leq 1$.  Then, by Proposition~\ref{prop:ineq_moment_p_01} and~\eqref{eq:S_n_1_tech1_beak_bound},
$$
\E |S_{n,1}(t)|^p \leq N_{n,1} \E \left[P_{n,k}^{-\frac{p\Re \beta_{n,l}(t)}{\sigma_1}}\ind_{P_{n,k} \leq 1}\right] \, \E |\tilde Z_{n,k}(t)|^p < C.
$$

\vspace*{2mm}
\noindent
\textsc{Case 2:} $1 \leq  p < 2$.
Then, by Proposition~\ref{prop:von_bahr_esseen_non_centered} and~\eqref{eq:S_n_1_tech1_beak_bound},
\begin{align*}
\E |S_{n,1}(t)|^p
&\leq
C N_{n,1}  \E \left[P_{n,k}^{-\frac{p\Re \beta_{n,l}(t)}{\sigma_1}} \ind_{P_{n,k} \leq 1}\right] \, \E |\tilde Z_{n,k}(t)|^p + C|\E S_{n,1}(t)|^p\\
&<
C + C|\E S_{n,1}(t)|^p.
\end{align*}
We need to estimate $\E S_{n,1}(t)$. Clearly,
$$
|\E S_{n,1}(t)| = N_{n,1}  \left| \E \left(P_{n,k}^{-\frac{\beta_{n,l}(t)}{\sigma_1}} \ind_{P_{n,k} \leq 1}\right)\right| \cdot |\E \tilde Z_{n,k}(t)|.
$$
Recall that $\beta_{n,l}(t)$ converges to $\beta_*$ uniformly in $t\in K$ and
note that $\sigma_* < \frac{\sigma_1}{p} < \sigma_1$ because $p\geq 1$.  By
Lemma~\ref{lem:P_n_k_moment_E1_part1}, we can estimate  the first factor on the
left-hand side by $C$. By the induction
assumption~\eqref{eq:Z_n_k_tilde_moment_induction_beak_bound} we have  the
estimate $\E  |\tilde Z_{n,k}(t)|^p \leq C$ (recall that $p <
\frac{\sigma_1}{\sigma_*} < \frac{\sigma_2}{\sigma_*}$). By Lyapunov's
inequality~\eqref{eq:lyapunov_ineq} (recall that $p\geq 1$), this implies that
$|\E  \tilde Z_{n,k}(t)| \leq C$. Hence, we obtain the estimate $|\E
S_{n,1}(t)|<C$.
\end{proof}

\vspace*{2mm} \noindent \textsc{Step 2.}  In this step, we obtain estimates for
the $p$-th moments of $S_n(t)-S_{n,T}(t)$  and $S_{n}^{\circ}(t)-
S_{n,T}^{\circ}(t)$. The main result of this step is
Lemma~\ref{lem:S_n_T_minus_S_n_GE_beak_bound}.
\begin{lemma}\label{lem:E_S_n_minus_S_n_T_beak_bound}
Fix some $3\leq l\leq d$ and let $\beta_*=\sigma_*+i\tau_*\in \C$ be such
that~\eqref{eq:beta_star_beak_boundary} holds and, additionally,
$\sigma_*<\sigma_2$. Let $K$ be a compact subset of  $\C$.  Then, there exist
constants $C=C(K)>0$ and $\eps=\eps(K)>0$  such that for all $t\in K$, $T\in
\N$, $n\in\N$,
\begin{equation}\label{eq:moment_S_n_T_tail2_beak_bound}
|\E (S_{n}(t) - S_{n,T}(t))| < C T \, \eee^{-\eps n}.
\end{equation}
\end{lemma}
\begin{remark}\label{rem:lem:E_S_n_minus_S_n_T_beak_bound}
In the case $l=2$, we will prove a weaker estimate $|\E (S_{n}(t) - S_{n,T}(t))| < C T$.
\end{remark}
\begin{proof}[Proof of Lemma~\ref{lem:E_S_n_minus_S_n_T_beak_bound} and Remark~\ref{rem:lem:E_S_n_minus_S_n_T_beak_bound}]
Fix some $2\leq l \leq d$.  The subsequent estimates are valid uniformly over $t\in K$.  Since $\beta_{n,l}(t)$ converges to $\beta_*$ and since $\sigma_*+|\tau_*|=\sigma_l > \sigma_1$, we can apply Lemma~\ref{lem:P_n_k_moment_E1_part2b} to obtain
$$
|\E (S_{n}(t) - S_{n,T}(t))| = N_{n,1} \left|\E P_{n,k}^{-\frac{\beta_{n,l}(t)}{\sigma_1}}\ind_{P_{n,k}\geq T}\right|\, |\E \tilde Z_{n,k}(t)|
\leq
CT\, |\E \tilde Z_{n,k}(t)|.
$$
We have to estimate $\E \tilde Z_{n,k}(t)$. By definition of $\tilde Z_{n,k}(t)$, see~\eqref{eq:Z_n_k_tilde_def_beak_bound} and~\eqref{eq:h_n_l_t_tilde_def}, we have
\begin{equation}\label{eq:E_Z_n_k_tilde_computation_beak_bound}
\E \tilde Z_{n,k}(t) =  \prod_{j=2}^l  \left(\eee^{-\beta_{n,l}(t)\sqrt{na_j} u_{n,j}} N_{n,j} e^{\frac 12 \beta_{n,l}^2(t) n a_j}\right).
\end{equation}
Note that the terms with $j>l$ are missing in the product because they are equal to $1$.

\vspace*{2mm}
\noindent
\textsc{Case 1:} $l=2$. In this case, the product in~\eqref{eq:E_Z_n_k_tilde_computation_beak_bound} has just one term which, by Lemma~\ref{lem:beak_bound_expect_extremes_e_t}, converges to $\eee^t$ uniformly in $t\in K$. We can estimate this term by $C$, thus proving Remark~\ref{rem:lem:E_S_n_minus_S_n_T_beak_bound}.

\vspace*{2mm}
\noindent
\textsc{Case 2:} $3\leq l \leq d$. By Lemma~\ref{lem:beak_bound_expect_extremes_e_t}, the last factor in~\eqref{eq:E_Z_n_k_tilde_computation_beak_bound} converges to $\eee^t$ uniformly in $t\in K$. Thus, we can estimate the last factor by $C$. However, there is at least one factor with $j\geq 2$ and $j < l$. For the latter one, we have (recall~\eqref{eq:asympt_N_nk} and~\eqref{eq:u_n_k_asympt})
\begin{align*}
|\eee^{-\beta_{n,l}(t)\sqrt{na_j} u_{n,j}} N_{n,j} e^{\frac 12 \beta_{n,l}^2(t) n a_j}|
&=
\eee^{na_j(-\sigma_* \sigma_j + \frac 12 \sigma_j^2 + \frac 12 (\sigma_*^2 - \tau_*^2)) + o(n)}\\
&=
\eee^{\frac 12 na_j((\sigma_j-\sigma_*)^2 - \tau_*^2) + o(n)}.
\end{align*}
Since $\sigma_j-\sigma_* < \sigma_l - \sigma_* = \tau_*$ and $\sigma_j-\sigma_*\geq \sigma_2-\sigma_*>0$, we can estimate the term by $\eee^{-\eps n}$, for some sufficiently small $\eps>0$ and all sufficiently large $n$.
For the right-hand side of~\eqref{eq:E_Z_n_k_tilde_computation_beak_bound} we obtain the estimate
$$
|\E \tilde Z_{n,k}(t)| < C \eee^{-\eps n}.
$$
This completes the proof of~\eqref{eq:moment_S_n_T_tail2_beak_bound}.
\end{proof}

\begin{lemma}\label{lem:S_n_T_minus_S_n_GE_beak_bound}
Fix some $2\leq l\leq d$ and some $\beta_*=\sigma_*+i\tau_*\in\C$ satisfying~\eqref{eq:beta_star_beak_boundary}. Let  $p\in (0,2)$ be such that $p < \frac{\sigma_2}{\sigma_*}$. Let $K$ be a compact subset of $\C$.
\begin{enumerate}

\item[\textup{(1)}]  If $l=2$, then there exist  constants $C=C(K)>0$   and and $\eps=\eps(K)>0$ such that, for all $t\in K$, $T\in \N$, $n\in\N$, we have
\begin{align}
& \E |S_{n}^{\circ}(t) - S_{n,T}^{\circ}(t)|^p \leq C T^{-\eps}.  \label{eq:lem:S_n_T_minus_S_n_case1_beak_bound}
\end{align}

\item[\textup{(2)}] If $3\leq l\leq d$, then there exist constants $C=C(K)>0$ and $\eps=\eps(K)>0$ such that, for all $t\in K$, $T\in \N$, $n\in\N$, we have
\begin{align}
& \E |S_{n}(t) - S_{n,T}(t)|^p \leq C T^{-\eps} + CT^2 \eee^{-\eps n}.  \label{eq:lem:S_n_T_minus_S_n_case2_beak_bound}
\end{align}
\end{enumerate}
\end{lemma}
\begin{remark}\label{rem:tech1_beak_bound}
Under the assumptions of Part~1, there exists a constant  $C=C(K)>0$ such that for all $t\in K$, $T\in \N$, $n\in\N$ we have
\begin{equation} \label{eq:lem:S_n_T_minus_S_n_case2_rem_beak_bound}
\E |S_{n}(t) - S_{n,T}(t)|^p \leq C T^p.
\end{equation}
To see this, note that $\beta_{n,l}(t)$ converges to $\beta_*$ and that $\sigma_*+|\tau_*|=\sigma_2 >\sigma_1$ so that we can use Lemma~\ref{lem:P_n_k_moment_E1_part2b} to obtain the estimate
\begin{equation}\label{eq:rem_tech1_beak_bound}
|(S_{n}(t) - S_{n,T}(t)) - (S_{n}^{\circ}(t) - S_{n,T}^{\circ}(t))|
=
N_{n,1} \left|\E \left(P_{n,k}^{-\frac{\beta_{n,l}(t)}{\sigma_1}} \ind_{P_{n,k}>T}\right)\right| < C T.
\end{equation}
Combining~\eqref{eq:lem:S_n_T_minus_S_n_case1_beak_bound} and~\eqref{eq:rem_tech1_beak_bound} and using Jensen's inequality~\eqref{eq:jensen_ineq}, we obtain~\eqref{eq:lem:S_n_T_minus_S_n_case2_rem_beak_bound}.
\end{remark}

\begin{proof}[Proof of Lemma~\ref{lem:S_n_T_minus_S_n_GE_beak_bound}]
We prove both parts of the lemma simultaneously. The subsequent estimates hold uniformly in $t\in K$.
Since $p<2$, $p<\frac{\sigma_2}{\sigma_*}$, $\sigma_1<\sigma_2$, $\sigma_*> \frac{\sigma_l}{2}>\frac{\sigma_1}{2}$, there exists a number $q$ such that
\begin{equation}\label{eq:q_choice_beak_bound}
\max\left\{p,\frac{\sigma_1}{\sigma_*}\right\} < q < \min\left\{\frac{\sigma_2}{\sigma_*}, 2\right\}.
\end{equation}
In~\eqref{eq:lem:S_n_T_minus_S_n_case1_beak_bound} and~\eqref{eq:lem:S_n_T_minus_S_n_case2_beak_bound}, it suffices to provide estimates for the $q$-th moment instead of the $p$-th moment since by Lyapunov's inequality~\eqref{eq:lyapunov_ineq} we have (recalling that $p<q$)
\begin{align*}
\E |S_{n}(t) - S_{n,T}(t)|^p &\leq \left(\E |S_{n}(t) - S_{n,T}(t)|^q\right)^{p/q},
\\
\E |S_{n}^{\circ}(t) - S_{n,T}^{\circ}(t)|^p &\leq \left(\E |S_{n}^{\circ}(t) - S_{n,T}^{\circ}(t)|^q\right)^{p/q}.
\end{align*}
For future use, note that there exist $C=C(K)>0$, $\eps=\eps(K)>0$ such that for all $t\in K$, $T\in \N$, $n > n_0$,
\begin{equation}\label{eq:CLT_tech1_beak_bound}
N_{n,1} \E \left[P_{n,k}^{-\frac{q \Re \beta_{n,l}(t)}{\sigma_1}}\ind_{P_{n,k}>T}\right] \, \E  |\tilde Z_{n,k}(\beta)|^q < C T^{-\eps}.
\end{equation}
Indeed, by Lemma~\ref{lem:P_n_k_moment_E1_part2a} (recall that $\Re \beta_{n,l}(t)$ converges to $\sigma_*$ and $\sigma_* q >\sigma_1$ by~\eqref{eq:q_choice_beak_bound}) we can estimate  the first factor on the left-hand side by $C T^{-\eps}$. By the induction assumption~\eqref{eq:Z_n_k_tilde_moment_induction_beak_bound}  we have $\E  |\tilde Z_{n,k}(\beta)|^q\leq C$ (recall that $q < \frac{\sigma_2}{\sigma_*}$ by~\eqref{eq:q_choice_beak_bound}). This proves~\eqref{eq:CLT_tech1_beak_bound}.

\vspace*{2mm}
\noindent
\textsc{Part 1.}
Assume that we are in the setting of Part~1 of Lemma~\ref{lem:S_n_T_minus_S_n_GE_beak_bound}. We prove~\eqref{eq:lem:S_n_T_minus_S_n_case1_beak_bound}. It follows from $l=2$ that we have the estimate $|\E \tilde Z_{n,k}(t)|<C$; see Case~1 in the proof of Lemma~\ref{lem:E_S_n_minus_S_n_T_beak_bound}.

\vspace*{2mm}
\noindent
\textsc{Case 1: $0<q\leq 1$.}
Using Lemma~\ref{lem:centered_moment_ineq},  we obtain
\begin{align}
\lefteqn{\E |S_{n}^{\circ}(t) - S_{n,T}^{\circ}(t)|^q}  \label{eq:GE_proof_tech3_beak_bound}\\
&\leq  C \E |S_{n}(t) - S_{n,T}(t)|^q
+ C \left|N_{n,1}\E\left(P_{n,k}^{-\frac{\beta_{n,l}(t)}{\sigma_1}}\ind_{P_{n,k}>T}\right)\right|^q |\E \tilde Z_{n,k}(t)|^q. \notag
\end{align}
By Proposition~\ref{prop:ineq_moment_p_01} (which is applicable in the case $0<q\leq 1$) and  by~\eqref{eq:CLT_tech1_beak_bound},
\begin{align*}
\E |S_{n}(t) - S_{n,T}(t)|^q
&\leq
C N_{n,1} \E \left[P_{n,k}^{-\frac{q \Re \beta_{n,l}(t)}{\sigma_1}} \ind_{P_{n,k}>T}\right]  \E |\tilde Z_{n,k}(\beta)|^q
\leq
CT^{-\eps}.
\end{align*}
The second term on the right-hand side of~\eqref{eq:GE_proof_tech3_beak_bound} can also be estimated by $CT^{-\eps}$.  Indeed, since $\sigma_* > \frac{\sigma_1} q >\sigma_1$ by~\eqref{eq:q_choice_beak_bound} and by the assumption $0<q\leq 1$, we can apply Lemma~\ref{lem:P_n_k_moment_E1_part2a} to obtain that
$$
\left|N_{n,1}\E\left(P_{n,k}^{-\frac{\beta_{n,l}(t)}{\sigma_1}}\ind_{P_{n,k}>T}\right)\right| < CT^{-\eps}.
$$
Also, recall the estimate $|\E \tilde Z_{n,k}(t)| < C$.

\vspace*{2mm}
\noindent
\textsc{Case 2: $1\leq q <2$.}
It follows from~\eqref{eq:S_n_circ_def_beak_bound} and~\eqref{eq:S_n_T_circ_def_beak_bound} that we can write
$$
S_{n}^{\circ}(t) - S_{n,T}^{\circ}(t)
=
\sum_{k=1}^{N_{n,1}} \left(P_{n,k}^{-\frac{\beta_{n,l}(t)}{\sigma_1}} \ind_{P_{n,k}>T}  \tilde Z_{n,k}(t) - \E\left[P_{n,k}^{-\frac{\beta_{n,l}(t)}{\sigma_1}}\ind_{P_{n,k}>T}\tilde Z_{n,k}(t)\right]\right).
$$
The summands on the right-hand side have zero mean. By Proposition~\ref{prop:von_bahr_esseen} (which is applicable in the case $1\leq q<2$) and by Lemma~\ref{lem:centered_moment_ineq} (where we use that $q\geq 1$), we have
\begin{align*}
\lefteqn{\E |S_{n}^{\circ}(t) - S_{n,T}^{\circ}(t)|^q}\\
&\leq
C N_{n,1} \E \left|P_{n,k}^{-\frac{\beta_{n,l}(t)}{\sigma_1}} \ind_{P_{n,k}>T}  \tilde Z_{n,k}(t) - \E\left[P_{n,k}^{-\frac{\beta_{n,l}(t)}{\sigma_1}}\ind_{P_{n,k}>T}\tilde Z_{n,k}(t)\right]\right|^q\\
&\leq
C N_{n,1} \E \left[P_{n,k}^{-\frac{q \Re \beta_{n,l}(t)}{\sigma_1}} \ind_{P_{n,k}>T}\right]  \E |\tilde Z_{n,k}(t)|^q.
\end{align*}
The right-hand side can be estimated by $CT^{-\eps}$ by~\eqref{eq:CLT_tech1_beak_bound}.

\vspace*{2mm}
\noindent
\textsc{Part 2.}
Assume that we are in the setting of Part~2 of Lemma~\ref{lem:S_n_T_minus_S_n_GE_beak_bound}. We prove~\eqref{eq:lem:S_n_T_minus_S_n_case2_beak_bound}.

\vspace*{2mm}
\noindent
\textsc{Case 1: $0<q\leq 1$.}
By Proposition~\ref{prop:ineq_moment_p_01} and~\eqref{eq:CLT_tech1_beak_bound}, we obtain that
\begin{equation*}
\E |S_{n}(t) - S_{n,T}(t)|^q
\leq
N_{n,1} \E \left[P_{n,k}^{-\frac{q \Re \beta_{n,l}(t)}{\sigma_1}}\ind_{P_{n,k}>T}\right] \, \E  |\tilde Z_{n,k}(\beta)|^q
\leq
C T^{-\eps}.
\end{equation*}

\vspace*{2mm}
\noindent
\textsc{Case 2: $1\leq q <  2$.}
By Proposition~\ref{prop:von_bahr_esseen_non_centered} (which is applicable in the case $1\leq q < 2$), we obtain that
\begin{align*}
\lefteqn{\E |S_{n}(t) - S_{n,T}(t)|^q}\\
&\leq C N_{n,1} \E \left[P_{n,k}^{-\frac{q\Re \beta_{n,l}(t)}{\sigma_1}}\ind_{P_{n,k}>T}\right] \, \E |\tilde Z_{n,k}(\beta)|^q + C |\E (S_{n}(t) - S_{n,T}(t))|^q
\end{align*}
The first summand on the right-hand side can be estimated by $CT^{-\eps}$ by~\eqref{eq:CLT_tech1_beak_bound}, whereas the second summand  can be estimated by $CT^2 \eee^{-\eps n}$ by Lemma~\ref{lem:E_S_n_minus_S_n_T_beak_bound}. The assumptions of this lemma are satisfied because $\sigma_* < \frac{\sigma_2}{q} < \sigma_2$ by~\eqref{eq:q_choice_beak_bound} and the assumption $1 \leq q <2$.

\vspace*{2mm}
The proof of Lemma~\ref{lem:S_n_T_minus_S_n_GE_beak_bound} is complete.
\end{proof}

\subsection{Proof of the functional limit theorem}

In this section, we prove Theorem~\ref{theo:functional_clt_GE_beak_boundary_restate}. We have to show that weakly on $\HHH(\C)$,
\begin{equation}\label{eq:func_clt1_restate_beak_bound}
S_n(t)=
\frac{\ZZZ_n(\beta_{n,l}(t))}{\eee^{h_{n,l}(t)}}
=
\sum_{k=1}^{N_{n,1}} P_{n,k}^{-\frac{\beta_{n,l}(t)}{\sigma_1}}\tilde Z_{n,k}(t)
\toweak
\eee^t \zeta^{(l-1)} + \zeta^{(l)}.
\end{equation}
We will use induction over $l$.  In the case $l=1$ (which is the basis of
induction), we proved~\eqref{eq:func_clt1_restate_beak_bound} in
Section~\ref{subsec:beak_bound_basis}. Take some $l\geq 2$ and assume
that~\eqref{eq:func_clt1_restate_beak_bound} has been established for all
smaller values of $l$. The random function $\tilde Z_{n,k}(t)$ is an analogue of
the random function $S_n(t)$ with $d$ and $l$ reduced by $1$. By the induction
assumption, we have the following weak convergence on $\HHH(\C)$:
\begin{equation}\label{eq:tech223_beak_bound}
\{\tilde Z_{n,k}(t) \colon t\in\C\} \toweak  \{\eee^t \tilde \zeta^{(l-2)} + \tilde \zeta^{(l-1)}\colon t\in\C\},
\end{equation}
where
$$
\tilde \zeta^{(l-2)} = \zeta_P\left(\frac{\beta_*}{\sigma_2}, \ldots, \frac{\beta_*}{\sigma_{l-2}}\right),
\;\;\;
\tilde \zeta^{(l-1)}= \zeta_P\left(\frac{\beta_*}{\sigma_2}, \ldots, \frac{\beta_*}{\sigma_{l-1}}\right).
$$
First, we will show that~\eqref{eq:func_clt1_restate_beak_bound} holds in the sense of weak convergence of finite-dimensional distributions. Fix some $t_1,\ldots,t_r\in \C$. We will prove that the random vector $\bS_n: = \{S_{n} (t_i)\}_{i=1}^r$ converges in distribution to $\bS_{\infty} = \{S_{\infty} (t_i)\}_{i=1}^r$, where
\begin{align*}
S_{n} (t)
=
\sum_{k=1}^{N_{n,1}} P_{n,k}^{-\frac{\beta_{n,l}(t)}{\sigma_1}}\tilde Z_{n,k}(t),
\quad
S_{\infty}(t)
&=
\eee^t \zeta^{(l-1)} + \zeta^{(l)}.
\end{align*}

\vspace*{2mm}
\noindent
\textsc{Case A:} $3\leq l \leq d$.
We will verify the conditions of Lemma~\ref{lem:interchange_limits} for the random vectors $\bS_{n,T}:= \{S_{n,T} (t_i)\}_{i=1}^r$ and $\bS_{\infty, T}:= \{S_{\infty, T}(t_i)\}_{i=1}^r$, where $T\in\N$ is a truncation parameter  and
\begin{align*}
S_{n,T} (t)
&=
\sum_{k=1}^{N_{n,1}} P_{n,k}^{-\frac{\beta_{n,l}(t)}{\sigma_1}} \ind_{P_{n,k}\leq T}  \tilde Z_{n,k}(t),\\
S_{\infty, T}(t)
&=
\sum_{k=1}^{\infty} P_k^{-\frac{\beta_*}{\sigma_1}} \ind_{P_{k}\leq T} (\eee^t \tilde \zeta_{k}^{(l-2)} + \tilde \zeta_{k}^{(l-1)}).
\end{align*}
Here, we denote by $(\tilde \zeta_{k}^{(l-2)}, \tilde \zeta_{k}^{(l-1)})$, $k\in\N$, independent copies of the random vector $(\tilde \zeta^{(l-2)}, \tilde \zeta^{(l-1)})$.

\vspace*{2mm}
\noindent
\textsc{Step A1.} We prove that $\bS_{n,T} \todistr \bS_{\infty,T}$ for every $T\in\N$.
From
Lemma~\ref{lem:adjoin_level} and~\eqref{eq:tech223_beak_bound}, it follows that
\begin{align*}
\left\{\sum_{k=1}^{N_{n,1}}  P_{n,k}^{-\frac{\beta_{n,l}(t_i)}{\sigma_1}} \ind_{P_{n,k}\leq T}  \tilde Z_{n,k}(t_i)\right\}_{i=1}^r
&\todistr
\left\{\sum_{k=1}^{\infty} P_k^{-\frac{\beta_*}{\sigma_1}} \ind_{P_{k}\leq T} (\eee^{t_i} \tilde \zeta_{k}^{(l-2)} + \tilde \zeta_{k}^{(l-1)}) \right\}_{i=1}^r. 
\end{align*}
This is the desired convergence.

\vspace*{2mm}
\noindent
\textsc{Step A2.} By Proposition~\ref{prop:zeta_recursion}, we have $\bS_{\infty, T} \todistrT \bS_{\infty}$ (at this point we use that  $l\geq 3$).

\vspace*{2mm}
\noindent
\textsc{Step A3.} Fix $t\in \C$. To verify the third condition of Lemma~\ref{lem:interchange_limits}, it suffices to prove that
$$
\lim_{T\to\infty}  \limsup_{n\to\infty} \E |S_{n}(t) - S_{n,T}(t)|^p=0.
$$
However, this has already been established in Lemma~\ref{lem:S_n_T_minus_S_n_GE_beak_bound}, Part~2. (Here, we again use that  $l\geq 3$).

\vspace*{2mm}
\noindent
\textsc{Case B:} $l=2$. We will prove that the random vector $\bS_n^{\circ}: = \{S_{n}^{\circ} (t_i)\}_{i=1}^r$ converges in distribution to $\bS_{\infty}^{\circ} = \{S_{\infty}^{\circ} (t_i)\}_{i=1}^r$, where
\begin{align*}
S_{n}^{\circ} (t)
&=
S_n(t) -  N_{n,1} \E \left[P_{n,k}^{- \frac{\beta_{n,2}(t)}{\sigma_1}}\ind_{1\leq P_{n,k}}\right]\, \E [\tilde Z_{n,k}(t)],\\
S_{\infty}^{\circ}(t)
&=
\eee^t \left(\zeta^{(1)} - \frac{\sigma_1}{\beta_*-\sigma_1}\right) + \zeta^{(2)}.
\end{align*}
This implies that $\bS_n= \{S_{n} (t_i)\}_{i=1}^r$ converges in distribution to $\bS_{\infty}=\{S_{\infty} (t_i)\}_{i=1}^r$ because
\begin{align}
&\lim_{n\to\infty} N_{n,1} \E \left[P_{n,k}^{- \frac{\beta_{n,2}(t)}{\sigma_1}}\ind_{1\leq P_{n,k}}\right] = \frac{\sigma_1}{\beta_*-\sigma_1},\label{eq:tech1_beak_bound}\\
&\lim_{n\to\infty} \E [\tilde Z_{n,k}(t)] =\lim_{n\to\infty} N_{n,2}\eee^{-\beta_{n,2}(t) \sqrt{na_2}u_{n,2} + \frac 12 \beta_{n,2}^2(t) na_2} = \eee^t.  \label{eq:tech2_beak_bound}
\end{align}
Note that~\eqref{eq:tech1_beak_bound} follows from Lemma~\ref{lem:P_n_k_moment_kljuv_1}, whereas~\eqref{eq:tech2_beak_bound} follows from~\eqref{eq:Z_n_k_tilde_def_beak_bound}, \eqref{eq:h_n_l_t_tilde_def} and Lemma~\ref{lem:beak_bound_expect_extremes_e_t}.

To prove that $\bS_n^{\circ}$ converges in distribution to $\bS_{\infty}^{\circ}$, we will verify the conditions of Lemma~\ref{lem:interchange_limits} for the random vectors $\bS_{n,T}^{\circ}:= \{S_{n,T}^{\circ} (t_i)\}_{i=1}^r$ and $\bS_{\infty, T}^{\circ}:= \{S_{\infty, T}^{\circ}(t_i)\}_{i=1}^r$, where $T\in\N$ is a truncation parameter  and
\begin{align*}
S_{n,T}^{\circ} (t)
&=
S_{n,T}(t) - N_{n,1} \E \left[P_{n,k}^{- \frac{\beta_{n,2}(t)}{\sigma_1}}\ind_{1\leq P_{n,k}\leq T}\right]\, \E [\tilde Z_{n,k}(t)],\\
S_{\infty, T}^{\circ}(t)
&=
  \sum_{k=1}^{\infty} P_k^{-\frac{\beta_*}{\sigma_1}} \ind_{P_{k}\leq T} (\eee^t + \tilde \zeta_{k}^{(1)}) - \eee^t \int_{1}^T y^{-\frac{\beta_*}{\sigma_1}} \dd y.
\end{align*}

\vspace*{2mm}
\noindent
\textsc{Step B1.}
We prove that $\bS_{n,T}^{\circ} \todistr \bS_{\infty,T}^{\circ}$ for every $T\in\N$. In the same way as in Step~A1 we have $\bS_{n,T} \todistr \bS_{\infty,T}$. To complete the proof, recall~\eqref{eq:tech2_beak_bound} and note that by Lemma~\ref{lem:exp_P_n_k_z},
$$
\lim_{n\to\infty} N_{n,1} \E \left[P_{n,k}^{- \frac{\beta_{n,2}(t)}{\sigma_1}}\ind_{1\leq P_{n,k}\leq T}\right] = \int_{1}^T y^{-\frac{\beta_*}{\sigma_1}} \dd y.
$$
This yields the desired convergence.

\vspace*{2mm}
\noindent
\textsc{Step B2.} By Proposition~\ref{prop:zeta_recursion}  and~\eqref{eq:zeta_P_anal_cont_d_1}, we have $\bS_{\infty, T}^{\circ} \todistrT \bS_{\infty}^{\circ}$.

\vspace*{2mm}
\noindent
\textsc{Step B3.} Fix $t\in \C$. To verify the third condition of Lemma~\ref{lem:interchange_limits}, it suffices to prove that
$$
\lim_{T\to\infty}  \limsup_{n\to\infty} \E |S_{n}^{\circ}(t) - S_{n,T}^{\circ}(t)|^p=0.
$$
However, this has already been established in Lemma~\ref{lem:S_n_T_minus_S_n_GE_beak_bound}, Part~1.

\vspace*{2mm}
Both in Case A and in Case B we showed that~\eqref{eq:func_clt1_restate_beak_bound} holds in the sense of finite-dimensional distributions.
To complete the proof of Theorem~\ref{theo:functional_clt_GE_beak_boundary_restate}, we need to show that the sequence of random functions $S_n(t)$ is tight on $\HHH(\C)$.
If $p>0$ is sufficiently small, then by Proposition~\ref{prop:moment_S_n_beak_bound} for every compact set $K$ there exists a constant $C=C(K)>0$  such that $\E |S_n(t)|^p<C$, for all $t\in K$ and all $n\in\N$.  By Proposition~\ref{prop:tightness_random_analytic}, the sequence of random analytic functions $S_n(t)$ is tight on $\HHH(\C)$, thus completing the proof of Theorem~\ref{theo:functional_clt_GE_beak_boundary_restate}.

\subsection{Functional limit theorem exactly on the boundary}
Fix some $1\leq l\leq d$ and take some $\beta_*=\sigma_*+i \tau_*\in \C$ such that~\eqref{eq:beta_star_beak_boundary} holds.  In Theorem~\ref{theo:functional_clt_GE_beak_boundary_restate}, we considered the fluctuations of $\ZZZ_n(\beta)$ in a small window located \textit{outside} $E_l$ at a distance of order $\text{const} \cdot \frac{\log n}{n}$ from $\beta_*$. The distance was chosen so that the ``line of zeros'' becomes visible in the limit. In the sequel, we study what happens if we look at the partition function $\ZZZ_n(\beta_*)$ \textit{exactly} on the beak shaped boundary of $E_l$.
Define a sequence of normalizing constants
\begin{equation}\label{eq:hat_h_n_l_beta}
\hat h_{n,l}(\beta_*) = \sum_{j=1}^{l-1} \beta_* \sqrt{na_j} u_{n,j} + \sum_{j=l}^d \left(\log N_{n,j} + \frac 12 \beta_*^2 na_j\right).
\end{equation}
\begin{theorem}\label{theo:fluct_beak_exact}
Fix some $1\leq l\leq d$ and some $\beta_*=\sigma_*+i \tau_*\in \C$ such that~\eqref{eq:beta_star_beak_boundary} holds.
Then,
$$
\frac{\ZZZ_n(\beta_*)}{\eee^{\hat h_{n,l}(\beta_*)}} \todistr \zeta_P\left(\frac{\beta_*}{\sigma_1},\ldots, \frac{\beta_*}{\sigma_{l-1}}\right).
$$
\end{theorem}
\begin{proof}
The proof of Theorem~\ref{theo:fluct_beak_exact} uses the same method as the proof of
Theorem~\ref{theo:functional_clt_GE_beak_boundary_restate}, so we just describe the idea.
We use induction over $l$.
In the case $l=1$, we already proved in Proposition~\ref{prop:theo:functional_E1_restate} that
\begin{equation}\label{eq:fluct_beak_exact_l_1_tech222}
\frac{\ZZZ_n(\beta_*)-\E \ZZZ_n(\beta_*)}{\eee^{\beta_* \sqrt{na_1} u_{n,1} + \tilde c_n(\beta_*)}}
\todistr
\zeta_P\left(\frac{\beta_*}{\sigma_1}\right),
\end{equation}
where we recall that
$
\tilde c_n(\beta_*)= \sum_{j=2}^d (\log N_{n,j} + \frac 12 \beta^2_* n a_j).
$
By the same computation as in the proof of Lemma~\ref{lem:beak_bound_expect_extremes_e_t} with $\delta_n=0$, we have
\begin{equation}\label{eq:tech444}
\frac{|\E \ZZZ_n(\beta_*)|}{\eee^{\beta_* \sqrt{na_1} u_{n,1} + \tilde c_n(\beta_*)}}
=
|N_{n,1} \eee^{\frac 12 \beta_*^2 na_1 - \beta_* \sqrt{na_1} u_{n,1}}|
=
\eee^{\frac{\sigma_*}{2\sigma_1} \log (4\pi n \log \alpha_1) + o(1)},
\end{equation}
which converges to $+\infty$.
It follows from~\eqref{eq:fluct_beak_exact_l_1_tech222} and~\eqref{eq:tech444} that in the case $l=1$ we have
\begin{equation}\label{eq:fluct_beak_exact_l_1_weaker}
\frac{\ZZZ_n(\beta_*)}{\E \ZZZ_n(\beta_*)} \todistr 1.
\end{equation}
This proves Theorem~\ref{theo:fluct_beak_exact} in the case $l=1$, thus establishing the basis of induction. The rest of the proof, namely the adjoining of glassy phase levels to~\eqref{eq:fluct_beak_exact_l_1_weaker}, is analogous to the proof of Theorem~\ref{theo:functional_clt_GE_beak_boundary_restate}.
\end{proof}

\section{Functional limit theorems in phases with at least one fluctuation level} \label{sec:adjoin_spin_glass_to_F}

In this section, we prove functional limit theorems describing the local
behavior of the partition function $\ZZZ_n(\beta)$ near some $\beta_* =
\sigma_*+i\tau_*$ located inside or on the boundary of the phase
$G^{d_1}F^{d_2}E^{d_3}$, where $d_2\geq 1$.  Note that the case $d_2=0$ has been
considered in Section~\ref{sec:func_CLT_GE}. Our main aim in this section is to
prove Theorem~\ref{theo:functional_clt}. The proofs of
Theorems~\ref{theo:functional_clt_line of_zeros_d2_geq2},
\ref{theo:functional_clt_line of_zeros_d2_eq1},
\ref{theo:clt_boundary_spin_glass}, \ref{theo:clt_triple_point} are all very
similar and will be discussed in Section~\ref{sec:func_CLT_F_analogous_proofs}.

\subsection{Notation}
Fix some $d_1,d_2,d_3\in \{0,\ldots,d\}$ such that $d_1+d_2+d_3=d$ and $d_2\geq 1$. Let $\beta_*=\sigma_*+i\tau_*\in \C$ be such that
\begin{equation}\label{eq:beta_star_fluct_level}
\beta_* \in G^{d_1}F^{d_2}E^{d_3},\;\;\;\sigma_*\geq 0, \;\;\;  \tau_*>0.
\end{equation}
Define a local coordinate near $\beta_*$ by
\begin{equation}\label{eq:beta_n_def_fluct_level}
\beta_n(t)=\beta_*+\frac{t}{\sqrt n}, \;\;\; t\in\C.
\end{equation}
For $1\leq k\leq d$, define the normalizing functions $c_{n,k}(\beta_*;t)$, where $t\in\C$, by
\begin{equation}\label{eq:def_c_nk_t_fluct_level}
c_{n,k}(\beta_*; t) =
\begin{cases}
\beta_n(t)\sqrt {na_k} u_{n,k}, & \text{if } \beta \in G_k,\\
\frac 12 \log N_{n,k}+ a_k (\sqrt n \sigma_* +t)^2,  & \text{if } \beta \in F_k,\\
\log N_{n,k}  + \frac 12 a_k (\sqrt n \beta_* +t)^2,  & \text{if } \beta\in E_k.
\end{cases}
\end{equation}
Define also the normalizing functions
\begin{align}
c_n(\beta_*; t) &= c_{n,1} (\beta_*; t)+\ldots+c_{n,d}(\beta_*;t),\label{eq:c_n_beta_*_aux_def1}\\
\tilde c_n(\beta_*; t) &= c_{n,2} (\beta_*; t)+\ldots+c_{n,d}(\beta_*;t).\label{eq:c_n_beta_*_aux_def2}
\end{align}
Note that these functions are linear or quadratic in $t$. Consider a random analytic function $\{S_n(t)\colon t\in \C\}$ defined by
\begin{equation*}
S_n(t) = \frac{\ZZZ_n(\beta_{n}(t))}{\eee^{c_{n}(\beta_*;t)}}.
\end{equation*}
Our aim is to show that $S_n(t)$ converges weakly on $\HHH(\C)$ and to identify the limiting process.

Define the random variables $P_{n,k}$, $n\in\N$, $1\leq k\leq N_{n,1}$, (the normalized contributions of the first level of the GREM) and $\tilde Z_{n,k}(t)$, $n\in\N$, $1\leq k\leq N_{n,1}$, (the
normalized contributions of the remaining $d-1$ levels of the GREM) by
\begin{align}
P_{n,k} &= \eee^{-\sigma_1 \sqrt {n a_1} (\xi_k - u_{n,1})}, \label{eq:P_n_k_def_fluct_level}\\
\tilde Z_{n,k} (t) &=  \eee^{-\tilde c_{n}(\beta_*; t)} \sum_{\tilde \eps \in \tilde \SSS_n} \eee^{\beta_n(t) \sqrt n (\sqrt{a_2} \xi_{k\eps_2} + \ldots + \sqrt{a_d} \xi_{k\eps_2\ldots\eps_d})} \label{eq:Z_n_k_tilde_def_fluct_level},
\end{align}
where $\tilde \SSS_n$ is as in~\eqref{eq:parameter_set_def_tilde}.
By the definition of the GREM, these random variables have the following properties, for every $n\in\N$:
\begin{enumerate}
\item $\tilde Z_{n,k}(t)$, $1\leq k\leq N_{n,1}$, is an i.i.d.\ collection of random processes.
\item $P_{n,k}$, $1\leq k\leq N_{n,1}$, is an i.i.d.\ collection of random variables.
\item These two collections are independent.
\end{enumerate}
In the case $d_1\geq 1$, we have the representation
\begin{equation}\label{eq:S_n_def_fluct_level}
S_n(t)
= \frac{\ZZZ_n(\beta_{n}(t))}{\eee^{c_{n}(\beta_*;t)}}
= \sum_{k=1}^{N_{n,1}} P_{n,k}^{- \frac{\beta_{n}(t)}{\sigma_1}} \tilde Z_{n,k}(t).
\end{equation}
For $T\in \N$, define the truncated version of $S_n(t)$ by 
\begin{equation}\label{eq:S_n_T_def_fluct_level}
S_{n,T} (t)
=
\sum_{k=1}^{N_{n,1}} P_{n,k}^{-\frac{\beta_{n}(t)}{\sigma_1}}  \ind_{P_{n,k}\leq T}  \tilde Z_{n,k}(t).
\end{equation}

\subsection{Moment estimates}
In this section, we prove estimates for the $p$-th moments of $S_n(t)$ and $S_{n,T}(t)$.
The main results are Proposition~\ref{prop:moment_S_n_fluct_level} and Lemma~\ref{lem:S_n_T_minus_S_n_GE_fluct_level}.
\begin{proposition}\label{prop:moment_S_n_fluct_level}
Let $\beta_*=\sigma_*+i\tau_*\in \C$ be such that~\eqref{eq:beta_star_fluct_level} holds with some $d_2\geq 1$. Fix $p\in (0,2)$ such that $p < \frac{\sigma_1}{\sigma_*}$.
Let $K$ be a compact subset of $\C$. Then, there exists a constant $C=C(K)>0$ such that for all $t \in K$ and all $n\in\N$,
\begin{equation}\label{eq:prop:moment_S_n_fluct_level}
\E |S_n(t)|^p < C.
\end{equation}
\end{proposition}
\begin{proof}
We will use induction over $d_1$, the number of glassy phase levels. We have already verified the case $d_1=0$ (which is the base of our induction) in~\eqref{eq:ZZZ_n_moment_est}. Indeed, if $d_1=0$, then for every compact set $K\subset \C$ we can find $c=c(K)>0$ such that
$$
\sqrt{\E |\ZZZ_n(\beta_n(t))|^2} >\sqrt{\Var \ZZZ_n(\beta_n(t))} > c \eee^{c_n(\beta_*; t)}
$$
by Proposition~\ref{prop:asympt_cov} (which holds uniformly in $t\in K$) and hence,
$$
\E |S_n(t)|^p
=
\E \left|\frac{\ZZZ_n(\beta_n(t))}{\eee^{c_n(\beta_*; t)}}\right|^p
\leq
C \E \left|\frac{\ZZZ_n(\beta_n(t))}{\sqrt{\E |\ZZZ_n(\beta_n(t))|^2}}\right|^p
\leq C,
$$
where the last step is by~\eqref{eq:ZZZ_n_moment_est}. Note at this point that although~\eqref{eq:ZZZ_n_moment_est} is stated for $2< p < \frac{\sigma_1^2}{2\sigma_*^2}$ (this interval is non-empty for $d_1=0$, $d_2\geq 1$), the same inequality continues to hold for $0<p\leq 2$ by the Lyapunov's inequality~\eqref{eq:lyapunov_ineq}.

Let us therefore take $d_1\geq 1$ and assume that Proposition~\ref{prop:moment_S_n_fluct_level} holds in the setting of $d_1-1$ glassy phase levels.  The random function $\tilde Z_{n,k}(t)$ is the analogue of $S_n(t)$ with $d_1-1$  glassy phase levels. Hence, our induction assumption reads as follows.
\begin{itemize}
\item [(IND)] Fix some  $r\in (0,2)$ such that $r < \frac{\sigma_2}{\sigma_*}$. Let $K$ be a compact subset of $\C$.
Then, there exists a constant $C=C(K)>0$ such that for all $t\in K$ and all $n\in\N$,
\begin{equation}\label{eq:Z_n_k_tilde_moment_induction_fluct_level}
\E |\tilde Z_{n,k}(t)|^r < C.
\end{equation}
\end{itemize}

\vspace*{2mm}
\noindent
\textsc{Step 1.} In this step, we estimate the moments of $S_{n,1}(t)$.
\begin{lemma}\label{lem:S_n_1_moments_fluct_level}
Let  $\beta_*=\sigma_*+i\tau_*\in \C$  be such that~\eqref{eq:beta_star_fluct_level} holds with some $d_1\geq 1$ and $d_2\geq 1$. Fix $p\in (0,2)$  such that $p < \frac{\sigma_1}{\sigma_*}$. Let $K$ be a compact subset of $\C$. Then, there is a constant $C=C(K)>0$ such that for all $t\in K$ and all $n\in\N$,
\begin{equation}\label{eq:S_n_1_moment_est_fluct_level}
\E |S_{n,1}(t)|^p  < C.
\end{equation}
\end{lemma}
\begin{proof}
The proof of Lemma~\ref{lem:S_n_1_moments_beak_bound} applies with straightforward changes.
\end{proof}

\vspace*{2mm}
\noindent
\textsc{Step 2.}  In this step, we obtain estimates for the $p$-th moment of $S_n(t)-S_{n,T}(t)$.
The main result of this step is Lemma~\ref{lem:S_n_T_minus_S_n_GE_fluct_level}.
\begin{lemma}\label{lem:E_S_n_minus_S_n_T_fluct_level}
Let  $\beta_*=\sigma_*+i\tau_*\in \C$  be such that~\eqref{eq:beta_star_fluct_level} holds with some $d_1\geq 1$ and $d_2\geq 1$. Assume, additionally, that $\sigma_*<\sigma_2$. Let $K$ be a compact subset of  $\C$.  Then, there exist constants $C=C(K)>0$ and $\eps=\eps(K)>0$  such that for all $t\in K$, $T\in \N$, $n\in\N$,
\begin{equation}\label{eq:moment_S_n_T_tail2_fluct_level}
|\E (S_{n}(t) - S_{n,T}(t))| < C T \, \eee^{-\eps n}.
\end{equation}
\end{lemma}
\begin{proof}
The subsequent estimates are valid uniformly over $t\in K$.  Since $\beta_{n}(t)$ converges to $\beta_*$ and since $\sigma_*+|\tau_*| > \sigma_1$ (because $d_1\geq 1$), we can apply Lemma~\ref{lem:P_n_k_moment_E1_part2b} to obtain
$$
|\E (S_{n}(t) - S_{n,T}(t))| = N_{n,1} \left|\E P_{n,k}^{-\frac{\beta_{n}(t)}{\sigma_1}}\ind_{P_{n,k}\geq T}\right|\, |\E \tilde Z_{n,k}(t)|
\leq
CT\, |\E \tilde Z_{n,k}(t)|.
$$
We need to estimate $\E \tilde Z_{n,k}(t)$. By definition of $\tilde Z_{n,k}(t)$, see~\eqref{eq:Z_n_k_tilde_def_fluct_level} and~\eqref{eq:def_c_nk_t_fluct_level}, we have
\begin{equation}
\E \tilde Z_{n,k}(t)
=
\prod_{j=2}^{d_1}  \left(N_{n,j} \eee^{\frac 12 \beta_*^2 a_j n-\beta_* \sqrt{na_j} u_{n,j}}\right) \prod_{j=d_1+1}^{d_1+d_2} \left(N_{n,j}^{1/2} \eee^{\frac 12 \beta_*^2 a_j n-a_j \sigma_*^2 n}\right) \cdot \eee^{O(\sqrt n)}.
\end{equation}
Note that the factors with $j>d_1+d_2$ are missing on the right-hand side because they are equal to $1$.
We will show that every term in any product on the right-hand side can be estimated by $\eee^{-\eps n}$, for sufficiently large $n$. Since $d_2\geq 1$, there is a least one such term and we obtain the required estimate. Consider some term with $2\leq j\leq d_2$:
$$
N_{n,j} \eee^{\frac 12 \beta_*^2 a_j n-\beta_* \sqrt{na_j} u_{n,j}} =
\eee^{\frac 1 2 n a_j((\sigma_j-\sigma_*)^2 - \tau_*^2) + o(n)} < \eee^{-\eps n},
$$
where the last estimate holds since $\sigma_*< \sigma_2\leq \sigma_j$ and $\sigma_* + |\tau_*| > \sigma_j$ (because $\beta\in G_j$).  Consider some term with $d_1< j\leq d_1+d_2$:
$$
|N_{n,j}^{1/2} \eee^{\frac 12 \beta_*^2 a_j n-a_j \sigma_*^2 n}| =
\eee^{\frac 12 n a_j (\frac 12 \sigma_j^2 - |\beta_*|^2)} < \eee^{-\eps n},
$$
where the last estimate holds since $2|\beta_*|^2 > \sigma_j$ (because $\beta_*\in F_j$).
\end{proof}

\begin{lemma}\label{lem:S_n_T_minus_S_n_GE_fluct_level}
Let $\beta_*=\sigma_*+i\tau_*\in\C$ be such that~\eqref{eq:beta_star_fluct_level} holds with $d_1\geq 1$, $d_2\geq 1$. Let  $p\in (0,2)$ be such that $p < \frac{\sigma_2}{\sigma_*}$. Let $K$ be a compact subset of $\C$.
Then, there exist constants $C=C(K)>0$ and $\eps=\eps(K)>0$ such that for all $t\in K$, $T\in \N$, $n\in\N$ we have
\begin{align}
\E |S_{n}(t) - S_{n,T}(t)|^p \leq C T^{-\eps} + CT^2 \eee^{-\eps n}.  \label{eq:lem:S_n_T_minus_S_n_case2_fluct_level}
\end{align}
\end{lemma}

\begin{proof}
The subsequent estimates hold uniformly in $t\in K$.
By the inequalities $p<2$, $p<\frac{\sigma_2}{\sigma_*}$, $\sigma_1<\sigma_2$, $\sigma_*>\frac {\sigma_1}{2}$, there exists a number $q$ such that
\begin{equation}\label{eq:q_choice_fluct_level}
\max\left\{p,\frac{\sigma_1}{\sigma_*}\right\} < q < \min\left\{\frac{\sigma_2}{\sigma_*}, 2\right\}.
\end{equation}
By Lyapunov's inequality~\eqref{eq:lyapunov_ineq} (recall that $p<q$), it suffices to establish~\eqref{eq:lem:S_n_T_minus_S_n_case2_fluct_level} with $p$-th moment replaced by the $q$-th moment.

For future use, note that there exist $C=C(K)>0$, $\eps=\eps(K)>0$ such that for all $t\in K$, $T\in \N$, and all sufficiently large $n\in \N$,
\begin{equation}\label{eq:CLT_tech1_fluct_level}
N_{n,1} \E \left[P_{n,k}^{-\frac{q \Re \beta_{n}(t)}{\sigma_1}}\ind_{P_{n,k}>T}\right] \, \E  |\tilde Z_{n,k}(t)|^q < C T^{-\eps}.
\end{equation}
We can prove~\eqref{eq:CLT_tech1_fluct_level} as follows. By Lemma~\ref{lem:P_n_k_moment_E1_part2a} (recall that $\Re\beta_{n}(t)$ converges to $\sigma_*$ and $\sigma_* q >\sigma_1$ by~\eqref{eq:q_choice_fluct_level}), we can estimate  the first factor on the left-hand side by $C T^{-\eps}$. By the induction assumption~\eqref{eq:Z_n_k_tilde_moment_induction_fluct_level},  we have $\E  |\tilde Z_{n,k}(t)|^q\leq C$ (recall that $q < \frac{\sigma_2}{\sigma_*}$ by~\eqref{eq:q_choice_fluct_level}).
We are now ready to prove~\eqref{eq:lem:S_n_T_minus_S_n_case2_fluct_level}.

\vspace*{2mm}
\noindent
\textsc{Case 1: $0<q\leq 1$.}
By Proposition~\ref{prop:ineq_moment_p_01} (which is applicable in the case $0<q\leq 1$) and  by~\eqref{eq:CLT_tech1_fluct_level},
\begin{align*}
\E |S_{n}(t) - S_{n,T}(t)|^q
&\leq
C N_{n,1} \E \left[P_{n,k}^{-\frac{q \Re \beta_{n}(t)}{\sigma_1}} \ind_{P_{n,k}>T}\right]  \E |\tilde Z_{n,k}(t)|^q
\leq
CT^{-\eps}.
\end{align*}

\vspace*{2mm}
\noindent
\textsc{Case 2: $1\leq q <  2$.}
By Proposition~\ref{prop:von_bahr_esseen_non_centered} (which is applicable in the case $1\leq q < 2$), we obtain that
\begin{align*}
\lefteqn{\E |S_{n}(t) - S_{n,T}(t)|^q}\\
&\leq C N_{n,1} \E \left[P_{n,k}^{-\frac{q\Re \beta_{n}(t)}{\sigma_1}}\ind_{P_{n,k}>T}\right] \, \E |\tilde Z_{n,k}(t)|^q + C |\E (S_{n}(t) - S_{n,T}(t))|^q
\end{align*}
The first summand on the right-hand side can be estimated by $CT^{-\eps}$ by~\eqref{eq:CLT_tech1_fluct_level}, whereas the second summand  can be estimated by $CT^2 \eee^{-\eps n}$ by Lemma~\ref{lem:E_S_n_minus_S_n_T_fluct_level}. The assumptions of this lemma are satisfied because $\sigma_* < \frac{\sigma_2}{q} < \sigma_2$ by~\eqref{eq:q_choice_fluct_level} and the assumption $1 \leq q <2$.

\vspace*{2mm}
The proof of Lemma~\ref{lem:S_n_T_minus_S_n_GE_fluct_level} is complete.
\end{proof}

To complete the proof of Proposition~\ref{prop:moment_S_n_fluct_level}, combine the results of Lemmas~\ref{lem:S_n_1_moments_fluct_level} and~\ref{lem:S_n_T_minus_S_n_GE_fluct_level}.   
\end{proof}

\subsection{Proof of the functional limit theorem}
In this section, we prove Theorem~\ref{theo:functional_clt}. Introduce the notation
$$
\sigma^{\vartriangle}_*
=
\left(\frac{\sigma_*}{\sigma_1},\ldots, \frac{\sigma_*}{\sigma_{d_1}}\right)\in \R^{d_1},
\quad
\tilde \sigma^{\vartriangle}_*
=
\left(\frac{\sigma_*}{\sigma_2},\ldots, \frac{\sigma_*}{\sigma_{d_1}}\right)\in \R^{d_1-1}.
$$
Consider random analytic functions $\{S_n(t)\colon t\in\C\}$ (which is the same as in~\eqref{eq:S_n_def_fluct_level}) and $\{S_{\infty}(t)\colon t\in\C\}$ given by
\begin{align*}
S_{n} (t)
&=
\sum_{k=1}^{N_{n,1}} P_{n,k}^{-\frac{\beta_{n}(t)}{\sigma_1}}  \tilde Z_{n,k}(t)
=
\frac{\ZZZ_n(\beta_n(t))}{\eee^{c_n(\beta_*; t)}},\\
S_{\infty}(t)
&=\sqrt{\zeta_P(2 \sigma_*^{\vartriangle})} \, \XXX(\kappa t),
\end{align*}
where $\{\XXX(t)\colon t\in\C\}$ is the plane Gaussian analytic function independent of $\zeta_P$ and $\kappa$ is the total variance of the fluctuation levels, as in Theorem~\ref{theo:functional_clt}.
We can now state Theorem~\ref{theo:functional_clt} as follows: Weakly on $\HHH(\C)$,
\begin{equation}\label{eq:func_clt1_restate_fluct_level}
\left\{S_n(t)\colon t\in\C\right\}
\toweak
\left\{S_{\infty}(t) \colon t\in\C\right\}.
\end{equation}

To prove~\eqref{eq:func_clt1_restate_fluct_level}, we will use induction over $d_1$, the number of glassy phase levels. In the case $d_1=0$ (which is the basis of induction) we already established~\eqref{eq:func_clt1_restate_fluct_level} in Theorem~\ref{theo:functional_clt_EF}. Note that in the case $d_1=0$ we have $\sigma_*^{\vartriangle}=\emptyset$ and hence, $\zeta_P(2\sigma_*^{\vartriangle})=1$ by convention. Take some $d_1\geq 1$ and assume that~\eqref{eq:func_clt1_restate_fluct_level} has been established in the setting of  $d_1-1$ glassy phase levels. The random function $\tilde Z_{n,k}(t)$ is an analogue of the random function $S_n(t)$ with $d_1$ reduced by $1$. By the induction assumption, we have the following weak convergence on $\HHH(\C)$:
\begin{equation}\label{eq:tech223_fluct_level}
\{\tilde Z_{n,k}(t) \colon t\in\C\} \toweak  \left\{\sqrt{\zeta_P(2 \tilde \sigma_*^{\vartriangle})} \, \XXX(\kappa t) \colon t\in\C\right\}.
\end{equation}

First, we will show that~\eqref{eq:func_clt1_restate_fluct_level} holds in the sense of weak convergence of finite-dimensional distributions. Fix some $t_1,\ldots,t_r\in \C$. We will prove that the random vector $\bS_n: = \{S_{n} (t_i)\}_{i=1}^r$ converges in distribution to $\bS_{\infty} = \{S_{\infty} (t_i)\}_{i=1}^r$.
We will verify the conditions of Lemma~\ref{lem:interchange_limits} for the random vectors $\bS_{n,T}:= \{S_{n,T} (t_i)\}_{i=1}^r$ and $\bS_{\infty, T}(\beta):= \{S_{\infty, T}(t_i)\}_{i=1}^r$, where $T\in\N$ is a truncation parameter  and
\begin{align*}
S_{n,T} (t)
&=
\sum_{k=1}^{N_{n,1}} P_{n,k}^{-\frac{\beta_{n}(t)}{\sigma_1}} \ind_{P_{n,k}\leq T}  \tilde Z_{n,k}(t),\\
S_{\infty, T}(t)
&=
\sum_{k=1}^{\infty} P_k^{-\frac{\beta_*}{\sigma_1}} \ind_{P_{k}\leq T}\, \sqrt{V_k}\, \XXX_k(\kappa t).
\end{align*}
Here, we denote by $V_k$, $k\in\N$, and $\{\XXX_k(\kappa t)\colon t\in\C\}$, $k\in\N$, independent copies of the random variable $\zeta_P(2 \tilde \sigma_*^{\vartriangle})$ and the random  analytic function $\{\XXX(\kappa t) \colon t\in\C\}$.

\vspace*{2mm}
\noindent
\textsc{Step 1.} We prove that $\bS_{n,T} \todistr \bS_{\infty,T}$.
This follows from Lemma~\ref{lem:adjoin_level} and~\eqref{eq:tech223_fluct_level}. Recall, in particular,  that $\beta_n(t)$ converges to $\beta_*$.

\vspace*{2mm}
\noindent
\textsc{Step 2.} We prove that $\bS_{\infty, T} \todistrT \bS_{\infty}$. Let $\calA_{P, V}$ be the $\sigma$-algebra generated by $\{P_k\colon k\in \N\}$ and $\{V_k\colon k\in\N\}$. Conditioning on $\calA_{P, V}$ and treating $P_k,V_k$, $k\in\N$,  as constants we have
\begin{align*}
\{S_{\infty, T}(t):t\in \C\} |  \calA_{P, V}
&\eqdistr
\left.\left\{\sum_{k=1}^{\infty} P_k^{-\frac{\beta_*}{\sigma_1}} \ind_{P_{k}\leq T}\, \sqrt{V_k}\, \XXX_k(\kappa t)\colon t\in\C\right\}\right|  \calA_{P, V}\\
&\eqdistr
\left.\left\{ \sqrt{\zeta_P(2 \sigma_*^{\vartriangle}; T)}\,  \XXX(\kappa t)\colon t\in\C\right\}\right|  \calA_{P, V}
\end{align*}
where $\zeta_P(2 \sigma_*^{\vartriangle}; T) = \sum_{k=1}^{\infty} P_k^{-\frac{2\sigma_*}{\sigma_1}} \ind_{P_{k}\leq T}\, V_k$. Integrating over $P_k, V_k$, $k\in\N$, we obtain that
\begin{align*}
\{S_{\infty, T}(t):t\in \C\}
\eqdistr
\left\{\sqrt{\zeta_P(2 \sigma_*^{\vartriangle}; T)}\, \XXX (\kappa t)\colon t\in\C\right\}.
\end{align*}
By Theorem~\ref{theo:zeta_abs_conv}, the random variable $\zeta_P(2 \sigma_*^{\vartriangle}; T)$ converges a.s.\ to  $\zeta_P(2 \sigma_*^{\vartriangle})$, as $T\to\infty$. This yields the statement of Step~2 and verifies the third condition of Lemma~\ref{lem:interchange_limits}.

\vspace*{2mm}
\noindent
\textsc{Step 3.} Fix $t\in \C$. The second condition of Lemma~\ref{lem:interchange_limits} is satisfied since by Lemma~\ref{lem:S_n_T_minus_S_n_GE_fluct_level} it holds that
$$
\lim_{T\to\infty}  \limsup_{n\to\infty} \E |S_{n}(t) - S_{n,T}(t)|^p=0.
$$

\vspace*{2mm} Applying Lemma~\ref{lem:interchange_limits}, we obtain
that~\eqref{eq:func_clt1_restate_fluct_level} holds in the sense of
finite-dimensional distributions. To complete the proof
of~\eqref{eq:func_clt1_restate_fluct_level}, we need to show that the sequence
of random functions $S_n(t)$ is tight on $\HHH(\C)$. The tightness follows from
Proposition~\ref{prop:moment_S_n_fluct_level} and
Proposition~\ref{prop:tightness_random_analytic}.

\subsection{Proofs of Theorems~\ref{theo:functional_clt_line of_zeros_d2_geq2},
\ref{theo:functional_clt_line of_zeros_d2_eq1},
\ref{theo:clt_boundary_spin_glass},
\ref{theo:clt_triple_point}}\label{sec:func_CLT_F_analogous_proofs} These proofs
follow the method of adjoining the glassy phase levels developed in
Sections~\ref{sec:moments_GE}, \ref{sec:func_CLT_GE},
\ref{sec:func_CLT_GE_boundary}, \ref{sec:adjoin_spin_glass_to_F} and do not
require any new ideas.  For this reason, we will just give the idea of the
proofs.

\begin{proof}[Idea of proof of Theorems~\ref{theo:functional_clt_line of_zeros_d2_geq2} and~\ref{theo:functional_clt_line of_zeros_d2_eq1}]
The normalizing sequence $f_n(\beta_*; t)$ in these theorems is given by
\begin{multline}\label{eq:f_n_beta_star_t}
f_n(\beta_*; t)
= \left(\beta_*+\frac tn \right)\sum_{j=1}^{d_1} \sqrt{na_j} u_{n,j} +\\
+\sum_{j=d_1+1}^{d_1+d_2} \left(\frac 12 \log N_{n,j} + a_j \sigma^2_* n \right) + \sum_{j=d_1+d_2+1}^d \left( \log N_{n,j} + \frac 12 a_j \beta^2_* n\right).
\end{multline}
The proof is by induction over $d_1$, the number of glassy phase levels. In the case $d_1=0$, Theorems~\ref{theo:functional_clt_line of_zeros_d2_geq2} and~\ref{theo:functional_clt_line of_zeros_d2_eq1} were already established in Theorems~\ref{theo:functional_clt_EF_boundary} and~\ref{theo:functional_clt_EF_boundary_k1}. Note that in the case $d_1=0$ the first sum in the definition of $f_n(\beta_*; t)$ vanishes, whereas the remaining two sums are equal to the normalizing constants used in Theorems~\ref{theo:functional_clt_EF_boundary} and~\ref{theo:functional_clt_EF_boundary_k1}, up to a factor of the form $\eee^{i c_n + o(1)}$, where $c_n$ is real constant. The phase factor $\eee^{ic_n}$ can be ignored since the limiting process is isotropic.  So, Theorems~\ref{theo:functional_clt_EF_boundary} and~\ref{theo:functional_clt_EF_boundary_k1} state that in the case $d_1=0$ we have that weakly on $\HHH(\C)$,
\begin{equation}\label{eq:tech222}
\left\{\eee^{-f_n(\beta_*;t)} \ZZZ_n\left(\beta_*+ \frac tn\right) \colon t\in \C\right\} \toweak
\{\eee^{\lambda't} N' + \eee^{\lambda'' t} N''\colon t\in \C\},
\end{equation}
where in the case $d_2=0$ (Theorem~\ref{theo:functional_clt_EF_boundary_k1}) we have to replace $N''$ by $1$.
The proof of Theorems~\ref{theo:functional_clt_line of_zeros_d2_geq2} and~\ref{theo:functional_clt_line of_zeros_d2_eq1} in the case $d_1\geq 1$ proceeds by adjoining $d_1$ glassy phase levels to~\eqref{eq:tech222} one by one as in Sections~\ref{sec:func_CLT_GE}, \ref{sec:func_CLT_GE_boundary}, \ref{sec:adjoin_spin_glass_to_F}.
\end{proof}

\begin{proof}[Idea of proof of Theorem~\ref{theo:clt_boundary_spin_glass}]
Let $\beta\in \C$ be such that $\sigma=\frac{\sigma_l}{2}$ and $\tau>\frac{\sigma_l}{2}$ for some $1\leq l\leq d$.  The normalizing sequence $r_n(\beta)$ from Theorem~\ref{theo:clt_boundary_spin_glass} is given as follows. If $\frac{\sigma_k}{\sqrt 2} < |\beta| < \frac{\sigma_{k+1}}{\sqrt 2}$ for some $l<k\leq d$, then
\begin{multline}\label{eq:m_n_beta_clt_bound_with_spin_glass}
r_n(\beta) =  \beta \sum_{j=1}^{l-1} \sqrt{na_j} u_{n,j} +\\
 +\sum_{j=l}^{k} \left(\frac 12 \log N_{n,j} + a_j \sigma^2 n \right) + \sum_{j=k+1}^d \left( \log N_{n,j} + \frac 12 \beta^2 n a_j\right).
\end{multline}
The proof of Theorem~\ref{theo:clt_boundary_spin_glass} is by induction over $l$. For $l=1$ (no glassy phase levels), we proved in Theorem~\ref{theo:clt_boundary} that
\begin{equation}\label{eq:tech225}
\frac{\ZZZ_n(\beta)-\E \ZZZ_n(\beta)}{\sqrt{\Var \ZZZ_n(\beta)}} \todistr \frac N {\sqrt 2},
\end{equation}
where $N\sim N_{\C}(0,1)$. By Proposition~\ref{prop:exp_vs_stand_dev} (note that
$|\beta|>\frac{\sigma_1}{\sqrt 2}$ by our assumptions), we can drop $\E
\ZZZ_n(\beta)$ in~\eqref{eq:tech225}. The asymptotic expression for $\Var
\ZZZ_n(\beta)$ given in Proposition~\ref{prop:asympt_exp_variance_log_scale} has
the form $\eee^{2r_n(\beta) + i c_n + o(1)}$, where $c_n\in\R$. Using the
rotational invariance of the complex normal distribution, we obtain that for
$l=1$,
$$
\frac{\ZZZ_n(\beta)}{\eee^{r_n(\beta)}} \todistr \frac N {\sqrt 2}.
$$
This verifies the basis of induction.  The rest of the proof consists of
adjoining the glassy phase levels by the same method as developed in
Section~\ref{sec:adjoin_spin_glass_to_F}.

Note finally that if $|\beta| = \frac{\sigma_k}{\sqrt 2}$ for some $l<k\leq d$, then we have to add  $\frac 12 \log 2$ to the expression for  $r_n(\beta)$. This is related to the additional factor of $2$ in the asymptotic expression for $\Var \ZZZ_n(\beta)$ in Proposition~\ref{prop:asympt_exp_variance_log_scale}.
\end{proof}

\begin{proof}[Idea of proof of Theorem~\ref{theo:clt_triple_point}]
Let $\beta\in \C$ be such that $\sigma=\tau=\frac{\sigma_l}{2}$ for some $1\leq l\leq d$.
The normalizing constant $r_n(\beta)$ in Theorem~\ref{theo:clt_triple_point} has the same form as in~\eqref{eq:m_n_beta_clt_bound_with_spin_glass}, with $k=l$. The proof is by  induction over $l$. For $l=1$ (no glassy phase levels), we proved in Theorem~\ref{theo:clt_boundary} that~\eqref{eq:tech225} holds. However, this time we cannot drop $\E \ZZZ_n(\beta)$ since by Propositions~\ref{prop:asympt_expect} and~\ref{prop:asympt_exp_variance_log_scale}  we have
$$
\lim_{n\to\infty} \eee^{-i\sigma \tau a n} \frac{\E \ZZZ_n(\beta)}{\sqrt{\Var \ZZZ_n(\beta)}}
=
\lim_{n\to\infty}  \frac{\E \ZZZ_n(\beta)}{\eee^{r_n(\beta)}}
=
\frac{1}{\sqrt 2}.
$$
It follows that we can write~\eqref{eq:tech225} as follows:
$$
\frac{\ZZZ_n(\beta)}{\eee^{r_n(\beta)}} \todistr \frac{N + 1}{\sqrt 2},
$$
where $N\sim N_{\C}(0,1)$. This verifies the basis of induction.  The rest of the proof consists of adjoining the glassy phase levels by the same method as developed in Section~\ref{sec:adjoin_spin_glass_to_F} and Section~\ref{sec:func_CLT_GE_boundary}.
\end{proof}

\section{Limiting log-partition function and global distribution of zeros}\label{sec:free_energy_global_zeros_proofs}
\subsection{Limiting log-partition function: Proof of Theorem~\ref{theo:free_energy}}\label{subsec:proof_theo_free_energy}
The idea is that we have already proved a distributional limit theorem for $\ZZZ_n(\beta)$, for every $\beta\in \C$.
From that, we can deduce that $F_n(\beta):=\frac 1n \log |\ZZZ_n(\beta)|$ converges in probability. The $L^q$-convergence will be established later, in Proposition~\ref{prop:free_energy_Lp}.
The next lemma is taken from~\cite{kabluchko_klimovsky}; see Lemma~3.9 there.
\begin{lemma}\label{lem:proof_log_scale}
Let $Z,Z_1,Z_2,\ldots$ be  random variables with values in $\C$ and let $m_n\in\C$, $v_n\in \C\bsl \{0\}$ be sequences of normalizing  constants  such that
\begin{equation}\label{eq:proof_log_scale1}
\frac{Z_n-m_n}{v_n}\todistr Z.
\end{equation}
The following two statements hold:
\begin{enumerate}
\item[\textup{(1)}] If $|v_n|=o(|m_n|)$ and $|m_n|\to\infty$ as $n\to\infty$, then $\frac{\log |Z_n|}{\log |m_n|}\toprobab 1$.
\item[\textup{(2)}] If $|m_n|=O(|v_n|)$ and $|v_n|\to \infty$ as $n\to\infty$ and $Z$ has no atoms, then $\frac{\log |Z_n|}{\log |v_n|}\toprobab 1$.
\end{enumerate}
\end{lemma}

Note that we can view $m_n$ as the asymptotic ``location'' and $v_n$ as the asymptotic ``fluctuations'' of $Z_n$.
There are two parts in the lemma depending on what parameter, location (Part~1), or fluctuations (Part~2), dominates.

\begin{proof}[Proof of Theorem~\ref{theo:free_energy}]
Recall that $p_k(\beta)$, where $1\leq k\leq d$, is given by~\eqref{eq:def_mk}. Let $p(\beta)=p_1(\beta)+\ldots+p_d(\beta)$. It is easy to check that $p(\beta)>0$ for all $\beta\in \C$. Our aim is to prove that
\begin{equation}\label{eq:tech442}
\plim_{n\to\infty} \frac 1n \log |\ZZZ_n(\beta)| = p(\beta).
\end{equation}
We  are going to verify the conditions of Lemma~\ref{lem:proof_log_scale} for $Z_n=\ZZZ_n(\beta)$ and suitable $m_n$, $v_n$, $Z$.  Theorem~\ref{theo:free_energy} is known for $\beta\in\R$, see~\cite{capocaccia_etal,Derrida_Gardner1,bovier_kurkova1,bovier_kurkova_review}, so in the sequel we always assume that $\beta\in \C\bsl \R$. By symmetry, see~\eqref{eq:symmetry1}, \eqref{eq:symmetry2}, we may assume that $\sigma\geq 0$ and $\tau>0$. We consider three cases.

\vspace*{2mm}
\noindent
\textsc{Case 1:} \textit{Location dominates}. Let $\beta \in E^d=E_1$.
By Corollary~\ref{cor:moment_S_n_1}, we can find a sufficiently small $\eps=\eps(\beta)>0$ such that
$$
\frac{\ZZZ_n(\beta)- \E \ZZZ_n(\beta) }{ \eee^{-\eps n} \E \ZZZ_n(\beta)} \todistr 0.
$$
Hence, we can apply Part~1 of Lemma~\ref{lem:proof_log_scale} with $m_n=\E \ZZZ_n(\beta)$ to obtain that
\begin{equation}\label{eq:tech888}
\plim_{n\to\infty} \frac 1n \log |\ZZZ_n(\beta)|
=
\lim_{n\to\infty} \frac 1n \log |\E \ZZZ_n(\beta)|
=
\log \alpha + \frac 1 2 (\sigma^2 - \tau^2) a
=
p(\beta),
\end{equation}
where we used Proposition~\ref{prop:asympt_expect} and~\eqref{eq:def_mk}.

\vspace*{2mm}
\noindent
\textsc{Case 2:} \textit{Fluctuations dominate}. Let $\beta \in G^{d_1} F^{d_2} E^{d_3}$, where $d_3\neq d$. By Theorem~\ref{theo:fluct}, we can apply Part~2 of Lemma~\ref{lem:proof_log_scale} with $m_n=0$ and  $v_n=\eee^{c_n(\beta)}$, where $c_n(\beta)$ is given by~\eqref{eq:def_c_nk}, \eqref{eq:def_c_n_beta}. The fact that the limiting variable $Z$ (which may be of three different types, see Proposition~\ref{theo:fluct}) has no atoms has been verified in the case $d_1>0$, $d_2=0$ in Proposition~\ref{prop:zeta_P_no_atoms} and is trivial in the remaining two cases. It follows that
$$
\plim_{n\to\infty} \frac 1n \log |\ZZZ_n(\beta)|
=
\lim_{n\to\infty} \frac 1n  \Re c_n(\beta)
=
p(\beta),
$$
where the last step follows by comparing~\eqref{eq:def_c_nk} with~\eqref{eq:def_mk}.

\vspace*{2mm}
\noindent
\textsc{Case 3:} \textit{The boundary case}. Here we assume that $\beta$ is located on the boundary of some phase $G^{d_1}F^{d_2}E^{d_3}$. There are $4$ subcases.

\vspace*{2mm}
\noindent
\textsc{Case 3A:} \textit{Beak shaped boundaries}.
Assume that $\beta\in \C$ is such that $\sigma+\tau=\sigma_l$, $\sigma>\frac {\sigma_l}{2}$, $\tau>0$, for some $1\leq l \leq d$.

If $2\leq l\leq d$, then  by Theorem~\ref{theo:fluct_beak_exact} we can apply Part~2 of Lemma~\ref{lem:proof_log_scale} with $m_n=0$ and  $v_n=\eee^{\hat h_{n,l}(\beta)}$, where $\hat h_{n,l}(\beta)$ is given by~\eqref{eq:hat_h_n_l_beta}. Note that the limiting variable $Z$ is given by a Poisson cascade zeta function with $l-1$ variables and has no atoms by Proposition~\ref{prop:zeta_P_no_atoms} (at this point we use that $l\neq 1$).
It follows that
$$
\plim_{n\to\infty} \frac 1n \log |\ZZZ_n(\beta)|
=
\lim_{n\to\infty} \frac 1n  \Re \hat h_{n,l}(\beta)
=
p(\beta),
$$
where the last step follows by comparing~\eqref{eq:hat_h_n_l_beta} and~\eqref{eq:def_mk}.

If $l=1$, then the above argument breaks down since the limiting variable $Z=1$ has atoms.  However, using~\eqref{eq:fluct_beak_exact_l_1_tech222} and~\eqref{eq:tech444} we see that we can apply Part~1 of Lemma~\ref{lem:proof_log_scale} with $m_n=\E \ZZZ_n(\beta)$. This yields~\eqref{eq:tech442} by the same computation as in~\eqref{eq:tech888}.

\vspace*{2mm}
\noindent
\textsc{Case 3B:} \textit{Arc shaped boundaries}.
Assume that $\beta\in \C$ is such that  for some $d_1,d_2,d_3\in \{0,\ldots, d\}$ with $d_1+d_2+d_3=d$ we have
$$
\frac{\sigma_{d_1}}{2} < \sigma < \frac{\sigma_{d_1+1}}{2}, \;\; \tau>0,\;\; \sigma^2 + \tau^2 = \frac{\sigma_{d_1+d_2}^2}{2}.
$$
Due to Theorems~\ref{theo:functional_clt_line of_zeros_d2_geq2}, \ref{theo:functional_clt_line of_zeros_d2_eq1}, we can apply Part~2 of Lemma~\ref{lem:proof_log_scale} with $m_n=0$ and  $v_n=\eee^{f_n(\beta; 0)}$. The limiting random variable $Z$ has the form
$$
Z = \begin{cases}
\sqrt{W} N + \zeta^{(d_1)}, &\text{if } d_2=1,\\
\sqrt{2W} N, &\text{if } 2\leq d_2\leq d;
\end{cases}
$$
see Theorems~\ref{theo:functional_clt_line of_zeros_d2_geq2}, \ref{theo:functional_clt_line of_zeros_d2_eq1}. Note that the random variable $W=\zeta_P(2 T^{d_1}(\sigma))$ has no atom at $0$ by Proposition~\ref{prop:zeta_P_no_atoms}. By Lemma~\ref{lem:no_atoms_linear_function}, the random variable $Z$ has no atoms. By Part~2 of  Lemma~\ref{lem:proof_log_scale},
$$
\plim_{n\to\infty} \frac 1n \log |\ZZZ_n(\beta)|
=
\lim_{n\to\infty} \frac 1n  \Re f_n(\beta;0)
=
p(\beta),
$$
where we used~\eqref{eq:f_n_beta_star_t} and~\eqref{eq:def_mk}.

\vspace*{2mm}
\noindent
\textsc{Case 3C:} \textit{Vertical boundaries}.
Assume that $\sigma = \frac{\sigma_l}{2}$ and $\tau>\frac{\sigma_l}{2}$ for some $1\leq l\leq d$. By Theorem~\ref{theo:clt_boundary_spin_glass}, we can apply Part~2 of Lemma~\ref{lem:proof_log_scale} with $m_n=0$, $v_n=\eee^{r_n(\beta)}$. The limiting random variable $Z$ has no atoms by Lemma~\ref{lem:no_atoms_linear_function} and Proposition~\ref{prop:zeta_P_no_atoms}. By Part~2 of Lemma~\ref{lem:proof_log_scale},
$$
\plim_{n\to\infty} \frac 1n \log |\ZZZ_n(\beta)|
=
\lim_{n\to\infty} \frac 1n  \Re r_n(\beta)
=
p(\beta),
$$
where we used~\eqref{eq:m_n_beta_clt_bound_with_spin_glass} and~\eqref{eq:def_mk}.

\vspace*{2mm}
\noindent
\textsc{Case 3D:} \textit{Triple points}.
Assume that $\sigma = \tau = \frac{\sigma_l}2$, for some $1\leq l \leq d$. By Theorem~\ref{theo:clt_triple_point}, we can apply Part~2 of Lemma~\ref{lem:proof_log_scale} with $m_n=0$ and $v_n=\eee^{r_n(\beta)}$. The limiting random variable $Z$ has no atoms by the same argument as in Case~3B.  As in Case~3C it follows that~\eqref{eq:tech442} holds.
\end{proof}

\subsection{Estimates for the concentration function and $L^q$-convergence}
We will need to bound the probability of the event $|\ZZZ_n(\beta)|\leq r$, where $r\geq 0$ is small. For this purpose, the notion of the concentration function is useful; see, e.g.,\ \cite[\S1.5]{petrov_book}.  Denote by $B_r(t)=\{z\in \C\colon |z-t|\leq r\}$ the disk of radius $r\geq 0$ centered at $t\in \C$. Given a random variable $X$ with values in $\C$ define its concentration function by
\begin{equation}\label{eq:def_Q_concentration_func}
Q(X; r) = \sup_{t\in\C} \P[X\in B_r(t)], \;\;\; r\geq 0.
\end{equation}
 The next fact follows immediately from the convolution formula and can be found in~\cite[\S1.5, Lemma~1.11]{petrov_book}: If $Y_1,\ldots,Y_m$ are independent random variables with values in $\C$, then
\begin{equation}\label{eq:Q_ineq}
Q(Y_1+ \ldots + Y_m;r ) \leq \min_{i=1,\ldots,m} Q(Y_i; r).
\end{equation}
\begin{lemma}\label{lem:Q_XY}
Let $X$ and $Y$ be independent random values with values in $\C$. Then, for every $r\geq 0$,
$$
Q(XY; r) \leq Q(X;\sqrt r) + \P[|Y|\leq \sqrt r].
$$
\end{lemma}
\begin{proof}
Let $t\in\C$. Then,
$$
\P[XY\in B_r(t)] \leq \P[XY\in B_r(t), |Y|>\sqrt r] + \P[|Y|\leq \sqrt r].
$$
Let $\mu_Y$ be the distribution of the random variable $Y$. Conditioning on $Y=w$, where $w\in\C$, and using the formula for the total probability, we get
$$
\P[XY\in B_r(t), |Y|>\sqrt r] = \int_{\C\bsl B_{\sqrt r}(0)} \P[wX\in B_r(t)] \, \mu_Y(\dd w)
\leq Q(X; \sqrt r),
$$
where the last inequality holds since we have $\P[wX\in B_r(t)] \leq Q(X; \sqrt r)$ as long as $|w| > \sqrt r$.
The required inequality follows.
\end{proof}

\begin{lemma}\label{lem:Q_e_beta_xi}
Let $K\subset \C\bsl\{0\}$ be a compact set and let $\eps>0$. Let $\xi$ be a real standard normal random variable.  Then, there exist constants $C>0$, $N\in\N$, $\delta>0$ (which depend on $K$ and $\eps$) such that for every $\beta\in K$, $n>N$, $r\in (0,\eee^{-\eps n})$ we have
$$
Q(e^{\beta \sqrt n \xi}; r) < Cr^{\delta}.
$$
\end{lemma}
\begin{proof}
See Eq.~(3.35) in~\cite{kabluchko_klimovsky}.
\end{proof}

\begin{lemma}\label{lem:Q_ZZZ_n_beta}
Let $K\subset \C\bsl\{0\}$ be a compact set and let $\eps>0$. Then, there exist constants $C>0$, $N\in\N$, $\delta>0$ (which depend on $K$ and $\eps$) such that, for every $\beta\in K$, $n>N$, $r\in (0,\eee^{-\eps n})$, we have
$$
Q(\ZZZ_n(\beta) ; r) < Cr^{\delta}.
$$
\end{lemma}
\begin{proof}
We will prove this by induction over $d$, the number of GREM levels. If $d>1$, then we assume that the statement of the lemma is true for the GREM with $d-1$ levels. For $d=1$, we don't need any assumption. For the partition function of the GREM with $d$ levels, we have a representation
$$
\ZZZ_n(\beta)= \sum_{k=1}^{N_{n,1}} \eee^{\beta \sqrt {na_1}\xi_k} \ZZZ_{n,k}^*(\beta).
$$
Here, for every $k=1,\ldots,N_{n,1}$,  $\ZZZ_{n,k}^*(\beta)$ is an analogue of $\ZZZ_n(\beta)$ with $d-1$ levels instead of $d$ levels.  Assume first that $d>1$. By Lemma~\ref{lem:Q_XY},
$$
Q(\eee^{\beta \sqrt {na_1}\xi_1}\ZZZ_{n,1}^*(\beta); r) \leq  Q(\eee^{\beta \sqrt {na_1}\xi_1}; \sqrt r) + \P[|\ZZZ_{n,1}^*(\beta)|\leq \sqrt r].
$$
The first term is bounded by $C_1r^{\delta_1}$ by Lemma~\ref{lem:Q_e_beta_xi} (in which we take $\eps/2$ instead of $\eps$).  The second term is bounded by $C_2 r^{\delta_2}$ by the induction assumption. Hence, for every sufficiently large $n$  and all $\beta\in K$, $r\in (0,\eee^{-\eps n})$, we have
$$
Q(\eee^{\beta \sqrt {na_1}\xi_1}\ZZZ_{n,1}^*(\beta); r)\leq C_3 r^{\delta_3}.
$$
This inequality is true in the case $d=1$, too, because in this case $\ZZZ_n^*(\beta)=1$ and we can directly use Lemma~\ref{lem:Q_e_beta_xi}.  For independent random variables $Y_1,\ldots,Y_m$, we have $Q(Y_1+\ldots+Y_m; r)\leq Q(Y_1;r)$; see~\eqref{eq:Q_ineq}.  Hence, we obtain that
$$
Q(\ZZZ_{n}(\beta); r)\leq Q(\eee^{\beta \sqrt {na_1}\xi_1}\ZZZ_{n,1}^*(\beta); r)\leq C_3 r^{\delta_3}.
$$
This completes the induction.
\end{proof}

Let $F_n(\beta)=\frac 1n \log |\ZZZ_n(\beta)|$ and recall that $p(\beta)=\sum_{k=1}^d p_k(\beta)$, where $p_k(\beta)$ has been defined in~\eqref{eq:def_mk}.
\begin{lemma}\label{lem:p_n_beta_bounded_Lr}
Let $K\subset \C\bsl\{0\}$ be a compact set and let $r>0$.  Then, we can find $C>0$ and $N\in\N$  depending on $K$ and $r$ such that for all $n>N$,
$$
\sup_{\beta\in K} \E |F_n(\beta)|^r <C.
$$
\end{lemma}
\begin{proof}
For $u>0$ and $\beta\in K$ we have
$$
\P[F_n(\beta)>u] = \P[|\ZZZ_n(\beta)|>\eee^{n u}]\leq \eee^{-nu} \E |\ZZZ_n(\beta)| \leq  \eee^{-nu} N_n \E \eee^{\sigma \sqrt {na} \xi}
\leq \eee^{(C-u)n},
$$
where $C=C(K)$. Consequently, for all $\beta\in K$,  $u>C$,  $n\in\N$ we have
$$
\P[F_n(\beta)>u]\leq \eee^{C-u}.
$$
To complete the proof, we need to estimate the lower tail of $F_n(\beta)$. By Lemma~\ref{lem:Q_ZZZ_n_beta}, we can find $C>0$, $\delta>0$, $N\in\N$ such that for all $\beta\in K$, $u>1$ and $n>N$,
$$
\P[F_n(\beta)< - u] = \P[|\ZZZ_n(\beta)|<\eee^{-u n}] < Q(\ZZZ_n(\beta); \eee^{-u n}) <  C\eee^{-\delta u n} < C\eee^{-\delta u}.
$$
The last two displays imply the claim.
\end{proof}

We have already shown in Section~\ref{subsec:proof_theo_free_energy} that for every $\beta\in\C$,  $F_n(\beta)$ converges to $p(\beta)$ in probability. Now we are able to prove the $L^q$-convergence.
\begin{proposition}\label{prop:free_energy_Lp}
Fix $q\geq 1$. For every $\beta\in \C$,  $F_n(\beta)$ converges to $p(\beta)$ in $L^q$.
\end{proposition}
\begin{proof}
The statement is trivial for $\beta=0$, so fix some $\beta\neq 0$. We need to show that the random variables $|F_n(\beta)|^{q}$, $n\in\N$, are uniformly integrable; see~\cite[Proposition~3.12]{kallenberg_book}.
For this, it suffices to show that the sequence $F_n(\beta)$ is bounded in $L^r$, for some $r>q$; see~\cite[p.~44]{kallenberg_book}. This fact has already been established in Lemma~\ref{lem:p_n_beta_bounded_Lr}.
\end{proof}

\subsection{Global distribution of zeros: Proof of Theorem~\ref{theo:zeros_global}}\label{sec:proof_theo:zeros_global}
The proof is analogous to the proof of Theorem~2.1 in~\cite{kabluchko_klimovsky}.

\vspace*{2mm}
\noindent
\textsc{Step 1.} We need to show that for every infinitely differentiable, compactly supported function $f\colon \C\to\R$,
\begin{equation}\label{eq:zeros_global_proof}
\frac 1n \sum_{\beta\in \C \colon \ZZZ_n(\beta)=0} f(\beta) \toprobab \frac 1 {2\pi} \int_{\C} f(\beta) \dd\Xi(\beta).
\end{equation}
This is equivalent to the statement of Theorem~\ref{theo:zeros_global}
by~\cite[Theorem~14.16]{kallenberg_book}. When
proving~\eqref{eq:zeros_global_proof} we can assume that $f$ vanishes in some
neighborhood of the origin. To see this, note that we can write $f=f_1+f_2$,
where $f_1, f_2\colon \C\to\R$ are compactly supported and infinitely differentiable,
$f_1$ vanishes in the disk $\{|z|\leq \frac{\sigma_1} 4\}$, while  $f_2$
vanishes outside the disk $\{|z|\leq \frac{\sigma_1} 3\}$. Since $f_2$ does not
contribute neither to the left-hand side of~\eqref{eq:zeros_global_proof} (with
probability approaching $1$, by Theorem~\ref{theo:no_zeros_E1}), nor to the
right-hand side of~\eqref{eq:zeros_global_proof} (by the definition of $\Xi$),
we can and will assume that $f = f_1$.

\vspace*{2mm}
\noindent
\textsc{Step 2.}
Let $\lambda$ be the Lebesgue measure on $\C$.  By the Poincar\'e--Lelong formula, see~\cite[\S2.4.1]{peres_etal_book},
\begin{equation}\label{eq:laplace_log_zeros}
\frac 1n \sum_{\beta\in\C \colon \ZZZ_n(\beta)=0} f(\beta)
=
\frac 1 {2\pi n}
\int_{\C} \log |\ZZZ_n(\beta)| \Delta f(\beta) \lambda(\dd \beta).
\end{equation}
Recall that by Theorem~\ref{theo:free_energy} the random variable $F_n(\beta) := \frac 1n \log |\ZZZ_n(\beta)|$ converges to $p(\beta)=\sum_{k=1}^d p_k(\beta)$ in $L^1$ for every $\beta\in\C$.
From~\eqref{eq:distr_laplace_p} (which will be established in Step~3 below), we conclude that Theorem~\ref{theo:zeros_global} is equivalent to
$$
\int_{\C}  F_n(\beta) \Delta f(\beta) \lambda(\dd \beta) \toprobab \int_{\C} p(\beta)  \Delta f(\beta)\lambda(\dd \beta).
$$
We will show that this holds even in $L^1$.  By  Fubini's theorem, it suffices to show that
\begin{equation}\label{eq:conv_zeros_L1}
\lim_{n\to\infty}   \int_{\C}  \E |F_n(\beta)-p(\beta)| |\Delta f(\beta)| \lambda(\dd \beta) = 0.
\end{equation}
We know from Theorem~\ref{theo:free_energy} that $\lim_{n\to\infty} \E
|F_n(\beta)-p(\beta)|=0$, for every $\beta\in \C$. To complete the proof, we
need to interchange the limit and the integral. Recalling that $f$ vanishes on a
neighborhood of $0$ and applying Lemma~\ref{lem:p_n_beta_bounded_Lr} we obtain
that there is $C>0$ such that for all $\beta\in \supp f$ and all sufficiently
large $n\in \N$,
$$
\E |F_n(\beta)-p(\beta)| |\Delta f(\beta)| < C.
$$
This justifies the use of the dominated convergence  theorem and completes the proof of~\eqref{eq:conv_zeros_L1}.

\vspace*{2mm}
\noindent
\textsc{Step 3.}
In this step, we show that $\Delta p=\Xi$ in the sense of generalized functions.
This means that for every compactly supported infinitely differentiable function $f\colon \C\to\R$,
\begin{equation}\label{eq:distr_laplace_p}
\int_{\C} p(\beta) \Delta f(\beta) \lambda(\dd\beta) = \int_{\C} f(\beta) \Xi(\dd\beta).
\end{equation}
Here, $p(\beta) = \sum_{k=1}^d p_k(\beta)$, where $p_k(\beta)$ is given
by~\eqref{eq:def_mk}, and $\Xi=\sum_{k=1}^d \Xi_k$, where $\Xi_k =
\Xi_k^{F}+\Xi_{k}^{EF}+\Xi_{k}^{EG}$ and the three terms were described in
Section~\ref{subsec:zeros_global}.

It suffices to show that $\Delta p_k=\Xi_k$ for every $1\leq k\leq d$. This means that
\begin{equation}\label{eq:distr_laplace_p_k}
\int_{\C} p_k(\beta) \Delta f(\beta) \lambda(\dd\beta) = \int_{\C} f(\beta) \Xi_k(\dd\beta).
\end{equation}
This computation has been performed by~\citet{Derrida_zeros} (who has $a_k =  \frac 12$, $\alpha_k=\log 2$), but for completeness we provide the details. Green's second identity applied to the domains $B= E_k, F_k, G_k$ gives
\begin{multline*}
\int_{B} p_k(\beta) \Delta f(\beta) \lambda(\dd \beta)
\\
=
\int_{B} \Delta p_k(\beta) f(\beta) \lambda(\dd \beta)
+
\oint_{\partial B} \left(f(\beta)\frac{\partial p_k(\beta)}{\partial \nn} - p_k(\beta)\frac{\partial f(\beta)}{\partial \nn}\right) |\dd \beta|.
\end{multline*}
Here, $\nn$ denotes the unit inward pointing normal to the boundary of $B$
and $\frac {\partial}{\partial \nn}$ is the corresponding directional
derivative. Adding these three identities, noting that the pointwise Laplacian of $p_k$ is given by
$$
\Delta p_k (\beta) =
\begin{cases}
2a_k, & \text{ if } \beta\in F_k,\\
0, &\text{ if } \beta \in E_k\cup G_k,
\end{cases}
$$
and that the terms involving $\frac{\partial f(\beta)}{\partial \nn}$ cancel (by the continuity of $p_k$),  we obtain
$$
\int_{\C} p_k(\beta) \Delta f(\beta) \lambda(\dd \beta)
=
2a_k \int_{F_k} f(\beta) \lambda(\dd \beta)
+
\oint_{\gamma} f(\beta) \left(\frac{\partial}{\partial \nn_+} + \frac{\partial}{\partial \nn_-}\right) p_k(\beta)  |\dd \beta|.
$$
Here, $\gamma$ is the union of lines and arcs which constitute the boundaries of $E_k, F_k, G_k$ and $\nn_+$ and $\nn_-$ denote the unit normals to $\gamma$ (with directions opposite to each other).
To complete the proof, we need to compute the jump of the normal derivative of $p_k$:
\begin{equation}\label{eq:jump_normal_der_pk}
\left(\frac{\partial}{\partial \nn_+} + \frac{\partial}{\partial \nn_-}\right) p_k(\beta).
\end{equation}
There are three cases.

\vspace*{2mm}
\noindent
\textsc{Case EF.}
On the boundary of $E_k$ and $F_k$ (two circular arcs), \eqref{eq:jump_normal_der_pk} equals
\begin{align*}
\frac{\sqrt 2}{\sigma_k} \left(\sigma \frac{\partial}{\partial \sigma} + \tau \frac{\partial}{\partial \tau}\right)
\left(\frac 12 \log \alpha_k + a_k\sigma^2 - \log \alpha_k - \frac 12 a_k (\sigma^2-\tau^2)\right)
= \sqrt {a_k \log \alpha_k}.
\end{align*}

\vspace*{2mm}
\noindent
\textsc{Case EG.} On the boundary of $E_k$ and $G_k$ (four line segments), \eqref{eq:jump_normal_der_pk} equals
\begin{align*}
\frac{1}{\sqrt 2} \left(\sgn \sigma \frac{\partial}{\partial \sigma} + \sgn \tau \frac{\partial}{\partial \tau}\right)
\left(|\sigma| \sqrt{2a_k\log \alpha_k} - \log \alpha_k - \frac 12 a_k (\sigma^2-\tau^2)\right)
= \sqrt 2 a_k |\tau|.
\end{align*}

\vspace*{2mm}
\noindent
\textsc{Case FG.}
On the boundary of $F_k$ and $G_k$ (four half-lines), \eqref{eq:jump_normal_der_pk} equals
\begin{align*}
\left(\sgn \sigma \frac{\partial}{\partial \sigma}\right)\left(|\sigma| \sqrt{2a_k\log \alpha_k} -\frac 12  \log \alpha_k - a_k \sigma^2\right)
= 0.
\end{align*}
Combining everything together, we obtain that $\int_{\C} p_k(\beta) \Delta f(\beta) \lambda(\dd \beta)$ is given by
\begin{align*}
2a_k \int_{F_k} f(\beta) \lambda(\dd \beta)
+
\sqrt {a_k \log \alpha_k} \oint_{\bar E_k\cap \bar F_k}  f(\beta) |\dd \beta|
+
\sqrt 2 a_k
\oint_{\bar E_k\cap \bar G_k} |\tau| f(\beta) |\dd \beta|
.
\end{align*}
This coincides with $\int_{\C} f(\beta) \Xi_k(\dd\beta)$ by definition of $\Xi_k$, see Section~\ref{subsec:zeros_global}. The proof of~\eqref{eq:distr_laplace_p_k} is complete.

\bibliographystyle{plainnat}
\bibliography{grem_bib}

\begin{thebibliography}{48}
\providecommand{\natexlab}[1]{#1}
\providecommand{\url}[1]{\texttt{#1}}
\expandafter\ifx\csname urlstyle\endcsname\relax
  \providecommand{\doi}[1]{doi: #1}\else
  \providecommand{\doi}{doi: \begingroup \urlstyle{rm}\Url}\fi

\bibitem[Abramowitz and Stegun(1964)]{abramowitz_stegun}
M.~Abramowitz and I.~Stegun.
\newblock \emph{Handbook of mathematical functions with formulas, graphs, and
  mathematical tables}, volume~55 of \emph{National Bureau of Standards Applied
  Mathematics Series}.
\newblock U.S.~Government Printing Office, Washington, 1964.

\bibitem[Billingsley(1999)]{billingsley_book}
P.~Billingsley.
\newblock \emph{{Convergence of probability measures}}.
\newblock {Chichester: Wiley}, 1999.

\bibitem[Biskup et~al.(2004{\natexlab{a}})Biskup, Borgs, Chayes, Kleinwaks, and
  Koteck{\'y}]{biskup_etal1}
M.~Biskup, C.~Borgs, J.~T. Chayes, L.~J. Kleinwaks, and R.~Koteck{\'y}.
\newblock Partition function zeros at first-order phase transitions: a general
  analysis.
\newblock \emph{Comm. Math. Phys.}, 251\penalty0 (1):\penalty0 79--131,
  2004{\natexlab{a}}.

\bibitem[Biskup et~al.(2004{\natexlab{b}})Biskup, Borgs, Chayes, and
  Koteck{\'y}]{biskup_etal2}
M.~Biskup, C.~Borgs, J.~T. Chayes, and R.~Koteck{\'y}.
\newblock Partition function zeros at first-order phase transitions:
  {P}irogov-{S}inai theory.
\newblock \emph{J. Statist. Phys.}, 116\penalty0 (1--4):\penalty0 97--155,
  2004{\natexlab{b}}.

\bibitem[Bovier(2006)]{bovier_book}
A.~Bovier.
\newblock \emph{Statistical mechanics of disordered systems. A mathematical
  perspective}.
\newblock Cambridge Series in Statistical and Probabilistic Mathematics.
  Cambridge University Press, Cambridge, 2006.

\bibitem[Bovier and Kurkova(2003)]{bovier_kurkova3}
A.~Bovier and I.~Kurkova.
\newblock Gibbs measures of derrida’s generalized random energy models and
  genealogies of neveu’s continuous state branching process.
\newblock Preprint 854, WIAS, 2003.

\bibitem[Bovier and Kurkova(2004{\natexlab{a}})]{bovier_kurkova1}
A.~Bovier and I.~Kurkova.
\newblock Derrida's generalised random energy models 1: models with finitely
  many hierarchies.
\newblock \emph{Ann.~Inst.~H.~Poincar{\'e} Probab.~Statist.}, 40\penalty0
  (4):\penalty0 439--480, 2004{\natexlab{a}}.

\bibitem[Bovier and Kurkova(2004{\natexlab{b}})]{bovier_kurkova2}
A.~Bovier and I.~Kurkova.
\newblock Derrida's generalized random energy models 2: models with continuous
  hierarchies.
\newblock \emph{Ann.~Inst.~H.~Poincar{\'e} Probab.~Statist.}, 40\penalty0
  (4):\penalty0 481--495, 2004{\natexlab{b}}.

\bibitem[Bovier and Kurkova(2007)]{bovier_kurkova_review}
A.~Bovier and I.~Kurkova.
\newblock Much ado about {D}errida's {GREM}.
\newblock In \emph{Spin glasses}, volume 1900 of \emph{Lecture Notes in Math.},
  pages 81--115. Springer, Berlin, 2007.

\bibitem[Bovier et~al.(2002)Bovier, Kurkova, and
  L{\"o}we]{bovier_kurkova_loewe}
A.~Bovier, I.~Kurkova, and M.~L{\"o}we.
\newblock Fluctuations of the free energy in the {REM} and the {$p$}-spin {SK}
  models.
\newblock \emph{Ann.~Probab.}, 30\penalty0 (2):\penalty0 605--651, 2002.

\bibitem[Capocaccia et~al.(1987)Capocaccia, Cassandro, and
  Picco]{capocaccia_etal}
D.~Capocaccia, M.~Cassandro, and P.~Picco.
\newblock On the existence of thermodynamics for the generalized random energy
  model.
\newblock \emph{J. Stat. Phys.}, 46\penalty0 (3--4):\penalty0 493--505, 1987.

\bibitem[{Derrida}(1980)]{Derrida_REM3}
B.~{Derrida}.
\newblock {Random-Energy Model: Limit of a Family of Disordered Models}.
\newblock \emph{Phys.~Rev.~Lett.}, 45:\penalty0 79--82, July 1980.

\bibitem[{Derrida}(1981)]{Derrida_REM1}
B.~{Derrida}.
\newblock {Random-energy model: An exactly solvable model of disordered
  systems}.
\newblock \emph{Phys.~Rev.~B}, 24\penalty0 (5):\penalty0 2613--2626, 1981.

\bibitem[Derrida(1985)]{Derrida_GREM}
B.~Derrida.
\newblock A generalization of the random energy model which includes
  correlations between energies.
\newblock \emph{J. Physique Lett.}, 46\penalty0 (9):\penalty0 401--407, 1985.

\bibitem[Derrida(1991)]{Derrida_zeros}
B.~Derrida.
\newblock {The zeroes of the partition function of the random energy model}.
\newblock \emph{Physica A: Stat.~Mech.~Appl.}, 177:\penalty0 31--37, September
  1991.

\bibitem[{Derrida} and {Gardner}(1986{\natexlab{a}})]{Derrida_Gardner1}
B.~{Derrida} and E.~{Gardner}.
\newblock {Solution of the generalised random energy model}.
\newblock \emph{J.~Phys.~C}, 19:\penalty0 2253--2274, May 1986{\natexlab{a}}.

\bibitem[{Derrida} and {Gardner}(1986{\natexlab{b}})]{Derrida_Gardner2}
B.~{Derrida} and E.~{Gardner}.
\newblock {Magnetic properties and the function {$q(x)$} of the generalised
  random-energy model}.
\newblock \emph{J.~Phys.~C}, 19:\penalty0 5783--5798, October
  1986{\natexlab{b}}.

\bibitem[Derrida et~al.(1993)Derrida, Evans, and Speer]{derrida_evans_speer}
B.~Derrida, M.~Evans, and E.~Speer.
\newblock Mean field theory of directed polymers with random complex weights.
\newblock \emph{Comm.~Math.~Phys.}, 156\penalty0 (2):\penalty0 221--244, 1993.

\bibitem[Dobrinevski et~al.(2011)Dobrinevski, Le~Doussal, and
  Wiese]{DobrinevskiLeDoussalWiese2011}
A.~Dobrinevski, P.~Le~Doussal, and K.~J. Wiese.
\newblock Interference in disordered systems: A particle in a complex random
  landscape.
\newblock \emph{Phys.~Rev.~E}, 83\penalty0 (6):\penalty0 061116, 2011.

\bibitem[Hough et~al.(2009)Hough, Krishnapur, Peres, and
  Vir{\'a}g]{peres_etal_book}
B.~Hough, M.~Krishnapur, Y.~Peres, and B.~Vir{\'a}g.
\newblock \emph{Zeros of {G}aussian analytic functions and determinantal point
  processes}, volume~51 of \emph{University Lecture Series}.
\newblock American Mathematical Society, Providence, RI, 2009.

\bibitem[Hudson and Veeh(2001)]{hudson_veeh}
W.~N. Hudson and J.~A. Veeh.
\newblock Complex stable sums of complex stable random variables.
\newblock \emph{J. Multivariate Anal.}, 77\penalty0 (2):\penalty0 229--238,
  2001.

\bibitem[Hudson et~al.(1988)Hudson, Veeh, and Weiner]{hudson_etal}
W.~N. Hudson, J.~A. Veeh, and D.~C. Weiner.
\newblock Moments of distributions attracted to operator-stable laws.
\newblock \emph{J. Multivariate Anal.}, 24\penalty0 (1):\penalty0 1--10, 1988.

\bibitem[It{\^o} and Nisio(1968)]{ito_nisio}
K.~It{\^o} and M.~Nisio.
\newblock On the convergence of sums of independent {B}anach space valued
  random variables.
\newblock \emph{Osaka J.~Math.}, 5:\penalty0 35--48, 1968.

\bibitem[Kabluchko(2011)]{kabluchko_FCLT_geom_BM}
Z.~Kabluchko.
\newblock Functional limit theorems for sums of independent geometric {L}\'evy
  processes.
\newblock \emph{Bernoulli}, 17\penalty0 (3):\penalty0 942--968, 2011.

\bibitem[Kabluchko and Klimovsky(2014)]{kabluchko_klimovsky}
Z.~Kabluchko and A.~Klimovsky.
\newblock Complex random energy model: zeros and fluctuations.
\newblock \emph{Probab. Th. Related Fields}, 158\penalty0 (1--2):\penalty0
  159--196, 2014.

\bibitem[Kallenberg(1997)]{kallenberg_book}
O.~Kallenberg.
\newblock \emph{Foundations of modern probability}.
\newblock Probability and its Applications. Springer--Verlag, New York, 1997.

\bibitem[Lacoin et~al.(2013)Lacoin, Rhodes, and Vargas]{lacoin_rhodes_vargas}
H.~Lacoin, R.~Rhodes, and V.~Vargas.
\newblock {C}omplex {G}aussian {M}ultiplicative {C}haos.
\newblock \emph{Preprint}, 2013.
\newblock Available at http://arxiv.org/abs/1307.6117.

\bibitem[Leadbetter et~al.(1983)Leadbetter, Lindgren, and
  Rootz{\'e}n]{leadbetter_book}
M.~R. Leadbetter, G.~Lindgren, and H.~Rootz{\'e}n.
\newblock \emph{Extremes and related properties of random sequences and
  processes}.
\newblock Springer Series in Statistics. Springer-Verlag, New York, 1983.

\bibitem[Lee and Yang(1952)]{lee_yang2}
T.~D. Lee and C.~N. Yang.
\newblock Statistical theory of equations of state and phase transitions. {II}.
  {L}attice gas and {I}sing model.
\newblock \emph{Physical Rev.~(2)}, 87:\penalty0 410--419, 1952.

\bibitem[Madaule et~al.(2013{\natexlab{a}})Madaule, Rhodes, and
  Vargas]{madaule_rhodes_vargas}
T.~Madaule, R.~Rhodes, and V.~Vargas.
\newblock The glassy phase of complex branching {B}rownian motion.
\newblock \emph{Preprint}, 2013{\natexlab{a}}.
\newblock Available at http://arxiv.org/abs/1310.7775.

\bibitem[Madaule et~al.(2013{\natexlab{b}})Madaule, Rhodes, and
  Vargas]{madaule_rhodes_vargas1}
T.~Madaule, R.~Rhodes, and V.~Vargas.
\newblock Glassy phase and freezing of log-correlated {G}aussian potentials.
\newblock \emph{Preprint}, 2013{\natexlab{b}}.
\newblock Available at http://arxiv.org/abs/1310.5574.

\bibitem[Meerschaert and Scheffler(2001)]{meerschaert_book}
M.~Meerschaert and H.-P. Scheffler.
\newblock \emph{{Limit distributions for sums of independent random vectors.
  Heavy tails in theory and practice.}}
\newblock {Chichester: Wiley}, 2001.

\bibitem[Obuchi and Takahashi(2012)]{obuchi_takahashi}
T.~Obuchi and K.~Takahashi.
\newblock Partition-function zeros of spherical spin glasses and their
  relevance to chaos.
\newblock \emph{J.~Phys.~A}, 45\penalty0 (12):\penalty0 125003, 2012.

\bibitem[Panchenko(2013)]{PanchenkoBook2013}
D.~Panchenko.
\newblock \emph{{The Sherrington-Kirkpatrick model}}.
\newblock Springer, 2013.

\bibitem[Petrov(1995)]{petrov_book}
V.~V. Petrov.
\newblock \emph{Limit theorems of probability theory. Sequences of independent
  random variables}, volume~4 of \emph{Oxford Studies in Probability}.
\newblock Oxford University Press, New York, 1995.

\bibitem[P\'olya and Szeg\H{o}(1998)]{polya_szegoe_book}
G.~P\'olya and G.~Szeg\H{o}.
\newblock \emph{{Problems and theorems in analysis, Volume II.}}
\newblock {Classics in Mathematics. Berlin: Springer.}, 1998.

\bibitem[Resnick(1987)]{resnick_book}
S.~I. Resnick.
\newblock \emph{Extreme values, regular variation, and point processes},
  volume~4 of \emph{Applied Probability. A Series of the Applied Probability
  Trust}.
\newblock Springer-Verlag, New York, 1987.

\bibitem[Rosenthal(1970)]{rosenthal}
H.~P. Rosenthal.
\newblock On the subspaces of {$L^{p}$} {$(p>2)$} spanned by sequences of
  independent random variables.
\newblock \emph{Israel J. Math.}, 8:\penalty0 273--303, 1970.

\bibitem[Ruelle(1987)]{ruelle_cascades}
D.~Ruelle.
\newblock A mathematical reformulation of {D}errida's {REM} and {GREM}.
\newblock \emph{Comm. Math. Phys.}, 108\penalty0 (2):\penalty0 225--239, 1987.

\bibitem[Saakian(2000)]{saakian1}
D.~B. Saakian.
\newblock Random energy model at complex temperatures.
\newblock \emph{Phys. Rev. E}, 61:\penalty0 6132--6135, June 2000.
\newblock \doi{10.1103/PhysRevE.61.6132}.

\bibitem[Saakian(2009)]{saakian2}
D.~B. Saakian.
\newblock Phase structure of string theory and the random energy model.
\newblock \emph{J.~Stat.~Mech.: Theor.~Exp.}, 2009\penalty0 (07):\penalty0
  P07003, 2009.

\bibitem[Samorodnitsky and Taqqu(1994)]{samorodnitsky_taqqu_book}
G.~Samorodnitsky and M.~Taqqu.
\newblock \emph{Stable non-{G}aussian random processes: Stochastic models with
  infinite variance}.
\newblock Stochastic Modeling. Chapman \& Hall, New York, 1994.

\bibitem[Shirai(2012)]{shirai}
T.~Shirai.
\newblock Limit theorems for random analytic functions and their zeros.
\newblock \emph{RIMS K{\^{o}}ky{\^{u}}roku Bessatsu}, 2012.
\newblock To appear.

\bibitem[Sodin and Tsirelson(2004)]{sodin_tsirelson}
M.~Sodin and B.~Tsirelson.
\newblock Random complex zeroes. {I}. {A}symptotic normality.
\newblock \emph{Israel J. Math.}, 144:\penalty0 125--149, 2004.

\bibitem[Takahashi(2011)]{takahashi}
K.~Takahashi.
\newblock Replica analysis of partition-function zeros in spin-glass models.
\newblock \emph{J.~Phys.~A}, 44\penalty0 (23):\penalty0 235001, 2011.

\bibitem[Takahashi and Obuchi(2013)]{obuchi_takahashi_short}
K.~Takahashi and T.~Obuchi.
\newblock Zeros of the partition function and dynamical singularities in
  spin-glass systems.
\newblock In \emph{Proceedings of the International Meeting on ``Inference,
  Computation, and Spin Glasses'', Sapporo, Japan}, 2013.

\bibitem[von Bahr and Esseen(1965)]{von_bahr_esseen}
B.~von Bahr and C.-G. Esseen.
\newblock Inequalities for the {$r$}th absolute moment of a sum of random
  variables, {$1\leq r\leq 2$}.
\newblock \emph{Ann. Math. Statist}, 36:\penalty0 299--303, 1965.

\bibitem[Yang and Lee(1952)]{lee_yang1}
C.~N. Yang and T.~D. Lee.
\newblock Statistical theory of equations of state and phase transitions. {I}.
  {T}heory of condensation.
\newblock \emph{Physical Rev.~(2)}, 87:\penalty0 404--409, 1952.

\end{thebibliography}
\end{document}